\documentclass[11pt, a4paper]{amsart}
\usepackage{amsfonts,amssymb,amsmath, amsthm}
\usepackage{mathrsfs}
\usepackage{lmodern}
\usepackage{qsymbols}
\usepackage{latexsym}
\usepackage[noadjust]{cite}
\usepackage{pdfsync}
\usepackage{bm}
\usepackage{enumitem}

\newtheorem{thm}{Theorem}[section]
\newtheorem{lem}[thm]{Lemma}
\newtheorem{prop}[thm]{Proposition}
\newtheorem{cor}[thm]{Corollary}
\theoremstyle{definition}
\newtheorem{defn}[thm]{Definition}
\newtheorem{rem}[thm]{Remark}
\numberwithin{equation}{section}
\usepackage[plainpages=false,pdfpagelabels,citecolor=red,bookmarksopen,bookmarksdepth=3]{hyperref}
\usepackage{xcolor}
\hypersetup{
 colorlinks,
 linkcolor={cyan!90!black},
 citecolor={magenta},
 urlcolor={green!40!black}
} 
\newcommand{\R}{\mathbb{R}}
\newcommand{\IC}{\mathbb{C}}
\newcommand{\IN}{\mathbb{N}}
\newcommand{\IP}{\mathbb{\pi}}
\newcommand{\IZ}{\mathbb{Z}}

\newcommand{\cH}{\mathcal{H}}

\newcommand{\mH}{\mathbb{L}^2}
\newcommand{\mHp}{\mathbb{L}^p}
\newcommand{\cF}{\mathcal{F}}
\newcommand{\cX}{\mathcal{X}}
\newcommand{\cY}{\mathcal{Y}}
\newcommand{\mE}{{\mathcal E}}
\newcommand{\mD}{{\mathcal D}} 
\newcommand{\mS}{{\mathcal S}} 
\newcommand{\loc}{\operatorname{loc}}
\renewcommand{\L}{\operatorname{L}} 
\newcommand{\Lloc}{\L_{\operatorname{loc}}} 
\newcommand{\C}{\operatorname{C}} 
\renewcommand{\H}{\operatorname{H}} 
\newcommand{\W}{\operatorname{W}}
\newcommand{\Hdot}{\dot{\H}\protect{\vphantom{H}}} 
\let\SS\S	
\renewcommand{\S}{\mathrm{S}} 

\newcommand{\E}{\mathsf{E}} 
\newcommand{\reu}{{\mathbb{R}^{n+1}_+}}
\newcommand{\ree}{{\mathbb{R}^{n+1}}}
\renewcommand{\P}{P}
\newcommand{\M}{M}
\newcommand{\N}{N}
\newcommand{\wM}{\widetilde \M}
\newcommand{\wh}{\widehat}
\newcommand{\gradx}{\nabla_x}
\newcommand{\gradlamx}{\nabla_{\lambda,x}}
\renewcommand{\div}{\operatorname{div}}
\newcommand{\curl}{\operatorname{curl}}
\newcommand{\curlx}{\curl_{x}}
\newcommand{\divx}{\div_x}
\newcommand{\dnuA}{\partial_{\nu_A}} 
\newcommand{\dnuAstar}{\partial_{\nu_{A^*}}} 
\newcommand{\pcg}{D_{{\! A}}}
\newcommand{\pcgb}{\widetilde D_{{\! A^*}}}
\newcommand{\pcgt}{\widetilde D_{{\! A}}}
\newcommand{\BVP}{BV\!P}
\newcommand{\dhalf}{D_t^{1/2}} 
\newcommand{\dalpha}{D_t^{\alpha}} 
\newcommand{\donefour}{D_t^{1/4}} 
\newcommand{\ihalf}{I_t^{1/2}} 
\newcommand{\ipar}{I_{t,x}^{1/2}} 
\newcommand{\HT}{H_t} 
\newcommand{\pe}{\perp}
\newcommand{\pa}{\parallel}
\newcommand{\te}{\theta}
\newcommand{\Pfull}{\begin{bmatrix} 0 & \divx & -\dhalf \\ -\gradx & 0 & 0 \\ -\HT \dhalf & 0 & 0 \end{bmatrix}} 
\newcommand{\Mfull}{\begin{bmatrix} {\hat A}_{\pe \pe} & {\hat A}_{\pe \pa} & \vphantom{\dhalf}0 \\ {\hat A}_{\pa \pe} & {\hat A}_{\pa \pa} & 0 \\ \vphantom{\dhalf}0& 0& 1 \end{bmatrix}} 
\newcommand{\NT}{\widetilde{N}_*} 
\newcommand{\NTone}{\widetilde{N}_{*,1}} 
\newcommand{\LQI}{{\Lambda \times Q \times I}}
\newcommand{\LQIt}{\wt{\Lambda} \times \wt{Q} \times \wt{I}}
\newcommand{\fLQIk}[2]{{#1\Lambda \times #1Q \times I_#2}}
\newcommand{\e}{\mathrm{e}} 
\let\ii\i
\renewcommand{\i}{\mathrm{i}} 
\renewcommand{\d}{\, \mathrm{d}} 
\newcommand{\eps}{\varepsilon} 
\renewcommand\Re{\operatorname{Re}}

\newcommand{\Lop}{\mathcal{L}} 
\newcommand{\cl}[1]{\overline{#1}} 
\DeclareMathOperator{\supp}{supp} 
\DeclareMathOperator{\id}{1} 
\DeclareMathOperator{\ran}{\mathsf{R}} 
\DeclareMathOperator{\nul}{\mathsf{N}} 
\DeclareMathOperator{\dom}{\mathsf{D}} 
\DeclareMathOperator{\Max}{\mathcal{M}} 
\DeclareMathOperator{\Hp}{\mathsf{H}^+}
\DeclareMathOperator{\Hm}{\mathsf{H}^{--}}
\DeclareMathOperator{\Hpm}{\mathsf{H}^\pm}
\newcommand{\dual}[2]{\langle #1,#2 \rangle}
\newcommand{\clos}[1]{\overline{#1}}
\newcommand{\dyadic}{\square}
\newcommand{\sgn}{\operatorname{sgn}}
\newcommand{\wt}{\widetilde}
\newcommand{\pd}{\partial}
\newcommand{\qe}[1]{\int_0^\infty\|#1\|_{2}^2\,\frac{\d\lambda}\lambda}
\def\Xint#1{\mathchoice
{\XXint\displaystyle\textstyle{#1}}%
{\XXint\textstyle\scriptstyle{#1}}%
{\XXint\scriptstyle\scriptscriptstyle{#1}}%
{\XXint\scriptscriptstyle%
\scriptscriptstyle{#1}}%
\!\int}
\def\XXint#1#2#3{{\setbox0=\hbox{$#1{#2#3}{%
\int}$ }
\vcenter{\hbox{$#2#3$ }}\kern-.6\wd0}}
\def\barint{\,\Xint -} 
\def\bariint{\barint_{} \kern-.4em \barint}
\def\bariiint{\bariint_{} \kern-.4em \barint}
\renewcommand{\iint}{\int_{}\kern-.34em \int} 
\renewcommand{\iiint}{\iint_{}\kern-.34em \int} 

\title[BVP for parabolic systems via a first order approach]{$\L^2$ well-posedness of boundary value problems \\for parabolic systems with measurable coefficients
}
\author{Pascal Auscher}
\address{Universit\'{e} Paris-Saclay, CNRS, Laboratoire de Math\'{e}matiques d'Orsay, 91405, Orsay, France
\vspace{3pt}  \newline \vspace{0pt} \hspace{7pt} Laboratoire Ami\'{e}nois de Math\'{e}matiques Fondamentales et Appliqu\'{e}es, UMR 7352 du CNRS, Universit\'{e} de Picardie-Jules Verne, 80039 Amiens, France}
\email{pascal.auscher@universite-paris-saclay.fr}
\author{Moritz Egert}
\address{Universit\'{e} Paris-Saclay, CNRS, Laboratoire de Math\'{e}matiques d'Orsay, 91405, Orsay, France}
\email{moritz.egert@universite-paris-saclay.fr}

\author{Kaj Nystr\"om}
\address{Department of Mathematics, Uppsala University, S-751 06 Uppsala, Sweden}
\email{kaj.nystrom@math.uu.se}
\thanks{P.A.\ and M.E.\ were supported by the ANR project ``Harmonic Analysis at its Boundaries'', ANR-12-BS01-0013. P.A.\ was also supported by the NSF grant no.~DMS-1440140, while being in residence at the MSRI in Berkeley, California, during the spring 2017 semester. M.E.\ was also supported by a public grant as part of the FMJH and thanks MSRI for hospitality. K.N.\ was supported by a grant from the G{\"o}ran Gustafsson Foundation for Research in Natural Sciences and Medicine. K.N.\ thanks Universit\'{e} Paris-Sud for support and hospitality.}
\subjclass[2010]{Primary: 35K40, 35K46, 42B37. Secondary: 26A33, 42B25, 47A60, 47D06}
\keywords{Second order parabolic systems, parabolic Kato square root estimate, parabolic Dirac operator, Dirichlet and Neumann problems, (non-tangential) maximal functions, square function estimates, \emph{a priori} representations, boundary layer operators, half-order derivative}
\dedicatory{In memory of Alan McIntosh}
\date{\today}
\begin{document}

\begin{abstract} We prove the first positive results concerning boundary value problems  in the upper half-space of second order parabolic systems  only assuming  measurability and some transversal regularity in the coefficients of the  elliptic part. To do so, we introduce and develop a first order strategy by means of a parabolic Dirac operator at the boundary to obtain, in particular, Green's representation for solutions in natural classes involving square functions and non-tangential maximal functions, well-posedness results with data in $\L^2$-Sobolev spaces together with invertibility of layer potentials, and perturbation results. In the way, we solve the Kato square root problem for parabolic operators with coefficients of the elliptic part depending measurably on all variables. The major new challenge, compared to the earlier results by one of us under time and transversally independence of the coefficients, is to handle non-local half-order derivatives in time which are unavoidable in our situation.   
\end{abstract}
\maketitle
\tableofcontents
\section{Introduction}
\label{sec:introduction}

In this paper we study boundary value problems for parabolic equations of type
\begin{align}
 \label{eq1}
\Lop u:= \partial_t u -\div_{X} A(X,t)\nabla_{X} u = 0.
\end{align}
Here, $X=(\lambda,x)\in \reu:= (0,\infty)\times \R^n$ and $t\in \R$. We assume that the matrix $A$ is measurable, strongly elliptic, and possibly depends on all variables. Parabolic systems will be included in our considerations as well. The focus on the \emph{upper} parabolic half-space $\R^{n+2}_+$ here is of course arbitrary and we could as well have decided to work in the lower one.

We stress that with mere measurable dependence in time there are no positive results on solvability of boundary value problems in this context up to now. Indeed, most studies concerning second order parabolic boundary value problems have focused on the heat equation in (parabolic) Lipschitz-type domains and some generalisations. For the heat equation, we mention \cite{B,B1} devoted to time-independent Lipschitz cylinders and the important achievements in \cite{H,Hofmann-Lewis,LM} for time dependent Lipschitz-type domains. As for generalisations, \cite{HL1} studies the Dirichlet problem for parabolic equations with singular drift terms (containing the ones coming from an elaborate change of variables from the heat equation) with coefficients having some smoothness and smallness in the sense of a Carleson condition. See also \cite{DH} and  \cite{DPP}.

In a series of recent papers \cite{N1, CNS, N2}, the solvability for Dirichlet, regularity and Neumann problems with $\L^2$ data was established for the class of parabolic equations \eqref{eq1} under the assumptions that the elliptic part is independent of both the time variable $t$ and the variable $\lambda$ transverse to the boundary, and that the elliptic part has either constant (complex) coefficients, real symmetric coefficients, or small perturbations thereof. This is proved by first studying fundamental solutions and proving square function estimates and non-tangential estimates for the single layer potential along the lines of the approach for elliptic equations in \cite{AAAHK}. In the case of real symmetric coefficients, solvability is then established using a Rellich-type argument.

An important step in \cite{N1} -- revisiting an earlier idea of \cite{Kaplan, Hofmann-Lewis} -- is the introduction of a reinforced notion of weak solution by means of a sesquilinear form associated with \eqref{eq1}, showing its coercivity and proving a fundamental square function estimate. In fact, what is implicitly shown in \cite{N1} by adapting the proof of the Kato conjecture for elliptic operators \cite{AHLMcT}, is that the operator $\pd_{t} - \div_{x}A(x)\nabla_{x}$ is maximal-accretive in $\L^2(\ree)$ and that the domain of its square root is that of the associated form. Note that this is for one dimension lower than $\eqref{eq1}$. With this at hand, \cite{CNS} proves a number of technical estimates for the single layer potential for $\Lop$ based on the DeGiorgi-Nash-Moser condition. Finally, \cite{N2} addresses solvability and uniqueness issues by proving invertibility of the layer potentials. However, the methods there are limited to time independent coefficients and removing this constraint  requires new ideas. 

We owe to   Alan McIntosh  the  insight that boundary value problems involving accretive sesquilinear forms could be addressed via a first order system of Cauchy-Riemann type. For elliptic equations this approach was pioneered in \cite{AAH} and carried out systematically in \cite{AAM} using a simpler setup. It relies on three main properties of the corresponding perturbed Dirac operator at the boundary: bisectoriality, off-diagonal estimates for its resolvents and a bounded holomorphic functional calculus on $\L^2$, which in this case has been obtained by a $T(b)$ argument based on a remarkable elaboration on the solution of the Kato conjecture in \cite{AKMc}.
Non-tangential maximal estimates to precise the boundary behaviour were also obtained. Uniqueness via a semigroup  representation,  later identified to layer potential integrals in the case of  real elliptic equations in \cite{R1}, are obtained in  \cite{AA}.

In this paper we   implement  this strategy for boundary value problems of parabolic equations. The outcome of our efforts is the possibility to address arbitrary parabolic equations (and systems) as in \eqref{eq1} with coefficients depending measurably on  time and on the transverse variable with additional transversal regularity,  the removal of time independence being the substantial step.

In the spirit of a first order approach for parabolic operators, we mention two articles  \cite{CKS, CSV} proposing a factorisation of the heat operator via a (non-homogeneous) Dirac operator valued in some Clifford algebra that involves only first and zero order partial derivatives. But this does not seem to meet our needs here:
First, we have non-smooth coefficients and, second,  half-order derivatives in time turn out to be part of the data for the boundary value problems. Hence, it is natural that our Dirac operator comes with such derivatives, too.

To proceed, we have to overcome  serious difficulties compared with the elliptic counterpart since half-order time-derivatives appearing in the Dirac operator are non-local and have poor decay properties. However, we are still able to obtain just enough decay in the off-diagonal estimates of its resolvents to obtain through a $T(b)$ argument the square function estimate that opens the door to all our further results concerning boundary value problems with $\L^2$ data. Additionally, we prove important non-tangential maximal estimates via some new reverse H\"older inequalities for reinforced weak solutions.

Concerning the Dirichlet problem, the first order approach only comprises uniqueness in the sense of $\L^2$ convergence with square function control. However, in the cases where we can solve, we also prove uniqueness in the classical sense of non-tangential convergence at the boundary with non-tangential maximal control, which is the largest possible class. This includes that the ``unique'' solution of the  Dirichlet problem with non-tangential maximal control in $\L^2$ also enjoys the square function estimates. Proof of uniqueness is done via new arguments,  not requiring any form of regularity of solutions such as the DeGiorgi-Nash-Moser estimates, and hence it does apply to parabolic systems in particular.

As a particular outcome of independent interest, we obtain that the parabolic operator $\pd_{t} - \div_{x}A(x,t)\nabla_{x}$ can be defined as a maximal-accretive operator in $\L^2(\ree)$ and that it satisfies a Kato square root estimate -- its square root has domain equal to that of the defining form. Note that we allow measurable $t$-dependence on the coefficients in contrast to \cite{N1}, which makes a huge difference.
This has a worth-mentioning consequence. In \cite{AEN2}, we use the solution of the parabolic Kato problem to prove that the $\L^p$ Dirichlet problem corresponding to \eqref{eq1} is well-posed if $A = A(x,t)$ is real valued but possibly non-symmetric, measurable, and $p \in (1,\infty)$ is sufficiently large, which is the best possible result in general. 

In the next section we will give a comprehensive presentation of our results, postponing details for later. It will come with the introduction of some necessary notation, but we try to make the presentation as  fluent as possible to give a general overview. Also the progression does not necessarily follow the order in which things are proved. The subsequent Sections~\ref{sec:Homogeneous Sobolev spaces} - \ref{sec:Dirichlet} are devoted to the proofs of our results. In the final Section~\ref{sec:miscellani} we formulate further generalisations of our results and state additional  open problems. Let us mention two of them here. One concerns  boundary value problems for \eqref{eq1} in time dependent Lipschitz-type domains $\lambda>\varphi(x,t)$. While Lipschitz changes of variables preserve the class of elliptic equations $-\div_X A\nabla_X u=0$, this is not the case for equations $\partial_tu  -\div_X A\nabla_X u=0$.  Already the simple change of variables $(\lambda,x,t)\mapsto (\lambda-\varphi(x,t),x,t)$ requires that $\varphi$ Lipschitz in both $x$ and $t$ and brings in a drift term. But the natural regularity on  $\varphi$  is some half-order regularity  in $t$  and requires more elaborate changes of variables such as the ones used in the works mentioned before. This results in creating drift terms and brings in leading   coefficients to which  existing results including ours do not apply. 
 Another one  would be to allow for data in other spaces which, if one uses the first order approach, requires to develop a Hardy space theory associated with parabolic Dirac operators in the first  place. 
\section{Main results}
\label{sec:main results}

\subsection{The coefficients}

For the time being, $A=A(X,t)=\{A_{i,j}(X,t)\}_{i,j=0}^{n}$ is assumed to be an $(n+1)\times (n+1)$-dimensional matrix with complex coefficients satisfying the uniform ellipticity condition
\begin{equation}
\label{eq2}
 \kappa|\xi|^2\leq \Re (A(X,t) \xi \cdot \cl{\xi}), \qquad
 |A(X,t)\xi\cdot\zeta|\leq C|\xi||\zeta|,
\end{equation}
for some $\kappa, C \in (0,\infty)$, which we refer to as the ellipticity constants of $A$, and for all $\xi,\zeta \in \IC^{n+1}$, $(X,t) \in \R^{n+2}_+$. Here $u\cdot v=u_0v_0+...+u_{n}v_{n}$, $\bar u$ denotes the complex conjugate of $u$ and $u\cdot \cl{v}$ is the standard inner product on $\IC^{n+1}$. We use $X=(x_{0},x)$ and $x=(x_{1}, \ldots, x_{n})$. Most often, we  specialize the variable transverse to the boundary by setting $\lambda = x_0$.

\subsection{Reinforced weak solutions}
\label{sec:rein}

If $\Omega$ is an open subset of $\ree$, we let $\H^1(\Omega)=\W^{1,2}(\Omega)= \W^{1,2}(\Omega;\IC)$ be the standard Sobolev space of complex valued functions $v$ defined on $\Omega$, such that $v$ and $\nabla v$ are in $\L^{2}(\Omega;\IC)$ and $\L^{2}(\Omega;\IC^n)$, respectively. A subscripted `$\loc$' will indicate that these conditions hold locally.
We shall say that $u$ is a \emph{reinforced weak solution} of \eqref{eq1} on $\reu\times \R$ if
\begin{align*}
 u\in \dot {\E}_{\loc}:= \Hdot^{1/2}(\R; \L^2_{\loc}(\reu)) \cap \Lloc^2(\R; \W^{1,2}_{\loc}(\reu))
\end{align*}
and if for all $\phi\in \C_0^\infty(\R^{n+2}_+)$,
\begin{eqnarray*}
\int_0^\infty\iint_{\ree}
 A\nabla_{\lambda,x} u\cdot\overline{\nabla_{\lambda,x} \phi}+ \HT\dhalf u\cdot \overline{\dhalf \phi} \d x\d t\d\lambda = 0 .
\end{eqnarray*}
Here, $\dhalf$ is the half-order derivative and $\HT$ the Hilbert transform with respect to the $t$ variable, designed in such a way that $\partial_{t}= \dhalf \HT \dhalf$. The space $\Hdot^{1/2}(\R)$ is the homogeneous Sobolev space of order 1/2. In Section~\ref{sec:Homogeneous Sobolev spaces} we shall review properties of this space, which we define as a subspace of $\Lloc^2(\R)$ in order to deal with proper distributions. Still our definition is equivalent to other common ones: it is the completion of $\C_0^\infty(\R)$ for the norm $\|\dhalf \cdot\|_{2}$ and, modulo constants, it embeds into the space $\mS'(\R)/\IC$ of tempered distributions modulo constants.

At this point we remark that for any $u\in \Hdot^{1/2}(\R)$ and $\phi\in \C_0^\infty(\R)$ the formula
\begin{align*}
 \int_{\R} \HT\dhalf u\cdot \overline{\dhalf\phi}\d t = - \int_{\R} u \cdot \cl{\partial_{t}\phi} \d t
\end{align*}
holds, where on the right-hand side we use the duality form between $\Hdot^{1/2}(\R)$ and its dual $\Hdot^{-1/2}(\R)$ extending the complex inner product of $\L^2(\R)$. It follows that a reinforced weak solution is a weak solution in the usual sense on $\reu\times \R$: it satisfies
$u\in \Lloc^2(\R; \W^{1,2}_{\loc}(\reu))$ and for all $\phi\in \C_0^\infty(\R^{n+2}_+)$,
\begin{align*}
  \int_\R\iint_{\reu} A\nabla_{\lambda,x} u\cdot\overline{\nabla_{\lambda,x} \phi} \d x \d\lambda \d t - \int_{\R} \iint_{\reu} u \cdot \cl{\partial_{t}\phi} \d x \d\lambda \d t=0.
\end{align*}
This implies $\pd_{t}u\in \Lloc^2(\R; \W^{-1,2}_{\loc}(\reu))$. Conversely, any  weak solution $u$ in    $ \Hdot^{1/2}(\R; \L^2_{\loc}(\reu))$ is a reinforced weak solution.

\subsection{Energy solutions}
\label{sec:energy}

As a first illustration of this concept of interpreting \eqref{eq1}, let us quickly explain how to obtain reinforced weak solutions using the form method. This seems to be surprising at first sight and relies on the fundamental revelation of coercivity for the parabolic operator $\Lop$ explored first in~\cite{Kaplan} and later on for instance in \cite{Hofmann-Lewis, N1, CNS, N2}. It will also allow us to touch some of the paper's central themes. For the moment, time or even transversal dependency of the coefficients will not be any obstacle.

We are looking for reinforced weak solutions that belong to the \emph{energy class} $\dot \E$ of all $v \in \dot{\E}_{\loc}$ for which
\begin{align*}
 \|v\|_{\dot \E} := \bigg(\|\gradlamx v\|_{\L^2(\R^{n+2}_+)}^2 + \|\HT \dhalf v\|_{\L^2(\R^{n+2}_+)}^2 \bigg)^{1/2} < \infty.
\end{align*}
Consequently, these are called \emph{energy solutions}. When considered modulo constants, $\dot \E $ is a Hilbert space and it is in fact the closure of $\C_0^\infty\!\big(\,\cl{\R^{n+2}_+}\,\big)$ for the homogeneous norm $\|\cdot\|_{\dot \E}$. Any $v \in \dot \E$ can be defined as a uniformly continuous function on $[0,\infty)$ with values in the \emph{homogeneous parabolic Sobolev space} $\Hdot^{1/4}_{\pd_{t} - \Delta_x}$. Here, $\Hdot^{s}_{\pm \pd_{t} - \Delta_x}$ is defined as the closure of Schwartz functions $v \in \mS(\ree)$ with Fourier support away from the origin for the norm $\|\cF^{-1}((|\xi|^2 \pm \i \tau)^s \cF v)\|_2$. This yields a space of tempered distributions if $s<0$ and of tempered distributions modulo constants in $\Lloc^2(\ree)$ if $0 < s \leq 1/2$. Of course the choice of the sign in front of $\i \tau$ defines the same space up to equivalent norms. We implicitly assume the branch cut for $z^s$ on $(-\infty, 0]$ but any other possible choice yields the same space with an isometric norm. Conversely, any $g \in \Hdot^{1/4}_{\pd_{t} - \Delta_x}$ can be extended to a function $v \in  \dot \E$ with trace $v|_{\lambda = 0} = g$. We will include proofs of these statements in Section~\ref{sec:Homogeneous Sobolev spaces several}.

By an energy solution to \eqref{eq1} with Neumann boundary data $\dnuA u|_{\lambda = 0} = f$ we mean a function $u \in \dot\E$ such that for all $v \in \dot \E$,
\begin{align}
\label{eq:variational}
 \iiint_{\R^{n+2}_+} A \gradlamx u \cdot \cl{\gradlamx v} + \HT \dhalf u \cdot \cl{\dhalf v} \d \lambda \d x \d t = - \langle f, v|_{\lambda = 0} \rangle,
\end{align}
where $\langle \cdot \,, \cdot \rangle$ denotes the pairing of $\Hdot^{1/4}_{\pd_{t} - \Delta_x}$ with its dual $\Hdot^{-1/4}_{-\pd_{t} - \Delta_x}$ extending the inner product on $\L^2(\ree)$ and
\begin{align}
\label{eq:dnuA}
 \dnuA u(\lambda, x,t) := [1,0, \ldots, 0] \cdot (A\nabla_{\lambda,x}u)(\lambda, x,t)
\end{align}
is the \emph{conormal derivative}, inwardly oriented on the upper half-space. In particular, $u$ is a reinforced weak solution to \eqref{eq1} and the Neumann boundary data $f \in \Hdot^{-1/4}_{\pd_{t} - \Delta_x}$ is attained in the weak sense prescribed by \eqref{eq:variational}. Similarly, by an energy solution to \eqref{eq1} with Dirichlet boundary datum $u|_{\lambda = 0} = f \in \Hdot^{1/4}_{\pd_{t} - \Delta_x}$ we mean a function $u \in \dot\E$ such that for all $v \in \dot \E_0$, the subspace of $\dot \E$ with zero boundary trace,
\begin{align*}
 \iiint_{\R^{n+2}_+} A \gradlamx u \cdot \cl{\gradlamx v} + \HT \dhalf u \cdot \cl{\dhalf v} \d \lambda \d x \d t = 0.
\end{align*}
The key to solving these problems is the introduction of the modified sesquilinear form
\begin{align*}
 a_\delta(u,v) := \iiint_{\R^{n+2}_+} A \gradlamx u \cdot \cl{\gradlamx (1-\delta \HT) v} + \HT \dhalf u \cdot \cl{\dhalf (1-\delta \HT) v} \d \lambda \d x \d t,
\end{align*}
where $\delta$ is a  real number yet to be chosen. The Hilbert transform $\HT$ is a skew-symmetric isometric operator with inverse $-\HT$ on both $\dot \E$ and $\Hdot^{1/4}_{\pd_{t} - \Delta_x}$.  Hence, $1-\delta \HT$ is invertible on these spaces for any $\delta \in \R$. The upshot is that if we fix $\delta>0$  small enough, then $a_\delta$ is coercive on $\dot \E$ since
\begin{equation}\label{eq:coer}
\Re a_\delta(u,u) \ge (\kappa-C\delta )\|\gradlamx u\|_2^2 + \delta \|\HT \dhalf u \|_2^2
\end{equation} where $\kappa,C$ are the constants in \eqref{eq2}. Hence, solving the Neumann problem with data $f$ means finding $u \in \dot\E$ such that
\begin{align*}
 a_\delta(u,v) = - \langle f, (1-\delta \HT)v|_{\lambda = 0} \rangle \qquad (v \in \dot \E)
\end{align*}
and the Lax-Milgram lemma applied to $a_{\delta }$ on $\dot \E$ yields a unique such $u$. We remark that this exactly means \eqref{eq:variational} upon replacing $(1-\delta \HT)v$ by $v$, so that $\dnuA u|_{\lambda = 0}=f$ holds in the respective sense. For the Dirichlet problem we take an extension $w \in \dot \E$ of the data $f$ and apply the Lax-Milgram lemma to $a_\delta$ on $\dot \E_0$ to obtain some $u \in \dot \E_0$ such that \begin{align*}
 a_\delta(u,v) = - a_{\delta}(w,v) \qquad (v \in \dot \E_0).                                                                                                                                                                                                                                                                                                                                                                                                                                                                               \end{align*}
Hence, $u + w$ is an energy solution with data $f$. Would there exist another solution $v$, then $a_\delta(u+w-v,u+w-v) = 0$ and hence $\|u +w - v \|_{\dot \E} = 0$ by coercivity.

We shall rephrase these observations by saying the the Dirichlet and Neumann problems associated with \eqref{eq1} are \emph{well-posed} for the energy class. In the case of $\lambda$-independent coefficients we will rediscover the energy solutions in Section~\ref{sec:specialcases} within a much broader context.

\subsection{The correspondence to a first order system}
\label{sec:correspondence}

Given a reinforced weak solution $u$ as in Section~\ref{sec:rein}, we create a vector with $n+2$ components, which we call the \emph{(parabolic) conormal differential} of $u$,
\begin{equation*}
 \pcg u(\lambda, x,t) := \begin{bmatrix} \dnuA u(\lambda, x,t) \\ \gradx u(\lambda, x,t) \\ \HT\dhalf u(\lambda, x,t) \end{bmatrix},
\end{equation*}
where the inwardly oriented conormal derivative was defined in \eqref{eq:dnuA}. Note that the last component of $\pcg u$ contains the half-order time derivative of $u$, which exists by assumption. Whether or not we choose to include the Hilbert transform in the definition of $\pcg u$ is not important at this stage: it only affects the algebra of the representation but not its analysis. Note also that in a pointwise fashion,
\begin{align*}
 |\pcg u|^2\sim |\nabla_{\lambda,x}u|^2+|\HT\dhalf u|^2.
\end{align*}
Throughout the paper we will represent vectors $\xi \in \IC^{n+2}$ as
\begin{equation*}
 \xi= \begin{bmatrix} \xi_\pe \\ \xi_\pa \\ \xi_\te \end{bmatrix},
\end{equation*}
where the \emph{normal} (or \emph{perpendicular}) part $\xi_\pe$ is scalar valued, the \emph{tangential} part $\xi_{\pa}$ is valued in $\IC^n$ and the \emph{time} part $\xi_{\te}$ is again scalar valued.

If $u$ is a reinforced weak solution to \eqref{eq1}, then the vector valued function $F := \pcg u$ belongs to the space $\Lloc^2(\R; \Lloc^2(\reu;\IC^{n+1})) \times \L^2(\R; \Lloc^2(\reu; \IC))$, the global square integrability for the $\te$-component being with respect to the $t$ variable, and it has the additional structure
 \begin{equation}
\label{eq:comp}
\curlx F_{\pa}=0, \quad \gradx F_{\te} = \HT\dhalf F_{\pa},
\end{equation}
when computed in the sense of distributions on $\reu\times \R$. In fact, the first equation is contained in the second one.

To eventually address boundary value problems, we will assume global square integrability with respect to $(x,t)$ for all components of $F$. Hence, with the help of Fubini's theorem we put $\pcg u$ in the space $\Lloc^2(\R_{+}; \L^2(\ree;\IC^{n+2}))$, where the local square integrability now is with respect to the $\lambda$ variable. Throughout the paper we let  $\mH$ be $\L^2(\ree;\IC^{n+2})$  and $\cH_{\loc}$ be the subspace of $\Lloc^2(\R_{+}; \mH)$ defined by the compatibility conditions \eqref{eq:comp}.

Next, we split the coefficient matrix $A$ as
\begin{align}
\label{eq:A}
 A(\lambda,x,t)= \begin{bmatrix} A_{\pe\pe}(\lambda,x,t) & A_{\pe\pa}(\lambda,x,t)\\ A_{\pa\pe}(\lambda,x,t) & A_{\pa\pa}(\lambda,x,t) \end{bmatrix}
\end{align}
and we recall from \cite{AA} that there is a purely algebraic transformation on the bounded, uniformly elliptic matrix functions that will eventually allow us to write \eqref{eq1} as system for the unknown $\pcg u$.

\begin{prop}[{\cite[Prop.~4.1, p.~68]{AA}}]
\label{prop:divformasODE}
The pointwise transformation
\begin{align*}
 A\mapsto \hat A:= \begin{bmatrix}1& 0\\ A_{\pa\pe} & A_{\pa\pa} \end{bmatrix} \begin{bmatrix} A_{\pe\pe} & A_{\pe\pa}\\ 0&1 \end{bmatrix}^{-1}=\begin{bmatrix} A_{\pe\pe}^{-1} & -A_{\pe\pe}^{-1} A_{\pe\pa} \\
  A_{\pa\pe}A_{\pe\pe}^{-1} & A_{\pa\pa}-A_{\pa\pe}A_{\pe\pe}^{-1}A_{\pe\pa} \end{bmatrix}
\end{align*}
is a self-inverse bijective transformation on the set of bounded matrices which are uniformly elliptic.
\end{prop}

We introduce the operators
\begin{equation}
\label{eq:DB}
\P:= \Pfull, \qquad \M:= \Mfull.
\end{equation}
For each fixed $\lambda>0$, the operator $\M$ is a multiplication operator on $\mH$. 
The operator $\P$ is independent of $\lambda$ and defined as an unbounded operator in $\mH$ 
 with maximal domain. The adjoint of $P$ is
\begin{align}\label{eq:DB+}
\P^*=\begin{bmatrix} 0 & \divx & \HT\dhalf \\ -\gradx & 0 & 0 \\ - \dhalf & 0 & 0 \end{bmatrix}.
\end{align}
In Section~\ref{sec:proof of correspondence} we will prove the following theorem, establishing the connection between reinforced weak solutions $u$ to \eqref{eq1} and a first order differential equation.

\begin{thm}
\label{thm:correspondence}
If $u$ is a reinforced weak solution $u$ to \eqref{eq1} and $F:= \pcg u\in \cH_{\loc}$, then
\begin{equation}
\label{eq:diffeq}
\iiint_{\R^{n+2}_{+}} F\cdot \overline{ {\pd_{\lambda} \phi}} \d\lambda \d x \d t = \iiint_{\R^{n+2}_{+}} {MF}\cdot { \overline{P^*\phi}} \d\lambda \d x \d t
\end{equation}
for all $\phi\in \C_{0}^\infty(\R^{n+2}_+; \IC^{n+2})$. Conversely, to any $F\in \cH_{\loc}$ satisfying \eqref{eq:diffeq} for all $\phi$, there exists a reinforced weak solution $u$ to \eqref{eq1}, unique up to a constant, such that $F=\pcg u$.
\end{thm}
In other words, up to additive constants one can construct all reinforced weak solutions $u$ to \eqref{eq1} with the property $\pcg u \in \Lloc^2(\R_{+}; \L^2(\ree;\IC^{n+2}))$ by solving the differential equation
\begin{equation}
\label{eq:diffeq2}
\pd_{\lambda} F + \P\M F=0
\end{equation}
in the distributional sense \eqref{eq:diffeq} in the space $\cH_{\loc}$.

The equation \eqref{eq:diffeq2} depends on the operator $\P\M$, hence on the choice of $A$ to represent $\Lop$. One can show that any choice of $A$ leads to the same conclusions as far as Dirichlet problems are concerned: solvability and uniqueness, respectively, hold for all choices or none. On the other hand, Neumann problems depend on the choice of $A$ in the formulation and in the conclusion. It could well be that one Neumann problem is solvable or well-posed but not all of them.

\subsection{The parabolic Dirac operator for transversally independent equations}
\label{sec:dirac}

The first case to study is that of equations with coefficients independent of the transverse variable $\lambda$. In this case also the entries of $M$ are independent of $\lambda$. Thus, the differential equation \eqref{eq:diffeq2} becomes autonomous and can be solved via semigroup techniques, provided the semigroup is well-defined. This requires that $\P\M$ has a \emph{bounded holomorphic functional calculus}.

Let us quickly recall that an operator $T$ in a Hilbert space is \emph{bisectorial} of \emph{angle} $\omega \in (0, \pi/2)$ if its spectrum is contained in the closure of the open double sector
\begin{align*}
 \S_\omega := \{z \in \IC: |\arg z| < \omega \text{ or } |\arg z - \pi| < \omega \}
\end{align*}
and if for each $\mu \in (\omega, \pi/2)$ the map $z \mapsto z(z - T)^{-1}$ is uniformly bounded on $\IC \setminus \S_\mu$. Throughout this paper we require basic knowledge of such operators and their functional calculus allowing to define $b(T)$ for suitable holomorphic functions $b$ on $\S_\mu$. A reader without a background in this field will find all the necessary results and further information in various comprehensive treatments including \cite{Haase, Mc, EigeneDiss}.

As can be seen from \eqref{eq:DB} and \eqref{eq:DB+}, the operator $\P$ contains the non-local operators $\dhalf$ and $\HT\dhalf$ and it is \emph{not} self-adjoint. These are the two major differences compared to the corresponding construction of the Dirac operator for elliptic equations. Note that $\P$ arises from the heat equation $\pd_{t}u-\Delta_x u=0$. The non-selfadjointness of $\P$ is reflected in the fact that the adjoint equation is not the heat equation itself but the backward heat equation $-\pd_{t}u -\Delta_x u=0$. Hence, we (unfortunately) cannot just quote results from the literature devoted to elliptic equations without investigating the arguments. Nevertheless, we (fortunately) can prove the similar results yielding a harvest of consequences from abstract reasoning regardless of the precise definition of $\P$ and $\M$.

We are ready to state a fundamental result of the paper:

\begin{thm}
\label{thm:bhfc} The operator $\P\M$ is a bisectorial operator  on $\mH$
 with range $\ran(\P\M) = \ran(\P)$. It satisfies the quadratic estimate
\begin{align*}
  \int_0^\infty\| {\lambda{}}\P \M(\id+{\lambda{}^2}\P\M\P\M)^{-1} h \|_{2}^2 \, \frac{\mathrm{d}\lambda{}}\lambda{} \sim \|h\|_{2}^2 \qquad (h \in \cl{\ran(\P \M)}).
\end{align*}
The angle $\omega$ of bisectoriality and implicit constants in the quadratic estimate depend upon $n$ and the ellipticity constants of $A$. In particular, $\P \M$ has a bounded holomorphic functional calculus on $\clos{\ran(\P \M)} = \clos{\ran(\P)}$ on open double sectors $\S_{\mu}$ for all $\mu \in (\omega, \pi/2)$. The same holds true for $\M \P$ on $\cl{\ran(\M \P)} = \M \cl{\ran(\P)}$.
\end{thm}

The proof is given in Sections~\ref{sec:diracBasic} and \ref{sec:proof bhfc}. We emphasize that the entries of $\M$ are allowed to depend on $x$ and $t$ in a merely measurable fashion. As a consequence of bisectoriality there are topological splittings
\begin{equation} \label{eq:kernel/range}
\mH
= \clos{\ran(\P)} \oplus {\nul(\P\M)} = \M \cl{\ran(\P)} \oplus {\nul(\P)}.
\end{equation}
Existence of a bounded holomorphic functional calculus for $\P \M$ on $\clos{\ran(\P)}$ on sectors $\S_{\mu}$ in our case means that for any bounded holomorphic function $b: \S_\mu \to \IC$ the functional calculus operator $b(\P \M)$ on $\cl{\ran(\P)}$ is bounded by $\|b(\P \M)\|_{\cl{\ran(\P)} \to \cl{\ran(\P)}} \lesssim \|b\|_{\L^\infty(\S_\mu)}$ with an implicit constant depending upon $\mu$, $n$ and the ellipticity constants of $A$. Due to the seminal result of McIntosh~\cite{Mc} this property is in fact equivalent to the quadratic estimate stated above. If $b$ is unambiguously defined at the origin, then $b(\P\M)$ extends to $\mH$
 by $b(0)$ on $\nul(\P \M)$. This general theory also applies to $\M \P$ on the closure of its range.

Based on Theorem~\ref{thm:bhfc} and the observation that $\cH_{\loc}$ identifies with $\L^2_{\loc}(\R_{+}; \clos{\ran(\P)})$, see Lemma~\ref{lem:characterisation ranP} below, we can construct solutions to the differential equation \eqref{eq:diffeq2}, using functional calculus.

We define the characteristic functions $\chi^+(z)$ and $\chi^-(z)$ for the right and left open half planes in $\IC$ and the exponential functions $e^{-\lambda [z]}$, $\lambda>0$. Here $[z]:=z\,\sgn(z)$ and $\sgn(z):= \chi^+(z)-\chi^-(z)$. This gives the generalized {\em spectral projections} $\chi^\pm(\P\M)$, which are bounded on $\clos{\ran(\P)}$ by Theorem~\ref{thm:bhfc}, and the holomorphic semigroup $(e^{-\lambda[\P\M]})_{\lambda>0}$ generated by $[\P\M]= \P\M \sgn(\P\M)$. The boundedness of the spectral projections yields a topological splitting into spectral spaces $\Hpm(\P\M):= \chi^\pm(\P\M)\clos{\ran(\P)}$, which can be called \emph{generalized Hardy spaces},
\begin{equation}
\label{eq:hardysplit}
 \clos{\ran(\P)}= \Hp(\P\M) \oplus \Hm(\P\M).
\end{equation}
It is convenient to introduce the following generalized \emph{Cauchy extension}: for $h\in \clos{\ran(\P)}$ and $\lambda\in \R \setminus \{0\}$,
\begin{equation}
\label{eq:Cauchyext}
(C_{0}h)_{\lambda}=(C_{0}h)(\lambda,\cdot):=
\begin{cases}
(C_{0}^+h)(\lambda,\cdot) = \e^{-\lambda \P\M} \chi^+(\P\M) h & (\text{if $\lambda>0$}), \\
 (C_{0}^-h)(\lambda,\cdot) = \e^{-\lambda \P\M} \chi^-(\P\M) h & (\text{if $\lambda<0$}).
\end{cases}
\end{equation}
Let us also define the \emph{semigroup extension} of $h$ for $\lambda>0$ by
\begin{align}\label{eq:sgext}
 (Sh)_{\lambda}=(Sh)(\lambda,\cdot) :=e^{-\lambda[\P\M]} h.
\end{align}
Hence,
\begin{align*}
 \begin{cases}
 (C_{0}^+h)_{\lambda}= (S \chi^+(\P\M)h)_{\lambda} &  (\text{if $\lambda>0$}),\\
 (C_{0}^-h)_{\lambda}= (S \chi^-(\P\M)h)_{-\lambda} & (\text{if $\lambda<0$}).
\end{cases}
\end{align*}
The above also applies to $\M\P$, which has a bounded holomorphic functional calculus on the closure of its range, too. However, unlike for $\P \M$,  the range of $\M \P$ is not independent of $\M$.

\begin{prop}
\label{prop:Cauchyextension}
The generalized Cauchy extension $F=C_{0}^+h$ of $h\in \clos{\ran(\P)} $ gives a solution to $\pd_\lambda F+ \P\M F=0$ in the strong sense $F\in \C([0,\infty); \clos{\ran(\P)})\cap \C^\infty((0,\infty); \dom(\P\M))$ with  bounds
\begin{align*}
 \sup_{\lambda>0}\|F_\lambda\|_2 \sim \|\chi^+(\P\M)h\|_2 \sim \sup_{\lambda>0} \barint_{\lambda}^{2\lambda} \|F_{\mu}\|_{2}^2 \d \mu
\end{align*}
and $\mH$ limits
\begin{align*}
 \lim_{\lambda\to 0} F_\lambda =\chi^+(\P\M)h, \qquad \lim_{\lambda\to\infty} F_\lambda =0.
\end{align*}
In addition, the square function estimate
\begin{align}
\label{eq:sf}
\qe{\lambda \pd_{\lambda} F} \sim \|\chi^+(\P\M)h\|^2_{2}
\end{align}
holds. Furthermore, $F$ is a solution to $\pd_\lambda F+ \P\M F=0$ also in the weak sense \eqref{eq:diffeq}.
\end{prop}

If $h$ belongs to the spectral space $\Hp(\P\M)$, then $h= \chi^+(\P\M)h$ is the initial value of its generalized Cauchy extension $C_0^+h$. The generalized Cauchy extension $C_{0}^-h$ gives a solution of the same equation with the analogous estimates for $\lambda\in \R_{-}$. Moreover, the counterpart of Theorem~\ref{thm:correspondence} for the lower half-space (with identical proof) provides the correspondence with parabolic conormal differentials of reinforced weak solutions to \eqref{eq1} on $\R^{n+2}_{-}$.

All statements in Proposition~\ref{prop:Cauchyextension} but the last one are well-known properties in semigroup theory and they apply to any bisectorial operator with a bounded holomorphic functional calculus \cite{Haase, Mc}. Note that in \eqref{eq:sf} we have $\pd_{\lambda} F = \phi(\lambda \P \M) h$ with holomorphic function $\phi(z) = z \e^{-[z]} \chi^+(z)$, and thus this square function estimate follows from the analogous one with $\phi(z) = z(1+z^2)^{-1}$ provided by Theorem~\ref{thm:bhfc}. The verification of the equation in the weak sense \eqref{eq:diffeq} is special to the operator $\P\M$ and follows by a simple integration by parts.

The converse to Proposition~\ref{prop:Cauchyextension} is a verbatim modification of the argument originally designed for elliptic systems \cite[Thm.~8.2]{AA}. In fact, this relies on an abstract approach that works for any bisectorial operator with a bounded holomorphic functional calculus, see \cite{AA-Note}. Hence, we can directly state the following result.

\begin{thm}
\label{thm:uniq}
Let  $F \in \L^2_{\loc}(\R_{+}; \clos{\ran(\P)}) $ be a solution of \eqref{eq:diffeq2} in the weak sense such that
\begin{equation}
\label{eq:uniqueness}
\sup_{\lambda>0} \barint_{\lambda}^{2\lambda} \|F_{\mu}\|_{2}^2 \d \mu <\infty.
\end{equation}
Then $F$ has a $\L^2$ limit $h\in \Hp(\P\M)$ at $\lambda = 0$ and $F$ is given by the Cauchy extension of $h$. The analogous result for weak solutions $F \in \L^2_{\loc}(\R_{-}; \clos{\ran(\P)}) $ of \eqref{eq:diffeq2} also holds upon using the spectral space $\Hm(\P \M)$.
\end{thm}

Finally, we mention that the whole functional calculus for $\P \M$ becomes analytic in $A$ equipped with the $\L^\infty$-norm (or equivalently, in $\M$ for the same norm). In particular, this provides us with Lipschitz estimates for the operator norm on $\cl{\ran(\P)}$: for all $b$ bounded and holomorphic in a sector $\S_{\mu}$ with $\omega<\mu<\pi/2$,
\begin{align}
\label{eq:lipschitz}
\|b(\P\M) - b(\P\M')\|_{\cl{\ran(\P)} \to \cl{\ran(\P)}} \lesssim \|b\|_{\infty}\|M-M'\|_{\infty},
\end{align}
where $M'$ is associated with an elliptic matrix $A'$ in the same class as $A$ and $\|M-M'\|_{\infty}$ is small. The implicit constant depends on $n$, $\mu$ and the ellipticity constants of $A$. This is again an abstract property of the functional calculus for operators of type $\P \M$ once it is known that implicit constants in the quadratic estimate of Theorem~\ref{thm:bhfc} depend on the coefficients $A$ only through the ellipticity constants. Details are written out for example in \cite[p.~269]{AAM} or \cite[Prop.~6.1.25]{EigeneDiss}.
\subsection{Kato square root estimate}\label{sec:kato}

As another important consequence of Theorem~\ref{thm:bhfc} we obtain the resolution of the \emph{Kato problem for parabolic operators} in full generality.

In order to set the context, we note that similar to Section~\ref{sec:energy} the use of the energy space $\mathsf{V} := \H^{1/2}(\R; \L^2(\R^n)) \cap \L^2(\R; \W^{1,2}(\R^n))$ allows us to define the parabolic operator $L:=\pd_{t} - \div_{x}A_{\pa \pa}(x,t)\nabla_{x}$ as an operator $\mathsf{V} \to \mathsf{V}^*$ via a sesquilinear form,
\begin{align*}
 \langle L u, v \rangle := \iint_{\ree} A_{\pa \pa} \nabla_x u\cdot\overline{\nabla_x v}+ \HT\dhalf u\cdot \overline{\dhalf v} \d x\d t \qquad (u,v \in \mathsf{V}).
\end{align*}
Carving out the surprising analogy with elliptic operators even further, \cite[Lem.~4]{Aus-Eg} shows that $L$ with maximal domain $\dom(L) = \{u \in \mathsf{V} : L u \in \L^2(\ree) \}$ in $\L^2(\ree)$ is \emph{maximal accretive}: That is to say $L$ is closed, $\Re \langle Lu, u \rangle \geq 0$ holds for all $u \in \dom(L)$, and $1+L$ is onto. We refer to \cite{Kato} for background on this class of operators. Note that all this is for one dimension lower than in Section~\ref{sec:energy} and that we are using an inhomogeneous energy space since we are studying $L$ as an operator in $\L^2(\ree)$. The reader will notice that the following result reads almost identically to the famous Kato problem for elliptic operators solved in \cite{AHLMcT}. We shall give its proof in Section~\ref{sec:parKato}. We emphasize that no assumptions on $A_{\pa \pa} = A_{\pa \pa}(x,t)$ besides measurability and uniform ellipticity have been imposed.

\begin{thm}\label{thm:Kato} The operator $L=\pd_{t} - \div_{x}A_{\pa \pa}(x,t)\nabla_{x}$ arises from an accretive form, it is maximal-accretive in $\L^2(\ree)$, the domain of its square root is that of the accretive form, that is, $\dom(\sqrt{L}) = \mathsf{V} = \H^{1/2}(\R; \L^2(\R^n)) \cap \L^2(\R; \W^{1,2}(\R^n))$. The two-sided estimate
\begin{align*}
\|\sqrt L\, u\|_{2} \sim \|\gradx u\|_{2}+ \| \dhalf u\|_{2} \qquad (u \in \mathsf{V})
\end{align*}
holds with implicit constants depending only upon $n$ and ellipticity constants of $A_{\pa \pa}$.
\end{thm}

The reader should remark that not even the case $A_{\pa \pa}^*=A_{\pa \pa}$ can be treated by abstract functional analysis (while this is the case for the elliptic Kato problem) because $L$ is never self-adjoint.

\subsection{Sobolev spaces}
\label{sec:sobolev}

The previously obtained characterisation of solutions to the differential equation in \eqref{eq:diffeq2} extends to other topologies.

For $s \in \R$ we can define the fractional powers $[\P\M]^s$ as closed and injective operators in $\clos{\ran(\P \M)} = \clos{\ran(\P)}$, using the functional calculus for $\P \M$. The homogeneous Sobolev space $\Hdot^s_{\P\M}$ based on $\P\M$ is the completion of $\dom([\P \M]^s) \cap \clos{\ran(\P \M)}$ under the norm $ \|[\P\M]^s \cdot \|_{2}$. Another option taken in \cite{Amenta-Auscher} is to use the fractional powers of $(PM)^s$ with branch cut on $\i (-\infty,0]$, which gives the same spaces with isometric norms. This yields an abstract scale of Hilbert spaces that can all be realised within the same linear Hausdorff space and the intersection of any two of them is dense in both of them. Of course, $\Hdot^{0}_{\P\M} = \clos{\ran(\P)}$. As a matter of fact, the bounded holomorphic functional calculus extends to $\Hdot^s_{\P\M}$, yielding in particular the spectral spaces $\Hdot^{s,\pm}_{\P\M}$, and $\P \M$ extends to an isomorphism from $\Hdot^s_{\P\M}$ onto $\Hdot^{s-1}_{\P\M}$. The same abstract construction can be done with $\M\P$ starting from the base space $\cl{\ran(\M \P)}$. Further background on these operator-adapted Sobolev spaces is provided in \cite{Auscher-McIntosh-Nahmod} or \cite[Ch.~6]{Haase}.

In particular, this construction applies to $\P$ and in this case we re-obtain concrete homogeneous parabolic Sobolev spaces as defined in Section~\ref{sec:energy}. To this end, we note that if $h \in \dom(\P^2) \cap \ran(\P)$, then
\begin{align}
\label{eq:Psquared}
 \P^2 h
 = \begin{bmatrix} \partial_t - \Delta_x & 0 & 0 \\ 0 & -\gradx \divx & \gradx \dhalf \\ 0 & - \HT \dhalf \divx & \pd_{t}  \end{bmatrix} h
 = \begin{bmatrix} \partial_t - \Delta_x & 0 & 0 \\ 0 & \partial_t - \Delta_x & 0\vphantom{\dhalf} \\ 0\vphantom{\dhalf} & 0& \partial_t - \Delta_x \end{bmatrix} h.
\end{align}
Taking the Fourier transform, we easily see that $\Hdot_\P^s$ is a closed subspace of $(\Hdot_{\pd_{t} - \Delta_x}^{s/2})^{n+2}$ and that their norms are equal. Moreover, \eqref{eq:Psquared} reveals the bounded projection $\IP_{\P}$ onto $\clos{\ran(\P)}$ along ${\nul(\P)}$, which is well-defined in virtue of the decomposition \eqref{eq:kernel/range} for $M = 1$, as the smooth parabolic singular integral of convolution $\IP_\P = \P^2(\pd_{t} - \Delta_x)^{-1}$. Thus, $\Hdot^s_\P$ is precisely the image of $(\Hdot^{s/2}_{\pd_{t}-\Delta_x})^{n+2}$ under the bounded extension of $\IP_{\P}$.

The following lemma allows one to work more concretely  with the Sobolev spaces associated with $\P \M$. It is a consequence of the bounded holomorphic functional calculus for $\P \M$ (and $\P$) and the fact that $\cl{\ran(\P \M)} = \cl{\ran(\P)}$ holds, see \cite[Prop.~4.5]{AMM} for details.

\begin{lem}
\label{lem:identification}
If $-1\le s \le 0$, then $\Hdot_{\P \M}^s = \Hdot_{\P}^s$ with equivalent norms. Equivalence constants depend only on $n$ and ellipticity constants of $A$.
\end{lem}

With this notation set up, we can give the alluded \emph{a priori} characterisations of weak solutions to \eqref{eq:diffeq2} with data in parabolic homogeneous Sobolev spaces. Note that all appearing abstract Sobolev spaces identify with concrete parabolic Sobolev spaces only in the range of the above lemma.

\begin{thm}
\label{thm:sob}
Let $-1\le s<0$. The generalized Cauchy extension $F=C_{0}^+h$ of $h\in\Hdot^{s}_{\P} $ gives a solution to $\pd_\lambda F+ \P\M F=0$, in the strong sense $F\in \C([0,\infty); \Hdot_\P^s) \cap \C^\infty((0,\infty); \Hdot^{s}_{\P} \cap \Hdot^{1+s}_{\P\M})$ and in the weak sense \eqref{eq:diffeq}, with $\Hdot^{s}_{\P}$ bounds
\begin{align*}
 \sup_{\lambda>0}\|F_\lambda\|_{\Hdot^{s}_{\P}}\sim \|\chi^+(\P\M)h\|_{\Hdot^{s}_{\P}},
\end{align*}
and $\Hdot^{s}_{\P}$ limits
\begin{align*}
 \lim_{\lambda\to 0} F_\lambda =\chi^+(\P\M)h, \qquad \lim_{\lambda\to\infty} F_\lambda =0.
\end{align*}
In addition, the square function estimate
\begin{align}
\label{eq:sfsob}
 \qe{\lambda^{-s} F} \sim \|\chi^+(\P\M)h\|^2_{\Hdot^{s}_{\P}}
\end{align}
holds. Furthermore, let  $F \in \L^2_{\loc}(\R_{+}; \clos{\ran(\P)})$ be a solution of \eqref{eq:diffeq2} in the weak sense \eqref{eq:diffeq} such that
\begin{align*}
\qe{\lambda^{-s} F} <\infty.
\end{align*}
Then $F$ has a $\Hdot^{s}_{\P}$ limit $h\in \Hdot^{s,+}_{\P\M}$ at $\lambda = 0$ and $F$ is given by the Cauchy extension of $h$. The analogous results for the same equation on $\R_{-}$ also hold upon replacing the positive spectral spaces and projections with their negative counterparts.
\end{thm}

As in the case of Proposition~\ref{prop:Cauchyextension} and Theorem~\ref{thm:uniq}, also Theorem~\ref{thm:sob} was first obtained in the context of elliptic equations but by an abstract argument which literally applies to any bisectorial operator  with bounded holomorphic functional calculus and to $\P\M$ in particular. References in the context of elliptic equations are \cite[Thm.~9.2]{AA} for $s=-1$ and \cite[Thm.~1.3]{R1} for $s \in (-1, 0)$. These results have been put into a much broader context in \cite[Sec.~11]{Auscher-Stahlhut_APriori}. In particular, in their proof it is not important whether or not $\P$ is self-adjoint as long as $\P$ itself is bisectorial with bounded holomorphic functional calculus. However, expressions involving $B^*D$ in these references should be replaced with $\M^*\P^*$ in our context.
\subsection{Three special cases}
\label{sec:specialcases}

In order to make things more explicit, we single out three special cases of Theorems~\ref{thm:uniq} and \ref{thm:sob}. Throughout, we let $u$ be a reinforced weak solution to \eqref{eq1} and we write $F=\pcg u$, where we recall $|\pcg u|^2\sim |\nabla_{\lambda,x}u|^2+|\HT\dhalf u|^2$.

Let us begin by revisiting the case $s=-1/2$, which already appeared in our discussion of energy solutions in Section~\ref{sec:energy}. Indeed, in this case \eqref{eq:sfsob} becomes
\begin{align}\label{eq:equiv}
  \iiint_{\R^{n+2}_+} |\nabla_{\lambda,x}u|^2+|\HT\dhalf u|^2  \d\lambda\d x\d t
 \sim \|h\|_{\Hdot^{-1/2}_{\P}}^2
 \sim \|\dnuA u|_{\lambda=0}\|_{\Hdot^{-1/4}_{\pd_{t}-\Delta_{x}}}^2 + \|u|_{\lambda=0}\|_{\Hdot^{1/4}_{\pd_{t}-\Delta_{x}}}^2,
\end{align}
which is the \emph{energy estimate} for reinforced weak solutions, and finiteness of the left-hand side exactly means that $u$ belongs to the energy class $\dot \E$ introduced in Section~\ref{sec:energy}. The \emph{a priori} information that $F \in \C_0([0,\infty); \Hdot^{-1/2}_{\P}) \cap \C^\infty((0,\infty); \Hdot_P^{-1/2} \cap \Hdot^{1/2}_{\P\M})$ satisfies $\pd_{\lambda} F + \P \M F = 0$ in the strong sense, allows to integrate by parts in $\lambda$ for every $v \in \dot \E$, so to obtain
\begin{align}\label{eq:IBP}
 \iiint_{\R^{n+2}_+} A \gradlamx u \cdot \cl{\gradlamx v} + \HT \dhalf u \cdot \cl{\dhalf v}\,  \d \lambda \d x \d t = - \langle \dnuA u|_{\lambda = 0}, v|_{\lambda = 0} \rangle.
\end{align}
Here, the trace of $\dnuA u$ is defined in $\Hdot^{-1/4}_{-\pd_{t} - \Delta_x}$ by continuity of $F = \pcg u$ and $\langle \cdot \,, \cdot \rangle$ denotes the duality pairing with $\Hdot^{1/4}_{\pd_{t} - \Delta_x}$ extending the inner product on $\L^2(\ree)$. Strictly speaking, we would carry out this integration for $v$ smooth and compactly supported in all variables first and then extend by density. A comparison with \eqref{eq:variational} reveals that this understanding of the Neumann boundary data is the same as in the abstract variational setup but the restriction to transversally independent coefficients also allows for a more concrete understanding in the sense of a $\L^2$ limit. In particular, we may plug in $v = (1-\delta \HT)u$ with $\delta>0$ small enough as in Section~\ref{sec:energy} to conclude by coercivity
\begin{align}
\label{eq:coercivity energy}
  \iiint_{\R^{n+2}_+} |\nabla_{\lambda,x}u|^2+|\HT\dhalf u|^2  \d\lambda \d x\d t \lesssim  \|\dnuA u|_{\lambda=0}\|_{\Hdot^{-1/4}_{\pd_{t}-\Delta_{x}}} \|u|_{\lambda=0}\|_{\Hdot^{1/4}_{\pd_{t}-\Delta_{x}}}.
\end{align}
Hence, we obtain the boundedness and invertibility of the \emph{Dirichlet to Neumann map} at the boundary expressed in the comparability
\begin{align}
\label{eq:DirToNeu}
 \|\dnuA u|_{\lambda=0}\|_{\Hdot^{-1/4}_{\pd_{t}-\Delta_{x}}}
 \sim \|u|_{\lambda=0}\|_{\Hdot^{1/4}_{\pd_{t}-\Delta_{x}}}
 \sim \|h\|_{\Hdot^{-1/2}_\P}.
\end{align}
We shall come back to this point later on in Section~\ref{sec:sgnPM results}.

When $s=0$, \eqref{eq:uniqueness} rewrites
\begin{align*}
\sup_{\lambda>0} \barint_{\lambda}^{2\lambda}\iint_{\ree} |\nabla_{\lambda,x}u|^2+|\HT\dhalf u|^2 \d x\d t\d\mu \sim \|h\|_2^2\sim \|\dnuA u|_{\lambda=0}\|_{2}^2 + \|u|_{\lambda=0}\|_{\Hdot^{1/2}_{\pd_{t}-\Delta_{x}}}^2
\end{align*}
and, as we will see in Section~\ref{sec:NTmaxDef}, the left hand side is also comparable to a non-tangential maximal norm on $\pcg u$. The comparability
\begin{align*}
 \|\dnuA u|_{\lambda=0}\|_{2} \sim \|u|_{\lambda=0}\|_{\Hdot^{1/2}_{\pd_{t}-\Delta_x}} \sim \|\gradx u|_{\lambda=0}\|_{2}
+ \|\HT\dhalf u|_{\lambda=0}\|_{2}
\end{align*}
for this class of $u$ would be the \emph{Rellich estimates} for proving solvability of regularity problems and Neumann problems with $\L^2$ data.

When $s=-1$, \eqref{eq:sfsob} becomes
\begin{align*}
 \int_o^\infty \iint_{\ree} |\lambda \nabla_{\lambda,x}u|^2+|\lambda \HT\dhalf u|^2 \d x\d t \, \frac{\mathrm{d}\lambda}{\lambda} \sim \|h\|_{\Hdot^{-1}_{\P}}^2 \sim \|\dnuA u|_{\lambda=0}\|_{\Hdot^{-1/2}_{\pd_{t}-\Delta_{x}}}^2 + \|u|_{\lambda=0}\|_2^2.
\end{align*}
In this case the inequality
\begin{align*}
\|\dnuA u|_{\lambda=0}\|_{\Hdot^{-1/2}_{\pd_{t}-\Delta_{x}}} \lesssim \|u|_{\lambda=0}\|_2
\end{align*}
would be the \emph{Rellich estimate} that provides one way to solve the Dirichlet problem knowing $u|_{\lambda=0}$ in $\L^2$. Note that $u$ is only known up to a constant from Theorem~\ref{thm:correspondence}. The semigroup representation for $\pcg u$ in Theorem~\ref{thm:sob} immediately implies a representation of $u$ itself by means of the functional calculus for $\M \P$. The argument in \cite[Thm.~9.3]{AA} applies verbatim.

\begin{thm}
\label{thm:chardir}
Let $u$ be a reinforced weak solution of \eqref{eq1} with
\begin{align*}
\int_{0}^\infty\iint_{\ree} |\lambda \nabla_{\lambda,x}u|^2+|\lambda \HT\dhalf u|^2 \d x\d t \, \frac{\mathrm{d}\lambda}{\lambda}<\infty.
\end{align*}
Then there exist unique $c\in \IC$ and $\wt{h} \in \Hp(\M\P)$ such that $u= c- (e^{-\lambda[\M\P]}\wt{h})_{\pe}$. In particular, $u\in\C_{0}([0,\infty); \L^2(\R^{n+1}))+ \IC$, where the subscripted $0$ means that $u$ vanishes at $\infty$. Moreover, $\P \wt{h}$ is the element $h \in \Hdot^{-1,+}_{\P\M}$ appearing in Theorem~\ref{thm:sob}.
\end{thm}
\subsection{Resolvent estimates}\label{sec:resolvent}

An important ingredient in the proof of Theorem~\ref{thm:bhfc} are off-diagonal estimates for the resolvents of $\P\M$. For anyone who knows how these estimates are obtained in the elliptic case, the non-locality of half-order time derivatives may seem as a serious obstacle. A key observation is the following.

\begin{lem}
\label{lem:LpLq}
There exists $\delta _{0}>0$, depending only on {dimension and} the ellipticity constants of the matrix $A$, such that for any real $p,q$ with $|\frac{1}{p}-\frac{1}{2}| <\delta _{0}$ and $|\frac{1}{q}-\frac{1}{2}| <\delta _{0}$, and any $\lambda\in \R$, the resolvent $(\id + \i\lambda \P\M)^{-1}$ is bounded on $\L^p(\R; \L^q(\R^n; \IC^{n+2}))$ with uniform bounds with respect to $\lambda$. The same result holds with $\M\P$, {$\P^* \M$ or $\M \P^*$} in place of $\P\M$.
\end{lem}

This result will be proved in Section~\ref{sec:resolvents} and it allows us to obtain enough decay for our arguments. We prove the following off-diagonal estimates.

\begin{prop}\label{prop:OffDiag} There exists $\eps_0>0$ and $N_{0}>1$ such that if $|\frac{1}{q}-\frac{1}{2}| <\eps_{0}$, then one can find $\eps=\eps(n,q,\eps_0)>0$ with the following property. Whenever $N\ge N_{0}$, then there exists a constant $C=C(\eps,N,q)<\infty$ such that for all balls $Q=B(x, r)$ in $\R^n$, all intervals $I=(t-r^2, t+r^2)$, and all parameters $j\in \IN$, $k\in \IN^*$ and $\lambda \in \R$ with $|\lambda| \sim r$:
\begin{align}
\label{eq:off}
\bariint_{Q \times 4^jI} |(\id + \i\lambda \P\M)^{-1}h|^q \d y \d s \le C N^{-q\eps k} \bariint_{C_{k}(Q \times 4^jI)} |h|^q \d y \d s,
\end{align}
provided $h \in (\L^2 \cap \L^q)(\ree; \IC^{n+2})$ has support in $C_{k}(Q \times 4^jI)$. Here,
\begin{align*}
C_{k}(Q \times J):= (2^{k+1}Q\times N^{k+1}J)\setminus (2^{k}Q\times N^{k}J)
\end{align*}
and $cQ$ and $cJ$ denote dilates of balls and intervals, respectively, keeping the center fixed and dilating the radius by $c$. Analogous estimates hold with $\P\M$ replaced by $\M\P$, $\P^* \M$ or $\M \P^*$.
\end{prop}

Some remarks explaining the nature of this result are necessary. First, $Q\times I$ is a standard parabolic cube in $\R^n\times \R$. Some estimates will require that we stretch $I$ leaving $Q$ fixed. Hence the use of $4^jI$. Next, since $|\lambda|$ is on the order of the radius of the parabolic cube there is no power of $r/|\lambda|$ involved nor is there dependence on $j$. Furthermore,  in the definition of $C_{k}(Q \times 4^jI)$ we take the usual double stretching in the $x$-direction, but we stretch by a factor $N$ in the $t$-direction. In principle $N$ will be chosen larger than any power of $2^n$ that we may need to {match} the decay in the $x$-direction. In the end, we still have a small power of $N^{-k}$ in front of the integral on the right hand side but note that we have normalized by taking averages (materialized by the dashed integrals) and observe that the Lebesgue measure of $C_{k}(Q \times 4^jI)$ is on the order of $2^{kn}N^k$ times that of $Q \times 4^jI$. Hence, this is much more decay than it looks like.

We do not know whether the resolvent has the classical off-diagonal decay as an operator acting on a space of homogeneous type (here, $\ree$ with parabolic distance and Lebesgue measure) in the sense of \cite{Auscher-MartellJEE}. So, it is kind of a novelty in this topic that a weaker form of off-diagonal estimates apply, using that our space has identified directions, which is not the case in general. In contrast, all decay estimates obtained in \cite{N1} for operators with $t$-independent coefficients have the parabolic homogeneity.
\subsection{Non-tangential maximal function estimates}
\label{sec:NTmaxDef}

The off-diagonal estimates in Proposition~\ref{prop:OffDiag} are strong enough to prove non-tangential maximal function estimates in Section~\ref{sec:reverse and NT}. For $(x,t)\in \ree$ we define the non-tangential maximal function
\begin{equation*}
\NT F(x,t)= \sup_{\lambda>0} \bigg(\bariiint_{\Lambda \times Q \times I} |F(\mu,y,s) |^2 \d \mu \d y \d s\bigg)^{1/2},
\end{equation*}
where $\Lambda=(c_{0}\lambda, c_{1}\lambda)$, $Q=B(x, c_{2}\lambda)$ and $I= (t-c_{3}\lambda^2, t+c_{3}\lambda^2)$ with positive constants $c_{i}$ and $c_{0}<c_{1}$.
The numerical values of the constants are irrelevant {since by a covering argument} any change gives equivalent norms.

\begin{thm}
\label{thm:NTmax} Let $ h\in \cl{\ran(\P \M)}$ and let $F=Sh$ be its semigroup extension as in \eqref{eq:sgext}. Then
\begin{equation*}
\| \NT F \|_{2} \sim \|h\|_{2},
\end{equation*}
where the implicit constants depend only on dimension and the ellipticity constants of $A$. Furthermore, there is almost everywhere convergence of Whitney averages
\begin{equation*}
\lim_{\lambda\to 0} \bariiint_{\Lambda \times Q \times I} |F(\mu,y,s)- h(x,t) |^2 \d \mu \d y \d s=0
\end{equation*}
for almost every $(x,t)\in \ree$.
\end{thm}

Recall that if $h\in \Hp(\P\M)$, then the semigroup extension is the same as the Cauchy extension of $h$, so that this theorem provides us with estimates for reinforced weak solutions to \eqref{eq1}. As the estimate
\begin{align*}
 \barint_{c_{0}\lambda}^{ c_{1}\lambda}\|F_{\mu}\|_{2}^2 \d \mu \lesssim \| \NT F \|_{2}^2
\end{align*}
holds for arbitrary functions $F$, see Lemma~\ref{lem:NT and QE} below for convenience, the non-tangential maximal estimate can be seen as a further regularity estimate for solutions in the uniqueness class of Theorem~\ref{thm:uniq}. Of course, the class of solutions to \eqref{eq:diffeq2} with $\| \NT  F\|_{2}<\infty$ is also a uniqueness class by restriction. In addition, the almost everywhere convergence of Whitney averages allows us to give a pointwise meaning to the boundary trace of such solutions. A similar remark applies to data in $\Hm(\P\M)$ for solutions with $\lambda\in \R_{-}$.

There is also a companion result that applies to Dirichlet problems with $\L^2$ data.

\begin{thm}
\label{thm:NTmaxDir} Let $\wt{h} \in \cl{\ran(\M \P)}$ and let $\wt{F}(\lambda,\cdot) = \e^{-\lambda [\M\P]} \wt{h}$ be its semigroup extension with respect to $\M\P$. Then
\begin{equation*}
\| \NT \wt{F} \|_{2} \sim \|\wt{h}\|_{2},
\end{equation*}
where the implicit constants depend only on dimension and the ellipticity constants of $A$. Furthermore, there is almost everywhere convergence of Whitney averages,
\begin{equation*}
\lim_{\lambda\to 0} \bariiint_{\Lambda \times Q \times I} |\wt{F}(\mu,y,s)- \wt{h}(x,t) |^2 \d \mu \d y \d s=0
\end{equation*}
for almost every $(x,t)\in \ree$.
\end{thm}

Indeed, we recall from Theorem~\ref{thm:chardir} that for $\wt{F}$ as above and $\wt{h} \in \Hp(\M\P)$ the perpendicular part $\wt{F}_{\pe}$ is a solution of the Dirichlet problem with data $\wt{h}_{\pe}$.

The non-tangential maximal estimates in Theorem~\ref{thm:NTmax} and Theorem~\ref{thm:NTmaxDir} depend on apparently new reverse H\"older estimates for reinforced weak solutions of \eqref{eq1} that we also prove in Section~\ref{sec:reverse and NT}. These estimates are valid in full generality for $\lambda$-dependent equations and also for solutions in the lower half-space up to the obvious changes of notation.

\begin{thm}\label{thm:rh} There is a constant $C$, depending only on the ellipticity constants of $A$ and the dimension $n$, such that any reinforced weak solution of \eqref{eq1} satisfies the reverse H\"older estimate
\begin{equation}\label{eq:rh}
\begin{split}
\bigg(&\bariiint_{\Lambda \times Q \times I} |\nabla_{\lambda,x}u |^2+ |\HT\dhalf u|^2 + |\dhalf u|^2 \d \mu \d y \d s \bigg)^{1/2} \\
&\leq C \sum_{k \in \IZ} \frac{1}{1+|k|^{3/2}} \bariiint_{8\Lambda \times 8Q \times I_k} |\nabla_{\lambda,x}u |+ |\HT\dhalf u| + |\dhalf u| \d \mu \d y \d s.
\end{split}
\end{equation}
Here, $\Lambda=(\lambda-r, \lambda+r)$, $Q=B(x, r)$, $I= (t-r^2, t+r^2]$ define a parabolic cylinder of radius $r < \lambda/8$ and $I_k := k \ell(I) + I$ are the disjoint translates of $I$ covering the real line.
\end{thm}

Note that the estimates in \eqref{eq:rh} are non-local in the time variable, reflecting the non-locality of half-order time derivatives. The estimates involves both $\HT \dhalf u$ and $\dhalf u$, which is reminiscent of the arbitrary choice that we made when setting up the first order approach. In fact, the proof of Theorem~\ref{thm:rh} will suggest that control for one of these functions on the left requires both of them on the right. The factor of the parabolic enlargement of $\Lambda$ and $Q$ on the right-hand side and the corresponding relation between $r$ and $\lambda$ can be changed from $8$ to any factor $c>1$ at the expense of $C$ then depending on $c$ as well. Again, this is best done by a covering argument.
\subsection{Characterisation of well-posedness of the BVPs (\texorpdfstring{$\lambda$}{lambda}-independent)}
\label{sec:char}

Let us eventually address boundary value problems for parabolic equations. We write $\Lop u=0$ to mean that $u$ is a reinforced weak solution of \eqref{eq1}. For $-1\le s\le 0$ we let
\begin{align*}
\|F\|_{\mE_{s}}:= \begin{cases}
 \|\NT(F )\|_{2} & (\text{if $s=0$}), \\
 (\qe{\lambda^{-s}F})^{1/2} & (\text{otherwise})
\end{cases}
\end{align*}
and define the solution classes
\begin{align*}
 \mE_{s}:=\{F\in \Lloc^2(\R^{n+2}_+;\IC^{n+2}) ; \|F\|_{\mE_{s}}<\infty\}.
\end{align*}
Given $s\in [-1,0]$, the regularity of the data, we consider the problems
\begin{align*}
 (R)_{\mE_{s}}^\Lop&:  \Lop u=0,\, \pcg u\in \mE_{s},\, u|_{\lambda=0}=f \in \Hdot^{s/2+1/2}_{\pd_{t}-\Delta_{x}},\\
 (N)_{\mE_{s}}^\Lop&: \Lop u=0,\, \pcg u\in \mE_{s},\, \dnuA u|_{\lambda=0}= f \in \Hdot^{s/2}_{\pd_{t}-\Delta_{x}}.
\end{align*}
The problem $(R)_{\mE_{s}}^\Lop$ is formulated as a Dirichlet problem with regularity $s/2+1/2$ on the data but due to
\begin{align*}
 \|f\|_{\Hdot^{s/2+1/2}_{\pd_{t}-\Delta_{x}}}\sim \|[\gradx f,\, \HT\dhalf f]\|_{(\Hdot^{s/2}_{\pd_{t}-\Delta_{x}})^{n+1}},
\end{align*}
it is equivalent to a regularity problem with (parabolic) regularity $s$ on the quantities $\gradx u$ and $\HT\dhalf u$ at the boundary. We shall adopt this second point of view because these quantities naturally appear in $\pcg u$. On the other hand, $(N)_{\mE_{s}}^\Lop$ is a Neumann problem. Note that we require the full knowledge $\pcg u\in \mE_{s}$, whereas in \cite{N2} the Neumann problem for $s=0$ is posed with the non-tangential maximal control of $\nabla_{\lambda,x}u$ only. However, therein  additional information, such as DeGiorgi-Nash-Moser estimates and invertibility of layer potentials, are used to prove uniqueness under this weaker assumption.

Solving the problems $(R)_{\mE_{s}}^\Lop$ and $(N)_{\mE_{s}}^\Lop$ means finding a solution with control from the data. For $(R)_{\mE_{s}}^\Lop$ and $f$ the Dirichlet data, we want
\begin{align*}
 \|\pcg u\|_{\mE_{s}} \lesssim \|f\|_{\Hdot^{s/2+1/2}_{\pd_{t}-\Delta_{x}}}
\end{align*}
and for $(N)_{\mE_{s}}^\Lop$ and $f$ the Neumann data we want
\begin{align*}
 \|\pcg u\|_{\mE_{s}} \lesssim \|f\|_{\Hdot^{s/2}_{\pd_{t}-\Delta_{x}}}.
\end{align*}
The implicit constant must be independent of the data. The behaviour at the boundary is the strong convergence specified by our results above. Existence means that this holds for all data in a chosen space. Uniqueness means that there is at most one solution. Well-posedness is the conjunction of both existence of a solution for all data and uniqueness.

It is instructive to write out the general theory for the three most prominent cases that we considered in Section~\ref{sec:specialcases}: In the case $s=0$, the boundary value problems are the regularity and Neumann problems with data in $\L^2(\ree)$. For the regularity problem, the data takes the form of an array $[\gradx f,\, \HT\dhalf f]$ for some $f\in \Hdot^{1/2}_{\pd_{t}-\Delta_{x}}$. Due to Theorem~\ref{thm:NTmax} the boundary behaviour can also be interpreted in the sense of almost everywhere convergence of Whitney averages. In the case $s=-1$, as said, $(R)_{\mE_{-1}}^\Lop$ is nothing but a Dirichlet problem with $\L^2$ data~$f$ {and thanks to Theorem~\ref{thm:NTmaxDir} almost everywhere convergence of Whitney averages of $u$ to its boundary data comes again as a bonus}. In the case $s=-1/2$ we are back to the class of energy solutions, within which we have already obtained well-posedness in Section~\ref{sec:energy}. In fact, we have seen in Section~\ref{sec:specialcases} that the class of energy solutions coincides with the class of reinforced weak solutions such that $\pcg u \in \mE_{-1/2}$.

We will say that the boundary value problem $(\BVP)_{\mE_{s}}^\Lop$, where $\BVP$ designates either $N$ or $R$, is \emph{compatibly well-posed} if it is well-posed and if the solution agrees with the energy solution of the problem $(\BVP)_{\mE_{-1/2}}^\Lop$ when the data belongs to both data spaces.

In order to formulate characterisations of well-posedness and compatible well-posedness, we set some new notation. Recall that the trace $h$ of a conormal differential contains three terms: the first one is called the scalar component and denoted by $h_{\pe}$, and the two others are the tangential component and the time component, which we concatenate by $h_{r}= [h_{\pa}, \, h_{\te}]$. Since the trace space of all conormal differentials of reinforced weak solutions has been identified as a generalized Hardy space in Theorems~\ref{thm:uniq} and \ref{thm:sob}, the conclusion is that the two maps $N_{\pe}: h\mapsto h_{\pe}$ and $N_{r}: h\mapsto h_{r}$ carry the well-posedness. We record this easy but important observation in the following theorem.

\begin{thm}\label{thm:wpequiv} The following assertions hold for $-1\le s \le 0$.
\begin{enumerate}
 \item $(R)_{\mE_{s}}^\Lop$ is well-posed if and only if $N_{r}: \Hdot^{s,+}_{\P\M} \to (\Hdot^{s}_{\P})_{r}$ is an isomorphism.
 \item $(N)_{\mE_{s}}^\Lop$ is well-posed if and only if $N_{\pe}: \Hdot^{s,+}_{\P\M} \to (\Hdot^{s}_{\P})_{\pe}$ is an isomorphism.
 \end{enumerate}
In each case, ontoness is equivalent to existence and injectivity to uniqueness, respectively.
\end{thm}

As the boundary problems for $s=-1/2$ are well-posed by the method of energy solutions in Section~\ref{sec:energy}, compatible well-posedness reduces to another simple statement.

\begin{prop}
\label{prop:cwp}
The following assertions hold for $-1\le s \le 0$.
\begin{enumerate}
 \item $(R)_{\mE_{s}}^\Lop$ is compatibly well-posed if and only if $N_{r}: \Hdot^{s,+}_{\P\M} \to (\Hdot^{s}_{\P})_{r}$ is an isomorphism and its inverse agrees with the one at $s=-1/2$ on
 $(\Hdot^{s}_{\P})_{r} \cap (\Hdot^{-1/2}_{\P})_{r}$.
 \item $(N)_{\mE_{s}}^\Lop$ is compatibly well-posed if and only if $N_{\pe}: \Hdot^{s,+}_{\P\M} \to (\Hdot^{s}_{\P})_{\pe}$ is an isomorphism and its inverse agrees with the one at $s=-1/2$ on $(\Hdot^{s}_{\P})_{\pe} \cap (\Hdot^{-1/2}_{\P})_{\pe}$.
 \end{enumerate}
\end{prop}

Let us recall that Theorems~\ref{thm:uniq} and \ref{thm:sob} also provide \emph{a priori} representations for the conormal gradients of solutions to \eqref{eq1} on the lower half-space $\R^{n+2}_{-}$, leading to similar characterisations of well-posedness upon replacing positive with negative spectral spaces. Given these characterisations, (compatible) well-posedness extrapolates and  compatible well-posedness interpolates. The full details for elliptic Dirac operators can be found in \cite[Thm.~7.8 \& 7.7]{Amenta-Auscher}. They remain unchanged in our setup.

\begin{thm}\label{thm:extra} The set of $s \in (-1,0)$ for which $(\BVP)_{\mE_{s}}^\Lop$ is well-posed (compatibly well-posed) is open.
\end{thm}

Note that in general we cannot extrapolate from the endpoint cases $s=0$ and $s=~-1$.

\begin{thm}\label{thm:inter} If $(\BVP)_{\mE_{s_{i}}}^\Lop$ is compatibly well-posed when $-1\le s_{0}, s_{1}\le 0$, then $(\BVP)_{\mE_{s}}^\Lop$ is compatibly well-posed for any $s$ between $s_{0}$ and $s_{1}$.
\end{thm}

We mention that all results are analytic with respect to $\L^\infty$-perturbations of the coefficients $A(x,t)$. Again this is an abstract result proved in \cite[Thm.~7.16]{Amenta-Auscher}, compare also with \cite[Rem.~7.17 - 7.19]{Amenta-Auscher}. The assumption of this theorem is that a local Lipschitz estimate holds for the functional calculus of $\P\M$ on $\Hdot^s_{\P}$.  This one has been described in \eqref{eq:lipschitz} when $s=0$ and it extends to any $s\in [-1,0]$ using Lemma~\ref{lem:identification}.

\begin{thm}\label{thm:sta} Let $s\in [-1,0]$. If $(\BVP)_{\mE_{s}}^\Lop$ is (compatibly) well-posed, then so is $(\BVP)_{\mE_{s}}^{\Lop'}$ for any $\Lop'$ with coefficients $A'$ such that $\|A-A'\|_{\infty}$ is sufficiently small, depending on $s$ and the ellipticity constants of $A$. Moreover, the operator norm of the inverse of the respective projection $N_\bullet$ is  Lipschitz continuous as a function of $A'$ in a neighbourhood of $A$. \end{thm}

Finally, the Lipschitz estimates \eqref{eq:lipschitz}, extended to $\Hdot^s_{\P}$ as explained above, allow us to use the method of continuity to perturb from any equation to any other within arcwise connected components. Usually, we perturb from the heat equation. We record this observation in the following result to be proved in Section~\ref{sec:continuitymethod proof}. We say a bounded operator has \emph{lower bounds} if $\|Tf\| \gtrsim \|f\|$ holds for all $f$.

\begin{prop}
\label{prop:continuitymethod}
Let $s\in [-1,0]$. Let $\mathbb{A}$ be a class of matrices $A(x,t)$ with uniform ellipticity bounds that is arcwise connected. If this class contains one matrix $A_{0}$ for which the problem $(\BVP)_{\mE_{s}}^{\Lop_{0}}$ is (compatibly) well-posed and if all the corresponding maps $N_{\bullet}$ have lower bounds on the trace space $\Hdot^{s,+}_{\P\M}$ for each $A \in \mathbb{A}$,
then $(\BVP)_{\mE_{s}}^{\Lop}$ is (compatibly) well-posed for all $A\in \mathbb{A}$.
\end{prop}
\subsection{The backward equation and duality results}
\label{sec:backward}

The  equation dual to \eqref{eq1} is the backward equation
\begin{align}
\label{eq1*}
\Lop^* v= -\partial_t v -\div_{X} A^*(X,t)\nabla_{X} v = 0,
\end{align}
where $A^*$ is the (complex) adjoint of $A$. Note the sign change in front of the time derivative. This equation can also be rephrased as a first order system. However, in this case it will be convenient to use a different representation to simplify the presentation of the duality results.

We use the \emph{backward (parabolic) conormal differential}
\begin{equation*}
 \pcgb v(\lambda, x,t) := \begin{bmatrix} \dnuAstar v(\lambda, x,t) \\ \gradx v(\lambda, x,t) \\ \dhalf v(\lambda, x,t) \end{bmatrix},
\end{equation*}
where the difference compared to $D_{\!A^*} v$ is in the third component. With this choice of conormal differential, the analogue of Theorem~\ref{thm:correspondence} is that there is, up to constants, a correspondence between reinforced weak solutions $v$ to \eqref{eq1*} and weak solutions $G\in \wt{\cH}_{\loc}:= \Lloc^2(\R_{+}; \clos{\ran(P^*)})$ to the first order system
\begin{equation}
\label{eq:diffeq2*}
 \pd_{\lambda} G + \P^* \wM G=0.
\end{equation}
Here $\wM= N\M^* N$, where $\M^*$ is the complex adjoint of $\M$ in the standard duality of $\IC^{n+2}$ and
\begin{equation}
\label{eq:N}
 \N= {\begin{bmatrix} -1 & 0& \vphantom{\dhalf}0 \\ 0 & \id& 0 \\ \vphantom{\dhalf}0& 0& 1 \end{bmatrix}}.
\end{equation}
Note that this is an operator in the same class as $M$. In the case of $\lambda$-independent coefficients, the identity $\P^* \wM= ( (\wM)^*\P)^*$ shows that bisectoriality of $\P^* \wM$ on $\mH$ and the bounded holomorphic functional calculus inherit from $(\wM)^*\P$. The same analysis as developed above applies to \eqref{eq:diffeq2*}. It is only when coming back to \eqref{eq1*} that we have to use a different correspondence.

As the reversal of time maps solutions of the backward equation $\Lop^*v=0$ to a solution of a forward equation with coefficients $A^*(x,-t)$, we also have reverse H\"older inequalities for reinforced weak solutions of the backward equation. At this occasion let us recall that the reverse H\"older inequalities of Theorem~\ref{thm:rh} already contain both $\HT\dhalf v$ and $ \dhalf v$ on the left hand side. All other results for the backward equation, including non-tangential estimates as in Theorem~\ref{thm:NTmax}, are obtained by literally repeating the arguments provided for the forward equation and we shall not give further details.

With the construction outlined above we can formulate duality results for the boundary value problems.

\begin{thm}
\label{thm:duality}
Let $\Lop$ and $\Lop^*$ be as above with $\lambda$-independent coefficients and let $-1\le s\le 0$. Then $(\BVP)_{\mE_{s}}^\Lop$ is (compatibly) well-posed if and only if $(\BVP)_{\mE_{-1-s}}^{\Lop^*}$ is (compatibly) well-posed.
\end{thm}

{In a similar spirit we obtain a generalisation of Green's formula, which is of independent interest and worth stating as it avoids integration by parts. Recall that the dual of $\Hdot^{s/2}_{\pd_{t}-\Delta_{x}}$ can be identified with $\Hdot^{-s/2}_{-\pd_{t}-\Delta_{x}}$, which coincides with $\Hdot^{-s/2}_{\pd_{t}-\Delta_{x}}$ up to equivalent norms, and the duality pairing extends the inner product on $\L^2(\ree)$}.

\begin{prop}\label{prop:green} Let $-1\le s\le 0$. For any weak solution $u$ to \eqref{eq1} with $\pcg u \in \mE_{s}$
 and any weak solution $w$ to \eqref{eq1*} with $\pcgb w\in \mE_{-s-1}$ it holds
 \begin{equation}
\label{eq:green}
\dual {\partial_{\nu_{A}}u|_{\lambda=0}} {w|_{\lambda=0}} = \dual {u|_{\lambda=0}}{\partial_{\nu_{A^*}}w|_{\lambda=0}} .
\end{equation}
Here, the first pairing is the $\dual{\Hdot^{s/2}_{\pd_{t}-\Delta_{x}}}{ \Hdot^{-s/2}_{-\pd_{t}-\Delta_{x}}}$ sesquilinear duality while the second one is the $\dual { \Hdot^{s/2+1/2}_{\pd_{t}-\Delta_{x}}}{ \Hdot^{-s/2-1/2}_{-\pd_{t}-\Delta_{x}}}$ sesquilinear duality. \end{prop}

The proof of these two results will be given in Section~\ref{sec:charWPDuality}.
\subsection{The role of \texorpdfstring{$\sgn(\P \M)$}{sgn(PM)}}
\label{sec:sgnPM results}

In Section~\ref{sec:char} we have characterised well-posedness separately for each BVP on each parabolic half-space by means of the spectral projections $\chi^\pm(\P \M)$ and the coordinate projections $N_\pe$, $N_r$. In contrast, most classical methods to solving BVPs, such as the method of layer potentials~\cite{N2, CNS, V}, the closely related Rellich estimates, or the study of Dirichlet-to-Neumann maps, usually work equally well on both half-spaces and yield well-posedness of either all boundary value problems \emph{simultaneously} -- or none. We shall see that the general theory presented here comprises all these approaches.

This link becomes apparent by studying the operator $\sgn(\P \M) = \chi^+(\P \M) - \chi^-(\P \M)$, which by the bounded holomorphic functional calculus is an involution on $\dot \H_{\P \M}^s = \dot \H_P^s$ for $-1 \leq s \leq 0$. Gathering again $\pa$- and $\te$-components of vectors in these spaces in the single component $r$, we represent $\sgn(\P\M)$ by a $(2\times2)$-matrix
\begin{align}
\label{eq:rep sgn}
 \sgn(\P \M)
 := \begin{bmatrix} s_{\pe \pe}(\P \M) & s_{\pe r}(\P \M) \\ s_{r \pe}(\P \M) & s_{r r}(\P \M) \end{bmatrix}.
\end{align}
Since we have $\chi^\pm = \frac{1}{2} (1 \pm \sgn)$, the four components of $\sgn(\P \M)$ reappear in the representations of the spectral projections as
\begin{align}
\label{eq:rep chi}
\begin{split}
 \chi^+(\P\M)&= \frac 1 2\begin{bmatrix}
 1+s_{\pe\pe}(\P \M)     & s_{\pe r}(\P \M)    \\
  s_{r\pe}(\P \M)    &  1+s_{rr}(\P \M)
\end{bmatrix},  \\
\chi^-(\P\M)&= \frac 1 2\begin{bmatrix}
 1-s_{\pe\pe}(\P \M)     & -s_{\pe r}(\P \M)    \\
  -s_{r\pe}(\P \M)    &  1-s_{rr}(\P \M)
\end{bmatrix}
\end{split}
\end{align}
and in view of Section~\ref{sec:char}, well-posedness of the BVPs on any of the two half-spaces has to be encoded in the six maps $1 \pm s_{\pe \pe}(\P \M)$, $s_{\pe r}(\P\M)$, $s_{r \pe}(\P\M)$, $1 \pm s_{r r}(\P \M)$. This can be made precise as follows.

\begin{thm}
\label{thm:six invertibilities}
Let $s \in [-1,0]$. The following hold true.
\begin{enumerate}
 \item $(R)_{\mE_s}^\Lop$ is well-posed on \emph{both} half-spaces if and only if $s_{r \pe}(\P \M): (\Hdot_P^s)_\pe \to (\Hdot_P^s)_r$ is invertible.
 \item $(N)_{\mE_s}^\Lop$ \!is well-posed on \emph{both} half-spaces if and only if $s_{\pe r}(\P \M): (\Hdot_P^s)_r \to (\Hdot_P^s)_\pe$ is invertible.
 \item The two operators $s_{r \pe}(\P \M): (\Hdot_P^s)_\pe \to (\Hdot_P^s)_r$, $s_{\pe r}(\P \M): (\Hdot_P^s)_r \to (\Hdot_P^s)_\pe$ are invertible if and only if the four operators $1\pm s_{\pe \pe}(\P\M): (\Hdot_P^s)_\pe  \to (\Hdot_P^s)_\pe$, $1\pm s_{r r}(\P\M): (\Hdot_P^s)_r  \to (\Hdot_P^s)_r$ are invertible.
 \item In particular, well-posedness of $(R)_{\mE_s}^\Lop$ \emph{and} $(N)_{\mE_s}^\Lop$ on \emph{both} half-spaces is equivalent to simultaneous invertibility of the six operators $1 \pm s_{\pe \pe}(\P \M)$, $s_{\pe r}(\P\M)$, $s_{r \pe}(\P\M)$, $1 \pm s_{r r}(\P \M)$ acting between the respective components of $\Hdot_P^s$.
\end{enumerate}
\end{thm}

The experienced reader in parabolic boundary value problems may wonder about the relation between the spectral projections and the so-called Calder\'on projections. We shall come back to this in Section \ref{sec:invertibility layer}.

The proof of Theorem~\ref{thm:six invertibilities}  will be provided in Section~\ref{sec:sgnPM proofs}. Let us stress that this theorem \emph{cannot} be improved by purely algebraic reasoning on complementary projections.

First, we could have invertibility of solely $s_{r \pe}(\P\M)$ and failure of invertibility of one $1 \pm s_{\pe \pe}(\P \M)$ as well as one of $1 \pm s_{r r}(\P \M)$. Indeed, using a representation as in \eqref{eq:rep chi} for general pairs of complementary projections $\chi^\pm$, we already see from the simple example in $\R^2$,
\begin{align*}
\chi^+ := \begin{bmatrix}
    1  & 0   \\
    1  &  0
\end{bmatrix}, \qquad
\chi^- := \begin{bmatrix}
   0   & 0   \\
    -1  &  1
\end{bmatrix},
\end{align*}
that $s_{r \pe}$ can be invertible, while $1-s_{\pe \pe}$ and $1 + s_{r r}$ are not. In the opposite direction, let $\ell^2(\IN)$ be the space of complex square-summable sequences, let $S: (a_n) \mapsto (0, a_1,a_2,\ldots)$ be the right shift with $S^*: (a_n) \mapsto (a_2, a_3,a_4,\ldots)$ its adjoint and let $z \in \IC$ be a root of $z^2 - z +1$. Taking into account the relations $S^*S = 1$ and $z + \cl{z} = 1$, we see that on $\ell^2(\IN) \times \ell^2(\IN)$ we have a pair of complementary projections
\begin{align*}
 \chi^+ := \begin{bmatrix}
    z S S^*  & S   \\
    S^*  &  \cl{z} S^* S
\end{bmatrix}, \qquad
\chi^- := \begin{bmatrix}
   S^*S - z S S^*   & -S  \\
    -S^*  &  z S^*S
\end{bmatrix}
\end{align*}
for which both $1 \pm s_{r r}$ are invertible, while $s_{r \pe} = 2 S^*$ has one-dimensional nullspace and $s_{\pe r} = 2 S$ is not onto. 

Upon interchanging $\pe$- and $r$-coordinates, the same two examples give the same conclusions with the roles of $\pe$ and $r$ reversed.

In Section~\ref{sec:energy} we have obtained well-posedness of both boundary value problems for the energy class on the upper half-space but the discussion there applies verbatim to the lower one. These are precisely the problems $(R)_{\mE_{-1/2}}^\Lop$ \emph{and} $(N)_{\mE_{-1/2}}^\Lop$, see Section~\ref{sec:specialcases}. Hence, we can record

\begin{cor}
\label{cor:six invertibilities energy}
Invertibility of the six operators in Theorem~\ref{thm:six invertibilities} above always holds at regularity $s=-1/2$.
\end{cor}

The inverses for $s=-1/2$ have an illuminating interpretation in terms of boundary operators. Indeed, let us define the \emph{Neumann to Dirichlet operator} on the upper half-space by
\begin{align*}
 \Gamma_{ND}^{\Lop,+} f := \begin{bmatrix} \gradx \\ \HT \dhalf \end{bmatrix} u|_{\lambda = 0},
\end{align*}
where $u$ is the energy solution to the Neumann problem for $\Lop u = 0$ on $\R^{n+2}_+$ with data $f$. Owing to well-posedness, this defines an isomorphism $(\Hdot_{\P}^{-1/2})_\pe \to (\Hdot_{\P}^{-1/2})_r$. Its inverse is the \emph{Dirichlet to Neumann operator} $\Gamma_{DN}^{\Lop,+} g := \dnuA v|_{\lambda = 0}$, where $v$ is the energy solution to the Dirichlet problem for $\Lop v = 0$ on $\R^{n+2}_+$ with regularity data $g$. Hence, the parabolic conormal differential of an energy solution to $\Lop u = 0$ on the upper half-space has a trace in $\Hdot_\P^{-1/2} = \Hdot_{\P \M}^{-1/2}$ given by $\pcg u|_{\lambda = 0} = [f, \Gamma_{ND}^{\Lop +} f]$. However, due to Theorem~\ref{thm:sob} these traces also describe the spectral space $\Hdot_{\P\M}^{-1/2,+}$. Consequently,
\begin{align}
\label{eq:representation H-1/2+}
 \Hdot_{\P\M}^{-1/2,+} = \big \{h \in \Hdot_{\P}^{-1/2} : h_r = \Gamma_{ND}^{\Lop,+} h_\pe \big \},
\end{align}
which can be rephrased as saying that for $h \in \Hdot^{-1/2}_\P$ the conditions $\sgn(\P \M) h = h$ and $h_r = \Gamma_{ND}^{\Lop,+} h_\pe$ are equivalent. Unravelling the latter equivalence gives
\begin{align}
\label{eq:GammaND factorization}
 \Gamma_{ND}^{\Lop,+} = s_{\pe r}(\P\M)^{-1}(1-s_{\pe\pe}(\P\M)) = (1-s_{rr}(\P\M))^{-1}s_{r \pe}(\P\M)
\end{align}
and
\begin{align}
\label{eq:GammaDN factorization}
 \Gamma_{DN}^{\Lop,+} = (\Gamma_{ND}^{\Lop,+})^{-1} = s_{r \pe}(\P\M)^{-1}(1-s_{r r}(\P\M)) = (1-s_{\pe \pe}(\P\M))^{-1}s_{\pe r}(\P\M).
\end{align}
Similarly, akin operators for the lower half-space can be rediscovered within the functional calculus of $\P \M$. In Section~\ref{sec:sgnPM proofs} we shall use the operators $\Gamma_{ND}^{\Lop,\pm}$, $\Gamma_{DN}^{\Lop,\pm}$ to prove the following.

\begin{prop}
\label{prop:sgn invertibilty implies WP}
Let $s \in [-1,0]$.
\begin{enumerate}
 \item If $1\pm s_{\pe \pe}(\P \M)$ is invertible between the respective components of $\Hdot_\P^s$ \emph{and} the inverse agrees with the one for $s=-1/2$, then $(R)_{\mE_s}^\Lop$ is compatibly well-posed on $\R^{n+2}_\mp$ .
 \item If $1\pm s_{r r}(\P \M)$ is invertible between the respective components of $\Hdot_\P^s$ \emph{and} the inverse agrees with the one for $s=-1/2$, then $(N)_{\mE_s}^\Lop$ is compatibly well-posed on $\R^{n+2}_\mp$.
\end{enumerate}
\end{prop}

It is instructive to compare Proposition~\ref{prop:sgn invertibilty implies WP} with Theorem~\ref{thm:six invertibilities} and the succeeding remark. Indeed, the additional topological restriction that inverses should be compatible with the ones on energy solutions allows us to overcome the general algebraic obstructions to proving that invertibility of $1\pm s_{\pe \pe}(\P \M)$ (or $1\pm s_{r r}(\P \M)$) implies invertibility of $s_{r \pe}(\P \M)$ (or $s_{\pe r}(\P \M)$, respectively).
\subsection{Layer potentials}
\label{sec:layer intro}

We continue to carve out the link with classical theory of BVPs by introducing abstract layer potentials in the spirit of \cite{R1}, following the presentation in \cite[Section 12.3]{Auscher-Stahlhut_APriori}. We also go further by constructing explicitly the inverse of $\pd_{t}- \div_{\lambda,x}A(x,t)\nabla_{\lambda,x}$ on $\R^{n+2}$ from the layer potentials and without local DeGiorgi-Nash-Moser regularity assumptions.

In the following we designate by $\IP_{\M\P}$ the projection onto $\cl{\ran(\M\P)}$ along $\nul(\M\P)=\nul(\P)$ in \eqref{eq:kernel/range}, noting that the notation $\IP_{\P}$ is consistent with Section~\ref{sec:dirac}. We collect two further properties of the functional calculi that are required for defining the layer potentials rigorously. The first one follows since $\IP_{\M \P}$ and $\IP_\P$ share the same nullspace and the second one comes from general functional calculus of these operators, see \cite[Sec.~11]{Auscher-Stahlhut_APriori}.

\newcommand{\map}[3]{#1 \colon #2 \rightarrow #3}
{
\begin{lem}\label{lem:proj}

The restrictions $\map{\IP_{\M\P}}{\cl{\ran(\P)}}{ \cl{\ran(\M\P)}}$ and $\map{\IP_{\P}}{\cl{\ran(\M\P)}}{ \cl{\ran(\P)}}$ are isomorphisms and mutual inverses of one another. For every $-1\le s\le 0$, they extend to isomorphisms $\map{\IP_{\M\P}}{\Hdot^{s+1}_{\P} }{\Hdot^{s+1}_{\M\P} }$ and $\map{\IP_{\P}}{\Hdot^{s+1}_{\M\P} }{\Hdot^{s+1}_{\P} }$, respectively. Moreover, the \emph{intertwining property}
$$\P\IP_\P b (\M\P)\IP_{\M\P}h =  b(\P \M) \P h$$
is valid for all $h\in \Hdot^{s+1}_{\P}$ and all bounded holomorphic functions $b: \S_{\mu} \to \IC$.
\end{lem}}

Note that the holomorphic functions $z\mapsto \e^{-\lambda z}\chi^{\sgn \lambda}(z)$ are bounded on $\Re z\ne0$ when $ \lambda\ne 0$.  Thus, Lemma~\ref{lem:identification} shows that if $-1\le s\le 0$, then $e^{-\lambda\P\M} \chi^{\sgn\lambda}(\P \M)$ is bounded on $\Hdot^{s}_{\P}$, while $\P$ is an isomorphism from
$\Hdot^{s+1}_{\P}$ onto $\Hdot^{s}_{\P}$.

We regroup again $\pa$- and $\te$-components of vectors as a single component denoted $r$. For instance, for a scalar function $f\in \Hdot^{s/2}_{\pd_{t}-\Delta_{x}}$, we let $h=[f,  0]$ be the vector function in $\Hdot^{s}_{\P}$ with $h_\pe =f$ and $h_r = 0$. Throughout, we assume $-1 \le s \le 0$.
For $\lambda\in \R \setminus \{0\}$ and $f\in \Hdot^{s/2 + 1/2}_{\pd_{t}-\Delta_{x}}$, we define the single layer operator by
\begin{equation}
\label{eq:st}
\mS_{\lambda}f:= \begin{cases}
 -\bigg(P^{-1} \e^{-\lambda\P\M}\chi^{+}(\P\M) \begin{bmatrix} f \\ 0\end{bmatrix}\bigg)_{\pe} &(\text{if } \lambda>0), \\
 + \bigg(P^{-1} \e^{-\lambda\P\M}\chi^{-}(\P\M) \begin{bmatrix} f \\ 0\end{bmatrix}\bigg)_{\pe} &(\text{if } \lambda<0)
\end{cases}
\end{equation}
and for $f\in \Hdot^{s/2}_{\pd_{t}-\Delta_{x}}$ we define the double layer operator
\begin{equation}
\label{eq:dt}
\mD_{\lambda}f:=
\begin{cases}
 - \bigg(\IP_\P \e^{-\lambda\M\P} \chi^{+}(\M\P)\IP_{\M\P} \begin{bmatrix} f \\ 0\end{bmatrix}\bigg)_{\pe} &(\text{if } \lambda>0), \\
 + \bigg(\IP_\P \e^{-\lambda\M\P} \chi^{-}(\M\P)\IP_{\M\P} \begin{bmatrix} f \\ 0\end{bmatrix}\bigg)_{\pe} &(\text{if } \lambda<0).
\end{cases}
\end{equation}
Note that these definitions differ from the ones in \cite{Auscher-Stahlhut_APriori} in that we have introduced some additional projectors in \eqref{eq:dt} to  yield rigorous formul\ae\ for all Sobolev functions while  usage of the projectors is not necessary for $f$ in dense spaces, see \cite[Sec.~7.4]{Amenta-Auscher} for details. To avoid further confusion, let us notice an unfortunate typo in \cite{Auscher-Stahlhut_APriori}: the four formul\ae\ ($81$) - ($85$) come with $e^{+t \ldots}$ instead of $e^{-t \ldots}$ when $t<0$.

It follows from Lemma~\ref{lem:proj}, that  $\mS_{\lambda}$ is a bounded operator $\Hdot^{s/2}_{\pd_{t}-\Delta_{x}} \to \Hdot^{s/2+1/2}_{\pd_{t}-\Delta_{x}}$, while $\mD_{\lambda}$ is a bounded operator $\Hdot^{s/2+1/2}_{\pd_{t}-\Delta_{x}} \to \Hdot^{s/2+1/2}_{\pd_{t}-\Delta_{x}}$ with uniform norm with respect to $\lambda$. Moreover, they are all compatible on common dense subspaces.

We can relate the  conormal differentials of the layer operators to the Cauchy extension operators of \eqref{eq:Cauchyext}. By a direct calculation from the intertwining property in Lemma~\ref{lem:proj},

\begin{equation*}
\pcg \mS_{\lambda}f= \begin{cases}
+  \e^{-\lambda\P\M}\chi^{+}(\P\M) \begin{bmatrix} f \\ 0\end{bmatrix}  & (\text{if } \lambda>0), \\
- \e^{-\lambda\P\M}\chi^{-}(\P\M) \begin{bmatrix} f \\ 0\end{bmatrix} & (\text{if } \lambda<0)
\end{cases}
\end{equation*}
and
\begin{equation*}
\pcg \mD_{\lambda}f= \begin{cases}
 + \e^{-\lambda\P\M}\chi^{+}(\P\M)P \begin{bmatrix} f  \\ 0 \end{bmatrix}  & (\text{if } \lambda>0), \\
 -\e^{-\lambda\P\M}\chi^{-}(\P\M)P \begin{bmatrix} f \\ 0\end{bmatrix} & (\text{if } \lambda<0).
\end{cases}
\end{equation*}
Note that $P \begin{bmatrix} f  \\ 0 \end{bmatrix}$ has zero $\pe$-component and $r$-component equal to $\begin{bmatrix} -\nabla_{x}f \\ -\HT\dhalf f \end{bmatrix}$. In other words,
\begin{equation}
\label{eq:relationlayerCauchy}
\pcg \mS_{\lambda}f= \sgn(\lambda)\bigg(C_{0}  \begin{bmatrix} f  \\ 0 \end{bmatrix}\bigg)_{\lambda}, \quad  \pcg \mD_{\lambda}f= \sgn(\lambda)\bigg(C_{0} P  \begin{bmatrix} f  \\ 0 \end{bmatrix}\bigg)_{\lambda},  \qquad (\lambda\ne 0),
\end{equation}
where, as we recall, $C_0$ is the generalized Cauchy extension operator. These expressions, together with the correspondence stated in Theorem~\ref{thm:correspondence} and the boundedness of the Cauchy extensions, show that $\mD_{\lambda}f$ and $\mS_{\lambda}f$ are reinforced weak  solutions to \eqref{eq1}. Moreover, the link with the functional calculus of $\P\M$ provides us with existence of limits at $\lambda = 0^\pm$ including the jump relations and the duality relations between layer potentials.

\begin{prop}
\label{prop:propertieslayer}
Let $-1\le s\le 0$. Given $f\in \Hdot^{s/2}_{\pd_{t}-\Delta_{x}}$, the limits
\begin{align*}
 \pcg\mS_{0^+}f= \chi^{+}(\P\M)\begin{bmatrix} f \\ 0\end{bmatrix}, \qquad \pcg\mS_{0^-}f= -\chi^{-}(\P\M)\begin{bmatrix} f \\ 0\end{bmatrix}
\end{align*}
exist in the strong topology of $\Hdot^s_{\P}$ and hence the jump relation for the conormal derivatives and the no-jump relation for tangential derivatives are encoded as
\begin{equation*}
 \pcg\mS_{0^+}f -\pcg\mS_{0^-}f= \begin{bmatrix} f \\ 0\end{bmatrix}.
\end{equation*}
Given $f\in \Hdot^{s/2+1/2}_{\pd_{t}-\Delta_{x}}$, the limits
\begin{align*}
 \mD_{0^+}f = - \bigg(\IP_{\P} \chi^{+}(\M\P)\IP_{\M\P} \begin{bmatrix} f \\ 0\end{bmatrix}\bigg)_{\pe}, \qquad
 \mD_{0^-}f = +\bigg(\IP_{\P} \chi^{-}(\M\P)\IP_{\M\P} \begin{bmatrix} f \\ 0\end{bmatrix}\bigg)_{\pe}
\end{align*}
exist in the strong topology and hence
\begin{align*}
 \mD_{0^+}f -\mD_{0^-}f = -f, \qquad \mD_{0^+}f +\mD_{0^-}f= - \bigg(\IP_{\P} \sgn(\M\P)\IP_{\M\P} \begin{bmatrix} f \\ 0\end{bmatrix}\bigg)_{\pe}= :2Kf,
\end{align*}
with $K$ being the double layer potential at the boundary.

In addition, there are the duality formul\ae\   for $f\in \Hdot^{s/2}_{\pd_{t}-\Delta_{x}}$ and $g \in \Hdot^{-s/2-1/2}_{-\pd_{t}-\Delta_{x}}$ given by
\begin{equation*}
 \dual {g} {\mS_{\lambda}^\Lop f}= \dual { \mS_{-\lambda}^{\Lop^*}g} f,
\end{equation*}
and for $f\in \Hdot^{s/2+1/2}_{\pd_{t}-\Delta_{x}}$ and $g \in \Hdot^{-s/2-1/2}_{-\pd_{t}-\Delta_{x}}$ given by
\begin{equation*}
 \dual {g} {\mD_{\lambda}^\Lop f}= \dual {\partial_{{\nu}_{A^*}} \mS_{-\lambda}^{\Lop^*}g} f.
\end{equation*}
Here, $\mS_{\lambda}^\Lop$, $\mD_{\lambda}^\Lop$ are the layer potential operators defined above and associated with the parabolic equation \eqref{eq1}, while $\mS_{\lambda}^{\Lop^*}$ is the single layer operator associated with the backward parabolic equation \eqref{eq1*} defined using the correspondence with the operator $P^*\wM$ of Section~\ref{sec:backward}.
\end{prop}

Due to the relation \eqref{eq:relationlayerCauchy} between layer potentials and the Cauchy extension, there is always a  representation by (abstract) layer potential operators,
\begin{align*}
\pcg u(\lambda,\cdot) = \pcg\mS_{\lambda}(\dnuA u|_{\lambda=0}) - \pcg\mD_{\lambda}(u|_{\lambda=0})
\end{align*}
provided all terms are defined. More precisely, the integrated version of this is as follows.

\begin{thm}[Green's representation]
\label{thm:rep}
Let $-1 \leq s \leq 0$. Any solution $u$ of \eqref{eq1} with $\lambda$-independent coefficients such that $\pcg u$ belongs to the class $\mE_{s}$ can be represented as
\begin{align*}
u(\lambda,\cdot) = \mS_{\lambda}(\dnuA u|_{\lambda=0}) - \mD_{\lambda}(u|_{\lambda=0}) + c, \qquad (\lambda>0),
\end{align*}
where $c$ is a constant. The representation holds in $\C([0,\infty); \Lloc^2(\ree))$ and, moding out the constant, in $\C([0,\infty); \Hdot^{s/2+1/2}_{\pd_{t}-\Delta_{x}})$. In the lower half-space, with our convention of the conormal derivative, the formula reads 
\begin{align*}
u(\lambda,\cdot) = -\mS_{\lambda}(\dnuA u|_{\lambda=0}) + \mD_{\lambda}(u|_{\lambda=0}) + c, \qquad (\lambda<0).
\end{align*}
\end{thm}

Again, Proposition~\ref{prop:propertieslayer} and Theorem~\ref{thm:rep} hold as abstract results for operators of type $\P \M$ with a bounded holomorphic functional calculus, see \cite[Sec.~12.3]{Auscher-Stahlhut_APriori} for the proof in the context of elliptic systems and also \cite[Section~7.4]{Amenta-Auscher} for a cleaner presentation.

The  statement in Theorem~\ref{thm:rep} looks like the representation by the familiar Green's formula although we neither have defined a fundamental solution nor can perform integration by parts. In fact, at this level of generality, one can only construct an inverse, whose kernel is to become the fundamental solution when one has regularity 	assumptions on solutions. Nevertheless, our method says that \emph{all solutions of the BVPs must be obtained by Green's representation}, once the abstract operators have integral representations.

We present two different constructions of the inverse. Both inversion formul\ae \, will be proved in Section~\ref{sec:layer}. For the first one, we consider the operator $\Lop$ of \eqref{eq1}  with all variables $(\lambda, x, t)\in \R^{n+2}$ and coefficients independent of $\lambda$. The discussion of Section~\ref{sec:energy} adapts and {in virtue of the Lax-Milgram lemma $\Lop$ becomes an isomorphism from $\Hdot^{1/2}_{\pd_{t}-\Delta_{\lambda,x}}$ onto $\Hdot^{-1/2}_{\pd_{t}-\Delta_{\lambda,x}}$.

\begin{thm}
\label{thm:inverse1}
Assume that the coefficients do not depend on $\lambda$. The inverse of $\Lop$ is formally given by the convolution with the operator valued kernel $\mS_{\lambda}$. More precisely, let $E=\L^2(\R; \Hdot^{{-1/2}}_{-\pd_{t}-\Delta_{x}})$. Given $f \in \Hdot^{{-1/2}}_{\pd_{t}-\Delta_{\lambda,x}}$, the integrals
\begin{equation*}
(\Gamma_{\eps,R}f)_{\lambda}= \int_{\eps<|\lambda-\mu|<R} \mS_{\lambda - \mu} f_{\mu} \d\mu,
\end{equation*}
where $f_{\mu}:=f(\mu, \cdot)$, exists for each $0<\eps<R<\infty$ in $E'$. They define uniformly bounded operators $\Gamma_{\eps, R}: \Hdot^{{-1/2}}_{\pd_{t}-\Delta_{\lambda,x}}\to E'$ and
\begin{equation*}
\lim_{\eps\to 0, R\to \infty}{\Gamma_{\eps, R}f}= {\Lop^{-1}f}
\end{equation*}
holds in the sense of strong convergence in $E'$.
\end{thm}

Note that $\Hdot^{1/2}_{\pd_{t}-\Delta_{\lambda,x}}$ continuously embeds into $E'$ and so the last equality makes sense. In other words, we have approximated $\Lop^{-1}f$ not within $\Hdot^{1/2}_{\pd_{t}-\Delta_{\lambda,x}}$ but within the larger space $E'$. It does not seem as if $\Gamma_{\eps,R}$ or any modification of this convolution maps into $\Hdot^{1/2}_{\pd_{t}-\Delta_{\lambda,x}}$ except if $f\in E$. Nevertheless, this representation is in agreement with the natural interpretation of the single layer potential being the fundamental solution computed with {pole at} $\mu=0$.

One can also consider a parabolic operator with coefficients depending on all variables and construct its inverse with our methods. This gives us a definition in much greater generality than in the literature. Let us do that for $L =\pd_{t} - \div_{x} A_{\pa \pa}(x,t)\nabla_{x}$ with $(x,t)\in \R^{n+1}$.
Consider $\Lop= \pd_{t} - \div_{x} A_{\pa \pa}(x,t)\nabla_{x} - \pd_{\lambda}^2$ in $n+2$ variables and let $\mS_{\lambda}$ be the single layer potential operator associated with $\Lop$.

\begin{thm}
\label{thm:inverse2}
The inverse of $L: \Hdot^{1/2}_{\pd_{t}-\Delta_{x}} \to \Hdot^{-1/2}_{\pd_{t}-\Delta_{x}}$ is given by $L^{-1}=- \int_{-\infty}^\infty \mS_{\lambda}\d\lambda.$ More precisely, given $f \in \Hdot^{-1/2}_{\pd_{t}-\Delta_{x}}$, the integrals
\begin{equation*}
\int_{\eps\le |\lambda|\le R} \mS_{\lambda } f \d\lambda \qquad (0<\eps<R<\infty)
\end{equation*}
define uniformly bounded operators from $\Hdot^{-1/2}_{\pd_{t}-\Delta_{x}}$ to $\Hdot^{1/2}_{\pd_{t}-\Delta_{x}}$ and as a strong limit in $\Hdot^{1/2}_{\pd_{t}-\Delta_{x}}$,
\begin{equation*}
\lim_{\eps\to 0, R\to \infty}\int_{\eps\le |\lambda|\le R} \mS_{\lambda } f \d\lambda= - L^{-1}f.
\end{equation*}
\end{thm}

In this result, the approximation is within $\Hdot^{1/2}_{\pd_{t}-\Delta_{x}}$, but is not directly related to the layer potentials of $L$.
\subsection{Invertibility of layer potentials}
\label{sec:invertibility layer}

Let us denote the (abstract) layer potentials from the previous section with superscripts $\Lop$ or $\Lop^*$ to distinguish them for the forward and backward equation. So far, we have obtained their boundedness. Provided that the boundary layers are invertible, we can solve the Dirichlet/regularity problem for $\Lop u=0$ with boundary data $f$ (recall that these can be thought as the same problem in different topologies) by
\begin{equation}
\label{eq:invlayerdir1}
u= \mS_{\lambda}^\Lop (\mS_{0}^\Lop)^{-1}f
\end{equation}
or
\begin{equation}
\label{eq:invlayerdir2}
u=\mD_{\lambda}^\Lop (\mD_{0^+}^\Lop)^{-1}f,
\end{equation}
while the Neumann problem $\Lop u=0$ with Neumann data $g$ can be solved via
\begin{equation}
\label{eq:invlayerdir3}
u=\mD_{\lambda}^\Lop ( {\partial_{{\nu}_{A}}}\mD_{0}^\Lop)^{-1}g
\end{equation}
or
\begin{equation}
\label{eq:invlayerdir4}
u=\mS_{\lambda}^\Lop ({\partial_{{\nu}_{A}}}{\mS_{0^+}^{\Lop}})^{-1}g= \mS_{\lambda}^\Lop {({\mD_{0^-}^{\Lop^*})}^*}^{-1}g.
\end{equation}
Here we consider  $\lambda>0$. Since $\mS_{\lambda}^\Lop$ and ${\partial_{{\nu}_{A}}}\mD_{\lambda}^\Lop$ do not jump across the boundary $\lambda=0$, we do not need to use the $0^\pm$ symbol. On the other hand, the double layer potentials $\mD_{\lambda}^\Lop$ and $\mD_{\lambda}^{\Lop^*}=({\partial_{{\nu}_{A}}}{\mS_{-\lambda}^{\Lop}})^*$ jump. In the lower half-space, we would assume $\lambda<0$ and change $0^+$ to $0^-$ and vice versa. Hence, including the lower half-space in the discussion, we have eight possible formul\ae, depending on the invertibility of six boundary operators, which is reminiscent of the connection between simultaneous well-posedness of BVPs in both half-spaces and invertibility of the entries of $\sgn(\P \M)$ discussed in Section~\ref{sec:sgnPM results}.

This is of course no coincidence -- there is a direct translation of the results on $\sgn(\P \M)$ into the language of layer potentials. Indeed, setting $D_{r}f= [\gradx f, \HT\dhalf f]$ and using the formul\ae \, for $\pcg \mS_{\lambda}^\Lop$ and $\pcg \mD_{\lambda}^\Lop$ as well as the jump relations from Proposition~\ref{prop:propertieslayer}, the reader may readily check that
\begin{align}
\label{eq:rep sgn layer}
\begin{split}
\chi^+(\P\M)& =  \begin{bmatrix}
 \partial_{{\nu}_{A}} \mS_{0^+}^{\Lop}      &  -\partial_{{\nu}_{A}} \mD_{0}^{\Lop} D_r^{-1}    \\
  D_{r}\mS_{0}^{\Lop}    &  -D_{r}\mD_{0^+}^{\Lop}D_{r}^{-1}
\end{bmatrix}, \\
\chi^-(\P\M)& =  \begin{bmatrix}
- \partial_{{\nu}_{A}} \mS_{0^-}^{\Lop}      &  +\partial_{{\nu}_{A}} \mD_{0}^{\Lop}D_r^{-1}    \\
  -D_{r}\mS_{0}^{\Lop}    &  D_{r}\mD_{0^-}^{\Lop}D_{r}^{-1}
\end{bmatrix},\\
\sgn(\P\M)& =  \begin{bmatrix}
 \partial_{{\nu}_{A}} \mS_{0^+}^{\Lop}  + \partial_{{\nu}_{A}} \mS_{0^-}^{\Lop}    & -2 \partial_{{\nu}_{A}} \mD_{0}^{\Lop}D_r^{-1}   \\
  2D_{r}\mS_{0}^{\Lop}    &  -D_{r}(\mD_{0^+}^{\Lop}+\mD_{0^-}^{\Lop})D_{r}^{-1}
\end{bmatrix}.
\end{split}
\end{align}
Comparison with \eqref{eq:rep chi} reveals $\pm2\partial_{{\nu}_{A}} \mS_{0^\pm}^{\Lop} = 1\pm s_{\pe \pe}(\P \M)$, $-2 \partial_{{\nu}_{A}} \mD_{0}^{\Lop} = s_{\pe r}(\P \M)$, $2D_{r}\mS_{0}^{\Lop} = s_{r \pe}(\P \M)$ and $\mp 2 D_{r}\mD_{0^+}^{\Lop}D_{r}^{-1} = 1 \pm s_{r r}(\P \M)$. Hence, Theorem~\ref{thm:six invertibilities} and Corollary~\ref{cor:six invertibilities energy} in the language of layer potentials become

\begin{thm}
\label{thm:invertequiv}
Let $s \in [-1,0]$. The following hold true.
\begin{enumerate}
 \item $(R)_{\mE_s}^\Lop$ is well-posed on \emph{both} half-spaces if and only if $\mS_{0}^\Lop :\Hdot^{s/2}_{\pd_{t}-\Delta_{x}} \to \Hdot^{s/2+1/2}_{\pd_{t}-\Delta_{x}}$ is invertible.
 \item $(N)_{\mE_s}^\Lop$ \!is well-posed on \emph{both} half-spaces if and only if ${\partial_{{\nu}_{A}}}\mD_{0}^\Lop : \Hdot^{s/2+1/2}_{\pd_{t}-\Delta_{x}}
 \to \Hdot^{s/2}_{\pd_{t}-\Delta_{x}}$ is invertible.
 \item The two operators
 \begin{align*}
  \mS_{0}^\Lop :\Hdot^{s/2}_{\pd_{t}-\Delta_{x}} \to \Hdot^{s/2+1/2}_{\pd_{t}-\Delta_{x}}, \qquad
  {\partial_{{\nu}_{A}}}\mD_{0}^\Lop : \Hdot^{s/2+1/2}_{\pd_{t}-\Delta_{x}} \to \Hdot^{s/2}_{\pd_{t}-\Delta_{x}}
  \end{align*} are invertible if and only if the four operators
  \begin{align*}
   {\partial_{{\nu}_{A}}}{\mS_{0^\pm}^{\Lop}}: \Hdot^{s/2}_{\pd_{t}-\Delta_{x}} \to \Hdot^{s/2}_{\pd_{t}-\Delta_{x}}, \qquad
   \mD_{0^\pm}^\Lop : \Hdot^{s/2+1/2}_{\pd_{t}-\Delta_{x}} \to \Hdot^{s/2+1/2}_{\pd_{t}-\Delta_{x}},
  \end{align*}
 are invertible.
 \item In particular, well-posedness of $(R)_{\mE_s}^\Lop$ \emph{and} $(N)_{\mE_s}^\Lop$ on \emph{both} half-spaces is equivalent to simultaneous invertibility of the six boundary layer operators appearing in (iii).
 \item The six operators are always invertible when $s=-1/2$.
\end{enumerate}
\end{thm}

The direction from invertibility to \emph{solvability} is well-known in this field. Our characterisations of well-posedness allowed us to include \emph{uniqueness} and the full converse. This statement was implicit in the elliptic counterpart in \cite{Auscher-Mourgoglou}. 

The conormal derivative of the double layer potential appearing in (ii) is hardly used in the literature and indeed, (iii) makes clear that it needs not be considered for a simultaneous treatment of regularity and Neumann problems. This is in accordance with the case of a real, symmetric and $t$-independent matrix $A$, where invertibility of the boundary layers for $s=0$ can be achieved through a Rellich-type estimate
\begin{align*}
 \|\dnuA u|_{\lambda = 0}\|_2 \sim \|\nabla_x u|_{\lambda = 0}\|_2 + \|\HT \dhalf u|_{\lambda = 0}\|_2,
\end{align*}
at the boundary for $u = \mS_\lambda f$ or $u = \mD_\lambda f$. Its proof is typically carried out by integration by parts and thus applies equally well in both half-spaces. As a consequence of this, there are lower bounds for all six operators in Theorem~\ref{thm:invertequiv} and invertibility as well as compatibility with the inverses for the energy class can be obtained from the method of continuity, perturbing from the heat equation to $\Lop$ as in Proposition~\ref{prop:continuitymethod}. This method has been pioneered for elliptic equations in~\cite{V}. The parabolic version was recently obtained in \cite[Sec.~11.4]{N2}.

In the general situation, however, Neumann and regularity problems on both half-spaces might not be simultaneously well-posed and then invertibility of one boundary layer operator needs not be necessary for solving the corresponding problem. Indeed, invertibility of ${\mD_{0^\pm}^{\Lop}}:\Hdot^{s/2+1/2}_{\pd_{t}-\Delta_{x}} \to \Hdot^{s/2+1/2}_{\pd_{t}-\Delta_{x}}$ is sufficient for solving $(R)_{\mE_{s}}^\Lop$ on $\R^{n+2}_\pm$ via \eqref{eq:invlayerdir2} but it does \emph{not} follow from a purely algebraic reasoning that invertibility, or even lower bounds, are also necessary for having well-posedness of $(R)_{\mE_{s}}^\Lop$ solely in both half-spaces, see the remarks after Theorem~\ref{thm:six invertibilities}. The same discussion applies to $(N)_{\mE_{s}}^\Lop$ and $\mD_{0^\mp}^{\Lop^*}$.

To finish this section, we comment on the connection between the spectral projections and the Calder\'on projections.  We have seen that the entries of  $\chi^+(\P\M)$ are given in terms of  layer potentials but using the representation at the level of $\Hdot^{s}_{\P }$, $-1\le s\le 0$. However, if we use  the boundary space $\cH^s := \Hdot^{s/2}_{\pd_{t}-\Delta_{x}}\times \Hdot^{s/2+1/2}_{\pd_{t}-\Delta_{x}}$, then through the isometry $[ \psi,\omega]\mapsto [\psi, -D_{r}\omega]$ onto $\Hdot^{s}_{\P }$, the projection $\chi^+(\P\M)$ is similar to
\begin{align*}
 \begin{bmatrix}
 \partial_{{\nu}_{A}} \mS_{0^+}^{\Lop}      &  \partial_{{\nu}_{A}} \mD_{0}^{\Lop}    \\
  -\mS_{0}^{\Lop}    &  -\mD_{0^+}^{\Lop}
  \end{bmatrix},
\end{align*}
defined on $\cH^s$, which is the celebrated \emph{Calder\'on projection} for our parabolic equation on the upper half-space. The same can be done for the lower half-space. As for $\sgn(\P\M)$, it is similar to the involution on $\cH^s$ given by 
\begin{align*}
2\mathcal{A}_{\Lop} := \begin{bmatrix}
 \partial_{{\nu}_{A}} \mS_{0^+}^{\Lop}+  \partial_{{\nu}_{A}} \mS_{0^-}^{\Lop}& 2 \partial_{{\nu}_{A}} \mD_{0}^{\Lop}    \\
 -2 \mS_{0}^{\Lop}    &  -\mD_{0^+}^{\Lop}  -\mD_{0^-}^{\Lop}\end{bmatrix}.
\end{align*}
In \cite[Section 3]{C}, these calculations are carried out for the layer potentials in the case of the heat equation on a bounded Lipschitz cylinder when $s=-1/2$. The built-in algebra provided by the Dirac operator $\P\M$ hence gives us immediate access to this technology. This reference shows also a certain coercivity of the involution, which is not available here since $t$ lives in a non-compact interval. In fact, using the natural duality between $\cH^{-1/2}$ and its dual, one can only show that
\begin{align*}
  \Re \bigg \langle \mathcal{A}_{\Lop}  \begin{bmatrix}
       \psi  \\
       - \omega 
\end{bmatrix} , \begin{bmatrix}
      -\omega  \\
        \psi 
\end{bmatrix} \bigg \rangle 
= \Re \iiint_{\R^{n+2}_{+}} A\nabla u^+\cdot \overline{\nabla u^+} + \Re \iiint_{\R^{n+2}_{-}} A\nabla u^-\cdot \overline{\nabla u^-},
\end{align*}
where $u^\pm$ are the energy solutions to \eqref{eq1} for $\lambda$-independent coefficients in $\R^{n+2}_\pm$ with data corresponding to the splitting of $[\psi, D_{r} \omega] \in \Hdot^{-1/2}_{\P}$ in virtue of the spectral spaces $\Hdot^{-1/2,\pm}_{\P\M}$. So, we are missing the terms with half-order derivatives in time and there does not seem to be a hidden coercivity either, as already the example of the heat equation and calculations via Fourier transform show. We leave further details to the reader.  
[For the interested reader,  the very same  calculations work with the $DB$ formalism of \cite{Auscher-Stahlhut_APriori}  for the elliptic equation $-\div_{\lambda,x} A(x)\nabla_{\lambda,x} u = 0$ on both half-spaces showing that  the coercivity holds because the right hand side (integrals on $\R^{n+1}_{+}$) would be equivalent to the square of the norm in $\Hdot^{-1/4}_{\Delta_{x}}\times \Hdot^{1/4}_{\Delta_{x}}= \Hdot^{-1/2}(\R^n)\times \Hdot^{1/2}(\R^n)$ of the data using the spectral splitting.] 
\subsection{Well-posedness results for equations with \texorpdfstring{$\lambda$}{lambda}-independent coefficients}

Finally, we shall address well-posedness of the BVPs. All structural properties of the coefficients $A$ refer to the schematic representation \eqref{eq:A}.

\begin{thm}
\label{thm:WP}
\begin{enumerate}
 \item $(R)_{\mE_{s}}^\Lop$ and $(N)_{\mE_{s}}^\Lop$ are compatibly well-posed when $-1\le s \le 0$ and $A(x,t)$ has block structure.
 \item $(R)_{\mE_{s}}^\Lop$ is compatibly well-posed when $-1/2 \le s \le 0$ and $A(x,t)$ is upper triangular.
 \item $(R)_{\mE_{s}}^\Lop$ is compatibly well-posed when $-1 \le s \le -1/2$ and $A(x,t)$ is lower triangular.
 \item $(N)_{\mE_{s}}^\Lop$ is compatibly well-posed when $-1/2 \le s \le 0$ and $A(x,t)$ is lower triangular.
 \item $(N)_{\mE_{s}}^\Lop$ is compatibly well-posed when $-1\le s\le -1/2 $ and $A(x,t)$ is upper triangular.
 \item $(R)_{\mE_{s}}^\Lop$ and $(N)_{\mE_{s}}^\Lop$ are compatibly well-posed when $-1 \le s \le 0$ and $A(x,t)=A(x)$ is Hermitian.
 \item $(R)_{\mE_{s}}^\Lop$ and $(N)_{\mE_{s}}^\Lop$ are compatibly well-posed when $-1 \le s \le 0$ and $A(x,t)= A$ is constant.
\end{enumerate}
\end{thm}

We remark that additionally all ranges for compatible well-posedness are relatively open in $(-1,0)$ by Theorem~\ref{thm:extra} and all results are stable under $\L^\infty$-perturbations of the coefficients as in Theorem~\ref{thm:sta}. The proof of this theorem will be given in Section~\ref{sec:wpresults}.

Let us mention that (vi) was proved in \cite{N2} in the case when  $\Lop$ is  a real equation and $s={-1,0}$. Here, we may even allow systems of complex parabolic equations, see the discussion in Section~\ref{sec:systems} at the end of this paper. Similarly, (vii) is done in \cite{N2} when $s=-1, 0$. Item (i) is also proved as a corollary of the estimates in \cite{N1} when $\Lop$ has additionally $t$-independent coefficients. In (ii) and (iii) the triangular condition can be improved to the effect that the lower left and upper right coefficient of $A$ can be real-valued with vanishing $x$-divergence, respectively, since such a change of the representing matrix results in a new elliptic operator that could also be represented by a triangular matrix \cite[Rem.~6.7]{AAM}.

The results on well-posedness are identical for the backward equation with coefficients $A(x,t)$: In fact, it suffices to revert time and to observe that $A(x,-t)$ satisfies the same requirements as $A(x,t)$ in the list above. Likewise, the results remain valid for BVPs on the lower half-space: Here, it suffices to observe that by even reflection at the boundary we obtain a one-to-one correspondence with weak solutions on the upper half-space for an elliptic equation with coefficients
\begin{align*}
 \begin{bmatrix} -1\vphantom{A_{\pa \pa}} & 0 \\ 0\vphantom{A_{\pa \pa}} & 1 \end{bmatrix}
 \begin{bmatrix} A_{\pe \pe} & A_{\pe \pa} \\ A_{\pa \pe} & A_{\pa \pa} \end{bmatrix}
 \begin{bmatrix} -1\vphantom{A_{\pa \pa}} & 0 \\ 0\vphantom{A_{\pa \pa}} & 1 \end{bmatrix},
\end{align*}
sharing the same structural properties.

In combination with Theorem~\ref{thm:invertequiv} we obtain \emph{a posteriori} all possible representation formul\ae \, for the unique solution proposed by layer potential theory.

\begin{thm}
\label{thm:WP with invertlayer}
Let $s\in [-1,0]$ and $u$ be a solution to $\Lop u = 0$ with $\pcg u \in \mE_{s}$ (on either the upper or the lower half-space). All eight representations in \eqref{eq:invlayerdir1} - \eqref{eq:invlayerdir4} are valid if $A$ is either of block form, or constant, or Hermitian and independent of $t$, or  if $A$ is arbitrary and $s$ is sufficiently close to $-1/2$ (depending on $A$). Moreover, \eqref{eq:invlayerdir1} remains valid for the setup of Theorem~\ref{thm:WP}(ii)-(iii) and \eqref{eq:invlayerdir3} for that of Theorem~\ref{thm:WP}(iv)-(v).
\end{thm}

The Hermitian case (vi) will be treated in Section~\ref{sec:wpresults} by a Rellich-type argument. It will then become clear that the assumption of $t$-independence is necessary in order to compensate some lack of compactness, stemming from the fact that $t$ varies over the whole real line. If we drop this additional assumption, then the natural solution class changes. The following result, also proved in Section~\ref{sec:wpresults}, captures this effect.

\begin{prop}\label{prop:Dir} Assume $A(x,t)$ is Hermitian with $\dhalf A(x,t) \in \L^\infty(\R^n; \mathrm{BMO}(\R))$.
\begin{enumerate}
\item Let $f \in \Hdot^{1/2}_{{\pd_{t}-\Delta_{x}}}~\cap~\Hdot^{1/4}_{{\pd_{t}-\Delta_{x}}}$. Then there is a unique reinforced weak solution $u$ to \eqref{eq1} such that $u|_{\lambda=0}=f$ and $\pcg u \in \mE_0 \cap \mE_{-1/2}$.
\item Let $f\in \L^2(\R^{n+1}) + \Hdot^{1/4}_{{\pd_{t}-\Delta_{x}}}$. Then there is a unique reinforced weak solution $u$ to \eqref{eq1} with $u|_{\lambda=0}=f$ and the additional property that $u$ can be written as the sum of two solutions $u_{1}+u_{2}$ with $\pcg u_{1}\in \mE_{-1}$ and $\pcg u_{2}\in \mE_{-1/2}$.
\item Well-posedness in the same classes holds for the Neumann problem: For (i), $\dnuA u|_{\lambda = 0}$ is taken in $ \L^2(\R^{n+1})~\cap~\Hdot^{-1/4}_{{\pd_{t}-\Delta_{x}}}$ and for (ii) it is taken in $\Hdot^{-1/2}_{{\pd_{t}-\Delta_{x}}}+ \Hdot^{-1/4}_{{\pd_{t}-\Delta_{x}}}$.
\end{enumerate}
\end{prop}

Note that (i) is a regularity problem with restricted data. The dual problem (ii) looks like a curious Dirichlet problem but if the data is restricted to $f \in \L^2(\ree)$, it will have a strong implication on the (classical) $\L^2$ Dirichlet problem discussed in the next section.

As before, also Proposition~\ref{prop:Dir} remains valid for BVPs on the lower half space. The interested reader may verify that the transposition into the language of layer potentials yields the following.

\begin{prop}
\label{prop:invert layer}
Assume $A(x,t)$ is Hermitian with $\dhalf A(x,t) \in \L^\infty(\R^n; \mathrm{BMO}(\R))$.
\begin{enumerate}
 \item $\mS_{0}^\Lop:\Hdot^{0}_{\pd_{t}-\Delta_{x}}\cap \Hdot^{-1/4}_{\pd_{t}-\Delta_{x}} \to \Hdot^{1/2}_{\pd_{t}-\Delta_{x}} \cap \Hdot^{1/4}_{\pd_{t}-\Delta_{x}}$ is invertible and \eqref{eq:invlayerdir1} holds  for the solution (modulo constants) of  $\Lop u=0$ with  $\pcg u\in \mE_{0}\cap \mE_{-1/2}$ and boundary data $f\in \Hdot^{1/2}_{\pd_{t}-\Delta_{x}} \cap \Hdot^{1/4}_{\pd_{t}-\Delta_{x}}$.

 \item $\mD_{0^+}^\Lop:\L^2(\ree)+ \Hdot^{1/4}_{\pd_{t}-\Delta_{x}} \to \L^2(\ree) + \Hdot^{1/4}_{\pd_{t}-\Delta_{x}}$ is invertible and   \eqref{eq:invlayerdir2} holds  for the solution  of  $\Lop u=0$ with  $\pcg u\in \mE_{-1}+ \mE_{-1/2}$ and boundary data $f\in\L^2(\ree) + \Hdot^{1/2}_{\pd_{t}-\Delta_{x}}$.
\end{enumerate}
\end{prop}
\subsection{Uniqueness for Dirichlet problems}
\label{sec:UniqDir}

So far, the  Dirichlet problem  with $\L^2$ data is only formulated as $(R)_{\mE_{-1}}^\Lop$, that is, within the class of reinforced weak solutions having controlled square function norm
\begin{align*}
 \|\pcg u\|_{\mE_{-1}}^2=\iiint_{\R^{n+2}_+} |\lambda \nabla_{\lambda,x}u|^2+|\lambda \HT\dhalf u|^2 \, \frac{\mathrm{d} \lambda \d x\d t}{\lambda}<\infty
\end{align*}
and  the boundary behaviour can understood in the sense of $\L^2$ convergence of $u(\lambda,\cdot)$ when $\lambda\to 0$ since $u$ is continuous with values in $\L^2(\ree)$ as a consequence of Theorem \ref{thm:chardir}. As we have seen in Theorem~\ref{thm:NTmaxDir}, such solutions also satisfy the non-tangential maximal bound $\|\NT(u)\|_{2} \lesssim  \|\pcg u\|_{\mE_{-1}}$ and enjoy a form of non-tangential convergence for Whitney averages. Thus it makes sense to pose the question of uniqueness in the larger and more classical class involving non-tangential approach in the spirit of earlier references such as \cite{Hofmann-Lewis, HL1, N2}.

\begin{defn} We say that $(D)_{2}^\Lop$   is well-posed if given any $f\in \L^2(\ree)$ there is a unique weak solution $\Lop u=0$ which satisfies  $\|\NT(u)\|_{2}<\infty$ and  \begin{align*}
\bariiint_{\Lambda \times Q \times I} |u(\mu,y,s)- f(x,t) | \d \mu \d y \d s \to 0                                                                                                                                                                                                                                                                                                                                                           \end{align*}
almost everywhere when $\lambda\to 0$, where $\Lambda \times Q \times I$ and $(x,t)$ are related as in the definition of $\NT$. Uniqueness of $(D)_{2}^\Lop$  means that if a weak solution $\Lop u=0$  satisfies  $\|\NT(u)\|_{2}<\infty$ and tends to zero in the sense above when $\lambda\to 0$, then $u=0$ almost everywhere.
\end{defn}

Note that this problem is concerned with general weak solutions but we will see in Lemma~\ref{lem:NT implies reinforced} below that the non-tangential control \emph{a priori} forces $u$ to be reinforced.  By reverse H\"older estimates on weak solutions (see Lemma \ref{lem:RHu} below), we may replace $\L^2$-averages of $u$ by $\L^1$-averages up to taking slightly bigger Whitney regions. Hence, this is the largest possible class in which uniqueness may be asked for. When DeGiorgi-Nash-Moser regularity on solutions is assumed, the conditions in our definition amount to the usual pointwise non-tangential control and non-tangential convergence to the boundary almost everywhere. As our main result on $(D)_{2}^\Lop$ we will prove in Section~\ref{sec:Dirichlet} the following

\begin{thm}
\label{thm:uniqueDir}
Existence for the adjoint regularity problem  $(R)_{\mE_{0}}^{\Lop^*} $ (that is, with regularity data in $\L^2(\ree; \IC^{n+1})$)  implies uniqueness for $(D)_{2}^\Lop$.
\end{thm}

This result is in spirit of the earlier results in \cite{Ken-Pip} on elliptic equations, see also \cite[Lem.~4.31]{AAAHK}.  For parabolic equations, compare with \cite[Lem.~6.1]{N2}, where much more is assumed to get the same conclusion, including in particular DeGiorgi-Nash-Moser regularity as in all earlier references.

\begin{cor}
Assume  the adjoint regularity problem  $(R)_{\mE_{0}}^{\Lop^*} $ is (compatibly) well-posed.  Then so is $(D)_{2}^\Lop$. Moreover, given $f\in \L^2(\ree)$, the weak solution $u$ with data $f$ is reinforced, continuous in $\lambda$ valued in $\L^2(\ree)$ and satisfies
\begin{equation*}
\|\NT(u)\|_{2} \sim  \|\pcg u\|_{\mE_{-1}} \sim \sup_{\lambda>0} \|u(\lambda, \cdot)\|_{2}\sim \|f\|_{2}.
\end{equation*}
\end{cor}

Indeed, uniqueness in the corollary above follows from the previous theorem, while existence follows from the previous discussion and Theorem \ref{thm:duality} on duality of boundary value problems. The last two equivalences are consequences of Theorem~\ref{thm:correspondence} and Theorem~\ref{thm:sob} (or Theorem~\ref{thm:chardir}). For the first equivalence, the comparison ``$\lesssim$'' was already mentioned in Theorem~\ref{thm:NTmaxDir} while the inequality $\|f\|_{2}\lesssim \|\NT(u)\|_{2}$ is an easy consequence of the non-tangential convergence at the boundary. Compatible well-posedness was already stated in Theorem \ref{thm:duality} and is preserved as well. Comparing with Theorem \ref{thm:WP}, we also obtain

\begin{cor}
$(D)_{2}^\Lop$ is compatibly well-posed if $A$ is either of block form, lower triangular, Hermitian and time independent, or constant.
\end{cor}

Finally, it turns out that there also is a natural solution class defined by the non-tangential maximal function within which we have uniqueness in the setup of Proposition \ref{prop:Dir}. By this means, we can complete the discussion of the previous section by adding the following result for the Dirichlet problem with Hermitian time \emph{dependent} coefficients.

\begin{thm}
\label{thm:WPDirTimeDependent}
Assume $A(x,t)$ is Hermitian with $\dhalf A(x,t) \in \L^\infty(\R^n; \mathrm{BMO}(\R))$. Let $f\in \L^2(\R^{n+1})$. There is a unique weak solution $\Lop u = 0$ which satisfies $\|\NT(\frac{u}{\max\{1,\sqrt\lambda\}})\|_{2}<\infty$ and
\begin{align*}
 \bariiint_{\Lambda \times Q \times I} |u(\mu,y,s)- f(x,t) | \d \mu \d y \d s \to 0
\end{align*}
almost everywhere when $\lambda\to 0$. This solution is the same as in Proposition \ref{prop:Dir}(ii).
\end{thm}

We note that the growth in $\lambda$ in the non-tangential control is reflecting the regularity of $A$ in the assumption. The proof will also be given in Section~\ref{sec:Dirichlet}.
\subsection{Consequences for \texorpdfstring{$\lambda$}{lambda}-dependent coefficients}

There is an automatic improvement of our results to a situation where $A$ depends on the transversal variable $\lambda$ by treating $A$ in \eqref{eq1} as a perturbation of $\lambda$-independent coefficients $A_{0}$. This is done at the level of the first order equation, reformulating $\pd_{\lambda}F+\P\M F= 0$ as
\begin{equation}
\label{eq:perturbation}
 \pd_{\lambda}F+\P\M_{0} F=\P(\M_{0}-\M)F.
\end{equation}
This inhomogeneous problem can then be treated by maximal regularity results obtained in \cite{AA} in the cases $s=0$ and $s=-1$ and in \cite{R1} for the intermediate cases $-1<s<0$ but changing $DB$ to $\P\M$ therein.  We could have just stopped here but we felt that it is worth helping the reader making it through the transposition to our setup.

Generally, the estimates for $-1<s<0$ transpose immediately since they only use semigroup theory. The estimates for $s=-1$ and $s=0$ require more harmonic analysis. They are obtained using an abstract framework that consists of the following:
\begin{enumerate}
 \item A Hilbert space $\cH$, vector valued Lebesgue spaces $\cY:= \L^2(\R_+, \lambda \d\lambda; \cH)$, $\cY^*:= \L^2(\R_+, \mathrm{d}\lambda/\lambda; \cH)$ and a space $\cX$ with continuous embedding $\cY^* \subset \cX \subset \Lloc^2(\R_+, \mathrm{d} \lambda; \cH)$.
 \item A bisectorial operator $D B_0$ with a bounded holomorphic functional calculus on $\cH$ such that $h \mapsto (\e^{- \lambda [DB_0]}h)_{\lambda > 0}$ and $h \mapsto (\e^{- \lambda [D^*B_0^*]}h)_{\lambda > 0}$ are bounded from $\cH$ into $\cX$.
\end{enumerate}
As for our setup, we let $\cH = \cl{\ran(\P)}$, $\cX = \{F \in \Lloc^2(\R_+,\mathrm{d} \lambda; \cH): \|\NT(F)\|_2 < \infty \}$ and use $\P \M_0$ instead of $D B_0$. Then the embedding $\cY^* \subset \cX$ is classical, see also Lemma~\ref{lem:NT and QE} below. The $\cH \to \cX$ boundedness of the $\P\M$-semigroup follows from the non-tangential estimate stated in Theorem~\ref{thm:NTmax} and for the $\P^*\M$-semigroup we use the analogue for the backward equation, see also Section~\ref{sec:backward}. The fact that \cite{AA} works within Euclidean space and Lebesgue measure is of no harm either since only the doubling character of the measure is used. This is also the case for $\ree$ equipped with parabolic distance and Lebesgue measure. The needed theory on tent spaces above such spaces of homogeneous type has been written in detail in~\cite{Amenta-Tent}. The needed estimates for tent spaces with Whitney averages, proved in \cite{AA} and next in \cite{HR, Yi}, transpose to such a setting, see the final remark in \cite{Yi}. Below, we will summarize the main results informally and give an outline of the general strategy, leaving the care of checking details to the interested readers.

The condition imposed on $A(\lambda,x,t)$ is that there exists some $\lambda$-independent coefficients $A_{0}(x,t)$, uniformly elliptic and bounded, such that the \emph{discrepancy}
\begin{align*}
 \|A-A_{0}\|_{s}=\begin{cases}
 \|A-A_{0}\|_{\infty} & (\text{if } -1<s<0), \\
 \|A-A_{0}\|_{*} & (\text{if } s=-1,0)
\end{cases}
\end{align*}
is finite. Here, $\|\cdot\|_{\infty}$ is the $\L^\infty(\R^{n+2}_{+})$-norm and $ \|\cdot\|_{*} $ is the Carleson-Dahlberg measure norm defined by $\|\cdot\|_* := \| C_{2}(W_{\infty}(\cdot))\|_\infty$ with
\begin{align*}
 C_{2}(g)(x,t) := \sup_{R\ni (x,t)} \left(\frac 1{|R|} \int_0^{\ell(R)} \iint_R |g(\lambda,y,s)|^2 \d y \d s \, \frac{\mathrm{d} \lambda}{\lambda}\right)^{1/2},
\end{align*}
where the supremum is taken over all parabolic cubes $R$ with sidelength $\ell(R)$ in $\R^{n+1}$ containing $(x,t)$ and
$W_{\infty}(\cdot)(\lambda,y,s)$ is the essential supremum over the Whitney region given by $\Lambda=(\lambda/2, 2 \lambda)$, $Q=B(y, \lambda)$ and $I= (s-\lambda^2, s+\lambda^2)$. The condition $ \|A-A_{0}\|_{*}<\infty$ is strictly stronger than $\|A-A_{0}\|_{\infty}<\infty$. Note that $\|A-A_{0}\|_{s} \sim\|\M-\M_{0}\|_{s}$ where $\M_{0}$ corresponds to $A_{0}$ just as $\M$ corresponds to $A$. We also remark that $A_{0}$ is uniquely defined when using the $\|\, \cdot\, \|_*$-norm and the ellipticity constants relate to those of $A$. On the other hand, when using only the $\L^\infty$-norm there are many different possible $A_{0}$. Nevertheless we implicitly assume that the ellipticity constants of $A$ and $A_{0}$ are comparable.

The operator related to solving \eqref{eq:perturbation} with initial data equal to zero writes formally in vector-valued form as
\begin{equation*} \label{eq:firstformalSAdefn}
 (S_A F)_\lambda := \int_0^\lambda \e^{-(\lambda-\mu)\P\M_0} E_{0}^+ \P (\M_{0}-\M_\mu) F_\mu \d\mu - \int_\lambda^\infty \e^{(\mu-\lambda)\P\M_0} E_0^- \P (\M_{0}-\M_\mu) F_\mu\d\mu.
\end{equation*}
Here, $E_{0}^\pm=\chi^\pm(\P\M_{0})$ are the projections defined in Section~\ref{sec:dirac}. We also use the notation $F_{\lambda}=F(\lambda,\cdot)$. The
integral can be appropriately defined when $-1\le s \le 0$ and enjoys the norm estimate
\begin{align*}
\|S_{A}F\|_{\mE_{s}} \le C \|A-A_{0}\|_{s} \|F\|_{\mE_{s}},
\end{align*}
where the $\mE_s$-norms were defined in Section~\ref{sec:char} and $C$ depends on ellipticity, $s$ and dimension. Also $S_{A}$ maps into $\Lloc^2(\R_+; \clos{\ran(P)})$. Then it is indeed true that for any solution of \eqref{eq1} with $\pcg u \in \mE_{s}$ there exists a unique $h^+\in \Hdot^{s, +}_{\P\M_{0}}$ such that we have an implicit representation
\begin{equation}
\label{eq:rep-pert}
 (\pcg u)_{\lambda}= \e^{-\lambda\P\M_{0}}h^+ + (S_{A}\pcg u)_{\lambda}
\end{equation}
with equality in $\mE_{s}$ and $\|h^+\|_{\Hdot^s_{\P}}\sim \|\pcg u \|_{\mE_{s}} $ with implicit constants depending on dimension, $s$ and ellipticity. The boundary behaviour of $\pcg u$ at $\lambda =0$ is captured by $h^+$ and the complementary data
\begin{equation*}
\label{eq:h-}
 h^-:= -\int_0^\infty \e^{\mu \P\M_{0}}E_0^-\P (\M_{0}-\M_\mu) (\pcg u)_\mu \d\mu \in \Hdot^{s, -}_{\P\M_{0}}
\end{equation*}
satisfies for implicit constant depending on dimension, $s$ and ellipticity,
\begin{align}
\label{eq:h- est}
 \|h^-\|_{\Hdot^s_{\P}}\lesssim \|A-A_{0}\|_{s} \|\pcg u \|_{\mE_{s}}.
\end{align}
Then, writing $h:=h^++h^-$, there are limits
\begin{equation*}
\label{eq:SAavlim}
 \lim_{\lambda\to 0} \lambda^{-1} \int_\lambda^{2\lambda} \| \pcg u -h \|_{\Hdot^s_{\P}}^2 \d\mu =0=
 \lim_{\lambda\to \infty} \lambda^{-1} \int_\lambda^{2\lambda} \| \pcg u \|_{\Hdot^s_{\P}}^2 \d\mu.
\end{equation*}
In the case $s=-1$, there is a direct representation of $u$ by means of another operator \cite[Thm.~9.2]{AA}, showing that $u-c\in \C_{0}([0,\infty); \L^2(\ree))$ for some constant $c$. Moreover, \cite[Thm.~10.1]{AA} can be adapted as well in order to show the control of the non-tangential maximal function by the square function
\begin{align*}
 \| \NT (u-c)\|_{2}\lesssim \|\pcg u\|_{\mE_{-1}}
\end{align*}
for arbitrary reinforced weak solutions. This uses the dual estimate to \eqref{eq:off} {with $\M^* \P^*$ in place of $\P \M$} when it comes to off-diagonal decay.

When $s=0$, the Whitney averages of $\pcg u$ converge to $h$ almost everywhere on $\ree$. This is best done by showing $\int_0^\lambda \|\pcg u -Sh\|_2^2 \frac{\mathrm{d} \lambda}{\lambda} <\infty$ with $Sh$ the semigroup extension of $h$ and using Theorem~\ref{thm:NTmax} for $Sh$. If $-1<s\le 0$, then one can adapt the proof of Lemma~\ref{lem:energy Whitney trace} below to show that Whitney averages of $u$ (which can be chosen in $\Lloc^2(\R^{n+2}_{+})$) converge almost everywhere on $\R^{n+1}$ and the limit turns out to be an element $u_{0}\in\Hdot^{s/2 + 1/2}_{\pd_{t}-\Delta_{x}}\cap \Lloc^2(\ree)$. Thus, the boundary equation can be understood in $ \Lloc^2(\ree)$ with $\nabla_{x}u_{0}=h_{\pa}$ and $\HT\dhalf u_{0}=h_{\te}$.

So far, we have seen that all conormal differentials $\pcg u\in \mE_{s}$ of reinforced weak solutions to \eqref{eq1} have a trace $h=\pcg u |_{\lambda = 0}$ and that this trace maps $\mE_{s}$ boundedly into $\Hdot^s_{\P}$. If now $\|A-A_{0}\|_{s}$ is small enough, then $\|h\|_{\Hdot^s_{\P}}\sim \|h^+\|_{\Hdot^s_{\P}}$ as a consequence of \eqref{eq:h- est}. Hence,
\begin{align*}
 \|\pcg u \|_{\mE_{s}}\sim \|\pcg u |_{\lambda = 0}\|_{\Hdot^s_{\P}}.
\end{align*}
In this case, the trace map is an isomorphism onto its range. Let us call $\H^{s,+}_{A}$ the range of this map just for the discussion. With this notation, $\H^{s,+}_{A_{0}}= \H^{s,+}_{\P\M_{0}}$. The conclusion is that if $\|A-A_{0}\|_{s}$ is small enough, then the equation \eqref{eq:rep-pert} can be used to construct solutions as the operator $\id -S_{A}$ becomes invertible on $\mE_{s}$: given any $g\in \Hdot^{s}_{\P}$, we can construct a reinforced weak solution to \eqref{eq1} with $\pcg u\in \mE_{s}$ by
\begin{align*}
\pcg u= (\id -S_{A})^{-1} (C_{0}^+g)
\end{align*}
with $C_{0}^+g$ the Cauchy extension of $g$ using the operator $\P\M_{0}$. Set $E_{A}^+g=h$ to be the trace of $\pcg u$. Taking limits, we find $h^+=E_{0}^+g$, $\|h^-\|_{\Hdot^s_{\P}}\lesssim \|A-A_{0}\|_{s} \|g\|_{\Hdot^s_{\P}}$ and so we obtain that
\begin{align*}
E^+_{A}g= E_{0}^+g -\int_0^\infty \e^{\mu \P\M_{0}}E_0^-\P (\M_{0}-\M_\mu) (\id -S_{A})^{-1} (C_{0}^+g)_\mu \d\mu
\end{align*}
is a bounded projection of $\Hdot^s_{\P}$ onto $\H^{s,+}_{A}$ along $\Hdot^{s,-}_{\P\M_{0}} $. Moreover, there is a Lipschitz estimate
\begin{align*}
\|E_{A}^+-E_{0}^+\| \lesssim \|A-A_{0}\|_{s}.
\end{align*}

Turning eventually to boundary value problems, we can use $h^+$ to construct $u$ from a boundary data and obtain the following theorem.

\begin{thm}
\label{thm:perturbation}
Let $-1 \leq s \leq 0$. If $\|A-A_{0}\|_{s}$ is small enough, then the following assertions hold.
\begin{enumerate}
 \item $(R)_{\mE_{s}}^\Lop$ is well-posed if and only if $N_{r}: \Hdot^{s,+}_{\P\M_{0}} \to (\Hdot^{s}_{\P})_{r}$ is an isomorphism.
 \item $(N)_{\mE_{s}}^\Lop$ is well-posed if and only if $N_{\pe}: \Hdot^{s,+}_{\P\M_{0}} \to (\Hdot^{s}_{\P})_{\pe}$ is an isomorphism.
\end{enumerate}
\end{thm}

In fact, as we have characterized the traces of conormal derivatives of solutions in $\mE_s$ provided that $\|A-A_{0}\|_{s}$ is small enough, well-posedness of $(R)_{\mE_{s}}^\Lop$ is equivalent to invertibility of $N_r:\H^{s,+}_{A}\to (\Hdot^{s}_{\P})_r$. By the Lipschitz estimate on the difference of projectors, we may equivalently replace
$\H^{s,+}_{A}$ by $\H^{s,+}_{A_{0}}= \H^{s,+}_{\P\M_{0}}$. The proof for $(N)_{\mE_{s}}^\Lop$ is the same.

This result is not formulated like this in \cite{AA} or \cite{R1} but it is an illuminating formulation. Indeed, note that the right hand conditions mean well-posedness for the $\lambda$-independent operator $\Lop_{0}$ with coefficients $A_{0}$ by Theorem~\ref{thm:wpequiv}. Hence, this is also a perturbation result in that well-posedness for $A_{0}$ implies well-posedness for nearby $A$ in the norm $\|\cdot \|_{s}$. We remark that compatible well-posedness is also stable in this process. (This is a little tedious but not difficult using a Neumann series arguments when computing inverses.) In the same spirit we can use that the right hand sides in Theorem~\ref{thm:perturbation} are also stable under the change $(\Lop_{0},s)$ to $(\Lop_{0}^*,-1-s)$ by Theorem~\ref{thm:duality}, to obtain

\begin{cor}\label{cor:per}
Let $-1\le s\le 0$ and assume that $\|A-A_{0}\|_{s}$ is small enough. Let $\Lop$ and $\Lop^*$ be as above with coefficients $A$ and $A^*$, respectively.  Then $(\BVP)_{\mE_{s}}^\Lop$ is (compatibly) well-posed if and only if
$(\BVP)_{\mE_{-1-s}}^{\Lop^*}$ is (compatibly) well-posed, where $\BVP$ designates either $N$ or $R$.
\end{cor}
\section{Homogeneous fractional Sobolev spaces}
\label{sec:Homogeneous Sobolev spaces}

We introduce the homogeneous Sobolev space $\Hdot^{1/2}(\R)$ and related spaces in several variables. It will be necessary to have them realised not only within the tempered distributions modulo polynomials but, equivalently and more concretely, within $\Lloc^2(\R)$ by means of a Riesz potential. In doing so we have to deviate from common literature \cite{BL, Adams-Hedberg}, see also \cite{Bourdaud}, since $1/2$ is a critical exponent for homogeneous $\L^2$-Sobolev space on the real line. For the reader's convenience we shall provide all necessary details.

In the following $\dhalf$ denotes the \emph{half-order time derivative}, defined $\L^2(\R) \to \mS'(\R)$ by the Fourier multiplier $\tau \mapsto |\tau|^{1/2}$, and $\HT$ is the \emph{Hilbert transform} corresponding to $\tau \mapsto \i \sgn(\tau)$. The sign convention for the Hilbert transform is in accordance with the textbook \cite{SamkoEtAl}.

\subsection{Spaces on the real line}

For $v \in \mS(\R)$ there are well-known kernel representations
\begin{align}
 \label{eq:kernel of dhalf}
 \dhalf v(t) = \frac{1}{2 \sqrt{2 \pi}} \int_\R \frac{1}{|t-s|^{3/2}} (v(t)-v(s)) \d s,
\end{align}
and
\begin{align}
 \label{eq:kernel of HTdhalf}
 \HT \dhalf v(t) = \frac{1}{2 \sqrt{2 \pi}} \int_\R \frac{\sgn(t-s)}{|t-s|^{3/2}} (v(t)-v(s)) \d s,
\end{align}
valid for a.e.\ $t \in \R$, see \cite[\SS 12.1]{SamkoEtAl}. These formul{\ae} can be used to prove that $\dhalf v$ and $\HT \dhalf v$ are integrable functions satisfying
\begin{align}
 \label{eq:decay HTdhalf}
 |\dhalf v(t)| + |\HT \dhalf v(t)| \leq C \min \{1, |t|^{-3/2} \} \qquad (t \in \R),
\end{align}
where $C$ depends only on the size of $v$ in the topology of $\mS(\R)$. In fact, for boundedness we would split the integral at height $|t - s| =1$ and treat both terms separately and in order to reveal the decay we would split at $|t-s| = |t|/2$. We will freely use these facts in the following without referring to them at each occurrence.

\begin{defn}
\label{defn: Hdot1/2}
The space $\Hdot^{1/2}(\R)$ consists of all $v \in \Lloc^2(\R) \cap \mS'(\R)$ such that firstly for every $\phi \in \mS(\R)$ the product $v \dhalf \phi \in \L^1(\R)$ and secondly there is $g \in \L^2(\R)$ with
\begin{align*}
 \int_\R v \cdot \clos{\dhalf \phi} \d t = \int_\R g \cdot \clos{\phi} \d t \qquad(\phi \in \mS(\R)).
\end{align*}
In this case we define $\dhalf v := g$ and $\|v\|_{\Hdot^{1/2}(\R)} := \|\dhalf v\|_{\L^2(\R)}$.
\end{defn}

\begin{rem}
\label{rem:consistency Hdot12}
For $v \in \L^2(\R)$ we have $v \in \Hdot^{1/2}(\R)$ if and only if $\wh{g} := |\tau|^{1/2} \wh{v} \in \L^2(\R)$. Hence, $\L^2(\R) \cap \Hdot^{1/2}(\R) = \H^{1/2}(\R)$ is the usual inhomogeneous fractional Sobolev space.
\end{rem}

There is an important growth estimate implicit in Definition~\ref{defn: Hdot1/2}.

\begin{lem}
\label{lem:seminorm on Hdot12}
Any $v \in \Hdot^{1/2}(\R)$ satisfies $\int_{\R} |v(t)| \, \frac{\mathrm{d} t}{1+|t|^{3/2}} < \infty$ and $\dhalf v = 0$ holds if and only if $v$ is constant.
\end{lem}

\begin{proof}
The estimate follows from $v \in \Lloc^1(\R)$ and since for any $\phi \in \C_0^\infty(\R)$ positive and non-zero at some point, \eqref{eq:kernel of dhalf} yields $|\dhalf \phi(t)| \sim |t|^{-3/2}$ for $|t|$ large. Clearly the constant functions are in $\Hdot^{1/2}(\R)$ and satisfy $\dhalf v = 0$. Conversely, let $v \in \Hdot^{1/2}(\R)$ be of that kind. We can write any $\psi \in \mS(\R)$ with Fourier transform supported away from $0$ as $\psi = \dhalf \phi$ for some $\phi \in \mS(\R)$. Hence $\langle v, \psi \rangle = 0$, showing that $\wh{v}$ is supported at the origin. Thus, $v$ is a polynomial and according to our estimate it must be constant.
\end{proof}

Further properties of $\Hdot^{1/2}(\R)$ follow from  a representation in terms of the Riesz-type potential
\begin{align*}
 \ihalf g(t) = \frac{1}{\sqrt{2 \pi}} \bigg(&\int_{(-1,1)} \frac{1}{|t-s|^{1/2}} g(s) \d s \\
 &+ \int_{^c (-1,1)} \bigg(\frac{1}{|t-s|^{1/2}} - \frac{1}{|s|^{1/2}} \bigg) g(s) \d s \bigg),
\end{align*}
where typically $g \in \L^2(\R)$. A similar modification of the classical Riesz potential was considered in \cite{Kurokawa_Riesz}. A few words concerning well-definedness: the Hardy-Littlewood-Sobolev inequality \cite[Sec.~V.1.2]{Stein} ensures that the integral over $(-1,1)$ converges absolutely for a.e.\ $t \in \R$ and the same argument applies to the integral over $^c(-1,1)$ near the singularity of the kernel. On the remaining part this kernel is controlled by $(1+|s|)^{-3/2}$ and the Cauchy-Schwarz inequality applies. This argument reveals that $\ihalf : \L^2(\R) \to \Lloc^2(\R)$ is bounded.

\begin{prop}
\label{prop:potential properties}
Let $g \in \L^2(\R)$ and $\phi \in \mS(\R)$. If $T$ denotes any of the operators $\id$, $\HT$, or $\HT \dhalf$, then $\ihalf g \cdot \dhalf T \phi \in \L^1(\R)$ and
\begin{align*}
 \int_\R \ihalf g \cdot \clos{\dhalf T \phi} \d t = \int_\R g \cdot \clos{T \phi} \d t.
\end{align*}
Moreover, $\ihalf \HT g$ agrees up to a constant with the potential
\begin{align*}
 -\frac{1}{\sqrt{2 \pi}} \bigg(&\int_{(-1,1)} \frac{\sgn(t-s)}{|t-s|^{1/2}} g(s) \d s + \int_{^c (-1,1)} \bigg(\frac{\sgn(t-s)}{|t-s|^{1/2}} - \frac{\sgn(s)}{|s|^{1/2}} \bigg) g(s) \d s \bigg).
\end{align*}
\end{prop}

\begin{proof}
For $g \in \mS(\R)$ a calculation, using the Riemann integral $\int_{-\infty}^\infty \e^{\i s}|s|^{-1/2} \d s = \sqrt{2 \pi}$, reveals
\begin{align*}
\int_\R \frac{\e^{\i t \tau}}{|\tau|^{1/2}} \wh{g}(\tau) \d \tau
= \sqrt{2 \pi} \int_\R \frac{1}{|t-s|^{1/2}} g(s) \d s.
\end{align*}
Hence, in this case $\ihalf g$ agrees with the following function of $t$ up to a constant:
\begin{align}
\label{eq1:potential properties}
\frac{1}{2 \pi} \int_{(-1,1)} \frac{\e^{\i t \tau} -1}{|\tau|^{1/2}} \wh{g}(\tau) \d \tau + \cF^{-1}\bigg(1_{^c(-1,1)}(\tau)\frac{\wh{g}(\tau)}{|\tau|^{1/2}} \bigg)(t).
\end{align}
This representation of $\ihalf g$ is better suited for the extension to general $g \in \L^2(\R)$. In fact, define $v_1$ to be the first function in \eqref{eq1:potential properties} and $v_2$ the second. Due to Plancherel's theorem $\|v_2\|_2 \leq \|g\|_2$ and, using $|\e^{\i t \tau} - 1| \leq 2|t \tau|^{1/4}$ for $t, \tau \in \R$,
\begin{align}
\label{eq2:potential properties}
 |v_1(t)|
 \leq \frac{1}{2 \pi} \int_{(-1,1)} \frac{|\e^{\i t \tau} -1|}{|\tau|^{1/2}} |\wh{g}(\tau)| \d \tau
 \leq 4|t|^{1/4} \|g\|_2 \qquad (t \in \R).
\end{align}
This proves that the maps $g \mapsto v_j$ are bounded $\L^2(\R) \to \Lloc^2(\R)$, just as is $\ihalf $. Thus, it remains true for all $g \in \L^2(\R)$ that the function defined in \eqref{eq1:potential properties} coincides with $\ihalf g$ up to a constant. The potential representation for $\ihalf \HT g$ follows by a similar reasoning with $\HT g$ in place of $g$ using the Riemann integral $\int_{-\infty}^\infty \e^{\i s}|s|^{-1/2} \sgn(s) \d s = \i \sqrt{2 \pi}$.

To prove the actual claim, we let $g \in \L^2(\R)$ and keep on denoting the two summands in \eqref{eq1:potential properties} by $v_1$ and $v_2$. In the cases $T = \id$ and $T = \HT$ we recall from \eqref{eq:decay HTdhalf} the bound
\begin{align*}
 |\dhalf T \phi(t)| \lesssim \min \{ 1, |t|^{-3/2} \} \qquad (t \in \R),
\end{align*}
with implicit constants depending only on the size of $\phi$ in $\mS(\R)$. In the case $T = \HT \dhalf$ we even have $\dhalf T \phi=\pd_{t}\phi \in \mS(\R)$ and so this bound holds all the more. After the preceding discussion we may directly argue on $v_1 + v_2$ instead of $\ihalf g$. Due to \eqref{eq1:potential properties},
\begin{align*}
 \int_\R |v_1(t)| |\dhalf T\phi(t)| \d t \leq 4 \|g\|_2 \int_\R |t|^{1/4} |\dhalf T\phi(t)| \d t < \infty,
\end{align*}
thereby justifying the use of Fubini's theorem when computing
\begin{align*}
 \int_\R v_1(t) \cdot \clos{\dhalf T \phi(t)} \d t
 &= \int_{(-1,1)} \frac{\wh{g}(\tau)}{|\tau|^{1/2}} \cdot \clos{\int_\R \dhalf T \phi(t) \e^{-\i t \tau} \d t} \d \tau \\
 &= 2 \pi \int_{(-1,1)} \wh{g}(\tau) \cdot \clos{\wh{T \phi}(\tau)} \d \tau,
\end{align*}
where in the first step we have used that $\dhalf T \phi$ has integral zero since its Fourier transform vanishes at the origin. Concerning $v_2$, Plancherel's theorem yields
\begin{align*}
 \int_\R |v_2(t)| |\dhalf T\phi(t)| \d t \leq \|g\|_2 \|\dhalf T\phi\|_2 < \infty
\end{align*}
and
\begin{align*}
 \int_\R v_2(t) \cdot \clos{\dhalf T \phi(t)} \d t = 2 \pi \int_{^c (-1,1)} \wh{g}(\tau) \cdot \clos{\wh{T \phi}(\tau)} \d \tau.
\end{align*}
Adding up the corresponding estimates and identities together with a final application of Plancherel's theorem leads us to the claim.
\end{proof}

The first consequence of Proposition~\ref{prop:potential properties} is a Riesz potential representation for $\Hdot^{1/2}(\R)$. It also allows to improve on Lemma~\ref{lem:seminorm on Hdot12}  obtaining a  continuous embedding, which is apparently new.

\begin{cor}
\label{cor:potential representation of Hdot1/2}
If $g \in \L^2(\R)$, then $\ihalf g \in \Hdot^{1/2}(\R)$ and $\dhalf \ihalf g = g$. In particular, for any $v \in \Hdot^{1/2}(\R)$ there exist $c, d \in \IC$ such that $v = c + \ihalf \dhalf v$ and
\begin{align*}
 \int_\R |v(t) - d| \, \frac{\mathrm{d} t}{1+|t|^{3/2}} \leq 42 \|v\|_{\Hdot^{1/2}(\R)}.
\end{align*}
\end{cor}

\begin{proof}
The first statement follows from Proposition~\ref{prop:potential properties} applied with $T=1$. This implies the representation of $v$ by taking $g = \dhalf v$ and then using Lemma~\ref{lem:seminorm on Hdot12}. From the proof of Proposition~\ref{prop:potential properties} we know that \eqref{eq1:potential properties} with $g = \dhalf v$ reproduces $v$ up to a constant $d \in \IC$ and therefore the final estimate follows from the bounds for $v_1$ and $v_2$ obtained therein, see \eqref{eq2:potential properties} and the lines above.
\end{proof}

Next, Proposition~\ref{prop:potential properties} with $T = \HT$ and $T = \HT \dhalf$ yields the fractional integration by parts formul\ae \ that justify to speak of reinforced weak solutions in the first place.

\begin{cor}
\label{cor:fractional integration by parts}
Let $v \in \Hdot^{1/2}(\R)$ and $\phi \in \mS(\R)$. Then $v \HT \dhalf \phi \in \L^1(\R)$ and there are fractional integration by parts formul\ae
\begin{align*}
 \int_\R v \cdot \clos{\HT \dhalf \phi} \d t = - \int_\R \HT \dhalf v \cdot \clos{\phi} \d t
\end{align*}
and
\begin{align*}
 \int_\R v \cdot \clos{\partial_t \phi} \d t = - \int_\R \HT \dhalf v \cdot \clos{\dhalf \phi} \d t.
\end{align*}
\end{cor}

The kernel representations \eqref{eq:kernel of dhalf} and \eqref{eq:kernel of HTdhalf} remain valid for $v \in \Hdot^{1/2}(\R)$ in a certain sense.

\begin{cor}
\label{cor:formula HTdhalf on Hdot1/2}
If $v \in \Hdot^{1/2}(\R)$, then for almost every $t \in {}^c(\supp v)$,
\begin{align*}
 \dhalf v(t) = - \frac{1}{2 \sqrt{2 \pi}} \int_\R \frac{v(s)}{|t-s|^{3/2}} \d s \qquad
 \HT \dhalf v(t) = - \frac{1}{2 \sqrt{2 \pi}} \int_\R \frac{\sgn(t-s) v(s)}{|t-s|^{3/2}} \d s.
\end{align*}
\end{cor}

\begin{proof}
Let $\phi \in \C_0^\infty(\R)$ be supported in $^c(\supp v)$. For the first formula we start out with
\begin{align*}
 \int_\R \dhalf v(t) \cdot \clos{\phi(t)} \d t
 = \int_\R v(t) \cdot \clos{\dhalf \phi(t)} \d t,
\end{align*}
represent $\dhalf \phi$ via \eqref{eq:kernel of dhalf}, and apply Fubini's theorem (justified by Lemma~\ref{lem:seminorm on Hdot12}), to obtain
\begin{align*}
\int_\R \dhalf v(t) \cdot \clos{\phi(t)} \d t
 = \int_\R \bigg(\frac{-1}{2 \sqrt{2 \pi}} \int_\R \frac{v(s)}{|t-s|^{3/2}} \d s\bigg) \clos{\phi(t)} \d t.
\end{align*}
For the second formula we start out with the first identity in Corollary~\ref{cor:fractional integration by parts} instead.
\end{proof}

We stress that our, seemingly new, definition of $\Hdot^{1/2}(\R)$ is in accordance with the relevant literature: It realises the closure of $\C_0^\infty(\R)$ for the homogeneous norm $\|\dhalf \cdot \|_{\L^2(\R)}$.

\begin{cor}
\label{cor:Hdot1/2 complete}
The space $\Hdot^{1/2}(\R)/\IC$ is a Hilbert space for the norm $\|\cdot\|_{\Hdot^{1/2}(\R)}$ and it contains $\C_0^\infty(\R)$ as a dense subspace.
\end{cor}

\begin{proof}
In view of Lemma~\ref{lem:seminorm on Hdot12}, $\|\cdot\|_{\Hdot^{1/2}(\R)}$ is a Hilbertian norm on $\Hdot^{1/2}(\R)/\IC$. Due to Corollary~\ref{cor:potential representation of Hdot1/2} we have an isomorphism
\begin{align*}
 \ihalf : \L^2(\R) \to \Hdot^{1/2}(\R) / \IC.
\end{align*}
Thus, $\Hdot^{1/2}(\R)/\IC$ is a Hilbert space. $\L^2$-functions with Fourier transform supported away from $0$ are dense in $\L^2(\R)$, hence their image under $\ihalf $ is dense in $\Hdot^{1/2}(\R)/\IC$. For such $g$, however, \eqref{eq1:potential properties} defines the $\L^2(\R)$-function $\cF^{-1}(|\tau|^{-1/2} \wh{g}(\tau))$ up to a constant. This proves that the inhomogeneous space $\H^{1/2}(\R) = \L^2(\R) \cap \Hdot^{1/2}(\R)$ is dense in $\Hdot^{1/2}(\R)/\IC$. The density of $\C_0^\infty(\R)$ in $\H^{1/2}(\R)$ is classical and can be proved by truncation and convolution with smooth kernels.
\end{proof}

\subsection{Spaces of several variables}
\label{sec:Homogeneous Sobolev spaces several}

The notions introduced so far can be extended straightforwardly to functions of several variables.

\begin{defn}
\label{defn: Hdot12 several variables}
The space $\Hdot^{1/2}(\R; \Lloc^2(\reu))$ consists of all $v \in \Lloc^2(\R^{n+2}_+)$ such that $v(\lambda,x, \cdot) \in \Hdot^{1/2}(\R)$ for almost every $(\lambda, x) \in \reu$ and such that
\begin{align*}
 \iint_K \int_{\R} |\dhalf v(\lambda,x,t)|^2 \d t \d x\d \lambda <\infty
\end{align*}
for every compact set $K \subset \reu$. Similarly, $\Hdot^{1/2}(\R; \Lloc^2(\R^n))$ is defined as a subspace of $\Lloc^2(\ree)$.
\end{defn}

There is an analogue of the embedding stated in Corollary~\ref{cor:potential representation of Hdot1/2} for the spaces in several variables.

\begin{lem}
\label{lem:Hdot12 several variables}
Every $v \in \Hdot^{1/2}(\R; \Lloc^2(\reu))$ satisfies $v \dhalf \phi, \, v \HT \dhalf \phi \in \L^1(\R^{n+2}_+)$ if $\phi \in \mS(\R^{n+2}_+)$ is compactly supported in $\lambda$ and $x$. A similar statement holds for $\Hdot^{1/2}(\R; \Lloc^2(\R^n))$.
\end{lem}

\begin{proof}
Corollary~\ref{cor:potential representation of Hdot1/2} and the growth estimate \eqref{eq:decay HTdhalf} yield the global integrability conditions with $v-d$ in place of $v$, where $d \in \Lloc^1(\R^{n+2}_+)$ does not depend on $t$. However, $d \dhalf \phi$ and $d \HT \dhalf \phi$ are in $\L^1(\R^{n+2}_+)$ by the support assumption on $\phi$.
\end{proof}

In view of Definition~\ref{defn: Hdot12 several variables} the ground space $\dot{\E}_{\loc}$ for reinforced weak solutions,
\begin{align*}
 \dot{\E}_{\loc}(\R^{n+2}_+) = \Hdot^{1/2}(\R; \L^2_{\loc}(\reu)) \cap \Lloc^2(\R; \W^{1,2}_{\loc}(\reu)),
\end{align*}
is unambiguously defined as a subspace of $\Lloc^2(\R^{n+2}_+)$ and within there is the energy space $\dot \E(\R^{n+2}_+)$ of those $v \in \dot \E_{\loc}(\R^{n+2}_+)$ for which
\begin{align*}
 \|v\|_{\dot \E(\R^{n+2}_+)} := \Big( \|\dhalf v\|_{\L^2(\R^{n+2}_+)}^2 + \|\gradlamx v\|_{\L^2(\R^{n+2}_+)}^2 \Big)^{1/2} < \infty.
\end{align*}
Note that in Section~\ref{sec:main results} we simply wrote $\dot \E_{\loc}$ and $\dot \E$. We stress the dependence of the domain since for the time being we shall concentrate on the analogous spaces of functions in the variables $x$ and $t$ only. Similar to the situation on the real line, there is a corresponding parabolic Riesz potential given by
\begin{align*}
 \ipar g(x,t) = \iint_{|\xi| \vee |\tau| < 1} \frac{\e^{\i x \cdot \xi} \e^{\i t \tau} \wh{g}(\xi, \tau) }{|\xi| + |\tau|^{1/2}} \d \xi \d \tau + \cF^{-1}\bigg(\frac{1_{\{|\xi| \vee |\tau| \ge 1\}} \wh{g}(\xi, \tau)}{|\xi| + |\tau|^{1/2}} \bigg).
\end{align*}
Here, $g \in \L^2(\ree)$, $\xi, \tau$ are the Fourier variables corresponding to $x, t$ and $a\vee b= \max (a,b)$.

\begin{lem}
\label{lem:parabolic potential}
Let $g \in \L^2(\ree)$ and let $v = \ipar g$. Then $v \in \dot \E(\ree)$ with derivatives
\begin{align*}
 \begin{bmatrix}
  \gradx v \\ \dhalf v
 \end{bmatrix}
= \cF^{-1} \left(\frac{1}{|\xi| + |\tau|^{1/2}} \begin{bmatrix} \i \xi \\ |\tau|^{1/2} \end{bmatrix} \wh{g} \right)
\end{align*}
and estimates $\|v\|_{\L^2 + \L^\infty} \lesssim \|g\|_{\L^2}$ and $\|v\|_{\dot \E(\ree)} \sim \|g\|_{\L^2}$. Conversely, each $v \in \dot \E(\ree)$ can be represented as $v = c +\ipar g$ for unique $g \in \L^2(\ree)$ and $c \in \IC$.
\end{lem}

\begin{proof}
Let us denote the two summands in the definition of $v = \ipar g$ by $v_1$ and $v_2$. Thanks to Plancherel's theorem we have $\|v_2\|_2 \leq \|g\|_2$ and
\begin{align*}
 |v_1(t)|
 \leq \iint_{|\xi| \vee |\tau| < 1} \frac{|\wh{g}(\xi, \tau)|}{|\xi| + |\tau|^{1/2}} \d \xi \d \tau
 \leq \bigg(\iint_{|\xi|\vee  |\tau| < 1} \frac{1}{|\xi|^2 + |\tau|} \d \xi \d \tau \bigg) \|g\|_2 \qquad (t \in \R),
\end{align*}
showing $v_1 \in \L^\infty(\ree)$. This already proves the first of the two estimates for $v$. For any $\phi \in \mS(\ree)$ the functions $v_1 \dhalf \phi$ and $v_2 \dhalf \phi$ are in $\L^1(\ree)$ thanks to \eqref{eq:decay HTdhalf}. In turn, this allows us to use the theorems of Fubini and Plancherel just as in the proof of Proposition~\ref{prop:potential properties} to conclude
\begin{align*}
 \iint_{\ree} v \cdot \cl{\begin{bmatrix} \gradx \phi \\ \dhalf \phi \end{bmatrix}} \d \xi \d \tau
 = \iint_\ree \cF^{-1} \left(\frac{1}{|\xi| + |\tau|^{1/2}} \begin{bmatrix} \i \xi \\ |\tau|^{1/2} \end{bmatrix} \wh{g} \right) \cdot \cl{\phi(\xi, \tau)} \d \xi \d \tau.
\end{align*}
Note that the Fourier multiplication operator appearing on the right-hand side is bounded on $\L^2(\ree)$. Taking $\phi$ in tensor form $\phi(x,t) = \phi_x(x) \phi_t(t)$ reveals $v \in \dot \E_{\loc}(\ree)$ and
\begin{align}
\label{eq1:parabolic potential}
 \begin{bmatrix}
  \gradx v \\ \dhalf v
 \end{bmatrix}
= \cF^{-1} \left(\frac{1}{|\xi| + |\tau|^{1/2}} \begin{bmatrix} \i \xi \\ |\tau|^{1/2} \end{bmatrix} \wh{g} \right) \in \L^2(\ree),
\end{align}
that is, $v \in \dot \E(\ree)$. Finally, the estimate $\|v\|_{\dot \E} \leq \|g\|_2$ is immediate from Plancherel's theorem and to complete the first part of the lemma we record the reverse estimate
\begin{align*}
 \|g\|_2 \leq
 \bigg\| \frac{|\xi|^{1/2}}{|\xi| + |\tau|^{1/2}} \wh{g} \bigg\|_2 +\bigg\| \frac{|\tau|^{1/2}}{|\xi| + |\tau|^{1/2}} \wh{g} \bigg\|_2 = \|\gradx v\|_2 + \|\dhalf v\|_2 = \|v\|_{\dot \E}.
\end{align*}

We turn to the second part where now $v \in \dot \E(\R^{n+2}_+)$ is given. Let $\phi \in \C_0^\infty(\ree; \IC^n)$.  It follows from Remark~\ref{lem:Hdot12 several variables}  that  $v \dhalf \divx \phi \in \L^1(\ree)$ and this justifies the use of Fubini's theorem when computing
\begin{align*}
 \int_\R \int_{\R^n} \gradx v \cdot \cl{\dhalf \phi} \d x  \d t
&= - \int_\R \int_{\R^n} v \cdot \cl{\dhalf \divx \phi} \d x \d t \\
&=  - \int_{\R^n} \int_\R v \cdot \cl{\dhalf \divx \phi} \d t \d x
= - \int_{\R^n} \int_\R \dhalf v \cdot \cl{\divx \phi} \d t \d x.
\end{align*}
With $h_\pa := \gradx v$ and $h_\te := \dhalf v$, Plancherel's theorem lets discover $|\tau|^{1/2} \wh{h_\pa} = \i \wh{h_\te} \xi$. So,
\begin{align*}
 g := \cF^{-1} \bigg(\frac{|\xi| + |\tau|^{1/2}}{|\tau|^{1/2}} \wh{h_\te}(\xi, \tau) \bigg) \in \L^2(\ree)
\end{align*}
solves \eqref{eq1:parabolic potential} and from the first part we deduce $\gradx v = \gradx \ipar g$ and $\dhalf v = \dhalf \ipar g $. Taking into account Lemma~\ref{lem:seminorm on Hdot12}, this implies that $c =v-\ipar g$ is constant. Finally, $g$ is unique due to the estimate in the first part of the lemma.
\end{proof}

As a consequence we obtain that our definition of the energy spaces $\dot{\E}$ gives a realisation of the closure of test functions with respect to their homogeneous norm. In particular, they are the same spaces that have been introduced by the third author and collaborators in this context, see \cite{N1, N2, CNS}.

\begin{cor}
\label{cor:testfunctions dens in energy space}
The space $\dot \E(\ree) / \IC$ is a Hilbert space for the norm $\| \cdot \|_{\dot \E(\ree)}$ and contains $\C_0^\infty(\ree)$ as a dense subset. The analogous statement holds for $\dot \E(\R^{n+2}_+) / \IC$ using $\C_0^\infty(\cl{\R^{n+2}_+})$ as dense subset.
\end{cor}

\begin{proof}
The statement for $\dot \E(\ree) / \IC$ can be obtained by a literal repetition of the proof of Corollary~\ref{cor:Hdot1/2 complete}, replacing $I_t$ by $I_{t,x}$ and the multiplier $|\tau|^{-1/2}$ by $(|\xi|+|\tau|)^{-1/2}$ therein. Of course, the same result holds in dimension $n+2$. Since even reflection around $\lambda = 0$ allows to extend any function in $\dot \E(\R^{n+2}_+)$ to a function in $\dot \E(\R^{n+2})$, we readily obtain the second claim.
\end{proof}

\begin{rem}
\label{rem:HdotsRealization}
\begin{enumerate}
 \item {Corollary~\ref{cor:testfunctions dens in energy space} can be rephrased as saying that $\dot \E(\ree)$ is a realisation of the parabolic Sobolev space $\Hdot^{1/2}_{\pd_{t} - \Delta_x}$ within the tempered distributions modulo constants. The same kind of argument can be used to obtain such realisations for $\Hdot^{s}_{\pd_{t} - \Delta_x}$, $0 < s < 1/2$.}
 \item Sobolev embeddings in parabolic space show that {$\dot\E(\R^{n+1})/\IC$ has a realisation in the Lebesgue space $\L^q(\R^{n+1})$, where $q = \frac{2(n+2)}{n}$.} We shall not use this fact.
\end{enumerate}
\end{rem}

Finally, we supply a proof of the trace theorem for $\dot \E(\R^{n+2}_+)$ that played an important role for the study of energy solutions in Section~\ref{sec:energy}. The reader may recall the definition of the parabolic Sobolev space $\Hdot^{1/4}_{\partial_t - \Delta_x}$ from Section~\ref{sec:sobolev}.

\begin{lem}
\label{lem:trace energy}
The space $\dot\E(\R^{n+2}_+)/\IC$ continuously embeds into $\C([0,\infty); \Hdot^{1/4}_{\partial_t - \Delta_x})$. Conversely, any $f \in \Hdot^{1/4}_{\partial_t - \Delta_x}$ has an extension $v \in \dot \E(\R^{n+2}_+)$ such that $v|_{\lambda = 0} = f$.
\end{lem}

\begin{proof}
Let $v \in \C_0^\infty(\cl{\R^{n+2}_+})$. For any $\lambda_0 > 0$ the fundamental theorem of calculus yields
\begin{align*}
\||\tau|^{1/4} \wh{v}|_{\lambda = \lambda_0}\|_2^2 + \||\xi|^{1/2}\wh{v}|_{\lambda = \lambda_0}\|_2^2
&= 2 \Re \int_{\lambda_0}^\infty (|\tau|^{1/2} \wh{v}, \pd_{\lambda} \wh{v}) + (|\xi| \wh{v}, \pd_{\lambda} \wh{v}) \d \lambda \\
&\leq \int_0^\infty \||\tau|^{1/2}\wh{v}\|_2^2 + \||\xi| \wh{v}\|_2^2 + 2 \|\pd_{\lambda} \wh{v}\|_2^2 \d \lambda,
\end{align*}
where norms, scalar products and the Fourier transform are in $\L^2(\ree)$. By Plancherel's theorem the left-hand side compares to the norm of $v|_{\lambda = \lambda_0}$ in $\Hdot^{1/4}_{\partial_t - \Delta_x}$ and the right-hand side to the norm of $v$ in $\dot \E(\R^{n+2})$. This proves uniform continuity  of $\lambda_{0}\mapsto v|_{\lambda = \lambda_0}$ into $\Hdot^{1/4}_{\partial_t - \Delta_x}$ and thus also the limit at $0$.  Due to Corollary~\ref{cor:testfunctions dens in energy space} we obtain the required embedding by density. Conversely, given $f \in \C_0^\infty(\ree)$, we can define an extension to $\R^{n+2}_+$ by
\begin{align*}
 v(\lambda,x,t) = \cF^{-1} (\e^{-\lambda(|\xi|^2 + \i \tau)^{1/2}} \wh{f})(x,t).
\end{align*}
From Plancherel's theorem we can infer $\|v\|_{\dot \E(\R^{n+2})} \lesssim \|f\|_{\Hdot^{1/4}_{\partial_t - \Delta_x}}$ and so this extension operator extends to $\Hdot^{1/4}_{\partial_t - \Delta_x}$ by density.
\end{proof}
\section{The parabolic Dirac operator}
\label{sec:diracBasic}

In this section we develop the basic operator theoretic properties of the parabolic Dirac operator
\begin{align*}
 \P=\Pfull
\end{align*}
with maximal domain
\begin{align*}
 \dom(\P)=\Big\{f\in\mH : \gradx f_{\pe}\in \L^2,\, \HT \dhalf f_{\pe} \in \L^2,\, \divx f_{\pa} - \dhalf f_{\te}\in \L^2\Big\}
\end{align*}
in $\mH=\L^2(\R^{n+1}; \IC^{n+2})$. The reader may recall relevant notation from Sections~\ref{sec:correspondence} and \ref{sec:dirac}. Clearly $\P$ is closed and its null space is
\begin{align*}
 \nul(\P)=\Big\{f\in\mH : f_{\pe}=0,\, \divx f_\pa = \dhalf f_\te \text{ in } \mS'(\ree)\Big\}.
\end{align*}
Here is an explicit description of $\clos{\ran(\P)}$ as a space of distributions.

\begin{lem}
\label{lem:characterisation ranP}
It holds
\begin{align*}
 \clos{\ran(\P)} = \Big\{ f \in \mH : \gradx f_\te = \HT \dhalf f_\pa \mathrm{\;in\;} \mS'(\ree; \IC^n) \Big\}.
\end{align*}
\end{lem}

\begin{proof}
The inclusion `$\subset$' follows from the definition of $\P$. Conversely assume that $f$ is contained in the right-hand space. Let $g \in \nul(\P^*)$, that is, $g_\pe = 0$ and $\divx g_\pa = - \HT \dhalf g_\te$. On the Fourier side the relation for $\pa$- and $\te$-components can be solved by
\begin{align*}
 \wh{f}_\pa = \frac{\wh{f}_\te}{\sgn (\tau) |\tau|^{1/2}} \ \xi, \qquad \wh{g}_\te = - \xi \cdot \frac{\wh{g}_\pa}{\sgn (\tau)|\tau|^{1/2}} \qquad (\text{a.e.\ on $\ree$}),
\end{align*}
where $\tau$ and $\xi$ are the Fourier variables corresponding to $t$ and $x$. Hence, $(\wh{f}, \wh{g} ) = 0$ and thus $(f, g) = 0$ by Plancherel's theorem. This proves that $f \in \nul(\P^*)^\pe$ and hence $f \in \clos{\ran(\P)}$.
\end{proof}

Bisectoriality of $\P \M$ and $\M \P$ can be obtained by implementing the `hidden coercivity' of the parabolic operator $\Lop$ discussed in Section~\ref{sec:energy} within the first order framework.

\begin{lem}
\label{lem:bisectoriality of Dirac}
The operators $\P\M$ and $\M\P$ are closed and bisectorial on $\mH$ with $\ran(\P \M) = \ran (\P)$ and $\cl{\ran(\M \P)} = \M \cl{\ran(\P)}$.
\end{lem}

\begin{proof} Closedness will follow from bisectoriality since the latter implies that the operator's resolvent set is non-empty. Consider for $\delta \in \R$ the transformation
\begin{align*}
 U_{\delta }= \frac{1}{1+\delta ^2}{\begin{bmatrix} 1-\delta \HT & 0 & 0 \\ 0 & 1+\delta \HT & 0 \\0 & 0 & \delta -\HT \end{bmatrix}}.
\end{align*}
Then, as $\HT^2=-\id$,
\begin{align*}
 \P U_{\delta }= \frac{1}{1+\delta ^2}{\begin{bmatrix} 0 & \divx (1+\delta \HT) & -\delta \dhalf+ \HT\dhalf \\ -\gradx (1-\delta \HT) & 0 & 0 \\ -\HT \dhalf - \delta \dhalf & 0 & 0 \end{bmatrix}}.
\end{align*}
Using that $\HT^*=-\HT$ and that $\HT$ commutes with derivatives in $x$, we see that $\P U_{\delta }$ is self-adjoint. Next, we claim that $U_{\delta }^{-1}\M$ is accretive for $\delta>0 $ small enough. Indeed, on recalling the definition of $\M$ from \eqref{eq:DB}, we can write $U_{\delta }^{-1}M$ as
\begin{align*}
 {\begin{bmatrix} 1+\delta \HT & 0 & 0\vphantom{\dhalf} \\ 0 & 1-\delta \HT & 0 \\0\vphantom{\dhalf} & 0 & \delta + \HT \end{bmatrix}} {\begin{bmatrix} \hat{A}_{\pe \pe} & \hat{A}_{\pe \pa} & \vphantom{\dhalf}0 \\ \hat{A}_{\pa \pe} & \hat{A}_{\pa \pa} & 0 \\ \vphantom{\dhalf}0& 0& 1 \end{bmatrix}}={\begin{bmatrix} (1+\delta \HT ) \hat{A}_{\pe \pe} & (1+\delta \HT) \hat{A}_{\pe \pa} & \vphantom{\dhalf}0 \\ (1-\delta \HT )\hat{A}_{\pa \pe} & (1-\delta \HT )\hat{A}_{\pa \pa} & 0 \\ \vphantom{\dhalf}0& 0& \delta + \HT \end{bmatrix}} .
\end{align*}
Since $\Re(\HT g, g)= 0$, we see that the lower block of the matrix on the right is accretive for all $\delta >0$. Since the upper block of the product rewrites as $\hat{A}$ plus a bounded perturbation of size $\delta $, it remains accretive if $\delta $ is small enough.

Based on the above we can conclude that $\P\M=(\P U_{\delta }) (U_{\delta }^{-1}\M)$ is a product of a self-adjoint and a bounded accretive operator. Such operators have exhaustively been studied in the literature and it is well-known that they are bisectorial and satisfy
\begin{align*}
 \ran(\P \M) = \ran((\P U_{\delta }) (U_{\delta }^{-1}\M)) = \ran(\P U_{\delta}) = \ran (\P),
\end{align*}
see for instance \cite{elAAM} or \cite[Prop.~6.1.17]{EigeneDiss}. Next, $\M\P=\M(\P\M)\M^{-1}$ is bisectorial by similarity and $\cl{\ran(\M \P)} = \M \cl{\ran(\P)}$ holds by accretivity of $\M$.
\end{proof}

\begin{rem}
\label{rem:PM and MP similar}
For parabolic systems $\M$ is only supposed to be accretive on the range of $\P$, see Section~\ref{sec:miscellani}. This property is also satisfied by $U_{\delta }^{-1}\M$ if $\delta $ is small. In that case, $\M^{-1}$ should be interpreted as the inverse of the restriction $\M|_{\cl{\ran(\P)}}$, so that $\M\P$ and $\P\M$ are similar on the closures of their ranges and they are null on their null spaces. None of our proofs in the following sections will use the invertibility of $\M$ in $\L^\infty(\R^{n+2}_+; \Lop(\IC^{n+2}))$.
\end{rem}
\section{The correspondence to a first order system: Proof of Theorem~\ref{thm:correspondence}}
\label{sec:proof of correspondence}

In this section we prove the correspondence between the parabolic equation \eqref{eq1} and the first order system $\partial_\lambda F + \P \M F = 0$. Whenever convenient, the reader may look up the definitions of $\P$ and $\M$ in \eqref{eq:DB}. We always consider $\P$ as a closed operator in $\mH$ with maximal domain $\dom(\P)$ and we write $(\cdot \,, \cdot )$ for the inner product on this space. Let us stress that in the following we do not use the bisectoriality of $\P \M$ and that $M$ may depend on $\lambda$.

\begin{defn}
\label{def: weak solutions first order}
A weak solution $F$ to the first order system $\partial_\lambda F + \P \M F = 0$ is a function $F \in \Lloc^2(\R_+; \clos{\ran(\P)})$ such that
\begin{align}
\label{eq:weak solutions first order}
\int_0^\infty ( F(\lambda), \partial_\lambda \phi(\lambda) ) \d \lambda = \int_0^\infty ( \M F(\lambda), \P^* \phi(\lambda) ) \d \lambda \qquad (\phi \in \C_0^\infty(\R^{n+2}_+; \IC^{n+2})).
\end{align}
\end{defn}

{This notion of weak solutions is in accordance with Section~\ref{sec:main results}} since Lemma~\ref{lem:characterisation ranP} asserts that $\Lloc^2(\R_+; \clos{\ran(\P)})$ is exactly the space $\cH_{\loc}$ introduced beforehand in Section~\ref{sec:correspondence}.

\subsection{Conormal differentials are weak solutions}
\label{sec:conormal differentials are weak solutions}

Let $u$ be a reinforced weak solution to \eqref{eq1} with the additional property
\begin{align*}
 \pcg u = \begin{bmatrix} \dnuA u \\ \gradx u \\ \HT \dhalf u \end{bmatrix} \in \Lloc^2(\R_+; \mH),
\end{align*}
where $\dhalf$ is in the sense of Definition~\ref{defn: Hdot12 several variables}. We put $F := \pcg u$ and need to demonstrate that $F$ is a weak solution to the first order system.

To verify the compatibility condition $F \in \Lloc^2(\R_+; \clos{\ran(\P)})$, let $\phi \in \C_0^\infty(\ree; \IC^{n})$ and note that in view of Lemma~\ref{lem:Hdot12 several variables} we have $u(\lambda) \cdot \HT \dhalf \divx \phi \in \L^1(\ree)$ for a.e.\ $\lambda >0$. Hence, we can compute by means of Fubini's theorem
\begin{align*}
 ( F(\lambda)_\te, \divx \phi )
&= - \int_{\R^n} \int_{\R} u(\lambda) \cdot \clos{\HT \dhalf \divx \phi} \d t \d x \\
&= - \int_{\R} \int_{\R^n} u(\lambda) \cdot \clos{\divx \HT \dhalf \phi} \d x \d t
= ( F(\lambda)_\pa, \HT \dhalf \phi ).
\end{align*}
This extends to all $\phi \in \mS(\ree)$ by approximation, showing $\gradx F(\lambda)_\te = \HT \dhalf F(\lambda)_\pa$ in the sense of $\mS'(\ree)$. Thus, $F(\lambda) \in \clos{\ran(\P)}$ thanks to Lemma~\ref{lem:characterisation ranP}.

In order to verify \eqref{eq:weak solutions first order}, let $\phi \in \C_0^\infty(\R^{n+2}_+; \IC^{n+2})$. Taking $\phi_\pe$ as a test function in the definition of reinforced weak solutions to the parabolic problem, we find
\begin{align*}
 \int_0^\infty ( F_\pe, \partial_\lambda \phi_\pe ) \d \lambda
&= \iiint_{\R^{n+2}_+} (A \gradlamx u)_\pe \cdot \clos{\partial_\lambda \phi_{\pe}} \d x \d t \d \lambda \\
&= \iiint_{\R^{n+2}_+} - (A \gradlamx u)_\pa \cdot \clos{\gradx \phi_{\pe}} + \HT \dhalf u \cdot \clos{\dhalf \phi_\pe} \d x \d t \d \lambda,
\end{align*}
which in the language of $\P$ and $\M$ reads
\begin{align*}
 \int_0^\infty ( F_\pe, \partial_\lambda \phi_\pe ) \d \lambda
= \int_0^\infty ( (\M F)_\pa, (\P^* \phi)_\pa ) + ( (\M F)_\te, (\P^* \phi)_\te ) \d \lambda.
\end{align*}
For the $\pa$-components we deduce, using integration by parts and the definitions of $\P$ and $\M$,
\begin{align*}
 \int_0^\infty ( F_\pa, \partial_\lambda \phi_\pa ) \d \lambda
&= \iiint_{\R^{n+2}_+} \gradx u \cdot \clos{\partial_\lambda \phi_\pa} \d \lambda
= \iiint_{\R^{n+2}_+} \partial_\lambda u \cdot \clos{\divx \phi_\pa} \d \lambda \\
&= \int_0^\infty ( (\M F)_\pe, (\P^* \phi)_\pe ) - ( \partial_\lambda u, \HT \dhalf \phi_\te ) \d \lambda.
\end{align*}
As for the $\te$-component we first note that $u \cdot \HT \dhalf (\partial_\lambda \phi_\te) \in \L^1(\R^{n+2}_+)$ holds since $u$ is a reinforced weak solution, see Lemma~\ref{lem:Hdot12 several variables}. Thus, by definition of $\dhalf$ and Fubini's theorem
\begin{align*}
 \int_0^\infty ( F_\te, \partial_\lambda \phi_\te ) \d \lambda
&=-\int_0^\infty \int_{\R^n} \int_\R u \cdot \clos{\HT \dhalf \partial_\lambda \phi_\te} \d t \d x \d \lambda \\
&= - \int_{\R^n} \int_\R \int_0^\infty u \cdot \clos{\partial_\lambda \HT \dhalf \phi_\te} \d \lambda \d t \d x
= \int_0^\infty ( \partial_\lambda u, \HT \dhalf \phi_\te ) \d \lambda,
\end{align*}
where we used $\HT \dhalf \phi(\cdot,x,t) \in \C_0^\infty(\R_+)$ for $(x,t) \in \ree$ fixed in the last step. Adding up the previous three identities gives \eqref{eq:weak solutions first order}.

\subsection{Every weak solution is a conormal differential}
\label{sec:weak solutions are conormal differentials}

Conversely, we are now given a weak solution $F \in \Lloc^2(\R_+; \clos{\ran(\P)})$ to the first order system and we have to construct a potential $u$ such that $u$ is a reinforced weak solution to the parabolic equation in \eqref{eq1} and  $\pcg u = F$. Such $u$ is necessarily unique up to a constant since if $\pcg u = 0$, then $\gradlamx u = 0$ showing that $u$ only depends on $t$ and then $\HT \dhalf u = 0$ implies that $u$ is constant, see Lemma~\ref{lem:seminorm on Hdot12}. The construction of a specific $u$ is in two steps.

\subsubsection*{Step 1: Adjusting the \texorpdfstring{$(x,t)$}{(x,t)}-direction}

By definition of $F$ we have $F(\lambda) \in \clos{\ran(\P)}$ for almost every $\lambda > 0$. In this case $\gradx F(\lambda)_\te = \HT \dhalf F(\lambda)_\pa$ in the sense of tempered distributions due to Lemma~\ref{lem:characterisation ranP} or, equivalently, $\i \xi \wh{F(\lambda)}_\te = \i \sgn(\tau) |\tau|^{1/2} \wh{F(\lambda)}_\pa$ by taking the Fourier transform in $(x,t)$. As usual, $(\xi, \tau)$ is Fourier variable corresponding to $(x,t)$. With these compatibility conditions at hand, we get for almost every $\lambda > 0$ that the measurable function
\begin{align*}
\wh{g(\lambda)} := \frac{|\xi| + |\tau|^{1/2}}{\i \sgn(\tau) |\tau|^{1/2}} \wh{F(\lambda)}_\te
\end{align*}
belongs to $\L^2(\ree)$ as $|\wh{g(\lambda)}|\le |\wh{F(\lambda)}_\te | + |\wh{F(\lambda)}_\pa |$ almost everywhere and the so-defined function $g$ is in $\Lloc^2(\R_+; \L^2(\ree))$. Now, Lemma~\ref{lem:parabolic potential} furnishes a parabolic potential $v(\lambda,x,t) := \ipar g(\lambda,x,t)\in \Lloc^2(\R^{n+2}_+)$ that satisfies $v(\lambda) \in \Hdot^{1/2}(\R; \Lloc^2(\R^n)) \cap \Lloc^2(\R; \W^{1,2}_{\loc}(\R^n))$ for almost every $\lambda>0$ with derivatives
\begin{align*}
 \begin{bmatrix}
  \gradx v(\lambda) \\ \dhalf v(\lambda)
 \end{bmatrix}
= \cF^{-1} \left(\frac{1}{|\xi| + |\tau|^{1/2}}
  \begin{bmatrix} \i \xi \\ |\tau|^{1/2} \end{bmatrix}
  \wh{g(\lambda)} \right).
\end{align*}
Taking into account the relation between $\wh{F(\lambda)}_\pa$ and $\wh{F(\lambda)}_\te$, we obtain
\begin{align}
\label{eq:correspondence: first potential}
 \begin{bmatrix}
  \gradx v  \\ \HT \dhalf v
 \end{bmatrix}
= \begin{bmatrix}
  F(\lambda)_\pa \vphantom{\gradx v} \\ F(\lambda)_\te \vphantom{\dhalf v} \end{bmatrix},
\end{align}
so that $v$ already behaves as desired when it comes to derivatives in $x$- and $t$-direction. We may also assume that there exists some bounded set $Q \times I \subseteq \ree$ with positive Lebesgue measure such that
\begin{align}
\label{eq:correspondence: vanishing average}
 \bariint_{Q \times I} v(\lambda,x,t) \d x \d t = 0 \qquad (\text{a.e.\ $\lambda > 0$})
\end{align}
since otherwise we could subtract this very average from $v(\lambda)$ for each $\lambda > 0$ without affecting any of the other properties of $v$ discussed above.

To compute the transversal derivative $\partial_\lambda v$, we let $\varphi \in \C_0^\infty(\R_+)$, $\psi \in \C_0^\infty(\R^n; \IC^n)$, $\eta \in \C_0^\infty(\R)$ and test \eqref{eq:weak solutions first order} with
\begin{align*}
 \phi(\lambda,x,t) = \begin{bmatrix} 0 \\ \varphi(\lambda) \psi(x) \eta(t) \\ 0 \end{bmatrix}.
\end{align*}
Spelling out the identity
\begin{align*}
 \int_0^\infty ( F, \partial_\lambda \phi ) \d \lambda = \int_0^\infty ( \M F, \P^* \phi ) \d \lambda
\end{align*}
and taking into account that for the terms involving $\varphi$ only the $\pa$-component of $\phi$ and the $\pe$-component of $\P^* \phi$ do not vanish, we readily find
\begin{align*}
- \int_\R \int_{\R^n} \bigg(\int_0^\infty v \cdot \clos{\partial_\lambda \varphi} \d \lambda \bigg) \clos{\eta \divx \psi} \d x \d t
= \int_\R \int_{\R^n} \bigg(\int_0^\infty (\M F)_\pe \cdot \clos{\varphi} \d \lambda \bigg) \clos{\eta \divx \psi} \d x \d t.
\end{align*}
Since $\psi$ and $\eta$ were arbitrary, the inner integrals over $(0,\infty)$ differ only by a constant, which we can compute by taking an average over $Q \times I \subset \ree$ and using \eqref{eq:correspondence: vanishing average}. Thus,
\begin{align*}
\int_0^\infty v \cdot \clos{\partial_\lambda \varphi} + (\M F)_\pe \cdot \clos{\varphi} \d \lambda
 = \int_0^\infty \bigg(\bariint_{Q \times I} (\M F)_\pe \d x \d t \bigg) \cdot \clos{\varphi} \d \lambda
\end{align*}
a.e.\ as functions defined on $\ree$. This means that in the sense of distributions
\begin{align}
\label{eq:correspondence: dlambda v}
\partial_\lambda v = (\M F)_\pe - \bariint_{Q \times I} (\M F)_\pe \d x \d t \in \Lloc^2(\R^{n+2}_+).
\end{align}
In particular, we have now found that $v \in \Hdot^{1/2}(\R; \Lloc^2(\reu)) \cap \Lloc^2(\R; \W^{1,2}_{\loc}(\reu))$ has the required regularity of a reinforced weak solution.

\subsubsection*{Step 2: Adjusting \texorpdfstring{$u$}{u} in \texorpdfstring{$\lambda$}{lambda}-direction}

From \eqref{eq:correspondence: first potential} and \eqref{eq:correspondence: dlambda v} we obtain
\begin{align*}
\pcg v
= \begin{bmatrix} A_{\pe \pe}\vphantom{\bariint_{Q \times I}} & A_{\pe \pa} & 0 \\ 0 & 1 & 0 \\ 0 & 0 & 1\end{bmatrix}
  \begin{bmatrix} \partial_\lambda v \vphantom{\bariint_{Q \times I}}\\ \gradx v \\ \HT \dhalf v \end{bmatrix}
= \begin{bmatrix} F_\pe \vphantom{\bariint_{Q \times I}} \\ F_\pa \\ F_\te \end{bmatrix}
  + \begin{bmatrix} A_{\pe \pe}\vphantom{\bariint_{Q \times I}} & A_{\pe \pa} & 0 \\ 0 & 1 & 0 \\ 0 & 0 & 1\end{bmatrix} \begin{bmatrix} \bariint_{Q \times I} (\M F)_\pe \d x \d t \\ 0 \\ 0\end{bmatrix}.
\end{align*}
Let now $w$ be an anti-derivative of the $\Lloc^2$-function $\lambda \mapsto \bariint_{Q \times I} (\M F)_\pe \d x \d t$, that is, $w \in \W^{1,2}_{\loc}(\R_+)$ and $\partial_\lambda w =   \bariint_{Q \times I} (\M F)_\pe \d x \d t$. Then the modified potential $u:=v - w$ still has the regularity of a reinforced weak solution since $w$ depends only on $\lambda$ {and from above we can read off $\pcg u = F$}. Hence, $u$ has the required property. Finally, we let $\varphi \in \C_0^\infty(\R^{n+2}_+)$ and take $\phi = [\varphi,\, 0,\, 0]$ as a test function in \eqref{eq:weak solutions first order}. Unravelling the resulting identity gives
\begin{align*}
 \int_0^\infty \int_\reu A\nabla_{\lambda,x} u\cdot\clos{\nabla_{\lambda,x}\varphi} + \HT \dhalf u \cdot \clos{\dhalf \varphi} \d x\d t\d\lambda=0.
\end{align*}
This confirms $u$ as a reinforced weak solution to \eqref{eq1} and the proof is complete.

\begin{rem}
Independently of the equation \eqref{eq1}, the argument above shows in fact that any $G\in \Lloc^2(\R_+; \L^2(\ree; \IC^{n+2}))$ with {$\partial_\lambda G_\pa - \gradx G_\pe=0$ and $\HT\dhalf G_{\pa}= \gradx G_{\te}$ has a potential $v\in \Lloc^2(\R_+; \L^2(\ree))$, unique up to constants, such that $[\partial_\lambda v,\; \gradx v,\;\HT \dhalf v] = G$}.
\end{rem}
\section{Resolvent estimates}
\label{sec:resolvents}

In this section we prove the resolvent estimates stated in Lemma~\ref{lem:LpLq} and Proposition~\ref{prop:OffDiag}.

\subsection{Proof of Lemma \ref{lem:LpLq}}

As there is nothing to prove for $\lambda=0$, we can in the following, by scaling $(x,t)\mapsto (|\lambda| x, |\lambda|^{1/2} t)$, assume $\lambda=\pm 1$ as long as implicit constants in all of our estimates only depend on dimension and the ellipticity constants of $A$. We shall only argue for $\lambda=1$ since the proof for $\lambda=-1$ is the same.

For $1<p,q<\infty$, we define inhomogeneous parabolic Sobolev spaces $\H_{p,q} (\R^{n+1})$ by
\begin{align*}
 \H_{p,q} (\R^{n+1}) := \L^p(\R; \W^{1,q}(\R^n)) \cap \H^{1/2, p}(\R; \L^q(\R^n))
\end{align*}
equipped with the mixed norm
\begin{align*}
\|u\|_{\H_{p,q}} := \Big\| \big\| |u|+ |\gradx u|+ |\dhalf u| \big\|_{\L^q(\R^n)} \Big\|_{\L^p(\R)}.
\end{align*}
We let $\H_{p,q}^*$ denote the  space dual to  $\H_{p,q}$ with respect to the $\L^2(\ree)$ inner product. For any $f\in \L^2(\R^{n+1}; \IC^{n+2})$, there exists a unique $ \wt{f} \in \L^2(\R^{n+1}; \IC^{n+2})$, such that
\begin{align*}
f= \begin{bmatrix}(A \wt{f})_{\pe} \\ \wt{f}_{\pa} \\ \wt{f}_{\theta} \end{bmatrix},
\end{align*}
{compare with Proposition~\ref{prop:divformasODE}}. Spelling out the definitions of $\P$ and $\M$ reveals that the equation $(\id + \i \P\M)^{-1}f= g$ is equivalent to $\wt{g}_{\pe}\in \H_{2,2}$ with
\begin{align*}
\label{eq1:LpLq}
\begin{cases}
 (A \wt{g})_\pe + \i\divx(A \wt{g})_\pa -\i \dhalf \wt{g}_{\theta}= f_\pe, \\
 \wt{g}_\pa - \i \nabla_x\wt{g}_\pe = f_\pa,
 \\
 \wt{g} _{\theta}- \i \HT\dhalf\wt{g}_{\pe}= f_{\theta}.
\end{cases}
\end{align*}
Using the second and third equations to eliminate $\wt{g}_\pa$ and $\wt{g}_{\theta}$ in the first one, this is equivalent to the system
\begin{align}
\begin{cases}
 (\pd_{t} + L_A)\wt{g}_\pe = f_\pe - A_{\pe \pa} f_\pa - \i\divx (A_{\pa\pa} f_\pa) + \i\dhalf f_{\theta} , \\
 \wt{g}_\pa - \i \nabla_x\wt{g}_\pe = f_\pa,
 \\
 \wt{g} _{\theta}- \i \HT\dhalf\wt{g}_{\pe}= f_{\theta},
\end{cases}
\end{align}
where
\begin{align*}
L_A:= \begin{bmatrix} 1 & \i\divx \end{bmatrix}
 A
 \begin{bmatrix} 1 \\ \i\nabla_x \end{bmatrix}.
\end{align*}
{The parabolic operator $\partial_t + L_A$ admits hidden coercivity in the same spirit as discussed for $\Lop$ in Section~\ref{sec:energy}, that is, it can be associated with a bounded and coercive sesquilinear form on $\H_{2,2}$.} As a consequence of the Lax-Milgram lemma we find that $\pd_{t} + L_A$ is a bounded and invertible operator from $\H_{2,2}$ to its dual $\H_{2,2}^*$. Moreover, $\pd_{t} + L_A$ clearly extends to a bounded operator from $\H_{p,q}$ to $\H_{p',q'}^*$ for any $1<p,q<\infty$. All of these bounds depend only upon the ellipticity constants of $A$. We intend to apply {\u{S}}ne{\u\ii}berg's lemma~\cite{Sneiberg-Original} in order to extend invertibility to a neighbourhood of $(2,2)$ in the $(p,q)$-plane. This is possible provided that $\H_{p,q}$ and $\H_{p,q}^*$ form complex interpolation scales.

\begin{lem}
\label{lem:interpolHpq}
Let $\eps>0$. For $1+ \eps < p_0, p_1, q_0, q_1 < 1 + \eps^{-1}$ and $0< \theta < 1$ the complex interpolation identities
\begin{align*}
 [\H_{p_0,q_0}, \H_{p_1, q_1}]_\theta = \H_{p,q} \qquad \text{and} \qquad [\H_{p_0,q_0}^*, \H_{p_1, q_1}^*]_\theta = \H_{p,q}^*
\end{align*}
up to equivalent norms, where $\frac{1}{p} = \frac{1-\theta}{p_0} + \frac{\theta}{p_1}$ and $\frac{1}{q} = \frac{1-\theta}{q_0} + \frac{\theta}{q_1}$. Equivalence constants can be chosen only in dependence of $\eps$ and $n$.
\end{lem}

\begin{proof}
For $1<p,q<\infty$ we abbreviate $\L^{p}(\L^{q}) = \L^{p}(\R; \L^{q}(\R^n))$. It is well-known that these spaces interpolate according to the rule
\begin{align*}
 [\L^{p_0}(\L^{q_0}), \L^{p_1}(\L^{q_1})]_\theta = \L^{p}(\L^{q})
\end{align*}
with \emph{equal} norms, see, for example, \cite[Thm.~5.1.1/2]{BL}. Let $T: \mS(\ree) \to \mS'(\ree)$ be defined on the Fourier side via multiplication by the symbol
\begin{align*}
 m(\tau,\xi) = \frac{1}{1+|\tau|^{1/2} + |\xi_1| + \cdots + |\xi_n|},
\end{align*}
where $\tau \in \R$ and $\xi \in \R^n$. Then $m$ satisfies Mihlin's condition
\begin{align*}
 \bigg|\tau^{\alpha_0} \xi_1^{\alpha_1} \cdots \xi_n^{\alpha_n} \frac{\partial^\alpha m}{\partial_\tau^{\alpha_0} \partial_{\xi_1}^{\alpha_1} \cdots \partial_{\xi_n}^{\alpha_n}} \bigg| \leq C(\alpha) < \infty \qquad (\tau, \xi_1,\ldots,\xi_n \neq 0)
\end{align*}
for all multi-indices {$\alpha \in \IN_0^{n+1}$}. The Marcinkiewicz multiplier theorem on $\L^p(\L^q)$, see \cite[Cor.~1]{LizorkinMultiplier}, yields that $T$ extends to a bounded operator on $\L^p(\L^q)$ for every pair $(p,q) \in (1,\infty) \times (1,\infty)$. {Complex interpolation, starting from the four extremal points with $p,q \in \{1+\eps, 1+\eps^{-1}\}$, yields the required uniform bound on the operator norm within their convex hull in the $(p,q)$-plane}. By the same argument applied to the multipliers $|\tau|^{1/2} m(\tau,\xi)$, $|\xi| m(\tau,\xi)$ and $\i \xi_j/ |\xi_j|$, we see that in fact the extension of $T$ is bounded $\L^p(\L^q) \to \H_{p,q}$ and invertible. Thus, the required interpolation rules for the $\H_{p,q}$-spaces including uniformity of the equivalence constants follow from those for $\L^p(\L^q)$.

As another consequence every space $\H_{p,q}$ contains $T(\mS(\ree))$ as a dense set since $\mS(\ree)$ is dense in $\L^p(\L^q)$. Thus, the interpolation rules for the dual spaces $\H_{p,q}^*$ follow from the duality theorem for the complex interpolation functor \cite[Thm.~4.5.1]{BL}.
\end{proof}

Lemma~\ref{lem:interpolHpq} allows to apply {\u{S}}ne{\u\ii}berg's result \cite{Sneiberg-Original}, see also \cite[App.~A]{ABES3} for a quantitative version, to the effect that $\pd_{t} + L_A$ remains invertible for $(p,q)$ in a neighbourhood of $(2,2)$. Technically speaking, this only applies to complex interpolation scales with one parameter: Seeing $(\frac{1}{p}, \frac{1}{q})$ on a line passing through $(\frac{1}{2},\frac{1}{2})$, it yields an interval for invertibility around $(2,2)$ on that line. However, the result is {quantitative} in that the length of that interval depends only upon upper/lower bounds for $\pd_{t} + L_A: \H_{2,2} \to \H_{2,2}^*$ and the equivalence constants in the interpolation from Lemma~\ref{lem:interpolHpq}. Therefore we can obtain the same interval on each line, eventually summing up to a two-dimensional neighbourhood in the $(p,q)$-plane. Let us also remark that all inverses are compatible on the intersection of any two such $\H_{p,q}^*$ spaces by abstract interpolation theory \cite[Thm.~8.1]{KMM}.

Returning to the equation $(\id + \i \P\M)^{-1}f=g$, we obtain from \eqref{eq1:LpLq} that if $(p,q)$ is as above, then given $f \in \L^2(\L^2) \cap \L^p(\L^q)$ we have $\|g\|_{\L^p(\L^q)} \lesssim \|f\|_{\L^p(\L^q)}$ with implicit constants depending only on dimension and ellipticity. We conclude by density.

Next, we prove the analogous result for $\M\P$. If $\M^{-1}$ is in $\L^\infty$, which is the case for equations, then the result is immediate by similarity. {For systems, however, we only have that $\M: \clos{\ran({\P})} \to \clos{\ran({\M\P})}$ is invertible. In this case, we first use the argument above to find that $(1+\i \P)^{-1}$ is bounded on $\L^p(\L^q)$ if $1<p,q<\infty$.} Therefore the domains $\dom_{p,q}$ of the maximal realisation of $\P$ in $\L^p(\L^q)$, equipped with the graph norm, form a complex interpolation scale. Since $\M\P$ is bisectorial on $\L^2(\L^2)$, we have that $1+\i \M\P$ is an isomorphism $\dom_{2,2} \to \L^2(\L^2)$. For all other $p,q$ it is at least bounded $\dom_{p,q} \to \L^p(\L^q)$. As above, $1 + \i \M\P$ remains an isomorphism $\dom_{p,q} \to \L^p(\L^q)$ for $(p,q)$ in a neighbourhood of $(2,2)$ with compatible inverses. Consequently, the resolvent estimate holds in $\L^p(\L^q)$ for those $(p,q)$.

{Finally, $\M^*$ is an operator in the same class as $\M$. Thus, we already have resolvent bounds on $\L^p(\L^q)$ for $\P \M^*$ and $\M^* \P$ and those for $\M \P^*$ and $\P^* \M$ follow by duality.}

\subsection{Proof of Proposition \ref{prop:OffDiag}}

It suffices to prove \eqref{eq:off} for $q=2$ and $0<\eps<\eps_{0}$ with $\eps_{0}=\delta _{0}$ of Lemma~\ref{lem:LpLq} and $N_{0}$ to be determined. Then the proposition follows by interpolation with the global bound provided by Lemma~\ref{lem:LpLq} as long as $q$ is close enough to $2$ to guarantee the announced decay. {We shall only argue for $\P \M$ as the proof for the other operators is literally the same}.

Let $h\in \L^2(\R^{n+1}; \IC^{n+2})=\L^2(\L^2)$. Let $Q,I, j,k, \lambda$ as in the statement and let $N\ge 4$. We set $J= 4^j I$ and split
$C_{k}(Q \times J)=C_{k}^1\cup C_{k}^2$ with
\begin{align*}
 C_{k}^1:= (2^{k+1}Q\setminus 2^{k}Q)\times N^{k}J \qquad \text{and} \qquad C_{k}^2:= 2^{k}Q\times (N^{k+1}J\setminus N^{k}J).
\end{align*}

We assume that $h$ is supported in $C_{k}$ and hence we can write $h=h_{1}+h_{2}$ with $h_{i}$ supported in $C_{k}^i$.

To estimate $\bariint_{Q \times J} |(\id + \i\lambda \P\M)^{-1}h_{1}|^2 \d y \d s$, we can rely on the same argument as in the elliptic situation since the $x$ support of $h_{1}$ is far from $Q$: More precisely, we have
\begin{align*}
 \bariint_{Q \times J} |(\id + \i\lambda \P\M)^{-1}h_{1}|^2 \d y \d s & \le \frac{1}{\ell(J)|Q|} \int_{\R}\int_{Q} |(\id + \i\lambda \P\M)^{-1}h_{1}|^2 \d y \d s
\end{align*}
{and by the argument in \cite[Prop.~5.1]{elAAM}, using a commutator with a cut-off function in $x$-direction only, we obtain for any $m\in \IN$,}
\begin{align}
\label{eq:OD1}
\begin{split}
\bariint_{Q \times J} |(\id + \i\lambda \P\M)^{-1}h_{1}|^2 \d y \d s
 & \lesssim \frac{2^{-km}}{\ell(J)|Q|} \int_{\R} \int_{\R^n} |h_{1}|^2 \d y \d s
 \\
 & = 2^{-km} N^{k+1} 2^{kn} \bariint_{C_{k}(Q \times J)} |h|^2 \d y \d s.
\end{split}
\end{align}

In order to estimate $\bariint_{Q \times J} |(\id + \i\lambda \P\M)^{-1}h_{2}|^2 \d y \d s$, we pick a smooth cut-off function $\eta \in \C_0^\infty(N^{k-1}J)$, which is valued in $[0,1]$, equal to $1$ on $N^{k-2}J$ and satisfies $(N^k\ell(J))\|\pd_{t}\eta\|_{\infty}\lesssim 1$ uniformly in $k$, $N$ and the size of $J$. With $p>2$ in the range of Lemma~\ref{lem:LpLq} we have
\begin{align*}
 \bariint_{Q \times J} |(\id + \i\lambda \P\M)^{-1}h_{2}|^2 \d y \d s & \le \frac{1}{|Q|} \barint_{J}\int_{\R^n} |(\id + \i\lambda \P\M)^{-1}h_{2}|^2 \d y \d s \\
 & \leq\frac{1}{|Q|} \left(\barint_{J}\bigg(\int_{\R^n} |(\id + \i\lambda \P\M)^{-1}h_{2}|^2 \d y\bigg)^{p/2} \d s\right)^{2/p}
 \\
 & \leq \frac{1}{|Q|\ell(J)^{2/p}} \left(\int_{\R}\bigg(\int_{\R^n} |\eta (\id + \i\lambda \P\M)^{-1}h_{2}|^2 \d y\bigg)^{p/2} \d s\right)^{2/p}.
\end{align*}
Since $\eta(t) h_{2}(x,t)=0$, we can re-express $\eta (\id + \i\lambda \P\M)^{-1}h_{2} $ using a commutator
\begin{align}
\label{eq:OffDiag commutator}
\begin{split}
 \eta (\id + \i\lambda \P\M)^{-1}h_{2} & = [\eta, (\id + \i\lambda \P\M)^{-1}]h_{2} \\
 & = (\id + \i\lambda \P\M)^{-1} [ \eta, \i\lambda \P\M] (\id + \i\lambda \P\M)^{-1}h_{2}
 \\
 & = (\id + \i\lambda \P\M)^{-1} \i\lambda [ \eta, \P] \M(\id + \i\lambda \P\M)^{-1}h_{2},
\end{split}
\end{align}
where
\begin{align*}
 [ \eta, \P] = \begin{bmatrix} 0 & 0 & -[\eta,\dhalf] \\ 0 & 0 & 0 \\ -[\eta,\HT \dhalf] & 0 & 0 \end{bmatrix}.
\end{align*}
The commutators in $[ \eta, \P]$ depend only on the $t$ variable and due to \eqref{eq:kernel of dhalf} and \eqref{eq:kernel of HTdhalf} they have kernels $K(t,s)$ bounded by
\begin{align}
 \label{eq:dhalf commutator kernel}
|K(t,s)| \le C \frac{|\eta(t)-\eta(s)|}{|t-s|^{3/2}} \lesssim \wt{h}^{-1/2}\wt{h}^{-1}\phi\bigg(\frac{t-s}{\wt{h}}\bigg)
\end{align}
with $\phi(t)= \min \{ |t|^{-1/2}, |t|^{-3/2} \}$ and $\wt{h}:=N^k\ell(J)$. Thus, by Young's inequality,
\begin{align*}
 \|[\eta,\dhalf]\|_{\L^2(\R) \to \L^p(\R)} + \|[\eta, \HT \dhalf]\|_{\L^2(\R) \to \L^p(\R)} \lesssim \tilde h^{-1/2-((1/2)- (1/p))}.
\end{align*}
Using usual vector-valued extension \cite[Sec.~II.5]{Stein}, it follows that $ [ \eta, \P] : \L^2(\L^2) \to \L^p(\L^2)$ with norm as above. This is the key point. Now, inserting \eqref{eq:OffDiag commutator} into the ongoing estimate, and applying the $\L^p(\L^2)$ boundedness of the resolvent and the just obtained commutator bound, we find
\begin{align*}
 \frac{1}{|Q|\ell(J)^{2/p}} &\left(\int_{\R}\bigg(\int_{\R^n} |\eta (\id + \i\lambda \P\M)^{-1}h_{2}|^2 \d y\bigg)^{p/2} \d s\right)^{2/p} \\ & \lesssim \frac{1}{|Q|\ell(J)^{2/p}} \left(\int_{\R}\bigg(\int_{\R^n} |\i\lambda [ \eta, \P] \M(\id + \i\lambda \P\M)^{-1}h_{2}|^2 \d y\bigg)^{p/2} \d s\right)^{2/p} \\
 & \lesssim \frac{|\lambda|^2}{|Q|\ell(J)^{2/p}} \cdot \frac{1}{(N^k\ell(J))^{1+(1-2/p)}} \int_{\R}\int_{\R^n} |\M(\id + \i\lambda \P\M)^{-1}h_{2}|^2 \d y \d s.
\intertext{Since the resolvents of $\P \M$ are $\L^2(\L^2)$-bounded, we obtain furthermore}
 & \lesssim {\frac{|\lambda|^2}{|Q|\ell(J)^2(N^k)^{2-2/p}}} \int_{\R}\int_{\R^n} |h_{2}|^2 \d y \d s
 \\
 & \lesssim \frac{2^{kn}}{(N^k)^{1-2/p}}\frac{|\lambda|^2}{\ell(J)} \bariint_{C_{k}(Q\times J)} |h|^2 \d y \d s.
\end{align*}
We remark that $\ell(J) = 4^j \ell (I) $ and $\lambda^2\sim \ell(I)$ by assumption. Hence, we have proved
\begin{align*}
 \bariint_{Q \times J} |(\id + \i\lambda \P\M)^{-1}h_{2}|^2 \d y \d s
 \lesssim \frac{2^{kn}}{(N^k)^{1-2/p}} \bariint_{C_{k}(Q\times J)} |h|^2 \d y \d s.
\end{align*}
{To conclude the proof, we only have to compare with \eqref{eq:OD1} and choose the parameters at our disposal appropriately}: First we pick $0<\eps<\eps_{0}$ and then $p$ with $\eps<\frac{1}{2} - \frac{1}{p} <\eps_{0}$. For $N$ large enough $2^{kn} \lesssim N^{-2k\eps}(N^k)^{1-2/p}$ and given any such choice of $N$, there is a choice of $m$ large verifying $ 2^{-km} N^{k+1} 2^{kn} \lesssim N^{-2k\eps}$.
\section{Quadratic estimates for \texorpdfstring{$PM$}{PM}: Proof of Theorem~\ref{thm:bhfc}}
\label{sec:proof bhfc}

Based on Lemma~\ref{lem:bisectoriality of Dirac} we see that in order to prove Theorem~\ref{thm:bhfc} it suffices to prove the quadratic estimate
\begin{equation}
\label{eq:sfMP}
 \int_0^\infty\| {\lambda{}}\M\P(\id+{\lambda{}^2}\M\P\M\P)^{-1} h \|_{2}^2 \, \frac{\mathrm{d}\lambda{}}\lambda{}\lesssim \|h\|_{2}^2 \qquad (h\in \mH)
\end{equation}
and the analogous estimate obtained upon replacing $\M \P$ by $\M^* \P^*$, where  $\mH=\L^2(\R^{n+1}; \IC^{n+2})$ as before: In fact, \eqref{eq:sfMP} for $\M \P$ implies the analogous estimate for $\P \M$ since these two operators are similar on the closure of their ranges by Remark~\ref{rem:PM and MP similar}. By the same argument, the estimate for $\M^* \P^*$ implies the one for $\P^* \M^*$. Once we have \eqref{eq:sfMP} for all four operators, an abstract duality argument gives the reverse inequalities for all four operators and hence the claim, see \cite{Mc} or \cite[Lem.~3.4.10]{EigeneDiss}.

Below, we shall only give the proof for $\M \P$ since the argument for $\M^* \P^*$ is identical upon minor changes of the algebra. Note that when $\M=\id$, then these estimates with $\P$ or $\P^*$ can simply be proved
 using a Fourier transform argument in the $(x,t)$ variables.

In proving \eqref{eq:sfMP}, we shall follow the algorithm developed in \cite{elAAM}. We set $R_\lambda=(\id+\i\lambda{} \M\P)^{-1}$ for $\lambda\in \R$. Then
\begin{align*}
 Q_{\lambda{}}= \frac{1}{2\i} (R_{-{\lambda{}}} -R_{{\lambda{}}})= {\lambda{}} \M\P(\id+{\lambda{}^2}\M\P\M\P)^{-1}
\end{align*}
and also $\frac{1}{2} (R_{-{\lambda{}}} +R_{{\lambda{}}})= (1+{\lambda{}^2} \M\P\M\P)^{-1}$. Since $\M \P$ is bisectorial, see Lemma~\ref{lem:bisectoriality of Dirac}, the operators $R_\lambda$, $Q_{\lambda{}}$
and $(\id+{\lambda^2} \M\P\M\P)^{-1}$ are uniformly bounded on $\mH$.

\subsection{Reduction to a Carleson measure}\label{sec:reduction}

As $Q_{\lambda{}}$ vanishes on $ \nul(\M\P)=\nul(\P)$, it is enough to prove the quadratic estimate (\ref{eq:sfMP}) for $h\in \clos{\ran(\M\P)}$. By density, $h\in {\ran(\M\P)}$ suffices. So, setting $\Theta_{\lambda{}}=Q_{\lambda{}} \M$, it amounts to showing
\begin{equation} \label{firsttermprop-}\qe{\Theta_{\lambda{}} \P v} \lesssim \|\P v\|_{2}^2 \qquad (v \in \dom(\P)).
\end{equation}
{We define the parabolic approximation of the identity $P_{\lambda}=(1+\lambda^2(\partial_{t}-\Delta_{x}))^{-1}$ for $\lambda \in \R$ and let it act componentwise on $\IC^{n+2}$ valued functions. By Plancherel's theorem the operators $P_\lambda$ are uniformly bounded on $\mH$.}

\begin{prop}
\label{whatisneeded}
For all $v \in \dom(\P)$ the estimate
\begin{equation*}
 \qe{\Theta _\lambda(\id-P_{\lambda{}}) \P v}\lesssim \|\P v\|_{2}^2.
\end{equation*}
\end{prop}

\begin{proof}
\label{proof:thm}
{From \eqref{eq:Psquared} we know that $\P^2$ agrees with $\partial_t - \Delta_x$ on $\dom(\P^2) \cap \ran(\P)$. Hence, $P_\lambda$ agrees with $(1 + \lambda^2 P^2)^{-1}$ on $\ran(\P)$ and to ease the computations we may think of $P_\lambda = (1 + \lambda^2 P^2)^{-1}$ just for the proof.} Since $P_{\lambda{}}$ and $\P$ commute and $(\id-P_{\lambda{}})v\in \dom(\P)$, we have
\begin{align*}
 \Theta_{\lambda{}} (\id-P_{\lambda{}}) \P v= (\Theta_{\lambda{}} \P) (\id-P_{\lambda{}})v= {\lambda{}}(\M\P)^2(\id+({\lambda{}}\M\P)^2)^{-1} (\id-P_{\lambda{}})v.
\end{align*}
Now, $({\lambda{}}\M\P)^2(\id+({\lambda{}}\M\P)^2)^{-1} =\id -(\id+({\lambda{}}\M\P)^2)^{-1}$ is uniformly bounded, hence
\begin{align*}
 \|\Theta_{\lambda{}} (\id-P_{\lambda{}}) \P v\|_{2} \lesssim \frac 1{ {\lambda{}}} \| (\id-P_{\lambda{}})v\|_{2}.
\end{align*}
Using that
\begin{align*}
 \frac{1}{\lambda}(\id-P_{\lambda{}})v= \lambda \P (\id+{\lambda{}^2}\P^2)^{-1}(\P v)
\end{align*}
we see by the square function estimate for $\P$ that
\begin{align*}
 \int_0^\infty\|\lambda \P (\id+{\lambda{}^2}\P^2)^{-1}(\P v) \|_{2}^2\,\frac{\mathrm{d}\lambda}{\lambda}\lesssim \|\P v\|_{2}^2. &\qedhere
\end{align*}
\end{proof}

Next, we perform the \emph{principal part approximation}. We use the following dyadic decomposition of $\R^{n+1}$ in parabolic dyadic cubes. Let $\dyadic= \bigcup_{j=-\infty}^\infty\dyadic_{2^j}$ where
$\dyadic_{2^j}:=\{ \delta _{j}(k+(0,1]^{n+1}) : k\in\IZ^{n+1} \}$ with $\delta _{j}(x,t)=(2^jx, 4^{j}t)$. For a parabolic dyadic cube $R\in\dyadic_{2^j}$ we denote by $\ell(R)=2^j$
its \emph{sidelength} and by $|R|= 2^{(n+2)j}$ its {\em volume}. We set $\dyadic_\lambda=\dyadic_{2^j}$ if $2^{j-1}<\lambda \le 2^j$. Let the {\em parabolic dyadic averaging operator} $S_\lambda{} : \mH \rightarrow \mH $ be given by
\begin{align*}
 S_\lambda{} u(x,t) := u_R:= \bariint_{R} u(y,s) \d y\d s = \frac{1}{|R|} \iint_R
 u(y,s)\d y\d s
\end{align*}
for every $(x,t) \in \R^{n+1}$ and $\lambda>0$, where $R $ is the unique parabolic dyadic cube in $ \dyadic_\lambda$ that contains $(x,t)$. We remark that $S^2_\lambda=S_\lambda{} $.

\begin{defn}
\label{defn:princpart}
By the {\em principal part} of $(\Theta_{\lambda{}})_{\lambda>0}$ we mean the multiplication operators $\gamma_\lambda{}$ defined by
\begin{align*}
 \gamma_\lambda{}(x,t)\zeta:= (\Theta_{\lambda{}} \zeta)(x,t)
\end{align*}
for every $\zeta\in \IC^{n+2}$. We view $\zeta$ on the right-hand side of the equation above as the constant function valued in $\IC^{n+2}$ defined on $\R^{n+1}$ by $\zeta(x,t):=\zeta$. We identify $\gamma_\lambda{}(x,t)$ with the (possibly unbounded) multiplication operator $\gamma_\lambda{}: f(x,t)\mapsto \gamma_\lambda{}(x,t)f(x,t)$.
\end{defn}

Of course $\zeta \notin \mH$ but still this definition is justified in virtue of the next lemma.

\begin{lem}
\label{lem:gammat}
The operator $\Theta_{\lambda{}}$ extends to a bounded operator from $\L^\infty$ into $ \Lloc^2$. In particular we have well defined functions $\gamma_\lambda{}\in \Lloc^2(\R^{n+1}; \Lop(\IC^{n+2}))$ with bounds
\begin{align*}
 \bariint_{R} |\gamma_\lambda{}(x,t)|^2 \d x\d t \lesssim 1
\end{align*}
for all $R\in\dyadic_\lambda$. Here, $|\cdot |$ is the operator norm of $\Lop(\IC^{n+2})$. Moreover, $\gamma_\lambda{} S_\lambda{}$ is bounded on $\mH$ with $\|\gamma_\lambda{} S_\lambda{} \|\lesssim 1$ uniformly for all $\lambda>0$.
\end{lem}

\begin{proof}
Fix a parabolic cube $R \in \dyadic_\lambda$ and $f \in \L^\infty(\R^n;\IC^{n+2})$ with $\|f\|_\infty=1$. Then write $f= f_1+f_2+\ldots$ where $f_1=f$ on $2R$ and $0$ elsewhere and if $j\ge 2$, $f_j=f$ on $C_{j}(R)$ with the notation of Proposition~\ref{prop:OffDiag} and $0$ elsewhere. Then apply $\Theta_{\lambda} = \frac{1}{2\i}(R_{-\lambda} - R_\lambda)\M$, use the off-diagonal estimates of Proposition~\ref{prop:OffDiag} for each term $\Theta_{\lambda{}} f_j$, with $N$ large enough, and sum to obtain {for $\Theta_\lambda f: = \sum_{j=1}^\infty \Theta_\lambda f_j$ the estimate}
\begin{align*}
 \bariint_{R} |(\Theta_{\lambda{}}f)(x,t)|^2 \d x \d t \le C.
\end{align*}
If we do this for the constant functions with values describing an orthonormal basis of $\IC^{n+2}$ and sum, we obtain an upper bound for the desired average of $\gamma_\lambda{}$.
Next, for a function $f\in \mH$,
\begin{align*}
\|\gamma_\lambda{} S_\lambda{} f\|_{2}^2 = \sum_{R\in \dyadic_\lambda} \iint_R \left|\gamma_\lambda{} (x,t)\bigg( \bariint_{R} f\bigg)\right|^2 \d x\d t \lesssim \sum_{R\in \dyadic_\lambda} |R|\left|\,\bariint_{R} f\right |^2 \le \|f\|_{2}^2.
\end{align*}
Here, $|\cdot|$ is the usual Hermitian norm on $\IC^{n+2}$.
\end{proof}

We have the following principal part approximation of $\Theta_{\lambda{}}$ by $\gamma_\lambda{}S_\lambda{} $.

\begin{lem}
\label{lem:ppa}
It holds
\begin{equation*}
 \qe{\Theta_{\lambda{}}\P_{\lambda{}}f-\gamma_\lambda{} S_\lambda{} f}
 \lesssim \|f\|_{2}^2 \qquad (f\in \mH).
\end{equation*}
\end{lem}

\begin{proof}
Let $\Phi_\lambda$ be a nice parabolic approximation of the identity on $\mH$, that is, the convolution operator by a real valued function $\lambda^{-n-2} \phi(\lambda^{-1}x, \lambda^{-2} t)$ with $\phi \in \C_0^\infty(B(0,1) \times (-1,1))$ satisfying $\iint_\ree \phi(x,t) \d x \d t = 1$. Write
\begin{align}
\label{eq1:ppa}
 \Theta_{\lambda{}}\P_{\lambda{}} -\gamma_\lambda{} S_\lambda{}
 = (\Theta_{\lambda{}}\P_{\lambda{}} -\gamma_\lambda{}S_\lambda{} \P_{\lambda{}} )
 + (\gamma_\lambda{}S_\lambda{} (\P_{\lambda{}}-\Phi_\lambda{} ))
 + (\gamma_\lambda{}S_\lambda{} (\Phi_{\lambda{}}-S_\lambda{} ))
 + (\gamma_\lambda{}S_\lambda ^2 - \gamma_\lambda{}S_\lambda{} ).
\end{align}
Because of $S_\lambda^2=S_\lambda{} $, the last term vanishes. Next, as the $\gamma_\lambda{}S_\lambda{} $ are uniformly bounded as operators on $\mH$,
\begin{align*}
 \int_0^\infty \|\gamma_\lambda{}S_\lambda{} (\P_{\lambda{}}-\Phi_\lambda{} )f\|_2^2 + \|\gamma_\lambda{}S_\lambda{} (\Phi_{\lambda{}}-S_\lambda{} )f\|_2^2 \, \frac{\mathrm{d} \lambda}{\lambda}
 \lesssim  \int_0^\infty \|(\P_{\lambda{}}-\Phi_\lambda{} )f\|_2^2 + \|(\Phi_{\lambda{}}-S_\lambda{} )f\|_2^2 \, \frac{\mathrm{d} \lambda}{\lambda}.
\end{align*}
A straightforward Fourier transform argument reveals control by $\|f\|_2^2$ for the square function of $\P_\lambda - \Phi_\lambda$. The analogous bound for the square function of $\Phi_\lambda - S_\lambda$ is done componentwise and for each component it is precisely the statement of \cite[Lem.~2.5]{CNS}. As for the first term in \eqref{eq1:ppa}, we will show that for $\alpha \in (0,1)$ small enough there is a bound
\begin{equation}
\label{eq:Poincare}
 \|(\Theta_{\lambda{}}\P_{\lambda{}} -\gamma_\lambda{}S_\lambda{} \P_{\lambda{}} )f \|_{2}^2 \lesssim \lambda^2 \|\nabla_{x} \P_{\lambda{}}f \|_{2}^2+ \lambda^{4\alpha} \|\dalpha \P_{\lambda{}}f\|_{2}^2,
\end{equation}
where $\dalpha$ is the fractional derivative of order $\alpha$ on $\R$ defined via the Fourier symbol $|\tau|^\alpha$. We postpone the proof of this claim until Section~\ref{sec:proofEqPoincare}. Hence, we obtain
\begin{align*}
\qe{\Theta_{\lambda{}}\P_{\lambda{}}f-\gamma_\lambda{} S_\lambda{} \P_{\lambda{}} f}
 \lesssim \qe{\lambda \nabla_{x} \P_{\lambda{}} f} + \qe{\lambda^{2\alpha}\dalpha \P_{\lambda{}} f} \lesssim \|f\|_{2}^2
\end{align*}
from a Fourier transform argument on each component of $f$, given the definition of $\P_{\lambda{}}$.
\end{proof}

Combining Lemma~\ref{lem:ppa} with Proposition~\ref{whatisneeded}, we obtain the principal part approximation
\begin{equation}\label{eq:ppa+}
 \qe{\Theta_{\lambda{}} \P v-\gamma_\lambda{} S_\lambda{} \P v}
 \lesssim \|\P v\|_{2}^2 \qquad (v\in \dom(\P)).
\end{equation}
Hence, to complete the proof of \eqref{firsttermprop-} we want to apply Carleson's theorem to conclude that
\begin{align*}
 \qe{\gamma_\lambda{} S_\lambda{} \P v} \lesssim \|\gamma_\lambda{}\|_C^2 \| \P v\|_{2}^2 \qquad (v\in \dom(\P)),
\end{align*}
where
\begin{equation*}
 \|\gamma_\lambda{}\|_C^2=\sup_{R\in \dyadic} \frac{1}{|R|}\iiint_{T(R)} |\gamma_\lambda{}(x,t)|^2 \, \frac{\mathrm{d} x\d t\d\lambda}\lambda.
\end{equation*}
Here $T(R):= (0,\ell(R)]\times R$ is the Carleson box over $R$. Hence, it remains to prove the Carleson measure estimate $\|\gamma_\lambda{}\|_C<\infty$, and the claim in \eqref{eq:Poincare}, to finish the proof of the quadratic estimate in \eqref{eq:sfMP}.

\subsection{The T(b) argument}
\label{sec:tb}

We first construct test functions which belong to the range of $\P$.

Let $\eta: \R \to [0,1]$ be a smooth non-negative function with support in $[-1,1]$ which is 1 on $[-1/2,1/2]$. Consider $\nu(t)= \eta(t)- \eta (t-z)$ for $z\in \R$, $|z|>2$. {Since $\wh{\nu}(0) = 0$, we can define $h := (-\HT\dhalf)^{-1} \nu$ by Fourier transform}. We claim that $h\in \L^2(\R)$ satisfies $|h(t)-h(0)| \le C|t|$ and that we can find some $z$ such that $h(0)\ne 0$. Indeed,
\begin{align*}
 h(t)= -\frac{1}{2\pi}\int_{\R} \e^{\i t\tau} \frac{1- \e^{\i z\tau}}{\i\, \sgn(\tau) |\tau|^{1/2}}\ \wh{\eta}(\tau)\d\tau,
\end{align*}
so that {the cheap estimate $|1- \e^{\i t\tau}| \le |t \tau|$ yields}
\begin{align*}
 |h(t)-h(0)| \le \frac{ |t|}{\pi} \int_{\R} { |\tau|^{1/2}} |\wh{\eta}(\tau)|\d \tau,
\end{align*}
and likewise, using $|1- \e^{\i z\tau}| \le |z\tau|$ and Plancherel's theorem, it follows that
\begin{align*}
 \|h\|_{2}^2
\le \frac{|z|^2}{2\pi}\int_{\R} {|\tau|} |\wh{\eta}(\tau)|^2 \d \tau <\infty.
\end{align*}
To see the last claim, assume to contrary that $h(0) = 0$ holds for any $|z|>2$. By differentiating the formula for $h(0)$ with respect to $z$, we find
\begin{align*}
\frac{1}{2\pi}\int_{\R} \e^{\i z\tau}|\tau|^{1/2} \wh{\eta}(\tau) \d \tau =0 \qquad (|z| > 2).
\end{align*}
This means that the Fourier transform of $\tau\mapsto |\tau|^{1/2} \wh{\eta}(\tau)$ is compactly supported, hence this function is analytic. In particular, it is analytic at the origin. But $\tau\mapsto \wh{\eta}(\tau)$ has the same property and {$\wh{\eta}(0) = \int_\R \eta \geq 1$}. Hence $\tau\mapsto |\tau|^{1/2}$ should be analytic at the origin, which is absurd. We now fix $z$ with $|z|>2$ such that all the properties above hold.

Let $\zeta\in \IC^{n+2}$ with $|\zeta|=1$. Let $\zeta_{i}$ denote the components of $\zeta$ and $\zeta_{\pa}= (\zeta_{2}, \ldots, \zeta_{n+1})$. We need also a parameter $0<\delta \le 1 $ which is small and will be chosen later. We let $\chi$ be a smooth function on $\R^n$, valued in $[0,1]$, which is 1 on the cube $[-1/2,1/2]^n$ and with support in $[-1,1]^n$.

Fix a parabolic dyadic cube $R=Q\times I$, its center being denoted $(x_{Q}, t_{I})$. Observe that with our notation $\ell(R)=\ell(Q)= \sqrt {\ell(I)}$. We set
\begin{align*}
 \chi_{Q}(x) &:= \chi \bigg( \frac{x-x_{Q}}{\ell(Q)}\bigg) \qquad (x\in \R^n),\\
 h_{I,\delta }(t) &:= \delta ^{-1/2} \ell(I)^{1/2} h\bigg(\delta\cdot \frac{t-t_{I}}{\ell(I)}\bigg) \qquad (t\in \R),\\
 f_{1}(x) &:= \zeta_{n+2} - \frac{\delta ^{1/2} \zeta_{\pa}\cdot (x-x_{Q})}{h(0)\ell(Q)} \qquad (x\in \R^n),\\
 f_{2}(x) &:= \frac{\delta ^{1/2} \zeta_{1} (x_{1}-x_{Q,1})}{h(0)\ell(Q)} \qquad (x\in \R^n)
\end{align*}
and finally we introduce $L^\zeta_{R,\delta }:\R^{n+1}\to\mathbb C^{n+2}$ through
\begin{align*}
 L^\zeta_{R,\delta }(x,t):= \begin{bmatrix}
 \chi_{Q}(x)f_{1}(x) \ h_{I,{\delta }}(t) \\
 \chi_{Q}(x) f_{2}(x) \ h_{I,{\delta }}(t) \\
 0\\
 \vdots \\
 0
\end{bmatrix} \qquad ((x,t)\in \R^{n+1}).
\end{align*}
Clearly $L^\zeta_{R,\delta } \in \mH $, but we do not need its $\mH$-norm and so we do not compute it. We claim that $ L^\zeta_{R,\delta }\in \dom(\P)$ with
\begin{equation}
\label{1}
\|\P L^\zeta_{R,\delta }\|_{2}^2 \le C\delta ^{-2} |R|
\end{equation}
and that
\begin{equation}
\label{2}
\iint_{R} |\P L^\zeta_{R,\delta }- \zeta|^2 \d x\d t \le C\delta |R|
\end{equation}
{with an absolute constant $C$.} Indeed, we first note that
\begin{align*}
\P L^\zeta_{R,\delta }(x,t)= \begin{bmatrix}
\ \ \partial_{x_{1}}( \chi_{Q}(x) f_{2}(x)) \ h_{I,{\delta }}(t) \\
 -\partial_{x_{1}}( \chi_{Q}(x) f_{1}(x)) \ h_{I,{\delta }}(t) \\
 \vdots\\
 -\partial_{x_{n}}( \chi_{Q}(x) f_{1}(x)) \ h_{I,{\delta }}(t)\\
 \ \chi_{Q}(x) f_{1}(x)\ (-\HT\dhalf )h_{I,\delta }(t)
 \end{bmatrix}.
\end{align*}
Furthermore, we observe by a change of variables that
\begin{align*}
\| h_{I,\delta }\|_{\L^2(\R)} = \delta ^{-1} \ell(I) \|h\|_{\L^2(\R)}
\end{align*}
and $(-\HT\dhalf )h_{I,\delta }(t)= \nu\big(\delta\cdot \frac{t-t_{I}}{\ell(I)}\big)$, so that
\begin{align*}
\| (-\HT\dhalf) h_{I,\delta }\|_{\L^2(\R)} = \delta ^{-1/2} \ell(I)^{1/2}\|\nu\|_{\L^2(\R)}.
\end{align*}
Using that the partial derivatives of the functions of $x$ are bounded (uniformly for $\delta \le 1$) by $C/\ell(Q)$ with support in $2Q$, we get
\begin{align*}
\|P L^\zeta_{R,\delta }\|_{2}^2 \le C (\ell(Q)^{n-2} \delta ^{-2} \ell(I)^2+ \ell(Q)^n \delta^{-1} \ell(I))\lesssim \delta ^{-2} |R|.
\end{align*}
Because of the support properties of $\chi_{Q}$ and $\nu$ we have for $(x,t)\in R$ that
\begin{align*}
\P L^\zeta_{R,\delta }(x,t)= \begin{bmatrix}
\ \ \partial_{x_{1}}f_{2}(x) \ h_{I,{\delta }}(t)\vphantom{\frac{t-t_{I}}{\ell(I)}} \\
 - \partial_{x_{1}} f_{1}(x) \ h_{I,{\delta }}(t)\vphantom{\frac{t-t_{I}}{\ell(I)}} \\
 \vdots\\
 - \partial_{x_{n}} f_{1}(x) \ h_{I,{\delta }}(t)\vphantom{\frac{t-t_{I}}{\ell(I)}}\\
 f_{1}(x)\vphantom{\frac{\delta ^{1/2} \zeta_{\pa}\cdot (x-x_{Q})}{h(0)\ell(Q)}}
 \end{bmatrix} = \begin{bmatrix}
 \zeta_{1}\ h(0)^{-1} h\big(\delta\cdot \frac{t-t_{I}}{\ell(I)}\big) \\
 \zeta_{2}\ h(0)^{-1} h\big(\delta\cdot \frac{t-t_{I}}{\ell(I)}\big) \\
 \vdots\\
 \zeta_{n+1}\ h(0)^{-1} h\big(\delta\cdot \frac{t-t_{I}}{\ell(I)}\big)\\
 \zeta_{n+2} - \frac{\delta ^{1/2} \zeta_{\pa}\cdot (x-x_{Q})}{h(0)\ell(Q)}
 \end{bmatrix},
\end{align*}
so that at the center of the cube $\P L^\zeta_{R,\delta }(x_{Q},t_{I})= \zeta$ and by the Lipschitz property of $h$ at $0$ and $|\zeta|=1$, we get
\begin{align*}
|\P L^\zeta_{R,\delta }(x,t) - \zeta| \le |h(0)|^{-1}( C\delta + \delta ^{1/2}) \qquad ((x,t)\in R).
\end{align*}
This completes the proof of \eqref{1} and \eqref{2}.

We now define the test functions $b^\zeta_{R,\eps,\delta }$ for $\eps, \delta \in (0,1)$ by
\begin{align*}
 b^\zeta_{R,\eps, \delta }: = \P v^\zeta_{R,\eps,\delta } \qquad v^\zeta_{R,\eps, \delta }:= (\id+\i\eps \ell\M\P)^{-1}L^\zeta_{R,\delta },
\end{align*}
where $\ell=\ell(R)$.

\begin{lem}
\label{lem:testfunctions}
There exists $C>0$ such that for each $\zeta\in \IC^{n+2}$ with $|\zeta|=1$, each parabolic dyadic cube $R\subset\R^{n+1}$ and each $\eps, \delta \in (0,1)$,
 \begin{equation}\label{eq:1}
\iint_{\R^{n+1}}|v^\zeta _{R,\eps, \delta } - L^\zeta_{R,\delta } |^2\d x\d t \le C (\eps \ell)^2 \delta^{-2}|R|,
\end{equation}
\begin{equation}\label{eq:2}
\iint_{\R^{n+1}}|b^\zeta _{R,\eps, \delta } - \P L^\zeta_{R,\delta }|^2 \d x\d t \le C\delta ^{-2} |R|,
\end{equation}
\begin{equation}\label{eq:3}
 \left| \bariint_{R} (b^\zeta_{R,\eps,\delta } - \zeta) \d x\d t \right| \le
 C (\eps^{1/3} \delta^{-1} + \delta^{1/2}),
\end{equation}
\begin{equation}\label{eq:3b}
\iint_{R}|b^\zeta _{R,\eps, \delta } - \zeta|^2 \d x\d t \le C\delta ^{-2} |R|,
\end{equation}
\begin{equation}\label{eq:4}
\iiint_{T(R)} |\gamma_\lambda{}(x,t) S_\lambda{} b^\zeta _{R,\eps, \delta }(x,t)|^2 \, \frac{\mathrm{d} x\d t \d\lambda}{\lambda} \le C\eps^{-2}\delta ^{-2} |R|.
\end{equation}
\end{lem}

\begin{proof}
Using $(\id+\i \lambda \M\P)^{-1}-\id = -\i \lambda \M\P(\id+\i \lambda \M\P)^{-1}$, we have
\begin{align*}
 v^\zeta _{R,\eps, \delta } - L^\zeta_{R,\delta} = -\i\eps \ell(\id+\i\eps \ell \M\P)^{-1}(\M\P L^\zeta_{R,\delta }).
\end{align*}
The inequality \eqref{1} for $\P L^\zeta_{R,\delta }$ and the uniform boundedness of $s \mapsto (\id+\i s \M\P)^{-1}\M$ imply \eqref{eq:1}. Applying $\P$, we see that
\begin{align*}
b^\zeta _{R,\eps} - \P L^\zeta_{R,\delta }= -\i\eps \ell \P(\id+\i\eps \ell \M\P)^{-1}(\M \P L^\zeta_{R,\delta }).
\end{align*}
Again the inequality \eqref{1} and the boundedness of $s \mapsto s \P(\id+\i s \M\P)^{-1}\M$ imply \eqref{eq:2}. Next, \eqref{eq:3b} follows directly from \eqref{2} and \eqref{eq:2}, keeping in mind that $\delta \in (0,1)$.

As for \eqref{eq:4}, we first note, using \eqref{eq:ppa+}, \eqref{1} and the fact that $b^\zeta_{R,\eps,\delta }=\P v^\zeta _{R,\eps,\delta }$, that it suffices to establish the estimate
\begin{equation}
\label{eq:5}
 \iiint_{T(R)} |\Theta_{\lambda{}}b^\zeta_{R,\eps,\delta }(x,t)|^2 \, \frac{\mathrm{d} x\d t \d\lambda}{\lambda} \le C\eps^{-2}\delta ^{-1} |R|.
\end{equation}
Now,
\begin{align*}
 \Theta_{\lambda{}}b^\zeta_{R,\eps,\delta }&= {\lambda{}}\M\P(\id + {\lambda{}^2}\M\P\M\P)^{-1}\M\P(\id+\i\eps \ell \M\P)^{-1}L^\zeta_{R,\delta } \\
&= {\lambda{}}\M\P(\id+ {\lambda{}^2}\M\P\M\P)^{-1}(\id+\i\eps \ell \M\P)^{-1}(\M\P L^\zeta_{R,\delta }) \\
& =(\lambda/\eps \ell) (\id + {\lambda{}^2}\M\P\M\P)^{-1}\eps \ell \M\P(\id+\i\eps \ell \M\P)^{-1}(\M\P L^\zeta_{R,\delta }).
\end{align*}
Since $(\id + {\lambda{}^2}\M\P\M\P)^{-1}$ and $\eps \ell \M\P(\id+\i\eps \ell \M\P)^{-1}$ are bounded uniformly with respect to $\lambda$ and $\eps \ell$, we have
\begin{align*}
\|\Theta_{\lambda{}}b^\zeta_{R,\eps,\delta }\|_{2}\le C(\lambda/\eps \ell)\|\P L^\zeta_{R,\delta }\|_{2}.
\end{align*}
Integrating over $\lambda \in (0, \ell]$, we obtain \eqref{eq:5}.

It remains to establish \eqref{eq:3}. Let $\varphi\colon \R^{n+1}\to [0,1]$ be a smooth function which is $1$ on $(1-s)Q \times (1-s^2)I$, supported on $R$, with $\|\nabla_{x} \varphi\|_\infty \le C(s\ell)^{-1}$ and $\|\partial_{t}\varphi\|_{\infty}\le C(s\ell)^{-2}$, where $s\in (0,1)$ still has to be chosen. We can write
\begin{align*}
 \iint_{R}( b^\zeta_{R,\eps,\delta } - \zeta)\d x\d t = \iint_{R} (b^\zeta_{R,\eps,\delta } - \P L^\zeta_{R,\delta } )\d x\d t + \iint_{R} (\P L^\zeta_{R,\delta }-\zeta)\d x\d t.
\end{align*}
Due to \eqref{2}, the second integral is bounded by $\delta ^{1/2}|R|$. Next, we split the first integral as
\begin{align*}
\iint_{R} \varphi \P (v^\zeta_{R,\eps,\delta } - L^\zeta_{R,\delta } ) \d x\d t + \iint_{R} (1-\varphi)(b^\zeta_{R,\eps, \delta } - P L^\zeta_{R,\delta })
=:\text{I}+\text{II}.
\end{align*}
Using \eqref{eq:2}, the properties of $\varphi$ and the Cauchy-Schwarz inequality, we obtain
\begin{align*}
| \text{II} | \le Cs^{1/2} \delta ^{-1} |R|.
\end{align*}
{The Euclidean norm of $\text{I} \in \IC^{n+2}$ can be computed by
\begin{align*}
 |\text{I}|^2 = \sum_{j=1}^{n+2} \bigg| \iint_{\R^{n+1}} (v^\zeta_{R,\eps,\delta } - L^\zeta_{R,\delta } ) \overline {\P^* \overrightarrow{\varphi_j} } \d x\d t\bigg|^2
\end{align*}
with $\overrightarrow{\varphi_j} =(0,\ldots,0,\varphi,0 \ldots, 0)\in \mH$ and the non-zero entry sitting at $j$-th position.} Now, $\P^* \overrightarrow{\varphi_j}$ can contain two types of terms: the first one are first order derivatives of $\varphi$ in the $x$ variable, and the second one are $\dhalf \varphi$ or $\HT\dhalf \varphi$ acting on the $t$ variable. Both are functions in $\mH$ with bounds on the same order of $(s\ell)^{-1}|R|^{1/2}$. For the first one this follows by construction. For the second one, we obtain the order $(s^{-2}\ell^{-2})^{1/2}|R|^{1/2}$ from the $\L^2$-bounds on $\varphi$ and $\partial_t \varphi$ and the inequality $\|\dhalf \varphi\|_2 \leq \|\varphi\|_2^{1/2} \|\partial_t \varphi\|_2^{1/2}$, which follows directly from Plancherel's theorem. Thus, we can apply \eqref{eq:1} to the effect that
\begin{align*}
| \text{I} | \le C(\eps/ s) \delta ^{-1} |R|.
\end{align*}
Hence, choosing $s= \eps^{2/3}$, we have shown \eqref{eq:3}.
\end{proof}

In conclusion, setting $b^\zeta_{R,\eps}:= b^\zeta_{R,\eps, \eps^{1/6}}$, we obtain
\begin{equation*}
 \left| \bariint_{R} (b^\zeta_{R,\eps} - \zeta) \d x \d t\right| \le C \eps^{1/12} \qquad \text{and} \qquad \bariint_{R} |b^\zeta_{R,\eps} - \zeta|^2 \d x \d t \le C \eps^{-1/3}.
\end{equation*}
At this point, we can run the stopping time argument presented in \cite[Sec.~3.4]{elAAM} with this family of test functions instead of $b^w_{Q,\eps}$ there, using Lemma~\ref{lem:testfunctions} exactly as Lemma 3.10 there, to obtain that $\gamma_\lambda{}(x,t)$ is a parabolic Carleson function. As the argument is  really line by line the same with  only few minor  cosmetic changes, we omit details and refer the reader to \cite{elAAM}. This finishes the proof of Theorem~\ref{thm:bhfc} {modulo the proof of \eqref{eq:Poincare} that we present below}.

\subsection{Proof of \eqref{eq:Poincare}}
\label{sec:proofEqPoincare}

To start the proof we first note, for $\lambda>0$ fixed and $(x,t)\in \R^{n+1}$, that
\begin{align*}
(\Theta_{\lambda{}}\P_{\lambda{}} -\gamma_\lambda{}S_\lambda{} \P_{\lambda{}} )f(x,t)= \Theta_{\lambda{}} \bigg(g- \bariint_{R} g \bigg) (x,t),
\end{align*}
where $g=\P_{\lambda{}}f$ and $R$ is the only parabolic dyadic cube in $\dyadic_\lambda$ containing $(x,t)$. Let us write $R=Q\times I$ as usual. Using the notation of Proposition~\ref{prop:OffDiag} with $C_{k}(R)=C_{k}(Q\times I)$ when $k\ge 1$ and $C_{0}(R)=2Q\times NI$, we obtain
\begin{align*}
 \|(\Theta_\lambda\P_\lambda -\gamma_\lambda S_\lambda \P_\lambda )f \|_{2}^2 &= \sum_{R\in \dyadic_\lambda} |R| \bariint_R \bigg| \Theta_\lambda\Big(g- \bariint_{R} g \Big)\bigg|^2 \\
& \le \sum_{R\in \dyadic_\lambda} |R| \left(\sum_{k\ge 0} \bigg(\bariint_R \bigg| \Theta_\lambda \bigg(1_{C_k(R)} \Big(g- \bariint_{R} g \Big)\bigg)\bigg|^2\ \bigg)^{1/2}\right)^{2}.
\intertext{{Since $\Theta_\lambda = \frac{1}{2\i}(R_{-\lambda} - R_\lambda)\M$ and as the bounded multiplication operator $\M$ commutes with $1_{C_k(R)}$, we can apply Proposition~\ref{prop:OffDiag} to  continue the estimate with}}
& \lesssim \sum_{R\in \dyadic_\lambda} |R| \left(\sum_{k\ge 0} N^{-k\eps} \bigg(\bariint_{C_k(R)}\Big|g- \bariint_{R} g \Big|^2\ \bigg)^{1/2}\right)^{2} \\
& \lesssim \sum_{R\in \dyadic_\lambda} \sum_{k\ge 1} |R| N^{-2k\eps} \bariint_{2^kQ \times N^kI } \Big|g- \bariint_{R} g \Big|^2,
\end{align*}
where in the second step we have used $C_k(R)$ is contained in  $ 2^{k+1}Q \times N^{k+1}I$ and has comparable measure, and shifted the index $k$. Now, we write
\begin{align}
\label{eq1:eqPoincare}
g- \bariint_{R} g= \bigg(g - \barint_{Q}g \bigg) + \bigg(\barint_{Q} \Big(g-\barint_{I}g\Big) \bigg).
\end{align}
Then, using Poincar\'e's inequality \cite[Lem.~7.12/16]{Gilbarg-Trudinger} in the $x$ variable,
\begin{align*}
\int_{2^kQ } \Big|g- \barint_{Q} g \Big|^2 \lesssim 2^{2k} \ell(Q)^2 \int_{2^{k}Q} |\nabla_{x} g |^2.
\end{align*}
Hence, integrating over $N^kI$,
\begin{align*}
 \sum_{R\in \dyadic_\lambda} \sum_{k\ge 1} |R| N^{-2k\eps} &\bariint_{2^kQ \times N^kI } \Big|g- \barint_{Q} g \Big|^2 \\
& \lesssim \sum_{R\in \dyadic_\lambda} \sum_{k\ge 1} |R| N^{-2k\eps} 2^{2k} \ell(Q)^2 \bariint_{2^{k}Q\times N^kI} |\nabla_{x} g |^2.
\intertext{Since $\lambda\sim \ell(Q)$ and $\sum_{R\in \dyadic_\lambda} 1_{2^{k}Q\times N^kI}= 2^{kn}N^k$, we can continue by}
& \lesssim \lambda^2 \sum_{k\ge 1} N^{-2k\eps} 2^{2k} \iint_{\R^{n+1}} |\nabla_{x} g |^2 \lesssim \lambda^2 \|\nabla_{x} g \|_{2}^2,
\end{align*}
where we have assumed $N^{2\eps}>4$ by taking $N$ large enough as we may.

For the second term in \eqref{eq1:eqPoincare}, we use the following lemma as a substitute for Poincar\'e's inequality. Recall that $\dalpha$ is the fractional derivative of order $\alpha$ on $\R$ defined via the Fourier symbol $|\tau|^\alpha$. We write $\H^\alpha(\R)$ for space of all $h \in \L^2(\R)$ with $\dalpha h \in \L^2(\R)$.

\begin{lem}
\label{lem:fractional poincare alpha}
Assume $h\in \H^\alpha(\R)$ with $\alpha \in (0, 1/2)$ and let $p,q \in [1,\infty)$ be such that $(1-\alpha)p<q \leq p$. Then for any interval $J$ and any $N\ge 2$,
\begin{align*}
\bigg(\barint_{J} \Big| h- \barint_{J}h\Big|^p \d s\bigg)^{1/p} \lesssim \ell(J)^{\alpha}  \bigg(\sum_{l \ge 1} N^{(\alpha -1)l} \barint_{N^l J} |\dalpha h|^q \d s \bigg)^{1/q}.
\end{align*}
The analogous inequality with $\HT \dalpha h$ instead of $\dalpha h$ on the right-hand side also holds and both inequalities remain valid for $\alpha = 1/2$ and $h \in \Hdot^{1/2}(\R)$.
\end{lem}

\begin{proof}
We treat the estimates with $\dalpha h$ first. If $\alpha < 1/2$, then we can represent $h$ as a classical Riesz potential
\begin{align*}
 h(s) = \frac{1}{2 \Gamma(\alpha) \cos(\alpha \pi/2)} \int_\R \frac{\dalpha h(\sigma)}{|s-\sigma|^{1-\alpha}} \d \sigma \qquad (\text{a.e.\ $s \in \R$}),
\end{align*}
see for example \cite[Sec.~12.1]{SamkoEtAl}. If $\alpha = 1/2$, then we use the modified representation obtained in Corollary~\ref{cor:potential representation of Hdot1/2}. Hence, in any case we have
\begin{align*}
 h(s) - h(s') = C_\alpha \int_\R \bigg(\frac{1}{|s-\sigma|^{1 - \alpha}} - \frac{1}{|s' - \sigma|^{1 - \alpha}} \bigg) \dalpha h(\sigma) \d \sigma \qquad(\text{a.e.\ $s,s' \in J$})
\end{align*}
for a constant $C_\alpha$ depending only on $\alpha$. Averaging first in $s'$ and then in $s$ yields
\begin{align*}
 \bigg(\barint_J \Big|h(s) - \barint_J h \Big|^p \d s \bigg)^{1/p}
 \leq C_\alpha \bigg(\barint_J \bigg(\int_{\R} \Phi(s,\sigma) |\dalpha h(\sigma)| \d \sigma\bigg)^p \d s\bigg)^{1/p},
\end{align*}
where
\begin{align*}
 \Phi(s,\sigma) := \barint_J \bigg|\frac{1}{|s-\sigma|^{1-\alpha}} - \frac{1}{|s'-\sigma|^{1-\alpha}}\bigg| \d s'
\end{align*}
is seen to satisfy
\begin{align*}
 |\Phi(s,\sigma)| \lesssim \ell(J)^{\alpha} \frac{1}{\ell(J)}\min \bigg\{ \bigg|\frac{s-\sigma}{\ell(J)}\bigg|^{\alpha-1}, \bigg|\frac{s-\sigma}{\ell(J)}\bigg|^{\alpha-2} \bigg\} \qquad (s \in J, \, \sigma \in \R),
\end{align*}
using either the triangle inequality (if $\sigma \in 2J$) or the mean-value theorem (if $\sigma \in {}^c(2J)$). In particular, $\int_\R \Phi(s, \sigma) \d \sigma \lesssim \ell(J)^{\alpha}$. So, writing $\Phi = \Phi^{1-1/q} \Phi^{1/q}$, we can apply H\"older's inequality to the inner integral to find
\begin{align*}
\bigg(\barint_J \bigg(\int_\R \Phi(s,\sigma) |\dalpha h(\sigma)| \d \sigma \bigg)^{p} \d s \bigg)^{1/p}
 &\lesssim \ell(J)^{\alpha - \alpha/q} \bigg(\barint_J \bigg(\int_\R \Phi(s,\sigma) |\dalpha h(\sigma)|^q \d \sigma \bigg)^{p/q} \d s \bigg)^{1/p}
\end{align*}
and since $p \geq q$ by assumption, we obtain from Minkowski's inequality
\begin{align}
\label{eq1:fractional poincare alpha}
\bigg(\barint_J \Big|h(s) - \barint_J h \Big|^p \d s \bigg)^{1/p}
 &\lesssim \ell(J)^{\alpha - \alpha/q} \bigg(\int_\R \bigg(\barint_J \Phi(s,\sigma)^{p/q} \d s \bigg)^{q/p} |\dalpha h(\sigma)|^q \d \sigma \bigg)^{1/q}.
\end{align}
Putting $C_0(J) := N J$ and $C_l(J) := N^{l+1} J \setminus N^l J$, $l \geq 1$, the pointwise bounds for $\Phi$ imply
\begin{align*}
 \bigg(\barint_J \Phi(s,\sigma)^{p/q}  \d s \bigg)^{q/p} \lesssim \ell(J)^{\alpha -1} N^{(\alpha-2)l} \qquad (\sigma \in C_l(J)),
\end{align*}
where it is the term for $l=0$ that requires the assumption $(\alpha-1)p/q > -1$. This being said, the claim follows upon splitting $\R = \bigcup_{l = 0}^\infty C_l(J)$ on the right-hand side of \eqref{eq1:fractional poincare alpha}.

The argument is exactly the same if we want to let $\HT \dhalf h$ appear on the right-hand side: Indeed, for $\alpha < 1/2$ we can again rely on a classical representation~\cite[Sec.~12.1]{SamkoEtAl}:
\begin{align*}
 h(s) = \frac{1}{2 \Gamma(\alpha) \sin(\alpha \pi/2)} \int_\R \frac{\sgn(s-\sigma) \HT \dalpha h(\sigma)}{|s-\sigma|^{1-\alpha}} \d \sigma \qquad (\text{a.e.\ $s \in \R$}).
\end{align*}
Likewise, for $\alpha = 1/2$ we have $h = c - \ihalf \HT (\HT \dhalf h)$ for some constant $c$ (Corollary~\ref{cor:potential representation of Hdot1/2}) and the required integral representation for $\ihalf \HT$ has been obtained in Proposition~\ref{prop:potential properties}.
\end{proof}

Returning to the second term in \eqref{eq1:eqPoincare}, a telescoping sum and Lemma~\ref{lem:fractional poincare alpha} applied with $p=q=2$ and an arbitrary $\alpha \in (0,1/2)$ yield the bound
\begin{align*}
 \bariint_{2^kQ \times N^kI } \bigg| \barint_{Q }\Big(g- \barint_{I} g\Big) \bigg|^2 &
 \le \barint_{Q}\barint_{N^kI} \bigg|g- \barint_{I} g\bigg|^2 \\
 & \lesssim \barint_{Q} (k+1)^2 \sum_{j=0}^k \barint_{N^jI} \bigg|g- \barint_{N^jI} g\bigg|^2
 \\
 & \lesssim \barint_{Q} (k+1)^2 \sum_{j=0}^k \ell(N^j I)^{2\alpha} \sum_{l \ge 1} N^{(\alpha-1)l} \barint_{N^l N^j I} |\dalpha g|^2.
 \end{align*}
So, using $\sum_{R\in \dyadic_\lambda} 1_{Q\times N^{l+j}I}= N^{l+j}$, we obtain
\begin{align*}
 \sum_{R\in \dyadic_\lambda} \sum_{k\ge 1} &|R| N^{-2k\eps} \bariint_{2^kQ \times N^kI } \bigg| \barint_{Q }\Big(g- \barint_{I} g\Big) \bigg|^2 \\
&\lesssim \sum_{k\ge 1} N^{-2k\eps} (k+1)^2 \sum_{j=0}^k N ^{2j\alpha}\ell(I)^{2\alpha} \sum_{l\ge 1} N^{(\alpha-1)l} \iint_{\R^{n+1}} |\dalpha g|^2.
\end{align*}
Choosing $\alpha<\eps$ allows us to sum and we finally obtain {a bound by $\ell(I)^{2\alpha} \|\dalpha g\|_{2}^2$. Since $\lambda^2\sim \ell(I)$, this completes the proof of \eqref{eq:Poincare} and thus the proof of Theorem~\ref{thm:bhfc}.}
\section{Reverse H\"older estimates and non-tangential maximal estimates}
\label{sec:reverse and NT}

The first part of this section contains the proof of the reverse H\"older estimate for the parabolic conormal differential of reinforced weak solutions alluded to in Theorem~\ref{thm:rh}. In the second part we shall derive some consequences of this estimate: the $\L^2$-bounds for the modified non-tangential maximal function $\NT$ and a.e.\ convergence of Whitney averages for reinforced weak solutions as stated in Theorems~\ref{thm:NTmax} and \ref{thm:NTmaxDir}.
\subsection{Proof of Theorem \ref{thm:rh} }

For convenience we will fix some notation which will be used throughout the proof but not always restated in all intermediate results.

We let $\LQI$ be a parabolic cylinder of sidelength $r>0$ defined by $\Lambda = (\lambda - r, \lambda + r)$, $Q = B(x,r)$ and $I = (t - r^2, t + r^2]$ and we assume $8r < \lambda$. Hence, the parabolic enlargement $8 \Lambda \times 8 Q \times 64 I$ is contained in $\R^{n+2}_+$. We fix a smooth cut-off $\eta: \R^{n+2}_+ \to [0,1]$ with support in $2 \Lambda \times 2 Q \times 4 I$ that is $1$ on an enlargement $\frac{3}{2} \lambda \times \frac{3}{2} Q \times \frac{9}{4} I$. For a reason which will become clear later on, we choose $\eta$ to have the product form
\begin{align*}
 \eta(\mu,y,s) = \eta_\Lambda(\mu) \eta_Q(y) \eta_I(s),
\end{align*}
where $\eta_I$ is symmetric about the midpoint of $I$. To keep control over implicit constants we may also assume that $\eta$ is obtained from the one fixed cut-off $\eta_0$ for the parabolic cylinder with sidelength $r=1$ centered at the origin $(\lambda,x,t) = (0,0,0)$ by the change of variables
\begin{align*}
 \eta(\mu,y,s) = \eta_0\bigg(\frac{\mu - \lambda}{r}, \frac{y - x}{r}, \frac{s - t}{r^2}\bigg).
\end{align*}
We denote the translates of an interval $J$ by $J_k = k \ell(J) + J$, $k \in \IZ$. For the sake of clarity we give a name to the translation sums
\begin{align*}
 \sum(v) := \sum_{k \in \IZ} \frac{1}{1+|k|^{3/2}} \bariiint_{\fLQIk{4}{k}} |v|,
\end{align*}
where $v$ is a measurable function on $\R^{n+2}_+$. For most of the proof we work with $4\Lambda \times 4Q$ on the right-hand side and assume only $4r < \lambda$. It is only in the final step where we shall enlarge to $8\Lambda \times 8Q$ as in the statement of the theorem. All bounds below will depend only on $n$ and the ellipticity constants of $A$, but for simplicity we keep on using the symbol $\lesssim$.

To begin with, let us recall the classical local estimates. Proofs can be found, for example, in Sections~3.2 and 4.1 of~\cite{AMP15}. (The proof for Caccioppoli's inequality there assumes $f=0$ but the argument is the same when $f\ne 0$).

\begin{lem}[Caccioppoli]
\label{lem:Caccioppoli}
Let $\Omega \subset \R^{n+2}_+$ be open and $c>1$. Let $\LQI \subset \Omega$ be a parabolic cylinder of sidelength $r>0$ such that $\cl{c\Lambda \times cQ \times c^2I} \subset \Omega$. Given $f \in \Lloc^2(\Omega)$, every weak solution $u$ to the inhomogeneous problem $\pd_{t} u - \div_{\lambda,x} A \gradlamx u = \div_{\lambda,x} f$ in $\Omega$ satisfies
\begin{align*}
 \bariiint_\LQI |\gradlamx u|^2 \lesssim \frac{1}{r^2} \bariiint_{c \Lambda \times c Q \times c^2 I} |u|^2 + \bariiint_{c Q \times c \Lambda \times c^2 I} |f|^2,
\end{align*}
with an implicit constant depending only on $c$, $n$ and the ellipticity constants of $A$.
\end{lem}

\begin{lem}[Reverse H\"older]
\label{lem:RHu}
Let $\Omega \subset \R^{n+2}_+$ be open and $c>1$. Suppose that $u$ is a weak solution to $\pd_{t} u - \div_{\lambda,x} A \gradlamx u = 0$ in $\Omega$. Then for every parabolic cylinder $\LQI \subset \Omega$ such that $c\Lambda \times cQ \times c^2I \subset \Omega$,
\begin{align*}
 \bigg( \bariiint_\LQI |u|^2 \bigg)^{1/2} \lesssim \bariiint_{c \Lambda \times c Q \times c^2 I} |u|,
\end{align*}
with an implicit constant depending only on $c$, $n$ and the ellipticity constants of $A$.
\end{lem}

The following variant of Lemma~\ref{lem:fractional poincare alpha} will be our instrument for localising in time.

\begin{lem}
\label{lem:fractional poincare 1/2}
Let $p,q \in [1,\infty)$ satisfy $p/2 < q \leq p$. Then for each interval $J \subset \R$ and every $h \in \Hdot^{1/2}(\R)$,
\begin{align*}
 \bigg(\barint_J \Big|h - \barint_J h \Big|^p \d s \bigg)^{1/p} \lesssim \sqrt{\ell(J)} \bigg(\sum_{k \in \IZ} \frac{1}{1+|k|^{3/2}} \barint_{J_k} |\dhalf h|^q \d s \bigg)^{1/q}.
\end{align*}
The analogous inequality with $\HT \dhalf h$ instead of $\dhalf h$ on the right-hand side also holds.
\end{lem}

\begin{proof}
The proof follows that of Lemma~\ref{lem:fractional poincare alpha} verbatim, the only difference being that in \eqref{eq1:fractional poincare alpha} we use the uniform estimate
\begin{align*}
 \bigg(\barint_J \Phi(s,\sigma)^{p/q} \d s \bigg)^{q/p} \lesssim \frac{1}{\sqrt{\ell(J)}} \frac{1}{1+|k|^{3/2}} \qquad(\sigma \in J_k)
\end{align*}
and split $\R = \bigcup_{k \in \IZ} J_k$ on the right-hand side.
\end{proof}

With these estimates at hand, we can already prove a reverse H\"older inequality of required type for the spatial gradient of reinforced weak solutions.

\begin{lem}
\label{lem:rh spatial}
Let $u$ be a  reinforced weak solution to \eqref{eq1} and let $c = \bariiint_{2 \Lambda \times 2Q \times I} u$. Then
\begin{align*}
 \bigg(\bariiint_{2\Lambda \times 2Q \times 4I} |\gradlamx u|^2 \bigg)^{1/2}
 \lesssim \frac{1}{r} \bariiint_{4\Lambda \times 4Q \times 16I} |u - c|
 \lesssim \sum(|\gradlamx u| + |\HT \dhalf u|),
\end{align*}
where $\HT \dhalf u$ could be replaced with $\dhalf u$ on the right-hand side.
\end{lem}

\begin{proof}
The first estimate is a direct consequence of Lemmas~\ref{lem:Caccioppoli} and \ref{lem:RHu} applied to the weak solution $u-c$. For the second estimate we separate averages in space and time by writing
\begin{align*}
 u - c = u - \bariint_{2\Lambda \times 2Q} u + \bariint_{2\Lambda \times 2Q} u - c.
\end{align*}
Poincar\'{e}'s inequality in the spatial variables \cite[Lem.~7.12/16]{Gilbarg-Trudinger} allows for an estimate
\begin{align*}
 \frac{1}{r} \bariiint_{4\Lambda \times 4Q \times 16I} \Big|u - \bariint_{2\Lambda \times 2Q} u \Big|
&\lesssim \bariiint_{4\Lambda \times 4Q \times 16I} |\gradlamx u| \\
&\lesssim \sum_{|k| \leq 8} \frac{1}{1+|k|^{3/2}} \bariiint_\fLQIk{4}{k} |\gradlamx u|.
\end{align*}
As for the averages in time, we abbreviate
\begin{align*}
 f(s) := \bariint_{2\Lambda \times 2Q} u(\mu,y,s) \d \mu \d y.
\end{align*}
Since the reinforced weak solution $u$ belongs to the class $\Hdot^{1/2}(\R; \Lloc^2(\reu))$, we obtain $f \in \Hdot^{1/2}(\R)$ and due to Fubini's theorem,
\begin{align}
\label{eq1:rh spatial}
 \HT \dhalf f(s) = \bariint_{2\Lambda \times 2Q} \HT \dhalf u(\mu,y,s) \d \mu \d y,
\end{align}
see also Lemma~\ref{lem:Hdot12 several variables}. By a telescoping sum
\begin{align*}
 \frac{1}{r} \bariiint_{4 \Lambda \times 4 Q \times 16 I} \Big|\bariint_{2\Lambda \times 2Q} u - c\Big|
 = \frac{1}{\sqrt{\ell(I)}} \barint_{16 I} \Big|f - \barint_I f \Big|
 \lesssim \frac{1}{\sqrt{\ell(16I)}} \barint_{16 I} \Big|f - \barint_{16 I} f \Big|,
\end{align*}
so that Lemma~\ref{lem:fractional poincare 1/2} applies to the effect that
\begin{align*}
\frac{1}{r} \bariiint_{4 \Lambda \times 4 Q \times 16 I} \Big|\bariint_{2\Lambda \times 2Q} u - c\Big|
 & \lesssim \sum_{k \in \IZ} \frac{1}{1+|k|^{3/2}} \barint_{16 I + 16 k \ell(I)} |\HT \dhalf f| \\
 & \lesssim \sum_{k \in \IZ} \frac{1}{1+|k|^{3/2}} \barint_{I_k} |\HT \dhalf f|,
\end{align*}
where in the last step we have split $16 I + 16 k \ell(I)$ into translates of $I$. We conclude by plugging in the identity \eqref{eq1:rh spatial}. Note that we could also decide to let $\dhalf f$ appear in \eqref{eq1:rh spatial} and our final estimate.
\end{proof}

Due to the non-locality of the half-order derivative in time, reverse H\"older estimates for $\HT \dhalf u$ and $\dhalf u$ are much harder to get. A part of the the following argument was inspired by~\cite{Hofmann-Lewis}.

Since $\HT \dhalf$ annihilates constants, we can write the average to be estimated as
\begin{align*}
\bariiint_\LQI |\HT \dhalf u|^2
= \bariiint_\LQI |\HT \dhalf(u-c)|^2, \qquad \text{where} \qquad c := \bariiint_{2 \Lambda \times 2Q \times I} u
\end{align*}
and the same can be done with $\dhalf$ in place of $\HT \dhalf$. We treat a local version first.

\begin{lem}
\label{lem:rh time local}
Let $u$ be a  reinforced weak solution to \eqref{eq1} and let $c = \bariiint_{2 \Lambda \times 2 Q \times I} u$. Then
\begin{align*}
 \bigg(\bariiint_\LQI |\HT \dhalf(\eta(u-c))|^2 \bigg)^{1/2}
 \lesssim \sum(|\gradlamx u| + |\HT \dhalf u|),
\end{align*}
where $\HT \dhalf$ can be replaced with $\dhalf$ on either side of the inequality.
\end{lem}

\begin{proof}
We begin with a parabolic rescaling of $\reu$, setting $\wt{u}(\mu,y,s) = u(r \mu, r y, r^2 s)$ and similarly $\wt{\eta}$, $\wt{A}$, as well as $\wt{\Lambda} = (r^{-1} \lambda - 1, r^{-1} \lambda +1)$ and similarly $\wt{Q}$, $\wt{I}$. Then,
\begin{align*}
 \frac{1}{r^{n+1}} \iiint_\LQI |\HT \dhalf(\eta(u-c))|^2 \d \mu \d y \d s
 &= \iiint_{\wt{\Lambda} \times \wt{Q} \times \wt{I}} |\HT \dhalf(\wt{\eta}(\wt{u}-c))|^2 \d \mu \d y \d s \\
 &\leq \int_\R \|\HT \dhalf (\wt{\eta}(\wt{u}-c)) \big \|_{\L^2(\ree)}^2 \d s.
\end{align*}
Let $\cF_t$ denote the Fourier transform in $t$-direction only. For the moment take for granted that all of the subsequent integrals are finite. We shall check this in the further course of the proof anyway. The Hilbert space valued version of Plancherel's theorem~\cite[Sec.~II.5]{Stein} yields
\begin{align*}
 \int_\R \big \|\HT \dhalf (\wt{\eta}(\wt{u}-c)) \big \|_{\L^2(\ree)}^2 \d s
&= 2 \pi \int_\R \big \|\i \sgn(\sigma) |\sigma|^{1/2} \cF_t(\wt{\eta}(\wt{u}-c)) \big\|_{\L^2(\ree)}^2 \d \sigma \\
&= 2 \pi \int_\R \big| \big( \cF_t(\wt{\eta}(\wt{u}-c)), |\sigma| \cF_t(\wt{\eta}(\wt{u}-c)) \big) \big|^2 \d \sigma
\end{align*}
with the duality brackets denoting the $\W^{1,2}(\ree)$-$\W^{-1,2}(\ree)$ duality. Since these are (isomorphic to) Hilbert spaces as well, we can apply Cauchy-Schwarz' inequality and Plancherel's theorem backwards in order to obtain altogether
\begin{align}
\label{eq1:rh time local}
\begin{split}
& \frac{1}{r^{n+1}} \iiint_\LQI |\HT \dhalf(\eta(u-c))|^2 \d \mu \d y \d s \\
&\qquad \leq \bigg(\int_\R \big\|\wt{\eta}(\wt{u}-c) \big\|_{\W^{1,2}(\ree)}^2 \d s \bigg)^{1/2} \bigg(\int_\R \big \|\partial_t(\wt{\eta}(\wt{u}-c)) \big\|_{\W^{-1,2}(\ree)}^2 \d s \bigg)^{1/2}.
\end{split}
\end{align}
Above we instantly used $|\i \sgn(\sigma)| \leq 1$ and so the same analysis would apply if we started out with $\dhalf(\eta(u-c))$. As $\eta$ is supported in $2 \Lambda \times 2 Q \times 4 I$, the first integral on the right is bounded by
\begin{align*}
 \int_\R \big\|\wt{\eta}(\wt{u}-c) \big\|_{\W^{1,2}(\ree)}^2 \d s
 &\lesssim \iiint_{\wt{2 \Lambda} \times \wt{2Q} \times \! \wt{\, 4 I \,}} |\wt{u}-c|^2 + |\gradlamx \wt{u}|^2 \d \mu \d y \d s \\
 &= \frac{1}{r^{n+3}} \iiint_{2 \Lambda \times 2 Q \times 4 I} |u-c|^2 + r^2 |\gradlamx u|^2 \d \mu \d y \d s.
\end{align*}
Next, let $\phi \in \C_0^\infty(\ree)$ and $\chi \in \C_0^\infty(\R)$, where $\phi$ acts in the $\lambda$- and $x$ variables and $\chi$ acts in $t$-direction. Using that $\wt{u}-c$ is a weak solution to a parabolic problem \eqref{eq1} with coefficients $\wt{A}$,
\begin{align*}
 \int_\R \bigg(\iint_{\ree} &\wt{\eta}(\wt{u}-c) \cdot \clos{\phi} \d \mu \d y \bigg) \clos{\partial_t \chi} \d s \\
&= \iiint_{\wt{2 \Lambda} \times \wt{2Q} \times \! \wt{\, 4 I \,}} (\wt{u}-c) \cdot \clos{\partial_t(\wt{\eta} \phi \chi)} - (\wt{u}-c) \cdot \clos{\phi \chi \partial_t \wt{\eta}} \d \mu \d y \d s \\
&= \iiint_{\wt{2 \Lambda} \times \wt{2Q} \times \! \wt{\, 4 I \,}} \wt{A} \gradlamx(\wt{u}-c) \cdot \clos{\gradlamx(\wt{\eta} \phi \chi)} - (\wt{u}-c) \cdot \clos{\phi \chi \partial_t \wt{\eta}} \d \mu \d y \d s \\
&=\int_{\wt{\, 4 I \,}} \iint_{\wt{2 \Lambda} \times \wt{2Q}\,} \bigg(\wt{A} \gradlamx \wt{u} \cdot \clos{\gradlamx(\wt{\eta} \phi)} - (\wt{u}-c) \cdot \clos{\phi \partial_t \wt{\eta}} \d \mu \d y \bigg) \chi \d s .
\end{align*}
This implies the bound for the second integral
\begin{align*}
\int_\R \big \|\partial_t(\wt{\eta}(\wt{u}-c)) \big\|_{\W^{-1,2}(\ree)}^2 \d s
&\lesssim \iiint_{\wt{2 \Lambda} \times \wt{2Q} \times \! \wt{\, 4 I \,}} |\wt{u}-c|^2 + |\gradlamx \wt{u}|^2 \d \mu \d y \d s \\
&= \frac{1}{r^{n+3}} \iiint_{2\Lambda \times 2 Q \times 4I} |u-c|^2 + r^2 |\gradlamx u|^2 \d \mu \d y \d s.
\end{align*}
Returning to \eqref{eq1:rh time local} and clearing the powers of $r$, we find
\begin{align*}
 \bariiint_\LQI |\HT \dhalf(\eta(u-c))|^2 \d \mu \d y \d s
\lesssim \bariiint_{2\Lambda \times 2 Q \times 4I} \frac{1}{r^2} |u-c|^2 + |\gradlamx u|^2 \d \mu \d y \d s.
\end{align*}
Lemma~\ref{lem:RHu} applied to the weak solution $u-c$ provides control by
\begin{align*}
\frac{1}{r^2} \bigg(\bariiint_{4\Lambda \times 4 Q \times 16I}  |u-c|\d \mu \d y \d s \bigg)^2 +  \bariiint_{2\Lambda \times 2 Q \times 4I} |\gradlamx u|^2 \d \mu \d y \d s
\end{align*}
and we conclude by Lemma~\ref{lem:rh spatial}.
\end{proof}

We turn to the corresponding non-local terms, which can be written as
\begin{align*}
 \bariiint_\LQI |\HT \dhalf (1-\eta)(u-c)|^2 = \bariiint_\LQI |\HT \dhalf (1-\eta_I)(u-c)|^2
\end{align*}
and similarly with $\dhalf$ in place of $\HT \dhalf$, since $\eta_\Lambda$ and $\eta_Q$ are independent of $t$ and satisfy $\eta_\Lambda \eta_Q = 1$ on $\Lambda \times Q$. Similar to the proof of Lemma~\ref{lem:rh spatial} we separate
\begin{align}
\label{eq:final split rh}
 (1-\eta_I)(u - c) = (1-\eta_I)\Big(u - \bariint_{2 \Lambda \times 2Q} u\Big) + (1-\eta_I)\Big(\bariint_{2 \Lambda \times 2Q} u- \bariiint_{2 \Lambda \times 2Q \times I} u\Big).
\end{align}
Both addends on the right-hand side are in $\Hdot^{1/2}(\R; \Lloc^2(\reu))$. To see this, we first note that $u$ is contained therein as a reinforced weak solution. As in the proof of Lemma~\ref{lem:rh time local} we also have
\begin{align*}
 \phi u \in \L^2(\R; \W^{1,2}(\ree)) \cap \H^1(\R; \W^{-1,2}(\ree)) \qquad (\phi \in \C_0^\infty(\R^{n+2}_+))
\end{align*}
and thus $\phi u \in \Hdot^{1/2}(\R; \Lloc^2(\reu))$. In particular, we obtain $\eta_I u \in \Hdot^{1/2}(\R; \Lloc^2(\reu))$ and for the other terms we note that $\Hdot^{1/2}$-regularity in time inherits to averages in space in virtue of Fubini's theorem, see also Lemma~\ref{lem:Hdot12 several variables}.

Next, we shall see that the estimate of the first term on the right-hand side of \eqref{eq:final split rh} can once more be reduced to Lemma~\ref{lem:rh spatial}.

\begin{lem}
\label{lem:rh time nonlocal 1}
Let $u$ be a  reinforced weak solution to \eqref{eq1}. Then
\begin{align*}
 \bigg(\bariiint_\LQI \bigg|\HT \dhalf \bigg((1-\eta_I)\Big(u - \bariint_{2\Lambda \times 2Q} u \Big)\bigg)\bigg|^2 \bigg)^{1/2}
 \lesssim \sum(|\gradlamx u|+|\HT \dhalf u|),
\end{align*}
where $\HT \dhalf$ can be replaced with $\dhalf$ on either side of the inequality.
\end{lem}

\begin{proof}
We introduce $v = (1-\eta_I)(u-\bariint_{2\Lambda \times 2Q} u)$. By construction $v(\mu,y,\cdot)$ is supported in $^c ( \frac{9}{4} I)$ and $v(\mu,y,\cdot) \in \Hdot^{1/2}(\R)$ for almost every $\mu$, $y$. This justifies to use the representation
\begin{align*}
 \HT \dhalf v(\mu,y,s) = \frac{1}{2 \sqrt{2 \pi}} \int_\R \frac{\sgn(s-\sigma)}{|s-\sigma|^{3/2}} v(\mu,y,\sigma) \d \sigma
\end{align*}
from Corollary~\ref{cor:formula HTdhalf on Hdot1/2} almost everywhere on $\LQI$ and we can decompose $\R = \bigcup_{j \in \IZ} I_j$ on the right-hand side to obtain the pointwise estimate
\begin{align*}
 |\HT \dhalf v(\mu,y,s)| \lesssim \frac{1}{r} \sum_{j \in \IZ \setminus\{0\}} \frac{1}{1+|j|^{3/2}} \barint_{I_j} |v(\mu,y,\sigma)| \d \sigma.
\end{align*}
Due to Minkowski's inequality we can bound the average over $\LQI$ by
\begin{align*}
 \bigg(\bariiint_\LQI |\HT \dhalf v|^2 \bigg)^{1/2}
 &\lesssim \frac{1}{r}\bigg(\bariint_{\Lambda \times Q} \bigg(\sum_{j \in \IZ} \frac{1}{1+|j|^{3/2}} \barint_{I_j} |v| \bigg)^2\ \bigg)^{1/2} \\
 &\leq \frac{1}{r} \sum_{j \in \IZ} \frac{1}{1+|j|^{3/2}} \bigg(\bariiint_{\LQI_j} |v|^2 \bigg)^{1/2}.
\end{align*}
The very same argument applies to $\dhalf v$ and yields the same bound. Now, by definition of $v$ and Poincar\'{e}'s inequality in the spatial variables
\begin{align*}
 \bariiint_{\LQI_j} |v|^2
 \lesssim \bariiint_{2 \Lambda \times 2Q \times 4I_j} \Big|u - \bariint_{2\Lambda \times 2Q} u \Big|^2
 \lesssim r^2 \bariiint_{2 \Lambda \times 2Q \times 4I_j} |\gradlamx u|^2,
\end{align*}
so that altogether,
\begin{align*}
 \bigg(\bariiint_\LQI |\HT \dhalf v|^2 \bigg)^{1/2}
\lesssim \sum_{j \in \IZ} \frac{1}{1+|j|^{3/2}} \bigg(\bariiint_{2 \Lambda \times 2Q \times 4I_j} |\gradlamx u|^2\bigg)^{1/2}.
\end{align*}
The upshot of all this is that each integral above falls under the scope of Lemma~\ref{lem:rh spatial}, giving after an index-shift the bound
\begin{align*}
\sum_{j \in \IZ} \frac{1}{1+|j|^{3/2}} \bigg(\sum_{k \in \IZ} \frac{1}{1+|k-j|^{3/2}} \bariiint_{\fLQIk{4}{k}} |\gradlamx u| + |\HT \dhalf u| \bigg).
\end{align*}
It remains to remark that the occurring discrete convolution satisfies
\begin{align}
\label{eq:discrete convolution}
\sum_{j \in \IZ} \bigg(\frac{1}{1+|j|^{3/2}} \cdot \frac{1}{1+|k-j|^{3/2}} \bigg) \lesssim \frac{1}{1+|k|^{3/2}} \qquad (k \in \IZ)
\end{align}
as can be checked by distinguishing the cases $j \leq |k|/2$ and $j>|k|/2$.
\end{proof}

At this stage of the proof we have reduced the matter to two final reverse H\"older estimates, which are the ones for
\begin{align*}
 \HT \dhalf \bigg((1-\eta_I)\Big(\bariint_{2 \Lambda \times 2 Q} u - \bariiint_{2 \Lambda \times 2 Q \times I} u \Big) \bigg)
\end{align*}
and its counterpart with $\dhalf$ in place of $\HT \dhalf$. Note that these functions only depend on the $t$ variable. As for the \emph{odd} half-order derivative $\HT \dhalf$, the estimate will be a consequence of the following lemma. It is at this point where the symmetry assumption on the cut-off $\eta_I$ comes into play and in fact the argument would not go through for the \emph{even} half-order derivative $\dhalf$.

\begin{lem}
\label{lem:HTdhalf outside bound}
Let $J$ be a bounded interval and $\eta: \R \to [0,1]$ be a smooth cut-off function with support in $4J$ that is identically $1$ on $\frac{9}{4} J$. Suppose furthermore that $\eta$ is symmetric about the midpoint of $J$. If $f \in \Hdot^{1/2}(\R)$ is such that $(1-\eta)f \in \Hdot^{1/2}(\R)$, then almost everywhere on $J$,
\begin{align*}
 \bigg|\HT \dhalf \bigg((1-\eta)\Big(f - \barint_J f \Big)\bigg)\bigg|  \lesssim \sum_{k \in \IZ} \frac{1}{1+|k|^{3/2}} \barint_{J_k} |\dhalf f| + |\HT \dhalf f| \d s.
\end{align*}
\end{lem}

\begin{proof}
After a shift we may assume for simplicity that $J$ is centered at the origin. Let $w := \HT \dhalf ((1-\eta)(f-\barint_J f))$. Thanks to the support properties of $\eta$, Corollary~\ref{cor:formula HTdhalf on Hdot1/2} yields
\begin{align}
\label{eq1:HTdhalf outside bound}
 w(\tau) = \frac{-1}{2 \sqrt{2 \pi}} \int_\R \frac{\sgn(\tau-s)}{|\tau-s|^{3/2}} (1-\eta(s))\bigg(f(s) - \barint_J f \bigg) \d s \qquad (\text{a.e.\ $\tau \in J$})
\end{align}
and this defines $w$ as a continuous function on $J$. For the rest of the proof we will consider $w$ only on $J$ and stick with this particular representative.

\subsubsection*{Step 1: The bound for \texorpdfstring{$w(\tau) - w(0)$}{w(t)-w(0)}}

Since $s \notin \frac{9}{4} J$ whenever the integrand in \eqref{eq1:HTdhalf outside bound} is non-zero, we obtain from the mean-value theorem
\begin{align*}
 |w(\tau) - w(0)| \lesssim \int_\R \frac{|\tau|}{|s|^{5/2}} |1 - \eta(s)| \Big|f(s) - \barint_J f \Big| \d s.
\end{align*}
Decomposing $\R = \bigcup_{k \geq 0} (J_k \cup J_{-k})$, we find
\begin{align}
\label{eq2:HTdhalf outside bound}
|w(\tau) - w(0)| \lesssim \frac{1}{\sqrt{\ell(J)}} \sum_{k = 1}^\infty \frac{1}{1+|k|^{5/2}} \bigg(\barint_{J_{k}} \Big|f - \barint_J f \Big| + \barint_{J_{-k}} \Big|f - \barint_J f \Big| \bigg).
\end{align}
Now, for every $k \geq 0$ the telescoping argument
\begin{align*}
 \barint_{J_k} \Big|f - \barint_J f \Big| \leq \barint_{J_k} \Big|f - \barint_{J_k} f \Big| + 6 \sum_{j=1}^k \barint_{{J_j} \cup J_{j-1}} \Big|f - \barint_{{J_j} \cup J_{j-1}} f \Big|
\end{align*}
brings us in a position where Lemma~\ref{lem:fractional poincare 1/2} applies in the form
\begin{align*}
 \frac{1}{\sqrt{\ell(J)}} \barint_{J_k} \Big|f - \barint_J f\Big|
 &\lesssim \sum_{m \in \IZ} \frac{1}{1+|m|^{3/2}} \bigg(\barint_{J_{k+m}} |\HT \dhalf f| + \sum_{j=1}^k \barint_{J_{j+m} \cup J_{j-1+m}} |\HT \dhalf f| \bigg) \\
 &\leq 2\sum_{m \in \IZ} \bigg(\sum_{j=0}^k \frac{1}{1+|m-j|^{3/2}} \bigg) \barint_{J_m} |\HT \dhalf f|,
\end{align*}
where the second step follows from a simple regrouping of terms. For the corresponding average over $J_{-k}$ on the left-hand side we obtain an analogous bound with an inner sum ranging over $-k \leq j \leq 0$, so that altogether we sum over $|j| \leq k$. These estimates inserted on the right-hand side of \eqref{eq2:HTdhalf outside bound} lead us to
\begin{align*}
 |w(\tau) - w(0)|
 & \lesssim \sum_{m \in \IZ} \sum_{j \in \IZ} \bigg(\sum_{k \geq |j|} \frac{1}{1+|k|^{5/2}} \bigg) \frac{1}{1+|m-j|^{3/2}} \barint_{J_m} |\HT \dhalf f|.
\end{align*}
By a Riemann sum argument on the sum over $k$ and \eqref{eq:discrete convolution},
\begin{align*}
 \sum_{j \in \IZ} \bigg(\sum_{k \geq |j|} \frac{1}{1+|k|^{5/2}} \bigg) \frac{1}{1+|m-j|^{3/2}}
\lesssim \sum_{j \in \IZ} \frac{1}{1+|j|^{3/2}} \frac{1}{1+|m-j|^{3/2}}
\lesssim \frac{1}{1+|m|^{3/2}},
\end{align*}
giving the required bound
\begin{align*}
|w(\tau) - w(0)| \lesssim \sum_{m \in \IZ} \frac{1}{1+|m|^{3/2}} \barint_{J_m} |\HT \dhalf f|.
\end{align*}

\subsubsection*{Step 2: The bound for \texorpdfstring{$w(0)$}{w(0)}}

The `kernel' in the formula \eqref{eq1:HTdhalf outside bound} for $w(0)$ is odd since $\eta$ is an even function. Hence, we can omit the average $\barint_J f$ and for a later purpose we write
\begin{align*}
 w(0) = \frac{1}{2 \sqrt{2 \pi}} \lim_{S \to \infty} \int_{(-S,S)} \frac{\sgn(s)}{|s|^{3/2}} (1-\eta(s))f(s) \; \d s.
\end{align*}
Due to Corollary~\ref{cor:potential representation of Hdot1/2} we can represent $f$ as a potential
\begin{align*}
 f(s) = c + \frac{1}{\sqrt{2 \pi}} \bigg(\int_{(-1,1)} \frac{g(\rho)}{|s-\rho|^{1/2}} \d \rho
 + \int_{^c (-1,1)} \frac{g(\rho)}{|s-\rho|^{1/2}} - \frac{g(\rho)}{|\rho|^{1/2}} \d \rho \bigg)
\end{align*}
with $g = \dhalf f$ and $c$ a constant. Since these integrals converge absolutely for almost every $s \in \R$, see Section~\ref{sec:Homogeneous Sobolev spaces}, we can plug in this representation into the one for $w(0)$ and apply Fubini's theorem. During this procedure the terms $c$ and $g(\rho) |\rho|^{-1/2}$ depending not on $s$ vanish again by oddness. So, after a substitution $s = \rho \sigma$,
\begin{align}
\label{eq3:HTdhalf outside bound}
 w(0)
 &= \lim_{S \to \infty} \frac{1}{4 \pi} \int_\R \bigg(\int_{|\sigma| \leq \frac{S}{|\rho|}} \frac{\sgn(\sigma)}{|\sigma|^{3/2}} \cdot \frac{1-\eta_J(\sigma \rho)}{|1-\sigma|^{1/2}} \d \sigma \bigg) \frac{g(\rho)}{\rho} \d \rho.
\end{align}
Let us inspect the inner integral for fixed $S \geq \ell(J)$: In the case $|\rho| < \ell(J)$ we can simply use the support properties of $\eta$ to bound this in absolute value by
\begin{align*}
 \int_{\frac{9 \ell(J)}{8|\rho|} \leq |\sigma|} \frac{1}{|\sigma|^{3/2}} \cdot \frac{1} {|1-\sigma|^{1/2}} \d \sigma \lesssim \frac{|\rho|}{\ell(J)},
\end{align*}
where we have used the estimate $|\sigma| \leq 9|1-\sigma|$ on the domain of integration in the second step. If $\ell(J) < |\rho| < 8\ell(J)$, then the estimate above remains true for the simple reason that both sides are comparable to absolute constants. If finally $|\rho| \geq 8 \ell(J)$, then we additionally make use of the symmetry of $\eta$ to rewrite the integral under consideration as
\begin{align*}
 &\int_{|\sigma| \leq \frac{S}{|\rho|}} \frac{\sgn(\sigma)}{|\sigma|^{3/2}} \bigg(\frac{1}{|1-\sigma|^{1/2}} - 1 \bigg) \d \sigma \\
 - &\int_{|\sigma| \leq \frac{4 \ell(J)}{|\rho|}} \frac{\sgn(\sigma)}{|\sigma|^{3/2}} \bigg(\frac{1}{|1-\sigma|^{1/2}} - 1 \bigg) \eta(\sigma \rho) \d \sigma.
\end{align*}
The first term is bounded uniformly in $S$ and $\rho$ since $||1-\sigma|^{-1/2} - 1| \lesssim |\sigma|$ for $|\sigma|$ small, and in the limit for $S\to \infty$ it converges to the value $C \in \R$ of the corresponding integral over the real line. Similarly, the second term can be controlled by $\ell(J)^{1/2}/|\rho|^{1/2}$ since we have $|\sigma| \leq 1/2$ in the domain of integration. This justifies to use Lebesgue's theorem when passing to the limit in $S$ in \eqref{eq3:HTdhalf outside bound}. As a result,
\begin{align}
\label{eq4:HTdhalf outside bound}
 |w(0)|
 \lesssim \barint_{8J} |g(\rho)| \d \rho +  \bigg|C \int_{^c(8J)} \frac{g(\rho)}{\rho} \d \rho \bigg| + \sqrt{\ell(J)} \int_{^c(8J)} \frac{|g(\rho)|}{|\rho|^{3/2}} \d \rho.
\end{align}
The truncated Hilbert transform in the middle term can be handled similar to the classical proof of Cotlar's lemma but for the reader's convenience we repeat the short argument. So, we let $g_1$ be equal to $g$ on $8J$ and zero otherwise. Then for any $0<\eps<2 \ell(J)$ and $s \in J$,
\begin{align*}
 \int_{^c(8J)} \frac{g(\rho)}{\rho} \d \rho = \int_{^c(8J)} \bigg(\frac{1}{\rho} - \frac{1}{s-\rho} \bigg) g(\rho) \d \rho + \int_{|s-\rho| > \eps} \frac{g(\rho)}{s-\rho} \d \rho - \int_{|s-\rho| > \eps} \frac{g_1(\rho)}{s-\rho} \d \rho.
\end{align*}
The mean-value theorem applies to the first term so that in the limit $\eps \to 0$,
\begin{align*}
 \bigg| \int_{^c(8J)} \frac{g(\rho)}{\rho} \d \rho \bigg| \lesssim \int_{^c(8J)} \frac{\ell(J) |g(\rho)|}{|\rho|^2} \d \rho + |\HT g(s)| + |\HT g_1(s)| \qquad (\text{a.e.\ $s \in J$}).
\end{align*}
In view of Tchebychev's inequality and the weak-type bound for the Hilbert transform of $g_1 \in \L^1(\R)$ we can find some $s \in J$ such that simultaneously
\begin{align*}
 |\HT g(s)| \lesssim \barint_{8J} |\HT g| \qquad \text{and} \qquad |\HT g_1(s)| \lesssim \barint_{8J} |g|.
\end{align*}
Returning to \eqref{eq4:HTdhalf outside bound} and plugging in $g = \dhalf f$,
\begin{align*}
 |w(0)| \lesssim \sqrt{\ell(J)} \int_{^c(8J)} |\dhalf f| \, \frac{\mathrm{d} \rho}{|\rho|^{3/2}} + \barint_{8J} |\HT \dhalf f| + |\dhalf f| \d \rho.
\end{align*}
Now, the usual split up into translates of $J$ gives
\begin{align*}
 |w(0)| \lesssim \sum_{m \in \IZ} \frac{1}{1+|m|^{3/2}} \barint_{J_m} |\dhalf f| + |\HT \dhalf f|,
\end{align*}
which, together with the outcome of Step~1, establishes the required bound even uniformly on $J$. The claim follows by averaging over $\tau \in J$.
\end{proof}

Applying this lemma to $f = \bariint_{2 \Lambda \times 2Q} u$ and $\eta = \eta_I$ on $J = I$ completes the reverse H\"older estimate for $\HT \dhalf u$ with an enlargement to $4 \Lambda \times 4Q$ in space on the right-hand side:

\begin{cor}
\label{cor:HTdhalf outside bound weak solution}
Let $u$ be a  reinforced weak solution to \eqref{eq1}. Then
\begin{align*}
 \left(\bariiint_\LQI \bigg|\HT \dhalf \bigg((1-\eta_I)\Big(\bariint_{2 \Lambda \times 2 Q} u - \bariiint_{2 \Lambda \times 2 Q \times I} u\Big)\bigg)\bigg|^2 \right)^{1/2}
 & \!\lesssim \sum(|\dhalf u| + |\HT \dhalf u|).
\end{align*}
\end{cor}

The missing piece in the proof of Theorem~\ref{thm:rh} is the corresponding estimate with $\dhalf$ on the left-hand side. We call the reader's attention to the fact that this will note be achieved by an argument for functions on $\R$ but by resorting to the \emph{full} reverse H\"older estimate for $\bariiint_{2 \Lambda \times 2Q \times 4I} |\HT \dhalf u|^2$ that we have just completed. Since we are using this estimate already for the enlargement $2 \Lambda \times 2Q$ in space, we finally obtain an enlargement $8 \Lambda \times 8Q$ on the right-hand side as stated in Theorem~\ref{thm:rh}.

\begin{lem}
\label{lem:dhalf outside bound weak solution}
Let $u$ be a  reinforced weak solution to \eqref{eq1}. Then
\begin{align*}
 &\left(\bariiint_\LQI \bigg|\dhalf \bigg((1-\eta_I)\Big(\bariint_{2 \Lambda \times 2 Q} u - \bariiint_{2 \Lambda \times 2 Q \times I} u\Big)\bigg)\bigg|^2\right)^{1/2}\\
 &\qquad \lesssim \sum_{k \in \IZ} \frac{1}{1+|k|^{3/2}} \bariiint_{\fLQIk{8}{k}} |\gradlamx u| + |\dhalf u| + |\HT \dhalf u| \d \lambda \d x \d t.
\end{align*}
\end{lem}

\begin{proof}
Let $f:=\bariint_{2\Lambda \times 2Q} u$ and $h := f - \barint_I f$. Rewriting the left-hand side of the estimate in question by means these two functions reveals that it suffices to prove
\begin{align}
\label{eq1:dhalf outside bound weak solution}
\begin{split}
 \bigg(\barint_I &|\dhalf(\eta_I h)|^2 \bigg)^{1/2} + \bigg(\barint_I |\dhalf f|^2\bigg)^{1/2} \\
 & \lesssim \sum_{k \in \IZ} \frac{1}{1+|k|^{3/2}} \bariiint_{\fLQIk{8}{k}} |\gradlamx u| + |\dhalf u| + |\HT \dhalf u| \d \lambda \d x \d t,
\end{split}
\end{align}
where we also used $\dhalf h = \dhalf f$ as the half-order derivative annihilates constants. We treat both terms separately.

\subsubsection*{Step 1: The first term in \eqref{eq1:dhalf outside bound weak solution}}

Since $\HT^2 = -1$ as bounded operators on $\L^2(\R)$,
\begin{align*}
 \bigg(\barint_I |\dhalf(\eta_I h)|^2\bigg)^{1/2} \lesssim \bigg(\frac{1}{\ell(I)} \int_{\R} |\HT \dhalf (\eta_I h)|^2\bigg)^{1/2}.
\end{align*}
Let $\eta_{4I}$ have the same properties as $\eta_I$ but for the interval $4I$ instead of $I$. In particular, $\eta_{4I}$ is supported in $16I$ and identically $1$ on $9I$, which contains the support of $\eta_I$. Using the relation $1-\eta_{4I} = (1-\eta_I)(1-\eta_{4I})$ in the second line of the following computation, we can use commutators to expand
\begin{align*}
 \HT \dhalf (\eta_I h)
&= \HT \dhalf h - \HT \dhalf ((1-\eta_I)(1-\eta_{4I} + \eta_{4I}) h) \\
&=\HT \dhalf h - \HT \dhalf((1-\eta_{4I})h) - [\HT \dhalf, 1-\eta_I](\eta_{4I} h) - (1-\eta_I)\HT \dhalf (\eta_{4I} h)\\
&= \eta_I \HT \dhalf f - \eta_I \HT \dhalf((1-\eta_{4I})h) + [\HT \dhalf, \eta_I](\eta_{4I} h).
\end{align*}
Pick some $p \in (1,2)$. From the kernel bounds \eqref{eq:dhalf commutator kernel} and Young's inequality we obtain that the commutator $[\HT \dhalf, \eta_I]$ is bounded $\L^p(\R) \to \L^2(\R)$ with norm controlled by $\ell(I)^{-1/p}$. Thus, the various support assumptions and the definition of $h$ yield
\begin{align*}
 \bigg(\barint_{I} |\dhalf(\eta_I h)|^2\bigg)^{1/2}
&\lesssim \bigg(\barint_{4I} |\HT \dhalf f|^2 \bigg)^{1/2} + \bigg(\barint_{4I} \bigg|\HT \dhalf\bigg((1-\eta_{4I})\Big(f - \barint_I f\Big) \bigg)\bigg|^2 \bigg)^{1/2} \\
&\quad + \frac{1}{\sqrt{\ell(I)}} \bigg(\barint_{16I}\Big|f-\barint_I f \Big|^p \bigg)^{1/p}.
\end{align*}
Plugging in the definition of $f$, the first term is bounded by the $\L^2$ average of $\HT \dhalf u$ on $2\Lambda \times 2Q \times 4 I$, which can be controlled as desired since we have already completed the reverse H\"older estimate stated in Theorem~\ref{thm:rh} for $\HT \dhalf u$ on the left-hand side with an enlargement factor of $4$ in the spatial directions. The second term can be controlled by Lemma~\ref{lem:HTdhalf outside bound} applied with $\eta = \eta_{4I}$ on $J=4I$: Indeed, we take the average of $f$ over $\frac{1}{4} J$ instead of $J$ here but the only change in the argument would be to use translates of $\frac{1}{4}J$ in \eqref{eq2:HTdhalf outside bound}, too. Finally, for the third term we use a telescopic sum and then Lemma~\ref{lem:fractional poincare 1/2} with $q=1$ to obtain
\begin{align*}
 \frac{1}{\sqrt{\ell(I)}} \bigg(\barint_{16I}\Big|f - \barint_I f \Big|^p \bigg)^{1/p}
 \lesssim \frac{1}{\sqrt{\ell(I)}} \bigg(\barint_{16I} \Big|f - \barint_{16I} f \Big|^p \bigg)^{1/p}
 \lesssim \sum_{k \in \IZ} \frac{1}{1+|k|^{3/2}} \barint_{(16I)_k} |\dhalf f|,
\end{align*}
and so all we have to do is to decompose the translates of $16I$ into those of $I$.  We note that this estimate required $p<2$. This completes the treatment of the first term in \eqref{eq1:dhalf outside bound weak solution}.

\subsubsection*{Step 2: The second term in \eqref{eq1:dhalf outside bound weak solution}}

Using again $\HT^2 = -1$ and the $\L^2$ boundedness of the Hilbert transform, we find
\begin{align*}
 \bigg(\barint_I |\dhalf f|^2\bigg)^{1/2} \lesssim \bigg(\barint_{4I} |\HT \dhalf f|^2\bigg)^{1/2} + \bigg(\barint_I |\HT (1_{{}^c (4I)} \HT \dhalf f)|^2\bigg)^{1/2}.
\end{align*}
From Step~1 we know how to handle the first integral on the right. As for the second one, we put $g:= \HT \dhalf f$ and obtain from the singular integral representation of the Hilbert transform
\begin{align*}
 \HT (1_{{}^c (4I)} \HT \dhalf f)(t) = \int_{{}^c (4I)} \frac{g(\rho)}{\rho-t} \d \rho \qquad (\text{a.e.\ $t \in I$}).
\end{align*}
Now, we can follow the Cotlar-type argument succeeding \eqref{eq4:HTdhalf outside bound} to deduce the uniform bound
\begin{align*}
 |\HT (1_{{}^c (4I)} \HT \dhalf f)(t)| \lesssim \sum_{m \in \IZ} \frac{1}{1+|m|^{3/2}} \barint_{I_m} |\dhalf f| + |\HT \dhalf f| \qquad (\text{a.e.\ $t \in I$}).
\end{align*}
We conclude by averaging over $t \in I$ and plugging in the definition of $f$.
\end{proof}

\begin{rem}
As the reader may have noticed, we could have also written Lemma~\ref{lem:dhalf outside bound weak solution} as a statement on real functions, more in the spirit of Lemma~\ref{lem:HTdhalf outside bound}. In fact, the proof above has revealed the following estimate, using the notation and assumptions of Lemma~\ref{lem:HTdhalf outside bound}:
\begin{align*}
 \left(\barint_J \bigg|\dhalf \bigg((1-\eta)\Big(f - \barint_J f \Big)\bigg)\bigg|^2 \right)^{1/2}  \lesssim \bigg(\barint_{4J} |\HT \dhalf f|^2 \bigg)^{1/2} + \sum_{k \in \IZ} \frac{1}{1+|k|^{3/2}} \barint_{J_k} (|\dhalf f| + |\HT \dhalf f|).
\end{align*}
\end{rem}
\subsection{Proof of Theorem \ref{thm:NTmax}}

Again $\LQI$ denotes a parabolic cylinder in $\R^{n+2}_+$ but to provide for the non-tangential character of the estimates in question we shall now assume its sidelength to be comparable to its distance to the boundary, say $\Lambda = (7 \lambda, 9 \lambda)$, $Q = B(x, \lambda)$, and $I = (t - \lambda^2, t + \lambda^2)$, for convenience.

Recall that for any measurable function $F$ on the upper half-space $\R^{n+2}_+$ the non-tangential maximal function is given by
\begin{align*}
 \NT F(x,t) = \sup_{\lambda > 0} \bigg(\bariiint_{\LQI} |F(\mu,y,s)|^2 \d \mu \d y \d s \bigg)^{1/2},
\end{align*}
where the supremum is taken over the family of parabolic cubes described above.

It is well-known that the non-tangential maximal function can be controlled by the corresponding square function. For convenience and later reference we include the short argument proving

\begin{lem}
\label{lem:NT and QE}
For any measurable function $F$ on the upper half-space $\R^{n+2}_+$ the estimate
\begin{align*}
 \sup_{\lambda > 0} \barint_{\frac{7}{8} \lambda}^{\frac{9}{8} \lambda} \|F_{\mu}\|_{2}^2 \d \mu
\lesssim \|\NT F\|_{2}^2
\lesssim \int_0^\infty \|F_{\mu}\|_{2}^2 \; \frac{\mathrm{d} \mu}{\mu}
\end{align*}
holds true where $F_{\mu}=F(\mu, \cdot)$. If the rightmost integral is finite, then for a.e.\ $(x,t) \in \ree$,
\begin{align*}
 \lim_{\lambda \to 0} \bariiint_{\LQI} |F|^2 \d \mu \d y \d s = 0.
\end{align*}
\end{lem}

\begin{proof}
Spelling out the norm $\|\NT F\|_{2}^2$ gives up to a multiplicative constant depending only on $n$,
\begin{align*}
\iint_\ree \sup_{\lambda >0} \bigg(\iiint_{\R^{n+2}_+} \frac{1_{(8\mu/9, 8\mu/7)}(\lambda) 1_{B(y, \lambda)}(x) 1_{(s-\lambda^2, s+\lambda^2)}(t)}{\lambda^{n+3}} |F(\mu,y,s)|^2 \d \mu \d y \d s \bigg) \d x \d t,
\end{align*}
where $1_E$ denotes the characteristic function of a set $E$. Hence, the lower estimate required in the lemma simply follows on pulling the supremum over $\lambda > 0$ outside the double integral and applying Tonelli's theorem. Likewise, for the upper estimate we use
\begin{align*}
\frac{1}{\lambda^{n+3}} \cdot 1_{(8/9 \mu, 8/7 \mu)}(\lambda) 1_{B(y, \lambda)}(x) 1_{(s-\lambda^2, s+\lambda^2)}(t) \leq \frac{2^{n+3}}{\mu^{n+3}} \cdot 1_{B(y, 2 \mu)}(x) 1_{(s-4\mu^2, s+4\mu^2)}(t),
\end{align*}
pull the supremum over $\lambda > 0$ into the triple integral, and apply Tonelli's theorem.

Now, we apply the first part of the lemma to $1_{(0,9\eps/8)}(\mu) F(\mu,y,s)$ in place of $F$, where $\eps > 0$. It follows
\begin{align*}
\iint_\ree \sup_{0<\lambda < \eps} \bigg(\bariiint_{\LQI} |F(\mu,y,s)|^2 \d \mu \d y \d s \bigg) \d x \d t
\lesssim \int_0^\eps \|F_\mu\|_{2}^2 \; \frac{\mathrm{d} \mu}{\mu},
\end{align*}
proving the a.e.\ convergence of averages claimed in the second part of the lemma.
\end{proof}

{Let us recall that the backward parabolic conormal differential of} reinforced weak solutions is defined by
\begin{equation*}
 \pcgt u(\lambda, x,t) = \begin{bmatrix} \dnuA u(\lambda, x,t) \\ \gradx u(\lambda, x,t) \\ \dhalf u(\lambda, x,t) \end{bmatrix}.
\end{equation*}
Using this notation, we can formulate a short hand variant of Theorem~\ref{thm:rh} that is better suited to maximal functions.

\begin{lem}
\label{lem:rh dyadic}
For any reinforced weak solution $u$ to \eqref{eq1} and any parabolic cube $\LQI$ with $8 \Lambda \subset (0,\infty)$,
\begin{align*}
 \bigg(\bariiint_{\LQI} |\pcg u|^2 \d \mu \d y \d s \bigg)^{1/2} \lesssim \sum_{m \geq 0} 2^{-m} \bariiint_{8\Lambda \times 8Q \times 4^m I} |\pcg u| + | \pcgt u| \d \mu \d y \d s.
\end{align*}
\end{lem}

\begin{proof} Observe that the right hand sides of the reverse H\"older inequalities in Theorem~\ref{thm:rh} involve integrals of $|\pcg u| + | \pcgt u|$. To convert the translates on intervals into the formulation above, use that for any measurable function $f$ on the real line
\begin{align*}
\sum_{k \in \IZ} \frac{1}{1+|k|^{3/2}} \barint_{I_k} |f|
\leq \barint_{I} |f| + \sum_{m \geq 0} 4^{-3m/2} \bigg(\sum_{4^m \leq |k| < 4^{m+1}} \barint_{I_k} |f| \bigg)
\leq \sum_{m \geq 0}2^{-m} \barint_{4^m I} |f|. \quad & \qedhere
\end{align*}
\end{proof}

Note that if we introduce the operator matrices
\begin{align*}
 \HT^\theta = \begin{bmatrix} \vphantom{\dhalf}1 & 0 & 0 \\ 0 & 1 & 0 \\ \vphantom{\dhalf}0 & 0 & \HT \end{bmatrix} \qquad \text{and} \qquad \wt{\P} = \begin{bmatrix} 0 & \divx & -\HT \dhalf \\ -\gradx & 0 & 0 \\ -\dhalf & 0 & 0 \end{bmatrix},
\end{align*}
then $\pcg u= \HT^\theta \pcgt u$, $\HT^\theta$ commutes with $\M$ and intertwines $\P$ and $\wt{\P}$ with $P\HT^\theta = \HT^\theta \wt{\P}$, so that the operator $ \wt{\P}\M$ has the same {operator theoretic} properties as $\P\M$. Also we have the equation
$\pd_{\lambda} \pcgt u+ \wt{\P} \M \pcgt u=0$ and the link with the parabolic equation guarantees the off-diagonal estimates for $\wt{\P}\M$ and $\M\wt{\P}$ as well. {Alternatively, this can be seen by a verbatim repetition of the proof of Proposition~\ref{prop:OffDiag}.} Thus, in the following we may argue with $\pcgt u$ as we do with $\pcg u$ and we shall only provide the details involving $\pcg u$.

After these preparations we are ready to prove Theorem~\ref{thm:NTmax}.

\subsubsection*{Step 1: The non-tangential estimate}

Let $h \in \cl{\ran(\P \M)}$ and let $F = Sh$ be its semigroup extension $F(\mu,y,s) = \e^{- \mu [\P \M]}h(y,s)$.

The lower bound for $\NT f$ follows from Lemma~\ref{lem:NT and QE} and strong continuity of the semigroup:
\begin{align*}
 \|\NT F \|_{2}^2 \gtrsim \lim_{\lambda \to 0} \barint_{\frac{7}{8} \lambda}^{\frac{9}{8} \lambda} \|\e^{- \mu [\P \M]}h\|_{2}^2 \d \mu = \|h\|_{2}^2.
\end{align*}

We turn to the interesting upper estimate. In view of \eqref{eq:hardysplit} we have spectral decompositions $h = h^+ + h^-$, where $h^\pm = \chi^\pm(\P \M) h \in \Hpm(\P \M)$, and thus $F = F^+ + F^-$, where $F^\pm = Sh^\pm$. Since $F^+=C_{0}^+h^+$ solves $\partial_\lambda F^+ + \P \M F^+ = 0$ in the weak sense in $\R_{+}$, Theorem~\ref{thm:correspondence} yields a representation $F^+ = \pcg u$ for some reinforced weak solution $u$ to \eqref{eq1}. Similarly, for $\lambda<0$, the reflection $G^-_{\lambda}=(F^-)_{-\lambda}=(C_{0}^-h^-)_{\lambda}$ solves $\partial_\lambda G^- + \P \M G^- = 0$ on $\R_{-}$ so that the similar analysis can be performed. So from now on, we may as well assume $h=h^+$ to simplify matters and that $F=C_{0}^+h=\pcg u $ for some reinforced weak solution $u$ to \eqref{eq1}.

We fix $(x,t)\in \R^{n+1}$ and take $\LQI$ as above. Setting $\wt{F}=\pcgt u$, Lemma~\ref{lem:rh dyadic} asserts
\begin{align*} \bigg(\bariiint_{\LQI} |F|^2 \d \mu \d y \d s \bigg)^{1/2} \lesssim
 \sum_{m \geq 0} 2^{-m} \bariiint_{8\Lambda \times 8Q \times 4^m I} |F| + |\wt{F}| \d \mu \d y \d s.
\end{align*}
To control the right hand side, as mentioned above, it is enough to obtain bounds for the terms involving $F$ using the $\L^2$ estimate on $h$; those for $\wt{F}$ would lead to the same estimate with $P$ replaced by $\wt{P}$ and $h$ replaced by $\HT^\theta h$, which is in $\L^2$ as well.

We introduce $R_\mu := (1+i \mu \P \M)^{-1}$ and split $F_\mu = (F_\mu-R_\mu h)+R_\mu h$. After an application of the Cauchy-Schwarz inequality, the same argument as in the proof of Lemma~\ref{lem:NT and QE} yields
\begin{align*}
 \iint_{\ree} &\sup_{\lambda > 0} \bigg(\sum_{m \geq 0} 2^{-m} \bariiint_{8\Lambda \times 8Q \times 4^m I} |F_\mu-R_\mu h| \d \mu \d y \d s \bigg)^2 \d x \d t \\
 &\lesssim \sum_{m \geq 0} 2^{-m} \iint_{\ree} \sup_{\lambda > 0} \bigg(\bariiint_{8\Lambda \times 8Q \times 4^m I} |F_\mu-R_\mu h|^2 \d \mu \d y \d s \bigg)^2 \d x \d t \\
 &\lesssim \sum_{m \geq 0} 2^{-m} \int_0^\infty \|F_\mu-R_\mu h\|_{2}^2 \; \frac{\mathrm{d} \mu}{\mu} \\
 &=\int_0^\infty \|F_\mu-R_\mu h\|_{2}^2 \; \frac{\mathrm{d} \mu}{\mu}.
\end{align*}
Now, we can use the functional calculus for $\P \M$ to write $F_\mu - R_\mu h = \psi(\mu \P M)h$ with $\psi(z) = \e^{-[z]} - (1+\i z)^{-1}$. As $\psi$ is holomorphic on some double sector containing the spectrum of $\P \M$ and has polynomial decay at $0$ and $\infty$, the quadratic estimates for $\P \M$ obtain in Theorem~\ref{thm:bhfc} yield
\begin{align*}
 \int_0^\infty \|F_\mu-R_\mu h\|_{2}^2 \; \frac{\mathrm{d} \mu}{\mu} = \int_0^\infty \|\psi(\mu \P \M) h\|_{2}^2 \; \frac{\mathrm{d} \mu}{\mu} \sim \|h\|_{2}^2.
\end{align*}
The remaining term is
\begin{align}
\label{eq2:NT max}
 \sum_{m \geq 0} 2^{-m} \bariiint_{8\Lambda \times 8Q \times 4^m I} |R_\mu h| \d \mu \d y \d s.
\end{align}
Let $q \in (1,2)$ be such that we have Proposition~\ref{prop:OffDiag} on $\L^q$ off-diagonal decay for $R_\mu$ at our disposal. With $N_0 \in \IN$ as provided by this proposition we define for $J \subset \R$ and $B \subset \R^n$ the annuli
\begin{align*}
 C_0(B \times J) = 2 B \times N_{0} J, \qquad C_k(B \times J) = (2^{k+1}B \times N_0^{k+1} J) \setminus (2^k B \times N_0^k J)
\end{align*}
and we obtain an estimate
\begin{align*}
 \bariint_{8Q \times 4^m I} |R_\mu h| \d y \d s
& \leq \bigg(\bariint_{8Q \times 4^m I} |R_\mu h|^q \d y \d s \bigg)^{1/q} \\
& \lesssim \sum_{k \geq 0} N_0^{- \eps k} \bigg(\bariint_{C_k(8Q \times 4^m I)} |h|^q \d y \d s \bigg)^{1/q} \\
& \lesssim \sum_{k \geq 0} N_0^{- \eps k} \bigg(\barint_{2^{k+4}Q} \bigg(\barint_{N_0^{k+1} 4^m I} |h|^q \d y \bigg) \d s \bigg)^{1/q}
\end{align*}
for some $\eps > 0$ and implicit constants independent of $m \geq 0$ and $\mu \in 8 \Lambda$. By means of the $q$-adapted maximal operators acting separately on each variable by
\begin{align*}
 \Max_t^q f(x,t) = \sup_{J \ni t} \bigg(\barint_J |f(x,s)|^q \d s \bigg)^{1/q} \quad \text{and} \quad \Max_x^q f(x,t) = \sup_{B \ni x} \bigg(\barint_B |f(y,t)|^q \d y\bigg)^{1/q},
\end{align*}
where $J \subset \R$ and $B \subset \R^n$ denote intervals and balls, respectively, the ongoing estimate can be completed as
\begin{align*}
 \bariint_{8Q \times 4^m I} |R_\mu h| \d y \d s \lesssim \Max_x^q (\Max_t^q h )(x,t) \qquad (m \ge 0, \, \mu \in 8 \Lambda).
\end{align*}
Since $q < 2$, the $q$-adapted maximal operators are bounded on $\L^2$. Whence, the term defined in \eqref{eq2:NT max} has non-tangential maximal function bounds
\begin{align*}
 \iint_{\ree} \sup_{\lambda > 0} \bigg(\sum_{m \geq 0} 2^{-m} \bariiint_{8\Lambda \times 8Q \times 4^m I} |R_\mu h| \d \mu \d y \d s \bigg)^2 \d x \d t
 &\lesssim \|\Max_x^q \Max_t^q h \|_{2}^2
 \lesssim \|h\|_{2}^2.
\end{align*}
Altogether, we obtain
\begin{align}
\label{eq:NTtranslatesBound}
 \|\NT(F)\|_2^2 \lesssim \iint_{\ree} \sup_{\lambda > 0} \bigg(\sum_{m \geq 0} 2^{-m} \bariiint_{8\Lambda \times 8Q \times 4^m I} |F(\mu,y,s)| \d \mu \d y \d s \bigg)^2 \d x \d t \lesssim \|h\|_2^2
\end{align}
as required.

\subsubsection*{Step 2: Almost everywhere convergence of averages}

In order to complete the proof of the theorem we have yet to show that for any $h \in \cl{\ran(\P \M)}$ there is convergence
\begin{align}
\label{eq1:Whitney convergence}
 \lim_{\lambda \to 0} \bariiint_\LQI |Sh(\mu,y,s) - h(x,t)|^2 \d \mu \d y \d s = 0 \qquad (\text{a.e.\ $(x,t) \in \ree$}).
\end{align}
As we have the $\L^2$ estimate on the maximal function, it suffices to prove this for $h$ in some dense subset of $\cl{\ran(\P M)}$.

\begin{lem}
\label{lem:dense class Whitney convergence}
There exists $\delta_0 > 0$ depending only on {dimension and} the ellipticity constants of $A$ such that if $|\frac{1}{p} - \frac{1}{2}|<\delta_0$, then $\{h \in \ran(\P M) \cap \dom(\P \M) : \P \M h \in \L^p(\ree; \IC^{n+2}) \}$ is dense in $\cl{\ran(\P \M)}$.
\end{lem}

\begin{proof}
We take $\delta _{0}$ as in Lemma~\ref{lem:LpLq}, thereby guaranteeing that all resolvents $R_\mu = (1+\i \mu \P \M)^{-1}$, $\mu \in \R$, map $\mHp \cap \mH$ into itself where $\mHp=\L^p(\R^{n+1}; \IC^{n+2})$. For $m \in \IN$ define
\begin{align*}
 T_m \in \Lop(\mH), \quad T_m h = \i m R_{\i m} \P \M R_{\i/m} h.
\end{align*}
Since $\P \M$ is densely defined and bisectorial, the $T_m$ are uniformly bounded with respect to $m$ and on $\cl{\ran(\P \M)}$ they converge strongly to the identity, see \cite[Prop.~3.2.2]{EigeneDiss} for an explicit statement. Now, let $h \in \cl{\ran(\P \M)}$. We can choose functions $h_m \in \mHp \cap \mH$ such that $h_m \to h$ in $\mH$. Then also $T_m h_m \to h$ in $\mH$. However, by definition we have $T_m h_m \in \ran(\P \M) \cap \dom(\P \M)$ and expanding
\begin{align*}
 \P M T_m h_m = \i m h_m + R_{\i/m} h_m + m^2 R_{\i m} h_m - \i m R_{\i m} R_{\i/m} h_m,
\end{align*}
we read off $\P \M T_m h_m \in \mHp \cap \mH$ as desired.
\end{proof}

We complete the proof of the theorem by demonstrating \eqref{eq1:Whitney convergence} if $h \in \ran(\P M) \cap \dom(\P \M)$ satisfies $\P \M h \in \L^p(\ree; \IC^{n+2})$ for some $p>2$. As in Step~1 we let $R_\mu = (1 + \i \mu \P \M)^{-1}$ and $\psi(z) = \e^{-[z]} - (1+\i z)^{-1}$. We start out by estimating
\begin{align*}
 \bariiint_\LQI |Sh(\mu,y,s) - h(x,t)|^2 \d \mu \d y \d s
 & \lesssim \bariiint_\LQI |\psi(\mu \P \M)h(y,s)|^2 \d \mu \d y \d s \\
 &\quad + \bariiint_\LQI |R_\mu h(y,s) - h(y,s)|^2 \d \mu \d y \d s \\
 &\quad + \bariint_{Q \times I} |h(y,s) - h(x,t)|^2 \d y \d s.
\end{align*}
Thanks to Lemma~\ref{lem:NT and QE} and the quadratic estimates for $\P \M$, the first term vanishes for a.e.\ $(x,t) \in \ree$ in the limit $\lambda \to 0$. The third term vanishes for every (parabolic) Lebesgue point of $h \in \L^2(\ree)$, that is, for a.e.\ $(x,t) \in \ree$. Since $h \in \dom(\P \M)$, the middle term can be rewritten as
\begin{align*}
 \bariiint_\LQI |\mu R_\mu \P \M h(y,s)|^2 \d \mu \d y \d s
 & \lesssim \lambda^2 \bariiint_\LQI |R_\mu \P \M h(y,s)|^2 \d \mu \d y \d s \\
 & \lesssim \lambda^2 \big(\Max_x^2 \Max_t^2( \P \M h )(x,t)\big)^2,
\end{align*}
where the second step follows from the $\L^2$ off-diagonal estimates for $R_\mu$ similar to the estimate of \eqref{eq2:NT max}. Hence, these terms vanish in the limit $\lambda \to 0$ if only the maximal function $\Max_x^2 \Max_t^2(\P \M h)(x,t)$ is finite. This happens to be almost everywhere thanks to our additional assumption on $h$: In fact, $\Max_t^2$ and $\Max_x^2$ are bounded on $\L^p(\ree)$ since $p>2$ and thus $\P \M h \in \L^p(\ree; \IC^{n+2})$ implies $\Max_x^2 \Max_t^2(\P \M h) \in \L^p(\ree)$.

\begin{rem} Although it is not needed for our applications, we observe that the almost everywhere convergence \eqref{eq1:Whitney convergence} holds also for $h\in \nul(\P\M)$. In this case, as $(Sh)_{\lambda}=h$ for all $\lambda>0$, we can control the averages by the parabolic Hardy-Littlewood maximal operator $\Max_{(x,t)}^2h(x,t)$, use its weak-type (2,2) and Lebesgue differentiation.
\end{rem}
\subsection{Proof of Theorem \ref{thm:NTmaxDir}}

The estimates for the semigroup extension with respect to $\M \P$ are easily obtained from those for $\P \M$.

Since $\M: \cl{\ran(\P \M)} \to \cl{\ran(\M \P)}$ is an isomorphism, it is a matter of fact for the functional calculus of these type of operators that there is an \emph{intertwining relation}
\begin{align}
\label{eq1:NTmaxDir}
 \e^{- \mu [\M \P]} h = \M \e^{-\mu [\P \M]} \M^{-1}h \qquad (h \in \cl{\ran(\M \P)}).
\end{align}
In fact, we can directly verify $(\lambda - \M \P)^{-1} \M = \M(\lambda -\P \M)^{-1}$ for $\lambda$ in a suitable double sector in the complex plane and extend this relation to the Cauchy integral defining the semigroups above.

Since $\M$ is a bounded multiplication operator, it acts boundedly on $\L^2(\LQI; \IC^{n+2})$ for every Whitney region $\LQI$. Let now $h \in \cl{\ran(\M \P)} = \M \cl{\ran(\P)}$. Since $\M^{-1} h \in \cl{\ran(\P \M)}$ falls under the scope of Theorem~\ref{thm:NTmax} and as $\M^{-1}: \cl{\ran(\M \P)} \to \L^2(\ree;\IC^{n+2})$ is bounded (though not necessarily a multiplication operator), the upper $\NT$-bound and the almost everywhere convergence of Whitney averages for the $\M \P$-semigroup follow from \eqref{eq1:NTmaxDir} and the corresponding result for the $\P \M$-semigroup. Finally the lower bound $\|\NT(\e^{- \cdot [\M \P]} h)\|_2 \gtrsim \|h\|_2$ follows from Lemma~\ref{lem:NT and QE} and the strong continuity of the semigroup for $\M \P$ just as in the case of $\P \M$.
\section{Duality results for \texorpdfstring{$\lambda$}{lambda}-independent operators}
\label{sec:charWPDuality}

In this section we prove Theorem~\ref{thm:duality} and the abstract Green's formula announced in Proposition~\ref{prop:green}. The main step is the following. We use the notation introduced in Section~\ref{sec:backward}. The relevant Sobolev spaces have been introduced in Section~\ref{sec:sobolev}. Recall that we gather $\pa$- and $\te$-components of vectors in a single component denoted $r$.

\begin{lem}[Simultaneous duality]
\label{lem:duality}
Let $-1\le s\le 0$.  There is a pairing simultaneously realising the pairs of spaces   $(\Hdot^s_{\P},\Hdot^{-s-1}_{\P^*})$,   $(\Hdot^{s, \pm}_{\P\M}, \Hdot^{-s-1,\mp}_{\P^*\wM})$, $((\Hdot^{s}_{\P})_{\pe}, (\Hdot^{-s-1}_{\P^*})_{r})$ and $((\Hdot^{s}_{\P})_{r}, (\Hdot^{-s-1}_{\P^*})_{\pe})$  as dual spaces. Moreover, in this pairing, the pairs  $(\Hdot^{s, \pm}_{\P\M}, \Hdot^{-s-1,\pm}_{\P^*\wM})$, $((\Hdot^{s}_{\P})_{\pe}, (\Hdot^{-s-1}_{\P^*})_{\pe})$ and $((\Hdot^{s}_{\P})_{r}, (\Hdot^{-s-1}_{\P^*})_{r})$  are  pairs of orthogonal spaces.
\end{lem}

\begin{rem}
\label{rem:duality}
In order to make the statement of Lemma~\ref{lem:duality} meaningful, we identify spaces of type $X_{\pe}$ and $X_r$ as subspaces of $X$ via
$X_{\pe}\cong X_{\pe} \oplus \{0\}$ and $X_{r}\cong \{0\} \oplus X_{r}$.
\end{rem}

\begin{proof}
With $N$ as defined in \eqref{eq:N} and $P^*$ the adjoint of $\P$ in the canonical $\L^2(\R^{n+1}; \IC^{n+2})$ inner product,  $N$ and $\P^*$ anti-commute: $\P^*N=- N\P^*$. So, with $\wM= N\M^*N= N\M^*N^{-1}$, we obtain that the adjoint of $\P\M$ in the complex form  $$(f,Ng)=\iint_{\R^{n+1}} f \cdot \overline {Ng}\d x\d t$$ in $\L^2(\R^{n+1}; \IC^{n+2})$ is $-\wM\P^*$. Thus, the dual of the Sobolev space $\Hdot^s_{\P\M}$ in the extension of this duality is $\Hdot^{-s}_{\wM\P^*}$. Next, the formal relation  $\P^*(\wM\P^*)= (\P^*\wM)\P^*$ allows one to show that $\P^*$ extends to an isomorphism between  $\Hdot^{-s}_{\wM\P^*}$ and $\Hdot^{-s-1}_{\P^*\wM}$. Therefore $(\Hdot^s_{\P\M},\Hdot^{-s-1}_{\P^*\wM})$    is a pair of dual spaces for the duality form $$\beta(f,g)=  (f, N{\P^*}^{-1}g). $$ As {this very argument also applies  to $\M = 1$, we} see that $(\Hdot^s_{\P},\Hdot^{-s-1}_{\P^*})$ are dual spaces for the same duality form. As $P$ and $P^*$ swap the $\pe$- and $r$-components we have the announced duality for the $\pe$- and $r$-spaces as well as the orthogonality. And from the anti-commutation of $N$ and $\P^*$, we derive the announced duality and orthogonality for the spectral spaces.
\end{proof}

The proof of Theorem~\ref{thm:duality} now follows by applying abstract results. We shall make the same identifications as in Remark~\ref{rem:duality}.

\begin{proof}[Proof of Theorem~\ref{thm:duality}]We have a pair of projectors $(\chi^+(\P\M), \chi^-(\P\M))$ associated with the splitting $\Hdot^{s}_{\P}=\Hdot^{s, +}_{\P\M} \oplus \Hdot^{s, -}_{\P\M}$ and due to the identifications made above, the component maps $(N_{\pe}, N_{r})$ become the projectors associated with $\Hdot^s_{\P}= (\Hdot^{s}_{\P})_{\pe} \oplus (\Hdot^{s}_{\P})_{r}$. In the duality $\beta$ their dual pairs are $(\chi^-(\P^*\wM), \chi^+(\P^*\wM))$ and $(N_{r}, N_{\pe})$ corresponding to the dual splittings $\Hdot^{-s-1}_{\P^*}=\Hdot^{-s-1, -}_{\P^*\wM} \oplus \Hdot^{-s-1, +}_{\P^*\wM}$ and
$\Hdot^{-s-1}_{\P^*}= (\Hdot^{-s-1}_{\P^*})_{r} \oplus (\Hdot^{-s-1}_{\P^*})_{\pe}$, respectively.

Let us consider for example the Neumann problem $(N)_{\mE_{s}}^\Lop$. By Theorem~\ref{thm:wpequiv} this is well-posed if and only if $N_{\pe}: \Hdot^{s, +}_{\P\M} \to (\Hdot^{s}_{\P})_{\pe}$ is an isomorphism.
In the duality above, the adjoint of  $N_{\pe}: \Hdot^{s, +}_{\P\M} \to (\Hdot^{s}_{\P})_{\pe}$
is $\chi^-(\P^*\wM): (\Hdot^{-s-1}_{\P^*})_{r} \to \Hdot^{-s-1, -}_{\P^*\wM}$. Applying an abstract result on pairs of projections~\cite[Lem.~13.2]{Auscher-Mourgoglou}, we have that invertibility of the latter is exactly equivalent to invertibility  of $N_{\pe}: \Hdot^{-s-1, +}_{\P^*\wM} \to (\Hdot^{-s-1}_{\P^*})_{\pe}$. By the analogue of Theorem~\ref{thm:wpequiv} for the backward equation, this in turn is equivalent to well-posedness of $(N)_{\mE_{-1-s}}^{\Lop^*}$.

The regularity problem can be handled in the same manner. Finally, preservation of compatible well-posedness can be obtained unravelling this argument.
\end{proof}

A proof of the abstract Green's formula can be given by an approximation procedure as in \cite[Thm.~1.7]{Auscher-Mourgoglou}. There, it was mentioned that one can also use the simultaneous duality above {and this is the approach that we shall present here.}

\begin{proof}[Proof of Proposition~\ref{prop:green}]  Set $h= \pcg u|_{\lambda=0} \in \Hdot^{s, +}_{\P\M} $ and $g=\pcgb w|_{\lambda=0} \in \Hdot^{-s-1, +}_{\P^*\wM}$.
 We have $(h,N{\P^*}^{-1}g)=0$ by the orthogonality relation and we claim that
 $$ (h,N{\P^*}^{-1}g)= \dual {\partial_{\nu_{A}}u|_{\lambda=0}} {w|_{\lambda=0}}- \dual {u|_{\lambda=0}}{\partial_{\nu_{A^*}}w|_{\lambda=0}}.
 $$
 Indeed, $\pe$ and $r$-component  of $h$ and $N{\P^*}^{-1}g$ can be computed in  appropriate dual spaces described in the argument  below and we write $(h,N{\P^*}^{-1}g)= \dual {h_{\pe}} { (N{\P^*}^{-1}g)_{\pe}}+ \dual{ h_{r}} {(N{\P^*}^{-1}g)_{r}}$.  First,
 $h_{\pe}= \partial_{\nu_{A}}u|_{\lambda=0} \in  (\Hdot^{s}_{\P})_{\pe}= \Hdot^{s/2}_{\pd_{t}-\Delta_{x}}$ and using Lemma~\ref{lem:proj} and the isomorphism property of  $\IP_{\P^*}$,
 \begin{align*}
  (N{\P^*}^{-1}g)_{\perp}=  w|_{\lambda=0} \in  (\IP_{\P^*}\Hdot^{-s}_{\wM\P^*})_{\pe}=(\Hdot^{-s}_{\P^*})_{\pe}= \Hdot^{-s/2}_{-\pd_{t}-\Delta_{x}}.
 \end{align*}
 To compute the other part, we may symmetrise the argument by changing $s$ to $-s-1$ and observing that $(h,N{\P^*}^{-1}g)=- (N\P^{-1} h,g)$ and that $N\P^{-1}$ swaps $r$- and $\pe$-components. Hence $\dual{h_{r}}{ (N{\P^*}^{-1}g)_{r})}= -\dual{(N\P^{-1} h)_{\pe}}{ g_{\pe}}$ formally and it remains to identify what this means. We have
 $(N\P^{-1} h)_{\pe}= u|_{\lambda=0} \in (\Hdot^{s+1}_{\P})_{\pe}= {\Hdot^{s/2+1/2}_{\pd_{t}-\Delta_{x}}}$ and $g_{\pe}= \partial_{\nu_{A^*}}w|_{\lambda=0} \in  (\Hdot^{-s-1}_{\P^*})_{\pe}= {\Hdot^{-s/2-1/2}_{-\pd_{t}-\Delta_{x}}}$. The result follows.
  \end{proof}
\section{The operator \texorpdfstring{$\sgn(\P\M)$}{sgn(PM)}}
\label{sec:sgnPM proofs}

We provide the rather algebraic proofs of the results on $\sgn(\P \M)$ stated in Section~\ref{sec:sgnPM results}. The appearing Sobolev spaces have been introduced in Section~\ref{sec:sobolev}.} We keep on representing vectors by $\pe$- and $r$-components and operators on spaces $X= \Hdot_\P^s$ by the associated $(2 \times 2)$-matrices and let $i_{\pe}:X_{\pe}\to X_{\pe} \oplus \{0\}$ and $i_{r}:X_{r}\to \{0\} \oplus X_r$ be the canonical embeddings. Then the maps $N_{\pe}, N_{r}$ are just the coordinate maps with this representation. From \eqref{eq:rep chi} we obtain
 \begin{eqnarray}
\label{eq:s1}
\begin{split}
   1\pm s_{\pe\pe}(\P\M)&= 2N_{\pe}\chi^\pm(\P\M) \, i_{\pe},\\
s_{ r\pe }(\P \M)&=2N_{r}\chi^+(\P\M)\, i_{\pe},\
\end{split} \qquad
\begin{split}
 s_{\pe r}(\P \M)&= 2N_{\pe}\chi^+(\P\M)\, i_{r},\\
 1\pm s_{rr}(\P \M)&= 2N_{r}\chi^\pm(\P\M)\, i_{r}.
\end{split}
\end{eqnarray}

With this notation set up, we are ready to give the

\begin{proof}[Proof of Theorem~\ref{thm:six invertibilities}]
\begin{enumerate}
  \item Based on Theorem \ref{thm:wpequiv} and its counterpart for the lower half-space, $(R)_{\mE_{s}}^\Lop$ is well-posed on \emph{both} half-spaces if, and only if, the two operators $N_{r}: \Hdot^{s,\pm}_{\P\M} \to (\Hdot^{s}_{\P})_{r}$  are invertible. By an abstract principle for pairs of complementary projections~\cite[Lem.~13.6]{Auscher-Mourgoglou}, this in turn is equivalent to invertibility of the composite map $N_r \chi^{+}(\P\M)\, i_{\pe}: (\Hdot^{s}_{\P})_{\pe} \to (\Hdot^{s}_{\P})_r$, which in view of \eqref{eq:s1} exactly means that $s_{r \pe}(\P \M):(\Hdot^{s}_{\P})_{\pe} \to (\Hdot^{s}_{\P})_r$ is invertible. In order to avoid confusion, we note that in \cite{Auscher-Mourgoglou} the results are equivalently formulated using the projectors $P_{\pe}= i_{\pe}\circ N_{\pe}$, $P_{r}=i_{r}\circ N_{r}$ of $X$ onto $X_{\pe} \oplus \{0\}$ and $\{0\} \oplus X_{r}$, respectively.

  \item Theorem~\ref{thm:wpequiv} yields again equivalence of the Neumann problem $(N)_{\mE_{s}}^\Lop$ being well-posed on both half-spaces with invertibility of the two operators $N_{\pe}: \Hdot^{s,\pm}_{\P\M} \to (\Hdot^{s}_{\P})_{\pe}$. By the same reasoning as before, this is equivalent to $s_{\pe r}: (\Hdot^{s}_{\P})_{r} \to (\Hdot^{s}_{\P})_\pe$ being invertible.

  \item Let us abbreviate $s_{\pe \pe}(\P \M)$ by $s_{\pe \pe}$ and so on. Spelling out the identity $\sgn(\P \M)^2 = 1$ reveals
  \begin{align*}
   \begin{bmatrix}
    s_{\pe \pe}^2 + s_{\pe r}s_{r \pe} & s_{\pe \pe}s_{\pe r} + s_{\pe r} s_{r r} \\
    s_{r \pe}s_{\pe \pe} + s_{r r} s_{r \pe} & s_{r \pe}s_{\pe r} + s_{r r}^2
   \end{bmatrix}
   = \begin{bmatrix} 1 & 0 \\ 0 & 1 \end{bmatrix}.
  \end{align*}
  In particular, we have $s_{\pe r}s_{r \pe} = (1 + s_{\pe \pe})(1 - s_{\pe \pe})$ and $s_{r \pe}s_{\pe r} = (1 + s_{r r})(1 - s_{r r})$. Now, suppose that $s_{\pe r}$ and $s_{r \pe}$ are invertible. Then we have invertible products of commuting operators $(1 + s_{\pe \pe})(1 - s_{\pe \pe})$ and $(1 + s_{r r})(1 - s_{r r})$, which implies that the four factors are invertible. Conversely, if the four operators $1 \pm s_{\pe \pe}$, $1 \pm s_{r r}$ are invertible, then both products $s_{\pe r}s_{r \pe}$, $s_{r \pe}s_{\pe r}$ are invertible, which implies that $s_{\pe r}$, $s_{r \pe}$ themselves are invertible.

  \item This is just the conjunction of the previous three equivalences. \qedhere
\end{enumerate}
\end{proof}

Finally, we give the proof of Proposition~\ref{prop:sgn invertibilty implies WP}. The reader may recall the definition of the Neumann-to-Dirichlet map $\Gamma_{ND}^{\Lop,+}$ on the upper half-space from Section~\ref{sec:sgnPM results}.

\begin{proof}[Proof of Proposition~\ref{prop:sgn invertibilty implies WP}]
The four statements are all proved in exactly the same way. We restrict ourselves to the Neumann problem $(N)_{\mE_s}^\Lop$ on the upper half-space. Invoking Proposition~\ref{prop:cwp}, we have to check that $N_\pe: \Hdot_{\P\M}^{s,+} \to (\Hdot_{\P}^{s})_\pe$ is an isomorphism and that its inverse agrees with the one obtained on $(\Hdot_{\P}^{-1/2})_\pe$. Due to \eqref{eq:representation H-1/2+}, \eqref{eq:GammaND factorization} and our assumption that $1 - s_{r r}(\P \M): (\Hdot_\P^s)_r \to (\Hdot_\P^s)_r$ is invertible and that the inverse is compatible on $(\Hdot_\P^{-1/2})_r$, we obtain that
\begin{align*}
 \begin{bmatrix} \id \\ \Gamma_{ND}^{\Lop,+} \end{bmatrix}: (\Hdot_{\P}^{s})_\pe \cap (\Hdot_{\P}^{-1/2})_\pe \to \Hdot_{\P\M}^{s,+} \cap \Hdot_{\P\M}^{-1/2,+}
\end{align*}
is well-defined and a two-sided inverse for $N_\pe: \Hdot_{\P\M}^{s,+} \cap \Hdot_{\P\M}^{-1/2,+} \to (\Hdot_{\P}^{s})_\pe \cap (\Hdot_{\P}^{-1/2})_\pe$. Again by assumption it extends to a bounded operator $(\Hdot_{\P}^{s})_\pe \to \Hdot_{\P\M}^{s,+}$. This yields the sought-after inverse for $N_\pe: \Hdot_{\P\M}^{s,+} \to (\Hdot_{\P}^{s})_\pe$ and compatibility on $(\Hdot_{\P}^{-1/2})_\pe$ follows by construction.
\end{proof}
\section{Inverting parabolic operators by layer potentials}
\label{sec:layer}

Here, we provide the remaining proofs for results on layer potentials discussed in Section~\ref{sec:layer intro}, where the relevant notation has been introduced.

\begin{proof}[Proof of Theorem~\ref{thm:inverse1}]
Let $f \in \Hdot^{-1/2}_{\pd_{t}-\Delta_{\lambda,x}}$ and $u\in \Hdot^{1/2}_{\pd_{t}-\Delta_{\lambda,x}}$. Note that $\pcg u \in \L^2(\R; \clos{\ran(P)})$. Following the strategy of proof of Theorem~\ref{thm:correspondence}, we see that $u=\Lop^{-1}f$ if and only if $\pd_{\lambda}F+ \P\M F= \begin{bmatrix} f \\ 0\end{bmatrix}$ in the weak sense, where $F=\pcg u$. Recall that for vectors $F \in \IC^{n+2}$ we gather here $F_\pa$ and $F_\te$ in one component denoted $F_r$. Replacing $\begin{bmatrix} f \\ 0\end{bmatrix}$ by the differential equation and integrating by parts, it follows that we have, in $\L^2(\R; \clos{\ran(P)})$ and for almost every $\lambda>0$,
\begin{align*}
 \int_{\eps<|\lambda-\mu|<R}\pcg{\mS_{\lambda - \mu}} f_{\mu}\, d\mu=&
+ \e^{-\eps \P\M} \chi^{+}(\P\M) F_{\lambda - \eps}
+\e^{\eps \P\M} \chi^{-}(\P\M) F_{\lambda+\eps}
\\
& \qquad - \e^{-R \P\M} \chi^{+}(\P\M) F_{\lambda - R}
- \e^{R \P\M} \chi^{-}(\P\M) F_{\lambda+R}.
\end{align*}
Applying $-\P^{-1}$ to both sides, we have an equality in $\L^2(\R; \Hdot^{1}_{\P})$. Finally, taking the $\perp$-component, we obtain in $\L^2(\R; \Hdot^{1/2}_{\pd_{t}-\Delta_{x}})= E'$,
\begin{align*}
 \int_{\eps<|\lambda-\mu|<R}{\mS_{\lambda - \mu}} f_{\mu}\, d\mu=&
 -(\P^{-1}e^{-\eps \P\M} \chi^{+}(\P\M) F_{\lambda-\eps} )_{\perp}
-(\P^{-1}e^{\eps \P\M} \chi^{-}(\P\M) F_{\lambda+\eps} )_{\perp} \\
& \qquad +(\P^{-1} \e^{-R \P\M} \chi^{+}(\P\M) F_{\lambda - R})_{\perp}
+ (\P^{-1}e^{R \P\M} \chi^{-}(\P\M) F_{\lambda + R})_{\perp}.
\end{align*}
This defines uniformly bounded operators from $\Hdot^{-1/2}_{\pd_{t}-\Delta_{\lambda, x}}$ to $E'$. Next, we take strong limits in $E'$. The terms corresponding to $R\to \infty$ tend to 0 using the bounded holomorphic functional calculus of $\P\M$ and that $F_{\lambda\pm R}$ tend to 0 in $\L^2(\R; \Hdot^{1}_{\P})$. The terms corresponding to $\eps\to 0$, converge to
\begin{align*}
 -(\P^{-1} \chi^{+}(\P\M) F )_{\perp}
-(\P^{-1} \chi^{-}(\P\M) F )_{\perp} = -(\P^{-1}F)_{\perp} = u.
\end{align*}
This proves the limit in the statement.
\end{proof}

\begin{proof}[Proof of Theorem~\ref{thm:inverse2}]
To $\Lop$ there correspond $\P$ and {$\M = \begin{bmatrix} 1 & 0 \\ 0 & A \end{bmatrix}$} with $A = \begin{bmatrix} A_{\pa \pa} & 0 \\ 0 & 1 \end{bmatrix}$  as in Section~\ref{sec:proof of correspondence}. Let $f\in \Hdot^{-1/2}_{\pd_{t}-\Delta_{x}}$. As $\begin{bmatrix} f \\ 0\end{bmatrix}\in \Hdot^{-1}_{\P}$, we find, using the functional calculus of $\P\M$ on $\Hdot^{-1}_{\P}$,
\begin{equation*}
 \int_{\eps}^R \mS_{\lambda } f  \d\lambda = \bigg((\P\M\P)^{-1}(e^{-R \P\M} - \e^{-\eps \P\M})\chi^+(\P\M) \begin{bmatrix} f \\ 0\end{bmatrix}\bigg)_\pe.
\end{equation*}
Similarly, we obtain
\begin{equation*}
 \int_{-R}^{-\eps} \mS_{\lambda } f  \d\lambda = -\bigg((\P\M\P)^{-1}(e^{\eps \P\M} - \e^{R \P\M})\chi^-(\P\M)\begin{bmatrix} f \\ 0\end{bmatrix}\bigg)_\pe.
\end{equation*}
Here, we have used that by construction $\P\M\P$ is an isomorphism from $\Hdot^1_{\P}$ onto $\Hdot^{-1}_{\P}$. Hence, taking strong limits for the semigroup on each spectral space of $\Hdot^{-1}_{\P}$, and adding terms we obtain
\begin{align*}
{\lim_{\eps \to 0, \, R \to \infty} \int_{\eps \leq |\lambda| \leq R} \mS_{\lambda } f \d\lambda} = -\bigg((\P\M\P)^{-1} (\chi^+(\P\M) +\chi^-(\P\M))\begin{bmatrix} f \\ 0\end{bmatrix}\bigg)_\pe= -\bigg((\P\M\P)^{-1} \begin{bmatrix} f \\ 0\end{bmatrix}\bigg)_\pe.
\end{align*}
From the form of $\Lop$ we see that $\P\M\P$ is block diagonal and the scalar block is precisely $\pd_{t}- \div_{x} A(x,t)\nabla_{x}=L$. This gives us the conclusion.
\end{proof}
\section{Well-posedness results for \texorpdfstring{$\lambda$}{lambda}-independent coefficients}
\label{sec:wpresults}

In this section we shall prove our well-posedness results formulated in Theorem~\ref{thm:WP} and obtain the resolution of the Kato problem for parabolic operators. We also supply proofs of the related Propositions~\ref{prop:continuitymethod} and \ref{prop:Dir}. We shall need again the Sobolev spaces introduced in Section~\ref{sec:sobolev}.
\subsection{The proof of Theorem~\ref{thm:Kato}}\label{sec:parKato}

For $A_{\pa \pa}(x,t)$ any bounded and uniformly elliptic $(n \times n)$-matrix we let
\begin{align*}
 \M:= \begin{bmatrix}
 1 & 0 & 0\vphantom{\dhalf} \\
 0 & A_{\pa\pa} & 0 \\ 0\vphantom{\dhalf} & 0 & 1
\end{bmatrix}, \qquad
 \P \M = \begin{bmatrix}
 0 & \divx A_{\pa \pa} & -\dhalf \\
 -\gradx & 0 & 0 \\ -\HT \dhalf & 0 & 0
\end{bmatrix}
\end{align*}
As $\P\M$ has a bounded holomorphic functional calculus on $\clos{\ran(\P)}$ by Theorem~\ref{thm:bhfc}, the operator $\sgn(\P\M)$ is a bounded involution on this space. Therefore $\P\M$ and $[\P\M] = \sgn(\P \M) \P \M$ share the same domain with $\|\P\M h\|_{2}\sim \| [\P\M]h\|_{2}$. But $[\P\M] = \sqrt{(\P\M)^2}$ by definition and
\begin{align*}
 (\P\M)^2= \begin{bmatrix}
 {L} & 0 & 0 \\ 0 & - \gradx \divx A_{\pa \pa} & \gradx \dhalf \\ 0 & -\HT \dhalf \divx A_{\pa \pa} & \pd_{t}
 \end{bmatrix}.
\end{align*}
Specialising to the first component, we see that the domain of $\sqrt{L}$ is  $\H^{1/2}(\R; \L^2(\R^n)) \cap \L^2(\R; \W^{1,2}(\R^n))$ together with the homogeneous estimate
\begin{align*}
{\|\sqrt{L}\|_{2}}\sim \|\gradx f\|_{2}+\|\HT\dhalf f\|_{2}= \|\gradx f\|_{2}+\|\dhalf f\|_{2}
\end{align*}
This solves the parabolic Kato problem.
\subsection{The block case (i)}

\label{sec:WPblock}
More generally, we assume that the coefficients $A(x,t)$ are in block form
\begin{align*}
 A:= \begin{bmatrix}
 A_{\pe\pe} & 0 \\
 0 & A_{\pa\pa}
\end{bmatrix}.
\end{align*}
Since results for this case will be re-used in the context of more general coefficients, we shall write $A(x,t) = A_b(x,t)$ and $\M = \M_b$ to avoid confusion. As usual we concatenate $\pa$- and $\te$-components of vectors to a single component denoted $r$. Similarly, we represent operators acting on them as $(2 \times 2)$-matrices. For example, $\M = \M_b$ can be written in block form
\begin{align*}
 \M_b = \begin{bmatrix} A_{\pe \pe}^{-1} & 0 \\ 0 & \M_{r r} \end{bmatrix},\qquad \text{where} \qquad
 \M_{r r} := \begin{bmatrix} A_{\pa \pa}\vphantom{A_{\pe \pe}^{-1}}& 0 \\ 0 & 1 \end{bmatrix}.
\end{align*}
For $\sgn(P M_b)$ we have already introduced such a representation as a $(2\times2)$-matrix in \eqref{eq:rep sgn}:
 \begin{align*}
 \sgn(\P \M_b)
 = \begin{bmatrix} s_{\pe \pe}(\P \M_b) & s_{\pe r}(\P \M_b) \\ s_{r \pe}(\P \M_b) & s_{r r}(\P \M_b) \end{bmatrix},
\end{align*}
which by the holomorphic functional calculus is bounded on $\clos{\ran(\P)}$ and satisfies $\sgn(\P \M_b)^2 =1$. On the other hand, $\M_b$ commutes with the matrix $N$ introduced in \eqref{eq:N} since both are in block form. Hence, $\P \M_b N = - N \P \M_b$, which carries over to the functional calculus by $\sgn(\P \M_b) N = - N \sgn(\P \M_b)$. These two properties of $\sgn(\P \M_b)$ can only hold if its diagonal blocks vanish and its off-diagonal blocks are inverses of each other, that is,
\begin{align}
\label{eq:sgnPM block}
 \sgn(\P \M_b)
 = \begin{bmatrix} 0 & s_{\pe r}(\P \M_b) \\ s_{\pe r}(\P \M_b)^{-1} & 0 \end{bmatrix}.
\end{align}
From Section~\ref{sec:sobolev} we know that the bounded functional calculus for $\P\M_b$ extends to the Sobolev spaces $\Hdot^s_{\P \M_b} = \Hdot^s_{\P}$, $-1 \leq s \leq 0$, and by density of $\cl{\ran(\P)} \cap \Hdot^s_{\P}$ in $\Hdot^s_{\P}$, the formula \eqref{eq:sgnPM block} extends, too. In particular, $1\pm s_{\pe \pe}(\P \M_b)$ and $1\pm s_{r r}(\P \M_b)$ all act as the identity on the respective components of $\Hdot^s_{\P}$. Thus, Proposition~\ref{prop:sgn invertibilty implies WP} yields compatible well-posedness of $(R)_{\mE_s}^\Lop$ and $(N)_{\mE_s}^\Lop$.
\subsection{The triangular cases (ii) and (iv)}

We adopt notation from the previously discussed block case. It suffices to prove well-posedness at regularity $s=0$: Indeed, in view of compatible well-posedness for block matrices, we can perturb along $\mathbb{A}:= \{\tau \id + (1-\tau)A : 0 \leq \tau \leq 1\}$ to the heat equation and obtain \emph{compatible} well-posedness at regularity $s=0$ from Proposition~\ref{prop:continuitymethod}. Then the intermediate cases follow from interpolation with (compatible) well-posedness in the energy class by means of Theorem~\ref{thm:inter}.

Appealing to Theorem~\ref{thm:six invertibilities}, we shall check invertibility of $s_{\pe r}(\P \M)$ and $s_{r \pe}(\P \M)$ for $A$ lower- and upper-triangular, respectively. The absence of one of the entries of $A$ will allow us to transfer invertibility from regularity $s=-1/2$ to $s=0$. In the context of elliptic systems this strategy has been put into place in \cite{AMM}. The argument presented below constitutes a simplification even in the elliptic case. In the following we utilize the block coefficients $A_b := \begin{bmatrix} A_{\pe\pe} & 0 \\ 0 & A_{\pa\pa} \end{bmatrix}$ and the corresponding multiplication operator $\M_b$. In the case of triangular matrices $A$ we will see that the Sobolev spaces adapted to $\P \M_b$ carry substantial information also on $\P \M$. Note that
\begin{align*}
 \P \M_b = \begin{bmatrix} 0 & \divx A_{\pa\pa}  & -\dhalf \\ -\gradx A_{\pe\pe}^{-1} & 0 & 0 \\ -\HT \dhalf A_{\pe\pe}^{-1} & 0 & 0 \end{bmatrix}
\end{align*}
acts independently on $\pe$- and $r$-components and hence so does $(\P \M_b)^2$. So, by the very definition of homogeneous Sobolev spaces we can decompose
\begin{align*}
 \Hdot_{\P \M_b}^s = (\Hdot_{\P \M_b}^s)_\pe \oplus (\Hdot_{\P \M_b}^s)_r
\end{align*}
for any $s \in \R$, see Section~\ref{sec:sobolev}, and besides the usual isomorphy $\P \M_b: \Hdot_{\P \M_b}^s \to \Hdot_{\P \M_b}^{s-1}$ we have the restricted isomorphisms
\begin{align}
\label{eq1:tri}
\P \M_b: (\Hdot_{\P \M_b}^s)_\pe \oplus \{0\} \stackrel{\cong}{\longrightarrow} \{0\} \oplus (\Hdot_{\P \M_b}^{s-1})_r
\end{align}
and
\begin{align}
\label{eq2:tri}
\P \M_b: \{0\} \oplus (\Hdot_{\P \M_b}^s)_r \stackrel{\cong}{\longrightarrow} (\Hdot_{\P \M_b}^{s-1})_\pe \oplus \{0\}.
\end{align}
The fact that we use these operators associated with the coefficients $A_b$ different from $A$ reflects that well-posedness, in contrast to \emph{a priori} representations, cannot be comprised only by the functional calculus arising from $A$. Below, we shall frequently and without further mentioning use Lemma~\ref{lem:identification} to the effect that $\Hdot_{\P \M_b}^s = \Hdot_\P^s = \Hdot_{\P \M}^s$ holds for $-1 \leq s \leq 0$.

\begin{proof}[Proof of (ii)]

As we have seen, it suffices to prove that $s_{\pe r}(\P \M): (\Hdot_{\P}^0)_r \to (\Hdot_{\P}^0)_\pe$ is invertible if $A$ is lower-triangular. Taking into account \eqref{eq1:tri}, this is actually equivalent to
\begin{align*}
 T_r(\P \M): f \mapsto \N_r \P \M_b \begin{bmatrix} s_{\pe r}(\P\M)f \\ 0 \end{bmatrix}
\end{align*}
providing an isomorphism $(\Hdot_{\P \M_b}^0)_r \to (\Hdot_{\P \M_b}^{-1})_r$ and it is this property that we shall check. From Corollary~\ref{cor:six invertibilities energy} and \eqref{eq1:tri} we can infer
\begin{align}
\label{eq3:tri}
 T_r(\P \M): (\Hdot_{\P \M_b}^{-1/2})_r \stackrel{\cong}{\longrightarrow} (\Hdot_{\P \M_b}^{-3/2})_r.
\end{align}
Since $\sgn(\P \M) \in \Lop(\Hdot^{-1/2}_\P)$ depends Lipschitz continuously on $\M$, see the discussion before Theorem~\ref{thm:sta}, also $T_r(\P \M)$ depends Lipschitz continuously on $\M$ as a bounded operator between the Sobolev spaces above. Now, we use the lower-triangular structure of $M$ for the one and only time to  see that for $f \in (\Hdot_{\P \M_b}^{1/2})_r \cap (\Hdot_{\P \M_b}^{-1/2})_r$ we can also write
\begin{align}
\label{eq35:tri}
\begin{split}
 T_r(\P \M) f
&= N_r \P \M_b \begin{bmatrix} s_{\pe r}(\P\M)f \\ 0 \end{bmatrix}
= N_r \P \M \sgn(\P\M)\begin{bmatrix} 0 \\ f \end{bmatrix} \\
&= N_r \sgn(\P\M) \P \M \begin{bmatrix} 0 \\ f \end{bmatrix}
= s_{r \pe}(\P\M) N_\pe \P \M_b \begin{bmatrix} 0 \\ f \end{bmatrix}.
\end{split}
\end{align}
Hence, \eqref{eq2:tri} and Corollary~\ref{cor:six invertibilities energy} yield that $T_r(\P \M)$ extends by density to an isomorphism
\begin{align}
\label{eq4:tri}
 T_r(\P \M): (\Hdot_{\P \M_b}^{1/2})_r \stackrel{\cong}{\longrightarrow} (\Hdot_{\P \M_b}^{-1/2})_r
\end{align}
and by the same argument as before we have Lipschitz continuous dependence on $\M$.

Of course we want to interpolate between \eqref{eq3:tri} and \eqref{eq4:tri} but interpolation of invertibility requires that the two obtained inverses agree on their common domain of definition. For the moment, we can at least state for all $f \in (\Hdot_{\P \M_b}^{1/2})_r \cap (\Hdot_{\P \M_b}^{-1/2})_r$ the comparability
\begin{align*}
 \|T_r(\P \M)f\|_{(\Hdot_{\P \M_b}^{-1/2})_r} + \|T_r(\P \M)f\|_{(\Hdot_{\P \M_b}^{-3/2})_r}
 \sim \|f\|_{(\Hdot_{\P \M_b}^{1/2})_r} + \|f\|_{(\Hdot_{\P \M_b}^{-1/2})_r}.
\end{align*}
Hence,
\begin{align}
\label{eq:TrPM intersection}
 T_r(\P \M):  (\Hdot_{\P \M_b}^{1/2})_r \cap (\Hdot_{\P \M_b}^{-1/2})_r \to (\Hdot_{\P \M_b}^{-1/2})_r \cap (\Hdot_{\P \M_b}^{-3/2})_r
\end{align}
is bounded from above and below. Here, the intersection spaces carry their natural sum norms. Since the bounded maps in \eqref{eq3:tri} and \eqref{eq4:tri} depend Lipschitz continuously on $\M$, the same holds true for $T_r(\P \M)$ when considered a bounded map between these two intersection spaces. We claim that in the block case $\M = \M_b$ it is also onto. Indeed, since the diagonal entries of $\sgn(\P \M_b)$ vanish by \eqref{eq:sgnPM block}, we can improve the calculation in \eqref{eq35:tri} to the effect that
\begin{align*}
 \begin{bmatrix} 0 \\ T_r(\P \M_b) f \end{bmatrix}
 = \P \M_b \sgn(\P \M_b) \begin{bmatrix} 0 \\ f \end{bmatrix}.
\end{align*}
Now, for each $g \in (\Hdot_{\P \M_b}^{-1/2})_r \cap (\Hdot_{\P \M_b}^{-3/2})_r$ we can define $f \in (\Hdot_{\P \M_b}^{1/2})_r \cap (\Hdot_{\P \M_b}^{-1/2})_r$ via
\begin{align*}
 \begin{bmatrix} 0 \\ f \end{bmatrix}
 := (\P \M_b)^{-1} \sgn(\P \M_b) \begin{bmatrix} 0 \\ g \end{bmatrix}
\end{align*}
and obtain $T_r(\P \M_b) f = g$. By the method of continuity, perturbing $\M$ to $\M_b$ along $\M(\tau) = \tau \M_b + (1-\tau)\M$, we obtain that the map in \eqref{eq:TrPM intersection} is an isomorphism. In particular, this implies that the inverses obtained in \eqref{eq3:tri} and \eqref{eq4:tri} are compatible.

Eventually, real interpolation midway between \eqref{eq3:tri} and \eqref{eq4:tri} for $T_r(\P \M)$ and its unambiguously defined inverse yields the required isomorphy
\begin{align*}
 T_r(\P \M): (\Hdot_{\P \M_b}^0)_r \stackrel{\cong}{\longrightarrow} (\Hdot_{\P \M_b}^{-1})_r.
\end{align*}
The interpolation argument uses~\cite[Thm.~5.3]{Auscher-McIntosh-Nahmod} for the real interpolation of the full $\Hdot_{\P\M_b}^s$-spaces and~\cite[Thm.~6.4.2]{BL} for the transfer to their complemented subspaces $(\Hdot_{\P\M_b}^s)_r$.
\end{proof}

The proof for upper-triangular matrices follows along the same lines and we shall only give the main steps.

\begin{proof}[Proof of (iv)]First, the question of invertibility of $s_{r \pe}(\P \M): (\Hdot_{\P}^0)_\pe \to (\Hdot_{\P}^0)_r$ is equivalent to proving that
\begin{align*}
 T_\pe(\P \M): f \mapsto N_\pe \P \M_b \begin{bmatrix} 0 \\ s_{r \pe}(\P\M)f \end{bmatrix}
\end{align*}
acts as an isomorphism $(\Hdot_{\P\M_b}^0)_\pe \to (\Hdot_{\P\M_b}^{-1})_\pe$. Due to Corollary~\ref{cor:six invertibilities energy} and \eqref{eq2:tri} it certainly is invertible from $(\Hdot_{\P\M_b}^{-1/2})_\pe$ into $(\Hdot_{\P\M_b}^{-3/2})_\pe$. The fact that $\M$ is upper-triangular lets us discover the identity
\begin{align*}
 T_\pe(\P \M) f
&= N_\pe \P \M_b \begin{bmatrix} 0 \\ s_{r \pe}(\P\M)f \end{bmatrix}
= N_\pe \P \M \sgn(\P\M)\begin{bmatrix} f \\ 0 \end{bmatrix} \\
&= N_\pe \sgn(\P\M) \P \M \begin{bmatrix} f \\ 0 \end{bmatrix}
= s_{ r\pe }(\P\M) N_r \P \M_b \begin{bmatrix} f \\ 0 \end{bmatrix}
\end{align*}
for $f \in (\Hdot_{\P \M_b}^{1/2})_\pe \cap (\Hdot_{\P \M_b}^{-1/2})_\pe$. Therefore $T_\pe(\P \M)$ extends to an isomorphism from $(\Hdot_{\P \M_b}^{1/2})_\pe$ to $(\Hdot_{\P \M_b}^{-1/2})_\pe$ and we conclude by interpolation as in the proof of (ii).
\end{proof}
\subsection{The triangular cases (iii) and (v)}

This is a consequence of Theorem~\ref{thm:duality}: (iii) is equivalent to (ii) and (v) is equivalent to (iv).
\subsection{The Hermitian case (vi)}

From the proof presented below it will become clear that we do neither use real coefficients nor equations but self-adjointness. It will also become clear at what instance we have to assume that coefficients do not depend on the time variable. The argument follows that of \cite{N2} using some integral identities to obtain Rellich estimates.

Let $h\in \Hp(\P\M)$. Then $F=F(\lambda,\cdot)= \e^{-\lambda \P\M}h$ satisfies $\pd_\lambda F=-\P\M F$ in the strong sense with strong limits $\lim_{\lambda\rightarrow\infty} \partial_\lambda^k F=0$, $k \ge 0$, and $\lim_{\lambda\rightarrow 0}F=h$ in $\mH = \L^2(\ree; \IC^{n+2})$. Let $\N$ be the reflection matrix introduced in \eqref{eq:N}. A computation shows that
\begin{align*}
 \P^* \N +\N \P= {\begin{bmatrix} 0 & 0& (\id+\HT){\dhalf} \\ 0 & 0& 0 \\ (\id-\HT){\dhalf}& 0& 0 \end{bmatrix}}.
\end{align*}
Next, by definition of $\M$ in \eqref{eq:DB}, the Hermitian condition $A^*=A$ is equivalent to $\M^*\N=\N \M$. Using the Hermitian inner product $(\cdot \,, \cdot )$ on $\mH$, we obtain
\begin{align*}
 \pd_\lambda (\N F, \M F)
 &= (\N\P\M F, \M F)+ (\N F, \M\P\M F)
 = (\M F, \P^*\N\M F)+ (\N \M F, \P\M F) \\
 &= (\M F , (\P^*\N+\N\P)\M F ) = 2 \Re ((\id+\HT)F_{\theta}, \dhalf (\M F)_\pe).
\end{align*}
By integration in $\lambda$ and the limits above,
\begin{equation*}
(\N h, \M h) = -2 \int_{0}^\infty \Re ((\id+\HT)F_{\theta}, \dhalf (\M F)_\pe)\d \lambda.
\end{equation*}
Using $\dhalf (\M F)_{\pe}= -\HT \pd_{\lambda}F_{\theta}$, and integration by parts with vanishing boundary terms,
\begin{align*}
\int_{0}^\infty \|\donefour (\M F)_{\pe}\|^2_{2} \d \lambda&=
-2\Re \int_{0}^\infty \lambda ( \donefour \pd_{\lambda}(\M F)_{\pe},\donefour (\M F)_{\pe}) \d \lambda \\
&
= 2 \Re
 \int_{0}^\infty \lambda ( \pd_{\lambda}(\M F)_{\pe},\HT \pd_{\lambda}F_{\theta}) \d \lambda
 \\
& \le 2 \left(\int_{0}^\infty \|\lambda \HT \pd_{\lambda} F_{\theta}\|^2_{2} \, \frac{\mathrm{d} \lambda}{\lambda}\right)^{1/2} \left(\int_{0}^\infty \|{( \lambda \M \pd_{\lambda} F)_{\pe}}\|^2_{2} \, \frac{\mathrm{d} \lambda}{\lambda}\right)^{1/2}
 \\
 & \le C \|h\|_{2}^2,
\end{align*}
where the last step follows from the boundedness of $\HT$ and $\M$ on $\mH$ and the square function estimates {for $\P \M$, see Theorem~\ref{thm:bhfc}.}

Thus, by the Cauchy-Schwarz inequality and the boundedness of $\HT$,
\begin{equation}
\label{eq:1011}
|(\N h, \M h)| \lesssim \|h\|_{2} \left(\int_{0}^\infty \|\donefour F_{\theta}\|^2_{2} {\d \lambda}\right)^{1/2}.
\end{equation}
The remaining integral on the right is the difficult term. In order to justify some calculations, we use the following lemma which we shall prove in Section~\ref{sec:lemsmooth} below.

\begin{lem}
\label{lem:smooth}
Assume that the coefficients of $\Lop$ are Lipschitz continuous with respect to $(x,t)$ and that $u$ is a reinforced weak solution of \eqref{eq1} with $\pcg u \in \cH_{\loc}$.  Then $\nabla_{\lambda,x}^2u$, $\pd_{t}u$ and $ \nabla_x \pd_{t}u$ are in  $\Lloc^2(\R_{+};\L^2(\R^{n+1}))$ and for any $r>0$,
\begin{equation}
\label{est}
\int_{r}^{2r} \|\nabla_{\lambda,x}^2u\|_{2}^2+   \|\pd_{t}u\|_{2}^2+ \|\nabla_{x}\pd_{t}u\|_{2}^2 \d \lambda \leq C
\int_{r/2}^{4r} \|\nabla_{\lambda,x}u\|_{2}^2\d \lambda
\end{equation}
{with $C$ depending on $r$, the ellipticity constants of $A$ and $\|\nabla A\|_{\infty}$}. In particular, a similar estimate  holds for $\nabla_{x}\HT\dhalf u$.
\end{lem}

Assuming the hypothesis of  Lemma~\ref{lem:smooth} and using the  correspondence of $F$ with a reinforced weak solution $u$, we observe that $F_{\theta}= \HT \dhalf u$ and $\dhalf F_\te = \pd_{t}u$. It will be simpler to come back to the usage of $u$ in the following calculation. Note that $\div_{\lambda,x} A \nabla_{\lambda,x} u = \pd_{t}u\in  \Lloc^2(\R_{+}; \L^2(\R^{n+1}))$ and that
\begin{align*}
 \div_{\lambda,x} A \nabla_{\lambda,x} u = \pd_{\lambda}(A\nabla_{\lambda,x}u)_{\pe}+ \divx (A \nabla_{\lambda,x}u)_{\pa}.
\end{align*}
{Since $F$ is the semigroup extension of $h$}, we know that $(A\nabla_{\lambda,x}u)_{\pe}= F_{\pe}$ is smooth in $\lambda$ as a function valued in $\mH$. Therefore we have $\divx (A\nabla_{\lambda,x}u)_{\pa} \in \Lloc^2(\R_{+};\L^2(\R^{n+1}))$ and thanks to the lemma also $\nabla_{x}F_{\theta}=  \nabla_{x}\HT \dhalf u$ is contained in this space. This allows us to obtain for almost every $\lambda > 0$ the following identities:
\begin{align*}
 \|\donefour F_{\theta}\|_{2}^2&= (\dhalf F_{\theta},F_{\theta}) \\
 &= (\pd_{t}u, \HT \dhalf u)  \\
 &= (\pd_{\lambda}(A\nabla_{\lambda,x}u)_{\pe}+ (\divx (A \nabla_{\lambda,x}u)_{\pa}, \HT \dhalf u)\\&
 = (\pd_{\lambda}(A\nabla_{\lambda,x}u)_{\pe}, \HT \dhalf u) - ((A\nabla_{\lambda,x}u)_{\pa}, \gradx \HT \dhalf u)\\
 &=  \frac{\mathrm{d}}{\mathrm{d}\lambda}((A\nabla_{\lambda,x}u)_{\pe}, \HT \dhalf u))\\
 &\qquad - ((A\nabla_{\lambda,x}u)_{\pe}, \pd_{\lambda} \HT \dhalf u)) -((A\nabla_{\lambda,x}u)_{\pa}, \gradx \HT \dhalf u)\\
 &= - \frac{\mathrm{d}}{\mathrm{d}\lambda}(F_{\pe}, F_{\theta}) -  (A {\gradlamx} u, \HT\dhalf {\gradlamx} u) \\
 &=- \frac{\mathrm{d}}{\mathrm{d}\lambda}(F_{\pe}, F_{\theta}) - \frac{1}{2} ({\gradlamx} u , [A, \HT\dhalf] {\gradlamx} u),
\end{align*}
where we used the  self-adjointness of $A$ and the skew-adjointness of $\HT\dhalf$ in the last line and  $[A, \HT\dhalf]$ denotes the commutator between the multiplication by $A$ and $\HT\dhalf$. Integrating the equality above and recalling the limits for $F$, we obtain
\begin{align}
\label{eq:donefour}
 \int_{0}^\infty \|\donefour F_{\theta}\|^2_{2} {\d \lambda}&= -(h_{\pe}, h_{\theta})
- \int_{0}^\infty ({\gradlamx} u , [A, \HT\dhalf] {\gradlamx} u){\d \lambda},
\end{align}
in the sense that the left hand integral  exists provided the right hand integral exists.

By Murray's theorem \cite{Murray}, the commutator between a bounded function $a$ on the real line and $\dhalf$  is bounded on $\L^2(\R)$ if and only if  $\dhalf a \in  \mathrm{BMO}(\R)$ with norm on the order of $\|\dhalf a\|_{\mathrm{BMO}}$.  This is the same upon replacing $\dhalf$ with $\HT\dhalf$. Actually, it can nowadays be seen as an easy corollary of the T(1) theorem.  Thus, on assuming $\|\dhalf A\|_{\L^\infty(\mathrm{BMO})}<\infty$ and recalling $\int_{0}^\infty \|F\|_{2}^2 \d \lambda\sim  \|h\|_{\Hdot^{-1/2}_{\P}}^2$ from Theorem~\ref{thm:sob} if additionally $h \in \Hdot_{\P}^{-1/2}$, we would obtain
\begin{equation}
\label{eq:donefour2}
\int_{0}^\infty \|\donefour F_{\theta}\|^2_{2} {\d \lambda} \lesssim {|(h_{\pe}, h_{\theta})|} + \|\dhalf A\|_{\L^\infty(\mathrm{BMO})}   \|h\|_{\Hdot^{-1/2}_{\P}}^2.
\end{equation}
Here, we wrote $\L^\infty(\mathrm{BMO}):=\L^\infty(\R^{n+1}; \mathrm{BMO}(\R))$. However, there is no chance in our situation that the integral involving the commutator converges just assuming $h\in \Hp(\P\M)$. It is at this point that we have to assume that $A$ does not depend on the $t$ variable. In this case $[A, \HT\dhalf]=0$ and we obtain from \eqref{eq:1011} and \eqref{eq:donefour} the estimate
\begin{align*}
|(\N h, \M h)|
\lesssim \|h\|_{2}\|h_{\pe}\|_{2}^{1/2}\|h_{\theta}\|_{2}^{1/2}.
\end{align*}
By definition of $N$ in \eqref{eq:N} we have
\begin{align*}
 |(h, \M h) - 2(h_\pe, (\M h)_\pe)| = |(\N h, \M h)| = |(h, \M h) - 2(h_\pa, (\M h)_\pa) - 2 (h_\te, h_\te)|.
\end{align*}
Taking into account the accretivity condition $\|h\|_{2}^2 \lesssim \Re (h,\M h)$, we deduce the Rellich estimate
\begin{align*}
\|h_{\pe}\|_{2}^2\sim \|h\|_2^2 \sim \|h_{\pa}\|_{2}^2+ \|h_{\theta}\|_{2}^2= \|h_{r}\|_{2}^2.
\end{align*}
This proves that $\N_{\pe}:\Hp(\P\M)\to (\H_{\P})_{\pe}$ and  $\N_{r}:\Hp(\P\M)\to (\H_{\P})_{r}$ have lower bounds. Up to now we have made the qualitative assumption that $A$ is Lipschitz continuous in $(x,t)$. However, since the Lipschitz norm of $A$ does not enter the estimate above, we may remove this extra assumption by an approximation argument that we shall present in the proof of Proposition~\ref{prop:Dir} below. For the time being we admit that this is possible.

Thus, we can apply Proposition~\ref{prop:continuitymethod} with $\mathbb{A} =\{\tau \id + (1-\tau) A : 0 \leq \tau \leq 1 \}$ to obtain compatible well-posedness of $(R)_{\mE_0}^\Lop$ and $(N)_{\mE_0}^\Lop$ from the analogous result for the heat equation in Section~\ref{sec:WPblock}. By reversing time, the same result holds for the backward equation with coefficients $A(x)$ and duality (Theorem~\ref{thm:duality}) along with $A=A^*$ yields compatible well-posedness of $(R)_{\mE_{-1}}^\Lop$ and $(N)_{\mE_{-1}}^\Lop$ for the forward equation. Finally, by Theorem~\ref{thm:inter} and (compatible) well-posedness in the energy classes for $s=-1/2$, we obtain the intermediate cases.
\subsection{The case of constant coefficients (vii)}

In this case, the maps $N_{\perp}$ and $N_{r}$ are  Fourier multiplier operators in the context of the Mihlin multiplier theorem, and their invertibility is done explicitly in \cite{N2}, when $s=0$. The other cases follow by reversing time, duality and interpolation just as in the Hermitian case above.
\subsection{Proof of Proposition~\ref{prop:Dir}}

We recall from \eqref{eq:DirToNeu} that for $u$ an energy solution to \eqref{eq1} with bounded and measurable coefficients and $h$ the boundary data of $\pcg u$, there is comparability
\begin{align}
\label{eq:DirToNeuh}
\|h_{r}\|_{\Hdot^{-1/4}_{{\pd_{t}-\Delta_{x}}}} \sim \|h\|_{\Hdot^{-1/2}_{\P}} \sim \|h_{\pe}\|_{\Hdot^{-1/4}_{{\pd_{t}-\Delta_{x}}}}.
\end{align}
{Here, $h_{r} = [\nabla_x u, \HT \dhalf u]|_{\lambda = 0}$ belongs to $(\Hdot^{-1/4}_{{\pd_{t}-\Delta_{x}}})^{n+1}$ and we forget the power of $n+1$ by abuse of notation. Moreover, implicit constants depend only on ellipticity of $A$ and dimension since these estimates were obtained from the Lax-Milgram lemma. Given $h\in \Hp(\P\M)\cap \Hdot^{-1/2, +}_{\P\M}$, we can follow the proof of (iv) and obtain from \eqref{eq:1011} and \eqref{eq:donefour2} that
\begin{align*}
|(\N h, \M h)| \lesssim \|h\|_{2} \Big(|(h_{\pe}, h_{\theta})|^{1/2} + \|h\|_{\Hdot^{-1/2}_{\P}}\Big).
\end{align*}
Taking into account \eqref{eq:DirToNeu}, the same reasoning as before reveals
\begin{align}
\label{eq:sim}
\|h_{r}\|_{2}^2+\|h_{r}\|_{\Hdot^{-1/4}_{{\pd_{t}-\Delta_{x}}}}^2 \sim {\|h\|_2^2 + \|h\|_{\Hdot^{-1/4}_{{\pd_{t}-\Delta_{x}}}}^2 \sim} \|h_{\pe}\|_{2}^2+ \|h_{\pe}\|_{\Hdot^{-1/4}_{{\pd_{t}-\Delta_{x}}}}^2,
\end{align}
where implicit constants depend only on ellipticity of $A$, dimension and $\|\dhalf A\|_{\L^\infty(\mathrm{BMO})}$. So far, this estimate holds under the \emph{a priori} assumption that $A$ is Lipschitz continuous in $(x,t)$. {Let us now remove this additional regularity and assume $\dhalf A \in \L^\infty(\mathrm{BMO})$ only.}

By convolution with smooth {positive} kernels we can  find a sequence of Lipschitz matrices $A_{j}$ with uniform ellipticity bounds that converge towards $A$  almost everywhere.  Since $\|\dhalf A_{j}\|_{\L^\infty(\mathrm{BMO})} \le \|\dhalf A\|_{\L^\infty(\mathrm{BMO})}$ holds due to the translation invariance of the  $\L^\infty(\mathrm{BMO})$-norm, the bounds \eqref{eq:sim} are uniform in $j$ for $h_{j}\in \Hp(\P\M_{j})\cap \Hdot^{-1/2, +}_{\P\M_{j}}$, where $M_{j}$ is the matrix corresponding to $A_{j}$.

We want to obtain the same equivalence for all $h\in \Hp(\P\M)\cap \Hdot^{-1/2, +}_{\P\M}$. To this end we shall use in a crucial way well-posedness of the BVPs in the energy class, see Section~\ref{sec:energy}. Consider for example the Neumann problem. Let $f \in \L^2(\R^{n+1}) \cap \Hdot^{-1/4}_{{\pd_{t}-\Delta_{x}}}$ and let $u$ and $u_{j}$ be the unique energy solutions to \eqref{eq1} with coefficients $A$ and $A_{j}$ and conormal derivative equal to $f$, respectively. Theorem~\ref{thm:sob} creates two vectors $h, h_j \in \Hdot^{-1/2, +}_\P$ with $h_{\pe}=(h_{j})_{\pe}=f$. Due to \eqref{eq:DirToNeuh} the $(h_j)_r$ are uniformly bounded in $\Hdot^{-1/4}_{\pd_{t}-\Delta_x}$. The hidden coercivity estimate \eqref{eq:coercivity energy} implies that the $u_{j}$ are uniformly bounded in the energy space $\dot\E(\R^{n+2}_{+})$, which means that the $D_{A_j} u_j$ are uniformly bounded in $\L^2(\R^{n+2}_{+}; \IC^{n+2})$. Without loss of generality we may assume weak convergence. Passing to the limit in the variational formulation
\begin{align*}
 \iiint_{\R^{n+2}_+} A_j \gradlamx u_j \cdot \cl{\gradlamx v} + \HT \dhalf u_j \cdot \cl{\dhalf v} \d \lambda \d x \d t = - \langle f, v|_{\lambda = 0} \rangle \qquad (v \in \dot{\E}),
\end{align*}
we discover that the weak limit solves the equation with  coefficients $A$ and by well-posedness in the energy class we deduce $D_{A_j} u_j \to \pcg u$ weakly in $\L^2(\R^{n+2}_{+}; \IC^{n+2})$. Now, Lemma~\ref{lem:trace energy} yields $u_j|_{\lambda = 0} \to u|_{\lambda = 0}$ weakly in $\Hdot^{1/4}_{\pd_{t} - \Delta_x}$ and hence $(h_j)_r \to h_r$
weakly in $\Hdot^{-1/4}_{\pd_{t} - \Delta_x}$. Thus, we can pass to the limit inferior in \eqref{eq:sim} for $h_j$ and get the one-sided inequality
\begin{align*}
 \|h_{r}\|_{2}^2+\|h_{r}\|_{\Hdot^{-1/4}_{{\pd_{t}-\Delta_{x}}}}^2 \lesssim \|h_{\pe}\|_{2}^2+ \|h_{\pe}\|_{\Hdot^{-1/4}_{{\pd_{t}-\Delta_{x}}}}^2.
\end{align*}
For the reverse inequality we argue similarly using the energy solutions with fixed regularity data.

As usual, \eqref{eq:sim} and the method of continuity (see the proof of Proposition~\ref{prop:continuitymethod} below for the specific application) allows us to conclude that the regularity problem with data $u|_{\lambda = 0} \in \Hdot^{1/2}_{{\pd_{t}-\Delta_{x}}}~\cap~\Hdot^{1/4}_{{\pd_{t}-\Delta_{x}}}$ and the Neumann problem with data $\dnuA u|_{\lambda = 0} \in \Hdot^{0}_{{\pd_{t}-\Delta_{x}}}~\cap~\Hdot^{-1/4}_{{\pd_{t}-\Delta_{x}}}$ are well-posed in the class of solutions with $\pcg u \in \mE_{0}\cap \mE_{-1/2}$. Indeed, by Theorems~\ref{thm:uniq} and \ref{thm:sob} this amounts to showing that $N_\bullet: \Hp(\P\M)\cap \Hdot^{-1/2, +}_{\P\M} \to (\cl{\ran(\P)} \cap \Hdot_\P^{-1/2})_\bullet$ is an isomorphism, where $\bullet$ designates either $r$ or $\pe$. Now, \eqref{eq:sim} provides lower bounds for this map and if $A$ is in block form, then  Proposition~\ref{prop:cwp} and Theorem~\ref{thm:WP}(i) yield invertibility. This yields (i) and the first claim in (iii).

We perform a duality argument to conclude the remaining assertions. First, the same analysis applies to the regularity problem for the backward in time equation with coefficients $A^*(x,t)$ since the hypothesis $\dhalf A\in \L^\infty( \mathrm{BMO})$  is preserved in changing $t$ to $-t$ and in taking adjoints. By the argument of Theorem~\ref{thm:duality}, each $f \in \L^2(\R^{n+1})+ \Hdot^{1/4}_{{\pd_{t}-\Delta_{x}}}$ corresponds to a unique $h\in \Hdot^{-1,+}_{\P\M}+ \Hdot^{-1/2,+}_{\P\M}$ (which is the dual of  $ \Hdot^{0,+}_{\P^*\wM}\cap \Hdot^{-1/2,+}_{\P^*\wM}$ in that duality) with $h_{r}= [\nabla_{x }f, \HT\dhalf f ]$. Then $F=e^{-\lambda \P\M}h$ is the parabolic conormal differential of a reinforced weak solution $u$ with $u|_{\lambda=0}=f$. By construction $u$ is the sum of two solutions $u_1 + u_2$ with $\pcg u_1 \in \mE_{-1}$ and $\pcg u_2 \in \mE_{-1/2}$ and owing to Theorem~\ref{thm:sob} it is unique under all weak solutions with this property. This proves (ii). The corresponding statement on the Neumann problem is obtained in the same way.
\subsection{Proof of Lemma \ref{lem:smooth}}
\label{sec:lemsmooth}

This is a consequence of Caccioppoli-type arguments for weak solutions. Assume first that $u$ is a weak solution of $\partial_t u -\div_{\lambda,x} A(x,t)\nabla_{\lambda,x} u =  0$. We let $\LQI$ be a parabolic cylinder in $\R^{n+2}_{+}$ of size $r$ (not necessary of Whitney type) and $\LQIt$ denote an enlargement contained in $\R^{n+2}_{+}$.

Since $A$ is Lipschitz continuous with respect to $X=(\lambda, x)$, Caccioppoli's inequality (Lemma~\ref{lem:Caccioppoli}) and the classical method of difference quotients imply that for $i=0,\ldots,n$ the function $v=\pd_{X_i}u$ is a weak solution to $\partial_t v -\div_{\lambda,x} A\nabla_{\lambda,x} v = \div_{\lambda,x} (\pd_{X_i} A)\nabla_{\lambda,x}u$. This being said, we can sum up Caccioppoli's inequality for all these solutions and obtain
\begin{align*}
 \iiint_{\LQI} |\gradlamx^2 u|^2 \lesssim \frac{1}{r^2} \iiint_{\LQIt} |\nabla_{\lambda,x}u|^2 + \|\gradlamx A\|_\infty^2 \iiint_{\LQIt} |\nabla_{\lambda,x}u|^2.
\end{align*}
Since now $u$ has second order derivatives in $X$ within $\Lloc^2(\R^{n+2}_+)$, we can interpret the equation $\pd_{t} u = \div_{\lambda,x} A \gradlamx u$ in $\Lloc^2(\R^{n+2}_+)$. It follows that $\iiint_{\LQI} |\pd_{t} u|^2$ can be bounded in the same way as $\iiint_{\LQI} |\gradlamx^2 u|^2$ above. Eventually, we use that $A$ is also Lipschitz continuous with respect to $t$ to conclude likewise that $v=\pd_{t}u$ is a weak solution to $\partial_t v -\div_{\lambda,x} A\nabla_{\lambda,x} v = \div_{\lambda,x} (\pd_{t} A)\nabla_{\lambda,x}u$. Caccioppoli's inequality for this solution $v$ reads
\begin{align*}
 \iiint_{\LQI }|\nabla_{\lambda,x}\pd_{t}u|^2  \lesssim \frac{1}{r^2} \iiint_{\LQIt}|\pd_{t}u|^2  + \|\pd_{t}A\|_{\infty}^2 \iiint_{\LQIt} |\nabla_{\lambda,x}u|^2.
\end{align*}
Adjusting the enlargements, we have obtained
\begin{align*}
 \iiint_{\LQI }|\pd_{t} u|^2 + |\nabla_{\lambda,x}^2u|^2+ |\nabla_{\lambda,x}\pd_{t}u|^2  \lesssim   \iiint_{\LQIt} |\nabla_{\lambda,x}u|^2,
\end{align*}
with an implicit constant depending on dimension, ellipticity, $r$ and $\|\gradlamx A\|_{\infty} + \|\pd_{t}A\|_\infty$.

If in addition $u$ is a reinforced weak solution with $\pcg u \in \cH_{\loc}$, then we can sum these estimates in $x$ and $t$. More precisely, fix $\Lambda=(r,2r)$, $\tilde \Lambda=(r/2,4r)$ and choose translates $Q_{k}\times I_{k}$ of $Q\times  I$ forming a partition of $\R^{n+1}$. {Then the cubes $\tilde Q_{k} \times \tilde I_{k}$ have finite overlap} and summing up, we obtain \eqref{est}.  By interpolation in time between, using for example Plancherel's theorem, the estimates for $\nabla_{\lambda,x}u$ and $\pd_{t}\nabla_{\lambda,x}u$, we obtain the one for $\dhalf \nabla_{\lambda,x} u$ and also for $\HT\dhalf \nabla_{\lambda,x} u$ as the Hilbert transform is isometric.
\subsection{Proof of Proposition~\ref{prop:continuitymethod}}
\label{sec:continuitymethod proof}

Given $A \in \mathbb{A}$, there is a continuous path $\gamma: [0,1] \to \mathbb{A}$ such that $\gamma(0) = A_0$ and $\gamma(1) = A$. We put $\gamma(\tau) := A_\tau$, let $M_\tau$ be the corresponding matrix defined in \eqref{eq:DB} and $\chi^+_\tau := \chi^+(\P \M_\tau)$ the associated spectral projection. With this notation, the assumption is that all maps $N_\bullet:  \chi^+_\tau \Hdot_\P^s \to  (\Hdot_\P^s)_\bullet$ have lower bounds and that invertibility holds for $\tau = 0$. In the following, we always consider subspaces of $\Hdot_\P^s$ as Hilbert spaces with the inherited norm. Due to Theorem~\ref{thm:wpequiv} and Proposition~\ref{prop:cwp}, the claim amounts to proving invertibility for all $\tau$ and that the inverses are compatible with those at regularity $s=-1/2$ provided this is true for $\tau = 0$. Let us recall from the discussion before Theorem~\ref{thm:sta} that $\chi^+_\tau: \Hdot_\P^s \to \chi^+_\tau \Hdot_\P^s$ is bounded and depends continuously on $\tau \in [0,1]$.

First, we claim there exists $\delta>0$ such that $\chi^+_\tau: \chi^+_\sigma \Hdot_\P^s \to \chi^+_\tau \Hdot_\P^s$ is invertible whenever $|\sigma - \tau| \leq \delta$. Indeed, choosing $\delta$ sufficiently small according to uniform continuity of $\tau \mapsto \chi^+_\tau$, we can define bounded operators $\chi^+_\tau \Hdot_\P^s \to \chi^+_\sigma \Hdot_\P^s$ by $(\id - \chi^+_\sigma(\chi^+_\sigma - \chi^+_\tau))^{-1}\chi^+_\sigma$ and $\chi^+_\sigma(\id - \chi^+_\tau(\chi^+_\tau - \chi^+_\sigma))^{-1}$, respectively, via a convergent Neumann series. Using the relations
\begin{align*}
 (\id - \chi^+_\sigma (\chi^+_\sigma - \chi^+_\tau))\chi^+_\sigma = \chi^+_\sigma \chi^+_\tau \chi^+_\sigma, \qquad \chi^+_\tau \chi^+_\sigma = \chi^+_\tau (\id - \chi^+_\tau(\chi^+_\tau - \chi^+_\sigma)),
\end{align*}
these turn out to be left- and right inverses for $\chi^+_\tau: \chi^+_\sigma \Hdot_\P^s \to \chi^+_\tau \Hdot_\P^s$.

Now, we start with $\sigma = 0$ and obtain for $|\tau| \leq \delta$ that $N_\bullet: \chi^+_\tau \Hdot_\P^s \to (\Hdot_\P^s)_\bullet$ is invertible if and only if the composite map $N_\bullet \chi_\tau^+: \chi_0^+ \Hdot_\P^s  \to (\Hdot_\P^s)_\bullet$ between \emph{fixed} spaces is invertible. The latter has lower bounds for any $|\tau| \leq \delta$ since this holds for $N_\bullet: \chi^+_\tau \Hdot_\P^s \to (\Hdot_\P^s)_\bullet$; and it is invertible for $\tau=0$ again by assumption and since $\chi_0^+$ is a projection. Thus, the method of continuity yields invertibility for all $|\tau| \leq \delta$. Iterating this process finitely many times, we obtain invertibility for every $\tau \in [0,1]$.

Finally, we turn to compatible well-posedness. From Theorem~\ref{thm:wpequiv} and well-posedness of the BVPs in the energy class, we can infer for every $\tau$ that $\N_\bullet: \chi_\tau^+ \H_\P^{-1/2} \to (\H_\P^{-1/2})_\bullet$ is invertible. Together with the assumption, this entails lower bounds for $\N_\bullet: \chi_\tau^+(\H_\P^{s} \cap \H_\P^{-1/2}) \to (\H_\P^{s} \cap \H_\P^{-1/2})_\bullet$, the intersections being Hilbert spaces for the sum of the two norms. Now, the crucial observation is that invertibility between these intersection spaces is equivalent to compatibility of the two inverses on $(\Hdot_\P^s)_\bullet$ and $(\Hdot_\P^{-1/2})_\bullet$ since their intersection is dense in both of them. This being said, we can simply repeat the argument above, replacing systematically $\Hdot_\P^s$ by $\Hdot_\P^s \cap \Hdot_\P^{-1/2}$, to deduce that the inverses of $N_\bullet$ are compatible for all $\tau$ provided they are for $\tau = 0$.
\section{Uniqueness for Dirichlet problems}
\label{sec:Dirichlet}

Here, we give the proofs of the uniqueness results on Dirichlet problems with non-tangential maximal function control stated in Theorems~\ref{thm:uniqueDir} and \ref{thm:WPDirTimeDependent}. For simplicity, we often fix Whitney parameters in the definition of $\NT$ in such a way that $\Lambda = (\lambda,2\lambda)$, $Q = B(x,\lambda)$, $I = (t-\lambda^2, t + \lambda^2)$ is a Whitney cylinder of sidelength $\lambda > 0$, where the center $(x,t) \in \R^{n+1}$ will be clear from the context. We shall freely use that different choices of Whitney parameters yield non-tangential maximal functions with comparable $\L^2$-norms, see Section~\ref{sec:NTmaxDef}. We also introduce a variant of the non-tangential maximal function
\begin{align*}
 \NTone ( u) (x,t):= \sup_{\lambda>0} \bariiint_{\Lambda \times Q \times I} |u(\mu,y,s) | \d \mu \d y \d s \qquad ((x,t) \in \ree),
\end{align*}
where the subscript is reminiscent of the fact that we are using $\L^1$-averages instead of $\L^2$-averages. Of course, $\NTone (u) \leq \NT (u)$  due to H\"older's inequality and if $\Lop u = 0$ in the weak sense, then $\NT (u)$ and $\NTone (u)$ are comparable in $\L^2$-norm thanks to the reverse H\"older estimate in Lemma~\ref{lem:RHu}.

In both theorems we have to deduce that a certain weak solution to $\Lop u = 0$ vanishes on $\R^{n+2}_+$. In order to bring into play the assumptions on the dual equations, we shall use test functions of the form $\phi = \Lop^*H$. Similar to Section~\ref{sec:energy} we use the homogeneous energy space $\dot{\E}(\R^{n+2})$  (considered modulo constants) to define
\begin{align*}
 \langle \Lop^* H, v \rangle := \iiint_{\R^{n+2}} A^*(x,t)\nabla_{\lambda,x} \phi \cdot\overline{\nabla_{\lambda,x} v} - \HT\dhalf \phi \cdot \overline{\dhalf v} \d \lambda \d x\d t \qquad (v \in \dot{\E}(\R^{n+2})),
\end{align*}
where $\langle \cdot\, , \cdot \rangle$ denotes the duality pairing on $\dot{\E}(\R^{n+2})$. Hidden coercivity reveals that $\Lop^*: \dot{\E}(\R^{n+2}) \to \dot{\E}(\R^{n+2})^*$ is an isomorphism and we say that $H$ is the energy solution to $\Lop^* H = \phi$ on $\R^{n+2}$. A similar discussion applies to energy solutions of the inhomogeneous equation for $\Lop$. We remark that we work on the whole space $\R^{n+2}$ to avoid the discussion of boundary values.
\subsection{Proof of Theorem~\ref{thm:uniqueDir}}

For the sake of clean arrangement we shall postpone proofs of technical lemmas used throughout the main argument until Section~\ref{sec:AuxiliaryDirichlet}.

Let us assume that $u$ is a weak solution to $\Lop u = 0$ in $\R^{n+2}_+$ that satisfies $\NT (u)  \in \L^2(\ree)$ and for almost every $(x,t) \in \ree$,
\begin{align}
\label{eq1:uniqueDir}
\lim_{\lambda \to 0} \bariiint_{W((x,t),\lambda)} |u| = 0,
\end{align}
where for the argument we prefer to use $W((x,t),\lambda) := (\lambda/2, 4\lambda) \times B(x,\lambda) \times (t-4\lambda^2, t+4\lambda^2)$ as the Whitney regions defining $\NT$. We need to show $u = 0$ almost everywhere.

Let $\phi \in \C_0^\infty(\R^{n+2}_+)$. For a reason that will become clear in the course of the proof, we test $u$ not against $\phi$ but against $\partial_\lambda^2 \phi$ and we shall show
\begin{align}
\label{eq2:uniqueDir}
 (u, \partial_\lambda^2 \phi) = 0,
\end{align}
where here and throughout $(\cdot \,,\cdot)$ denotes the inner product on $\L^2(\R^{n+2}_+)$. This suffices to conclude: Indeed, $\partial_\lambda^2 u = 0$ implies $u(\lambda,x,t) = f(x,t) + \lambda g(x,t)$ for some $f,g \in \Lloc^2(\ree)$ and from \eqref{eq1:uniqueDir} and Lebesgue's differentiation theorem we deduce $f = 0$ almost everywhere. As now $u(\lambda,x,t) = \lambda g(x,t)$, we obtain for any $\lambda > 0$ by means of Tonelli's theorem
\begin{align}
\label{eq3:uniqueDir}
 \iint_{\ree} |g(x,t)|^2 \d x \d t
= \frac{1}{\lambda^2} \iint_\ree \bigg(\bariiint_{W((x,t),\lambda)}|u|^2 \bigg) \d x \d t
\leq \iint_\ree \frac{1}{\lambda^2} \NT(u)(x,t)^2 \d x \d t.
\end{align}
Since $\NT(u) \in \L^2(\ree)$ by assumption, we discover $g = 0$ almost everywhere on letting $\lambda \to \infty$.

We turn to the proof of \eqref{eq2:uniqueDir}. Let $\eps \in (0,1)$ and $R \in (1,\infty)$ be two degrees of freedom that will be chosen sufficiently small and large, respectively. Let $\eta = \eta(\lambda)$ be smooth with $\eta(\lambda) = 0$ for $\lambda \leq 1$ and $\eta(\lambda) = 1$ for $\lambda \geq 2$ and put $\theta(\lambda) = \eta(\lambda/\eps) - \eta(\lambda/R)$. For $\eps$ and $R$ sufficiently small and large, respectively, we have $(u, \partial_\lambda^2 \phi) = (\theta u, \partial_\lambda^2 \phi)$ since $\phi$ has compact support in $\R^{n+2}_+$. Since $\partial_\lambda^2 \phi \in \dot{\E}(\R^{n+2})^*$, we can introduce the dual equation by writing $\partial_\lambda^2 \phi = \Lop^* v$ for some $v \in \dot{\E}(\R^{n+2})$. We have

\begin{lem}
\label{lem:theta u testfunction}
The functions $\theta u$ and $\theta v$ are test  functions for the reinforced weak formulation of parabolic equations, that is, they are in the closure of $\C_0^\infty(\R^{n+2}_+)$ for the norm $\|\gradlamx \cdot\|_2 + \|\dhalf \cdot \|_2$.
\end{lem}

Since $\theta$ does not depend on $t$, we obtain in particular $u \in \Hdot^{1/2}(\R; \Lloc^2(\reu))$, that is, $u$ is a reinforced weak solution, see also Corollary~\ref{cor:fractional integration by parts}. In this way we can justify the calculation
\begin{align*}
 (u, \partial_\lambda^2 \phi)
&= (\theta u, \Lop^* v)\\
&= (\gradlamx(\theta u), A^* \gradlamx v) + (\HT \dhalf (\theta u), \dhalf v) \\
&= (u \gradlamx \theta, A^* \gradlamx v) - (\gradlamx u, A^* \gradlamx \theta v) + (\gradlamx u, A^* \gradlamx(\theta v)) + (\HT \dhalf u, \dhalf(\theta v)) \\
&= (u \gradlamx \theta, A^* \gradlamx v) - (\gradlamx u, A^* \gradlamx \theta v),
\end{align*}
where in the third step we have used that $\theta = \theta(\lambda)$ is real valued and the final step follows from the reinforced formulation of the equation $\Lop u = 0$. At this stage of the proof we want to bring the existence hypothesis into play. So, we need further properties of $v$ stated in the following

\begin{lem}
\label{lem:energy solution compact right}
Let $\phi \in \C_0^\infty(\R^{n+2})$. The energy solution to $\Lop^* v = \partial_\lambda^2 \phi$ satisfies $v \in \C^\infty(\R; \L^2 \cap \Hdot^{1/2}_{\partial_t - \Delta_x})$. Moreover, $\NTone(\frac{v - v_0}{\lambda}) \in \L^2(\ree)$, where $v_0 := v|_{\lambda = 0}$.
\end{lem}

Since in particular $v_0 \in \Hdot^{1/2}_{\partial_t - \Delta_x}$, the assumption provides a reinforced weak solution to $\Lop^* w = 0$ that satisfies $\NT(\pcgb w) \in \L^2(\ree)$ and $w|_{\lambda = 0} = v_0$ in the sense of $\Hdot^{1/2}_{\partial_t - \Delta_x}$. The backward conormal differential $\pcgb$ has been introduced in Section~\ref{sec:backward}. Due to $v_0 \in \L^2(\ree)$ we can realise $w$ by an additive constant such that $w|_{\lambda = 0} = v_0$ holds also in $\L^2(\ree)$. Upon replacing $v$ with $w$, the same line of reasoning as above then reveals
\begin{align*}
0 = (\theta u, \Lop^* w) = (u \gradlamx \theta, A^* \gradlamx w) - (\gradlamx u, A^* \gradlamx \theta w),
\end{align*}
invoking the following lemma to justify calculations.

\begin{lem}
\label{lem:theta w testfunction}
The function $\theta w$ is in the closure of $\C_0^\infty(\R^{n+2}_+)$ for the norm $\|\gradlamx \cdot\|_2 + \|\dhalf \cdot \|_2$.
\end{lem}

Putting things together, we conclude
\begin{align*}
 (u,\partial_\lambda^2 \phi) = (u \gradlamx \theta, A^* \gradlamx (v-w)) - (\gradlamx u, A^* \gradlamx \theta (v-w)).
\end{align*}
Here, both $u$ and $v-w$ vanish at $\lambda = 0$ and in fact this convergence at the boundary was the reason to introduce $w$. On recalling our choice of $\theta$, we obtain the straightforward bound
\begin{align}
\label{eq7:uniqDir}
|(u,\partial_\lambda^2 \phi)| \lesssim \text{I}_\eps + \text{I}_R + \text{II}_\eps + \text{II}_R,
\end{align}
where
\begin{align*}
 \text{I}_\sigma := \barint_\sigma^{2 \sigma} \iint_\ree |u| |\gradlamx(v-w)|, \qquad II_\sigma := \barint_\sigma^{2 \sigma} \iint_\ree |\gradlamx u| |v-w|.
\end{align*}
We prepare the limiting arguments $\eps \to 0$ and $R \to \infty$. Define parabolic cylinders $\frac{1}{2} W((x,t),\lambda) := (\lambda,2\lambda) \times B(x,\lambda) \times (t-\lambda^2, t+ \lambda^2)$, of which $W((x,t),\lambda)$ are enlargements. By an averaging trick in $(x,t)$ and Tonelli's theorem we conclude
\begin{align}
\label{eq8:uniqDir}
\begin{split}
 \text{I}_\sigma
&= \barint_\sigma^{2 \sigma} \iint_\ree \bigg(\bariint_{B(x,\sigma) \times (t-\sigma^2, t+ \sigma^2)} |u| |\gradlamx(v-w)| \d y \d s\bigg) \d x \d t \d \mu \\
&\leq \iint_\ree \bigg(\bariiint_{\frac{1}{2}W((x,t),\sigma)} |u|^2 \d \mu \d y \d s \bigg)^{1/2} \bigg(\bariiint_{\frac{1}{2}W((x,t),\sigma)} |\gradlamx(v-w)|^2 \d \mu \d y \d s \bigg)^{1/2} \d x \d t \\
&\lesssim \iint_\ree \bigg(\bariiint_{W((x,t),\sigma)} |u| \d \mu \d y \d s \bigg) \bigg(\bariiint_{W((x,t),\sigma)} \Big|\frac{v-w}{\mu}\Big| \d \mu \d y \d s\bigg) \d x \d t \\
&\leq \Big\|\NTone(u)\Big\|_{\L^2(\ree)} \Big\|\NTone\Big(\frac{v-w}{\lambda}\Big)\Big\|_{\L^2(\ree)},
\end{split}
\end{align}
where in the second to last step we have used the reverse H\"older inequality (Lemma~\ref{lem:RHu}) on $u$ and Caccioppoli's inequality (Lemma~\ref{lem:Caccioppoli}) followed by the reverse H\"older inequality on $v-w$. The latter can be justified for $\sigma$ sufficiently small or large, in which case $\Lop^*(v-w) = \partial_\lambda^2 \phi = 0$ holds on a neighbourhood of $W((x,t),\sigma)$. By interchanging the roles of $u$ and $v-w$ we also get
\begin{align}
\label{eq8b:uniqDir}
\begin{split}
 \text{II}_\sigma
&\lesssim \iint_\ree \bigg(\bariiint_{W((x,t),\sigma)} \Big|\frac{u}{\mu}\Big| \d \mu \d y \d s \bigg) \bigg(\bariiint_{W((x,t),\sigma)} |v-w| \d \mu \d y \d s\bigg) \d x \d t \\
&\lesssim \Big\|\NTone(u)\Big\|_{\L^2(\ree)} \Big\|\NTone\Big(\frac{v-w}{\lambda}\Big)\Big\|_{\L^2(\ree)}.
\end{split}
\end{align}
In both estimates $\NTone(u) \in \L^2(\ree)$ holds by assumption and writing $v-w = (v-v_0) - (w-v_0)$, we obtain $\NTone(\frac{v-w}{\lambda}) \in \L^2(\ree)$ from Lemma~\ref{lem:energy solution compact right} and the subsequent result. This is a parabolic version of \cite[Thm.~3.1]{Ken-Pip} but in contrast to its elliptic counterpart, it is limited to solutions to the homogeneous equation as will become clear from the proof given in Section~\ref{sec:AuxiliaryDirichlet}.

\begin{lem}
\label{lem:Ken-Pip parabolic}
Let $w$ be a reinforced weak solution to $\Lop^* w = 0$ on $\R^{n+2}_+$ satisfying $\NT(\pcgb w) \in \L^2(\ree)$. Then there exists a trace $w_0 \in \Lloc^2(\ree)$ such that for almost every $(x,t) \in \ree$,
\begin{align*}
 \lim_{\lambda \to 0} \bariiint_{\LQI} |w-w_0(x,t)| \d \mu \d y \d s = 0.
\end{align*}
Furthermore, $\NTone(\frac{w-w_0}{\lambda}) \in \L^2(\ree)$ and $\gradx w_0 = (\pcgb w)_\pa|_{\lambda = 0}$. An analogous statement holds for reinforced weak solutions to $\Lop u = 0$ upon replacing $\pcgb$ by $\pcg$.
\end{lem}

The upshot is that we can compute limits in \eqref{eq8:uniqDir} and \eqref{eq8b:uniqDir} by means of Lebesgue's theorem. From \eqref{eq1:uniqueDir} we conclude $I_\eps \to 0$ and $II_\eps \to 0$ as $\eps \to 0$. In order to obtain $\text{I}_R \to 0$ and $\text{II}_R \to 0$ as $R \to \infty$, we use $\NTone(u) \in \L^2(\ree)$ and the generic lemma on non-tangential maximal functions:

\begin{lem}
\label{lem:big averages}
If $u \in \Lloc^1(\R^{n+2}_+)$ satisfies $\NTone(u) \in \L^2(\ree)$, then for every $(x,t) \in \ree$,
\begin{align*}
 \lim_{\lambda \to \infty} \bariiint_{\LQI} |u| = 0.
\end{align*}
\end{lem}

Going back to \eqref{eq7:uniqDir}, we see that \eqref{eq2:uniqueDir} holds and the proof of Theorem~\ref{thm:uniqueDir} is complete.
\subsection{Non-tangential estimates for energy solutions}
\label{sec:AuxiliaryDirichlet}

Here, we supply the auxiliary results used in the proof of Theorem~\ref{thm:uniqueDir}. As a general observation, but also to set a strategy of proof that will be re-used several times, we begin with an extension of Lemma~\ref{lem:trace energy} to the effect that the trace of the homogeneous energy space can also be understood in the sense of non-tangential convergence. This is a parabolic analogue of \cite[Thm.~6.42]{Amenta-Auscher}.
\begin{lem}
\label{lem:energy Whitney trace}
For each $v \in \dot{\E}(\R^{n+2}_+)$ the trace $v_0:= v|_{\lambda = 0} \in \Hdot^{1/4}_{\partial_t - \Delta_x}$ can be realised in $\Lloc^2(\ree)$ such that $\NTone(\frac{v-v_0}{\sqrt{\lambda}}) \in \L^2(\ree)$ and for almost every $(x,t) \in \ree$,
\begin{align*}
 \lim_{\lambda \to 0} \bariiint_\LQI |v-v_0(x,t)| \d \mu \d y \d s = 0.
\end{align*}
\end{lem}

\begin{proof}
Abbreviating $F := |\gradlamx v| + |\HT \dhalf v|$, we can infer from Poincar\'e's inequality as in the proof of Lemma~\ref{lem:rh spatial},
\begin{align}
\label{eq:Poincare energy}
 \bariiint_\LQI \bigg| v - \bariiint_\LQI v \bigg|
\lesssim \lambda \sum_{k \in \IZ} \frac{1}{1+|k|^{3/2}} \bariiint_{\Lambda \times Q \times I_k} |F|
\leq \lambda \sum_{m \geq 0} 2^{-m} \bariiint_{\Lambda \times Q \times 4^mI} |F|,
\end{align}
where the second step follows by a simple rearrangement of terms as in the proof of Lemma~\ref{lem:rh dyadic}. Let $\Max_x$ and $\Max_t$ denote again the maximal operators in the $x$ and $t$ variable, respectively. By H\"older's inequality,
\begin{align}
\label{eq1:NT energy}
\begin{split}
 \bariiint_\LQI \bigg| v - \bariiint_\LQI v \bigg|
&\lesssim \lambda \sum_{m \geq 0} 2^{-m} \bigg(\barint_\Lambda \bigg(\bariint_{ Q \times 4^mI} |F| \bigg)^2 \bigg)^{1/2} \\
& \lesssim \sqrt{\lambda} \bigg(\int_0^\infty (\Max_x \Max_t F_{\mu})(x,t)^2 \d \mu \bigg)^{1/2}
=: \sqrt{\lambda} V(x,t).
\end{split}
\end{align}
Since $F \in \L^2(\R^{n+2}_+)$ by assumption on $v$, we obtain $V \in \L^2(\ree)$ from the $\L^2$ boundedness of the maximal operators. In particular, we may concentrate on fixed $(x,t) \in \ree$ for which $V(x,t)$ is finite.

Let $h(\lambda) := \bariiint_\LQI v$, noting that for us $\LQI$ depends only on $\lambda$ as we have fixed $(x,t)$. From \eqref{eq1:NT energy} we derive that for every $\delta \in (0,1)$ there is a constant $C_\delta$ such that for every $\lambda > 0$ we have an estimate $|h(\lambda) - h(\delta \lambda)| \leq C_\delta \sqrt{\lambda}$. This implies that $\lim_{\lambda \to 0} h(\lambda) =: w_0(x,t)$ exists. Therefore we can use a telescopic sum of the estimates \eqref{eq1:NT energy} for Whitney regions of sidelength $2^{-k} \lambda$, $k \in \IN$, to obtain the stronger estimate
\begin{align}
\label{eq2:NT energy}
\bariiint_{\LQI} |v - w_0(x,t)| \d \mu \d y \d s \lesssim \sqrt{\lambda} V(x,t).
\end{align}
Moreover, given $(x,t), (y,s) \in \ree$, we can use \eqref{eq1:NT energy} on a Whitney region of length $\lambda  \sim |x-y| + |t-s|^{1/2}$ to get
\begin{align*}
 |w_0(x,t) - w_0(y,s)| \lesssim (|x-y| + |t-s|^{1/2})^{1/2}(V(x,t) + V(y,s)).
\end{align*}
In particular, this implies $w_0 \in \Lloc^2(\ree) \cap \mS'(\ree)$ and together with \eqref{eq2:NT energy} we finally obtain
\begin{align*}
 \bariiint_\LQI |v - w_0| \d \mu \d y \d s \lesssim \sqrt{\lambda}(V(x,t) + \Max_x \Max_t(V)(x,t)).
\end{align*}
Since $V \in \L^2(\ree)$, this implies $\NTone(\frac{v-w_0}{\sqrt{\lambda}}) \in \L^2(\ree)$.

In order to conclude, we have yet to check that $w_0 = v_0$ holds on $\ree$ in the sense of distributions modulo constants. To this end let $\phi \in \C_0^\infty(\R^{n+1})$. Using Fubini's theorem, we perform an averaging trick in $(x,t)$ and write
\begin{align*}
 \iint_\ree \bigg(\barint_{\lambda}^{2\lambda} v(\mu,y,s) \d \mu - w_0(y,s)\bigg) \phi(y,s) \d y \d s = \sqrt{\lambda} \iint_\ree \bigg(\bariiint_{W((x,t),\lambda)} \frac{v-w_0}{\sqrt{\lambda}} \phi \bigg) \d x \d t,
\end{align*}
where $W((x,t),\lambda):=(\lambda,2\lambda) \times B(x,\lambda) \times (t-\lambda^2, t+\lambda^2)$. The integral on the right is bounded uniformly in $\lambda$ since $\NTone(\frac{v-w_0}{\sqrt{\lambda}}) \in \L^2(\ree)$ and $\phi$ has compact support. Thus, $\barint_{\lambda}^{2\lambda} v \d \mu \to w_0$ in the sense of distributions as $\lambda \to 0$. On the other hand, due to Lemma~\ref{lem:trace energy} we have $v(\lambda,\cdot) \to v_0$ as $\lambda \to 0$ in $\Hdot^{1/4}_{\partial_t - \Delta_x}$ and hence in the sense of distributions modulo constants.
\end{proof}

We adapt this strategy to give the proofs of Lemmas~\ref{lem:Ken-Pip parabolic} and \ref{lem:energy solution compact right}.

\begin{proof}[Proof of Lemma~\ref{lem:Ken-Pip parabolic}]
Since notation is set up for the forward equation throughout the paper, we prefer to argue for a reinforced weak solution to $\Lop u = 0$ that satisfies $\NT(\pcg u) \in \L^2(\ree)$. This reduction is justified by the usual transposition of results between forward and backward equations, see also Section~\ref{sec:backward}. The proof follows that of Lemma~\ref{lem:energy Whitney trace} verbatim, once we have shown that there exists $U \in \L^2(\ree)$ such that
\begin{align*}
 \bariiint_{\LQI} \bigg|u- \bariiint_{\LQI}u\bigg| \lesssim \lambda U(x,t).
\end{align*}
Here, $\LQI$ is a fixed Whitney region of sidelength $\lambda$ centered at $(x,t)$. Let $F = \pcg u$. From \eqref{eq:Poincare energy} we know that we can take
\begin{align*}
 U(x,t):= \sup_{\Lambda' \times Q' \times I'} \sum_{m=0}^\infty 2^{-m} \bariiint_{\Lambda' \times Q' \times 4^m I'} |F|,
\end{align*}
with the supremum over all Whitney regions $\Lambda' \times Q' \times I'$ centered at $(x,t)$. However, this is \emph{not} the non-tangential maximal function that we control by assumption and it is for this reason that we have to assume that $u$ is a reinforced weak solution to the homogeneous equation. Indeed, $F$ is a solution to the first order system \eqref{eq:diffeq2}, see Theorem~\ref{thm:correspondence}, and as $\NT(F) \in \L^2(\ree)$ implies the Dini condition $\sup_{\lambda >0} \barint_{\lambda}^{2\lambda} \|F\|_{\L^2(\ree)} \d \mu < \infty$, Theorem~\ref{thm:uniq} provides a representation $F(\lambda,\cdot) = \e^{-\lambda [\P \M]} h$ for some $h \in \Hp(\P\M)$. Now, $\|U\|_2 \lesssim \|h\|_2<\infty$ follows from \eqref{eq:NTtranslatesBound}.
\end{proof}

\begin{proof}[Proof of Lemma~\ref{lem:energy solution compact right}]
For any $\psi \in \C_0^\infty(\R^{n+2})$ we have $\partial_\lambda \psi \in \dot{\E}(\R^{n+2})^*$. Thus, can define $\Lop^{* -1}(\partial_\lambda \psi)$ by hidden coercivity and there is an \emph{a priori} estimate
\begin{align*}
 \|\Lop^{* -1}(\partial_\lambda \psi)\|_{\dot{\E}} \lesssim \|\partial_\lambda \psi\|_{\dot{\E}^*} \leq \|\psi\|_{2}.
\end{align*}
In particular, we have $\partial_\lambda \Lop^{* -1}(\partial_\lambda \psi) \in \L^2(\R^{n+2})$. As $\Lop^*$ has $\lambda$-independent coefficients, we can use the method of difference quotients in $\dot{\E}(\R^{n+2})$ and Caccioppoli's estimate (Lemma~\ref{lem:Caccioppoli}) to deduce
\begin{align*}
 \partial_\lambda \Lop^{*-1}(\partial_\lambda \psi) = \Lop^{*-1}(\partial_\lambda^2 \psi) \in \dot{\E}(\R^{n+2}).
\end{align*}
Applying these observations with $\psi = \phi, \partial_\lambda \phi$ reveals $ v, \partial_\lambda v \in \dot{\E}(\R^{n+2}) \cap \L^2(\R^{n+2})$, that is, the four functions $v$, $\gradx v$, $\HT \dhalf v$, $\partial_\lambda v$ as well as their $\lambda$-derivatives are in $\L^2(\R^{n+2})$. Integration in $\lambda$ shows that they can be defined continuously with values in $\L^2(\ree)$. Thus, $v \in \C(\R;\L^2 \cap \, \Hdot^{1/2}_{\partial_t - \Delta_x})$ and $\partial_\lambda v \in \C(\R; \L^2)$. Now, we can iterate this argument taking $\psi = \partial^k_{\lambda} \phi$, $k\geq 2$, to obtain $\C^\infty(\R;\L^2 \cap \, \Hdot^{1/2}_{\partial_t - \Delta_x})$.

Let us turn to the non-tangential estimate. Due to Lemma~\ref{lem:energy Whitney trace} we only need to check it for Whitney regions $\LQI$ centered at $(x,t) \in \ree$ with sidelength $\lambda \leq 1$. The same lemma asserts that $v$ attains its trace $v_0$ in the sense of almost everywhere convergence of Whitney averages. Therefore we can apply \eqref{eq:Poincare energy} iteratively on the Whitney regions of sidelength $2^{-k}\lambda$, $k \in \IN$, centered at $(x,t)$ to find by a telescopic sum
\begin{align}
\label{eq:pure NT estimate energy}
 \bariiint_\LQI |v-v_0| \lesssim \lambda \sup_{k \in \IN} \bigg(\sum_{m \geq 0} 2^{-m} \bariint_{2^{-k} Q \times 4^{m-k} I} \barint_{2^{-k}\lambda}^{2^{-k+1}\lambda}  |\partial_\lambda v| + |\gradx v| + |\HT \dhalf v|\bigg).
\end{align}
In order to treat the integrals over $|\gradx v| + |\HT \dhalf v|$, we write
\begin{align*}
 v(\mu,y,s) = v_0(y,s) + \int_0^\mu \partial_\lambda v(\sigma,y,s) \d \sigma \qquad ((\mu,y,s) \in \R^{n+2}_+)
\end{align*}
to obtain, taking into account $2^{-k+1} \lambda \leq 2$,
\begin{align*}
  \bariint_{2^{-k} Q \times 4^{m-k} I} \barint_{2^{-k}\lambda}^{2^{-k+1}\lambda}  |\gradx v| + |\HT \dhalf v|
 &\lesssim \bariint_{2^{-k} Q \times 4^{m-k} I} \int_0^2 |\gradx \partial_\lambda v| + |\HT \dhalf \partial_\lambda v| \\
 &\quad + \bariint_{2^{-k} Q \times 4^{m-k} I} |\gradx v_0| + |\HT \dhalf v_0|.
\end{align*}
For the integrals over $\partial_\lambda v$ in \eqref{eq:pure NT estimate energy} a similar use of the fundamental theorem of calculus leads us to
\begin{align*}
 \bariint_{2^{-k} Q \times 4^{m-k} I} \barint_{2^{-k}\lambda}^{2^{-k+1}\lambda} |\partial_\lambda v|
&\lesssim \bariint_{2^{-k} Q \times 4^{m-k} I} \int_0^2 |\partial_\lambda^2 v| + \bariint_{2^{-k} Q \times 4^{m-k} I} |\partial_\lambda v|_{\lambda =0}|.
\end{align*}
Due to $v \in \C^\infty(\R; \L^2 \cap \, \Hdot^{1/2}_{\partial_t - \Delta_x})$ we have
\begin{align*}
 F_0:= |\gradx v_0| + |\HT \dhalf v_0| + |\partial_\lambda v|_{\lambda = 0}| \in \L^2(\ree)
\end{align*}
and
\begin{align*}
 F:= \int_0^2 |\gradx \partial_\lambda v| + |\HT \dhalf \partial_\lambda v| + |\partial_\lambda^2 v| \d \mu \in \L^2(\ree).
\end{align*}
The previous two estimates then show that the right-hand side of \eqref{eq:pure NT estimate energy} can be controlled by $\lambda \Max_x \Max_t(F + F_0)(x,t)$ and $\L^2$ boundedness of the maximal operators yields the claim.
\end{proof}

Next, we supply a general fact on weak solutions implying in particular that every weak solution to $(D)_2^\Lop$ is reinforced. This was announced already in Section~\ref{sec:UniqDir} and used in Lemma~\ref{lem:theta u testfunction}.

\begin{lem}
\label{lem:NT implies reinforced}
Let $\phi \in \C_0^\infty(\R^{n+2}_+)$ and let $u$ be a weak solution to $\Lop u = \phi$ on $\R^{n+2}_+$ that satisfies $\NTone(u/\omega) \in \L^2(\ree)$ for some locally bounded, measurable $\omega: (0,\infty) \to (0,\infty)$ depending only on $\lambda$. Then for any $\theta \in \C_0^\infty(0,\infty)$, depending only on $\lambda$, it holds $\theta u \in \L^2(\R^{n+2}_+) \cap \dot{\E}(\R^{n+2}_+)$.
\end{lem}

\begin{proof}
The statement we have to prove is purely qualitative and so we shall not care about dependence of implicit constants. First, we pick $a>0$ and $c_0<1$ such that $\supp(\theta) \subset (c_0^{-1}a, c_0 a)$. It will turn out convenient to use $W((x,t),\lambda) := ((2c_0)^{-1}\lambda, 2 c_0 \lambda) \times B(x,2\lambda) \times (t-4\lambda^2, t+4\lambda^2)$ as Whitney cylinder around $(x,t)$ to define $\NTone$. We also introduce $\frac{1}{2}W((x,t),\lambda) := ((c_0)^{-1}\lambda, c_0 \lambda) \times B(x,\lambda) \times (t-\lambda^2, t+\lambda^2)$. By an averaging trick in $(x,t)$, we find
\begin{align}
\label{eq:averaging trick in xt only}
\begin{split}
 \iiint_{\R^{n+2}_+} |u \theta|^2 \d \lambda \d x \d t
&\lesssim \int_{c_0^{-1}a}^{c_0a} \iint_{\ree} \bigg(\bariint_{B(x,a) \times (t-a^2, t+a^2)} |u|^2 \d y \d s\bigg) \d x \d t \d \lambda \\
&\lesssim \iint_\ree \bigg(\bariiint_{\frac{1}{2} W((x,t),a)} |u|^2 \d \lambda \d y \d s \bigg) \d x \d t.
\end{split}
\end{align}
Since $\phi$ is compactly supported, there is a parabolic cube $Q \times I \subseteq \ree$ such that $\Lop u = 0$ in a neighbourhood of $W((x,t),a)$ for every $(x,t) \in {}^c (Q \times I)$. This makes the reverse H\"older estimate for weak solutions (Lemma~\ref{lem:RHu}) applicable and we deduce
\begin{align}
\label{eq1:NT implies reinforced}
 \bariiint_{\frac{1}{2}W((x,t),a)} |u|^2 \lesssim \bigg(\bariiint_{W((x,t),a)} |u| \bigg)^2 \lesssim \NTone(u/\omega)(x,t)^2 \qquad ((x,t) \in {}^c (Q \times I)).
\end{align}
Note that we have used the local boundedness of $\omega$ in the second step. Thus,
\begin{align*}
 \iiint_{\R^{n+2}_+} |u \theta|^2 \lesssim  \iint_{Q \times I} \bigg(\bariiint_{\frac{1}{2}W((x,t),a)} |u|^2 \bigg) \d x \d t +
 \iint_{{}^c (Q \times I)} \NTone(u/\omega)^2(x,t) \d x \d t.
\end{align*}
Here the second integral is finite by assumption on $u$ and so is the first one since $u$ as a weak solution is locally square-integrable. So far, this proves $\theta u \in \L^2(\R^{n+2}_+)$. Next, we claim $\gradlamx(\theta u) \in \L^2(\R^{n+2}_+; \IC^{n+1})$. Indeed,
\begin{align*}
 \gradlamx(\theta u) = u \gradlamx \theta + \theta \gradlamx u,
\end{align*}
where the first term can be treated just as $\theta u$ above and for the second one the only change in the argument is to apply Caccioppoli's estimate (Lemma~\ref{lem:Caccioppoli}) to eliminate the gradient of $u$ before using a reverse H\"older estimate as in \eqref{eq1:NT implies reinforced}. Hence, $\theta u \in \L^2(\R; \W^{1,2}(\reu))$.

At this point, we have all the ingredients to follow the interpolation argument given in the proof of Lemma~\ref{lem:rh time local} to first show $\partial_t(\theta u) \in \L^2(\R; \W^{-1,2}(\reu))$ and then deduce the missing piece of information $\HT \dhalf(\theta u) \in \L^2(\R^{n+2}_+)$. The fact that here $u$ solves the inhomogeneous equation $\Lop u = \phi$ does not pose any difficulties in this argument since we have $\theta \phi \in \L^2(\R^{n+2}_+)$.
\end{proof}

As a direct application we obtain the

\begin{proof}[Proofs of Lemma~\ref{lem:theta u testfunction} and Lemma~\ref{lem:theta w testfunction}]
Let us start with $\theta v$ and let $v_0 := v|_{\lambda = 0}$. On noting that for almost every $(\lambda,x,t) \in \R^{n+2}_+$,
\begin{align*}
 \bigg|\frac{v(\lambda,x,t)}{\max\{1, \lambda \}}\bigg| \leq \bigg|\frac{v(\lambda,x,t) - v_0(x,t)}{\lambda}\bigg| + |v_0(x,t)|,
\end{align*}
we deduce the pointwise bound
\begin{align*}
 \NTone\Big(\frac{v}{\max\{1, \lambda \}}\Big) \leq \NTone\Big(\frac{v -v_0}{\lambda}\Big) + \Max_x \Max_t (v_0).
\end{align*}
By Lemma~\ref{lem:energy solution compact right} and $\L^2$ boundedness of the maximal operators the right-hand side is in $\L^2(\ree)$. Thus, Lemma~\ref{lem:NT implies reinforced} yields $\theta v \in \L^2(\R^{n+2}_+) \cap \dot{\E}(\R^{n+2}_+)$. In this \emph{inhomogeneous} energy space we can approximate $\theta v$ by the usual cut-off and mollification procedure. Since $\theta v$ is supported away from the boundary $\lambda = 0$, these approximants are in $\C_0^\infty(\R^{n+2}_+)$ as required.

Referring to Lemma~\ref{lem:Ken-Pip parabolic} instead, the same argument applies to $\theta w$ and for $\theta u$ we may simply use the assumption $\NT(u) \in \L^2(\ree)$.
\end{proof}

We complete the treatment of postponed auxiliary results with the

\begin{proof}[Proof of Lemma~\ref{lem:big averages}]
As usual, let $\LQI$ be a Whitney region centered around $(x,t)$ with sidelength $\lambda$. By a change of Whitney parameters we may assume $\NTone$ is defined with respect to the larger regions $\Lambda \times 2Q \times 4I$. With this convention, we have for every $(z,r) \in Q \times I$,
\begin{align*}
 \bariiint_\LQI |u(y,s)| \d y \d s \leq 2^{n+2} \NTone(u)(z,r).
\end{align*}
Hence,
\begin{align*}
 \bariiint_\LQI |u(y,s)| \d y \d s \leq 2^{n+2} \bariint_{Q \times I} \NTone(u)(z,r) \d z \d r \lesssim \lambda^{-1-n/2} \|\NTone(u)\|_{\L^2(\ree)}
\end{align*}
and the conclusion follows.
\end{proof}
\subsection{Proof of Theorem~\ref{thm:WPDirTimeDependent}}

Let $f \in \L^2(\ree) \subset \L^2(\R^{n+1})+ \Hdot^{1/4}_{{\pd_{t}-\Delta_{x}}}$ and let $u$ be the corresponding solution to $\Lop u = 0$ provided by part (ii) of Proposition~\ref{prop:Dir}, written as the sum of two solutions $u_1+u_2$ to the same equation with $\pcg u_1 \in \mE_{-1}$ and $\pcg u_2 \in \mE_{-1/2}$. Thanks to Theorem~\ref{thm:sob} there are traces
\begin{align*}
 u_{1,0}:=u_1|_{\lambda = 0} \in \L^2(\ree), \qquad u_{2,0}:=u_2|_{\lambda = 0} \in \Hdot^{1/4}_{\partial_t - \Delta_x},
\end{align*}
which by construction satisfy $u_{1,0} + u_{2,0} = f$ in the sense of distributions modulo constants. In particular, we may modify $u$ by an additive constant to obtain also $u_{2,0} = f - u_{1,0} \in \L^2(\ree)$.

Next, Theorem~\ref{thm:chardir} provides a representation $u_1(\lambda,\cdot) = c - (\e^{-\lambda[\M \P]}\wt{h})_\pe$ for some unique $c \in \IC$ and $\wt{h} \in \Hp(\M\P)$. Since $u_{1,0} \in \L^2(\ree)$, we must have $c=0$, in which case Theorem~\ref{thm:NTmaxDir} yields $\NT(u_1) \in \L^2(\ree)$ and a.e.\ convergence of Whitney averages of $u_1$ towards $u_{1,0}$ as $\lambda \to 0$. As for $u_2$, we recall that $\pcg u_2 \in \mE_{-1/2}$ precisely means $u_2 \in \dot{\E}(\R^{n+2}_+)$. Thus, Lemma~\ref{lem:energy Whitney trace} yields $\NTone(\frac{u_2 - u_{2,0}}{\sqrt{\lambda}}) \in \L^2(\ree)$ and a.e.\ convergence of Whitney averages of $u_2$ to $u_{2,0}$ as $\lambda \to 0$. Adding up, we obtain Whitney average convergence of $u$ towards the boundary data $f$. Moreover, on splitting for almost every $(\lambda,x,t) \in \R^{n+2}_+$,
\begin{align*}
 \bigg|\frac{u(\lambda,x,t)}{\max\{1, \sqrt{\lambda} \}}\bigg|
 \leq |u_1(\lambda,x,t)| + \bigg|\frac{u_2(\lambda,x,t) - u_{2,0}(x,t)}{\sqrt{\lambda}}\bigg| + |u_{2,0}(x,t)|,
\end{align*}
we get a pointwise bound
\begin{align*}
 \NTone\bigg(\frac{u}{\max\{1,\sqrt{\lambda}\}}\bigg)
\leq \NTone(u_1) + \NTone\bigg(\frac{u_2 - u_{2,0}}{\sqrt{\lambda}}\bigg) + \Max_x\Max_t(u_{2,0})
\end{align*}
and we have seen that all three maximal functions are in $\L^2(\ree)$. The reverse H\"older inequality in Lemma~\ref{lem:RHu} finally yields $\NT(\frac{u}{\max\{1,\sqrt{\lambda}\}}) \in \L^2(\ree)$ as required.

It remains to prove uniqueness. Let $u$ be a weak solution to $\Lop u = 0$ with $\NT(\frac{u}{\max\{1,\sqrt{\lambda}\}}) \in \L^2(\ree)$ such that for almost every $(x,t) \in \ree$,
\begin{align*}
 \lim_{\lambda \to 0} \bariiint_\LQI |u| = 0.
\end{align*}
In proving $u=0$, we can follow the proof of Theorem~\ref{thm:uniqueDir}. Indeed, given $\phi \in \C_0^\infty(\R^{n+2}_+)$, we can define $v := \Lop^{*-1}(\partial_\lambda^2\phi)$ as before. Since the regularity assumption on $A$ is preserved under a reversal of time and taking adjoints, we also have at hand Proposition~\ref{prop:Dir} for the adjoint equation. In particular, since  Lemmas~\ref{lem:trace energy} and \ref{lem:energy solution compact right} pay for $v|_{\lambda = 0} \in \Hdot^{1/2}_{\partial_t - \Delta_x} \cap \Hdot^{1/4}_{\partial_t - \Delta_x}$, we can find a reinforced weak solution to $\Lop^* w = 0$ with trace $w|_{\lambda = 0} = v|_{\lambda = 0}$ and the same estimate $\pcgb u \in \mE_0$ as before. As an additional feature, $\pcgb u \in \mE_{-1/2}$, that is, $w \in \dot{\E}(\R^{n+2}_+)$.

Now that the same setup is settled, there are only two instances in the proof of Theorem~\ref{thm:uniqueDir} at which the precise form of the non-tangential control for $u$ matters, namely \eqref{eq3:uniqueDir} and the limiting argument relying on \eqref{eq8:uniqDir} and \eqref{eq8b:uniqDir}. For the former, having $\L^2$-control only on $\NT(u/\sqrt{\lambda})$ is sufficient. For the latter, we need a different argument only when $\sigma = R$ is large, in which case we simply distribute the factor $1/\mu$ to obtain
\begin{align*}
  \text{I}_\sigma + \text{II}_\sigma
\lesssim \iint_\ree \bigg(\bariiint_{W((x,t),\sigma)} \Big|\frac{u}{\sqrt{\mu}}\Big| \d \mu \d y \d s \bigg) \bigg(\bariiint_{W((x,t),\sigma)} \Big|\frac{v-w}{\sqrt{\mu}}\Big| \d \mu \d y \d s\bigg) \d x \d t.
\end{align*}
Here, the averages on $u/\sqrt{\lambda}$ are under control by assumption and for those of $(v-w)/\sqrt{\lambda}$ we can use the additional information on $w$: As $v-w \in \dot{\E}(\R^{n+2}_+)$ has trace $0$, the required $\L^2$-control is due to Lemma~\ref{lem:energy Whitney trace}. This being said, we can complete the proof as in the case of Theorem~\ref{thm:uniqueDir}.
\section{Miscellaneous generalisations and open problems}
\label{sec:miscellani}

\subsection{Systems}
\label{sec:systems}

All of our results apply to parabolic systems without any changes but in the ellipticity condition. More precisely, we can deal with systems of $m$ equations given by
\begin{equation*}
\label{eq:divform}
 \pd_{t}u^\alpha - \sum_{i,j=0}^n\sum_{\beta= 1}^m \pd_i\big( A_{i,j}^{\alpha, \beta}(\lambda,x,t) \pd_j u^{\beta}\big) =0,\qquad \alpha=1,\ldots, m,
\end{equation*}
in $\R^{n+2}_+$, where $\pd_0= \pd_{\lambda}$ and $\pd_i= \pd_{x_{i}}$ if $i=1,\ldots,n$, and where we impose the following assumptions on the matrix
$ A(\lambda,x,t)=(A_{i,j}^{\alpha,\beta}(\lambda,x,t))_{i,j=0,\ldots, n}^{\alpha,\beta= 1,\ldots,m}$. First,
\begin{equation*}
\label{eq:boundedmatrix}
 A(\lambda,x,t)=(A_{i,j}^{\alpha,\beta}(\lambda,x,t))_{i,j=0,\ldots, n}^{\alpha,\beta= 1,\ldots,m}\in \L^\infty(\R^{n+2};\Lop(\IC^{m(1+n)}))
\end{equation*}
is bounded and measurable. Second, $A$ satisfies, for some $\kappa>0$ and uniformly for all $\lambda$, the ellipticity condition
\begin{equation*}
\label{eq:accrassumption}
 \Re \iint_{\R^{n+1}} (A(\lambda,x,t)f(x,t)\cdot \overline{f(x,t)}) \d x\d t\ge \kappa
 \sum_{i=0}^n\sum_{\alpha=1}^m \iint_{\R^{n+1}} |f_i^\alpha(x,t)|^2\d x \d t,
\end{equation*}
for all $f=(f_{j}^\alpha)_{j=0, \ldots, n}^{\alpha=1, \ldots, m} \in \L^2(\R^{n+1}; \IC^{m(1+n)})$ with $\curlx (f_{j}^\alpha)_{j=1,\ldots,n}=0$ for all $\alpha$.

\subsection{Adding lower order terms to the elliptic part}

There is no real difficulty to adapt our techniques and proofs to the case of (strictly elliptic) equations with lower order terms: \begin{align*}
\Lop u= \partial_t u -\div_{X} (A(X,t)\nabla_{X} u+ b u) + c\nabla_{X}u +du= 0
\end{align*}
where $b,c$ are $n+1$ vectors and $d$ is a scalar, all coefficients are bounded and measurable, and we assume pointwise ellipticity (in the case of one equation). In this case the associated parabolic conormal differential becomes
\begin{equation*}
\label{eq:conormallower}
 \pcg u(\lambda, x,t) = \begin{bmatrix} \dnuA u(\lambda, x,t) + b_{\pe} u(\lambda,x,t) \\ \gradx u(\lambda, x,t) \\ u(\lambda,x,t) \\ \HT\dhalf u(\lambda, x,t) \end{bmatrix}
\end{equation*}
and the corresponding parabolic Dirac operator is given by
\begin{equation*}
\label{eq:newPM}
\P= {\begin{bmatrix} 0\vphantom{\wt{b}_{\pe}} & \divx & \id& -\dhalf \\ -\gradx\vphantom{\wt{b}_{\pe}} & 0 & 0&0 \\ \id\vphantom{\wt{b}_{\pe}}&0&0&0\\ -\HT \dhalf & 0 & 0 &0\end{bmatrix}}, \quad \M= {\begin{bmatrix} B_{\pe \pe} & B_{\pe \pa} & \wt{b}_{\pe} &\vphantom{\dhalf}0 \\ B_{\pa \pe} & B_{\pa \pa} & \wt{b}_{\pa}& 0 \\
\wt{c}_{\pe}& \wt{c}_{\pa}& \wt{d}&0
\\
 \vphantom{\dhalf}0& 0& 0& 1 \end{bmatrix}},
\end{equation*}
where
\begin{align*}
 {\begin{bmatrix} B_{\pe \pe} & B_{\pe \pa} & \wt{b}_{\pe} \\ B_{\pa \pe} & B_{\pa \pa} & \wt{b}_{\pa} \\
\wt{c}_{\pe}&\wt{c}_{\pa}& \wt{d}
 \end{bmatrix}} = \begin{bmatrix}1 &0& 0\vphantom{\wt{b}_{\pe}} \\ A_{\pa\pe}\vphantom{\wt{b}_{\pe}} & A_{\pa\pa} & b_{\pa}
 \\ c_{\pe} & c_{\pa}& d\vphantom{\wt{b}_{\pe}}
 \end{bmatrix}
 \begin{bmatrix} A_{\pe\pe}\vphantom{\wt{b}_{\pe}} & A_{\pe\pa}& b_{\pe}\\ 0\vphantom{\wt{b}_{\pe}}&1&0
 \\ 0&0&1\vphantom{\wt{b}_{\pe}}
 \end{bmatrix}^{-1}.
\end{align*}

The upshot is the proof of the square function implied in Theorem~\ref{thm:bhfc}. One can either adapt the proof given for the pure second order case or prove a lower order perturbation result. We do not get into details.

\subsection{The setup of cylindrical domains.}

We may replace $(X,t)\in \R^{n+2}_{+}$ by $(X,t) \in B\times \R$, where $B$ is the unit ball of $\R^{n+1}$. In that case, the normal independence of the coefficients means that they do not depend on the radial variable as in \cite{Ken-Pip}. Adapting \cite{AR} to the parabolic setting, one can develop the first order approach. Again the proof of the square function in Theorem~\ref{thm:bhfc} is the main part and should be handled by mixing arguments from here with those of \cite{AR}.

We remark that if $t$ belongs to the unbounded line, then this makes it impossible to treat the Rellich estimates in the self-adjoint case with $t$-dependent coefficients via Fredholm methods even though the angular variable lives on the sphere. However, one should be able to treat boundary value problems with initial condition on bounded intervals $[0,T]$ imposing that the commutators between coefficients and half-order time derivatives are compact, uniformly in the angular variable. We leave this idea and further details to interested readers.

\subsection{Degenerate parabolic equations and systems}

In the setting of the parabolic upper half-space, it may also be interesting to consider the case of degenerate equations and systems. To expand a bit on this in the case of equations, one may consider equations
 \begin{eqnarray}\label{eq1deg}
\lambda_1(X,t)\partial_t u -\div_{X} A(X,t)\nabla_{X} u = 0,
 \end{eqnarray}
 with
 \begin{equation*}
 \kappa\lambda_2(X,t)|\xi|^2\leq \Re (A(X,t)\xi \cdot \cl{\xi}), \qquad
 |A(X,t)\xi\cdot\zeta|\leq C\lambda_2(X,t) |\xi||\zeta|,
\end{equation*}
where the weights $\lambda_1$ and $\lambda_2$ are non-negative measurable functions subject to conditions.

In the case when  $\lambda_1(X,t)=w(x)=\lambda_2(X,t)$ for some $w\in A_{2}(\R^{n})$, one may try to develop a weighted version of our paper by following \cite{ARR}. Again, the matter is to prove the weighted analogue to Theorem~\ref{thm:bhfc} for $\lambda$-independent coefficients. We claim that this is possible. It is an interesting open problem to construct test functions in the case when the weight also depends on the $t$ variable and is, for instance, an $A_{2}$ weight for the parabolic doubling space $\R^{n+1}$.

Although not directly related, it is here worth noting that in the case of real parabolic equations, interior regularity of weak solutions to \eqref{eq1deg}, when
$\lambda_1(X,t)=w(X,t)=\lambda_2(X,t)$, or $\lambda_1(X,t) = 1$ and $\lambda_2(X,t)=w(X,t)$, was studied in \cite{CS,CS3} under various  integrability conditions on the weight. We remark that the conclusions there in terms of Muckenhoupt conditions are not exactly as for the elliptic case, showing that the $A_{2}$ condition is not the optimal one to obtain regularity.

\subsection{Boundary value problems with other spaces of data}

Another possible development is the formalisation of the first order setup to approach boundary value  problems with data in other spaces than the $\L^2$ based Sobolev spaces.

Starting with $\L^p$ spaces, it amounts to first study \emph{a priori} estimates enjoyed by the Cauchy extension and next to identify trace spaces of solutions having some square function control and prove that their conormal differentials are given by a Cauchy extension.

As a consequence of the results in Section~\ref{sec:resolvent}, in particular, Lemma~\ref{lem:LpLq} and Proposition~\ref{prop:OffDiag}, one can extend the bounded holomorphic functional calculus of $\P\M$ to $\L^p$ for $p$ near $2$ and this implies (after some work) the alluded square function estimates. However, if one wants to develop results covering larger ranges of $p$, one should  look at  Hardy space theory  and uniqueness results  as studied in \cite{Auscher-Stahlhut_APriori, Auscher-Mourgoglou} for the elliptic situation. This theory has been extended in \cite{Amenta-Auscher} to cover fractional Besov-Hardy-Sobolev spaces. While the strategy is in place, the details could reveal more complicated again due to the presence of half-order derivatives in time.

\subsection{Higher order problems and more}

One can also study other types of problems either for the purpose of boundary value problems with $\lambda$ independent coefficients or to prove the Kato square root estimate.

We have in mind  parabolic systems with higher order  elliptic part in divergence form. The  first order method for the elliptic case was discussed in \cite{elAAM}.

One could also think of replacing $\pd_{t}$ by other operators. If we replace it by $\pd_{t}^2$, we come back to an elliptic problem. But we could also think of $\pd_{t}^3 $ or fractional derivatives. In principle, the algebraic formalism can be set but the conditions to proceed with the analysis are unclear.

Finally, one could try to treat  elliptic parts of some fractional order but in this case the scary thing is the lack of spatial off-diagonal estimates.
\def\cprime{$'$} \def\cprime{$'$} \def\cprime{$'$}

\end{document}